\documentclass[a4paper,12pt,oneside,leqno]{amsbook}

\usepackage[a4paper,divide={2.5cm,*,2.5cm},nofoot,dvips]{geometry}
\usepackage{amsfonts}
\usepackage{amssymb}
\usepackage{amscd}
\usepackage{stmaryrd}
\usepackage{floatflt}

\usepackage[dvips]{graphicx}
\usepackage{pstricks}
\usepackage{psfrag}

\input{diagxy.tex}
\usepackage{indentfirst}
\usepackage{fancyhdr}
\usepackage{enumitem}

\newcommand{\plainPagin}{
\renewcommand{\headrulewidth}{0.4pt}
\renewcommand{\footrulewidth}{0pt}
\setlength\footskip{0pt}
\setlength\headheight{14pt}
\fancyhf{}
\chead{\fancyplain{}{}}
\rhead{\fancyplain{}{\small\thepage}}
}

\newcommand{\titlePagin}{
\setlength\footskip{16pt}
\setlength\headheight{0pt}
\renewcommand{\headrulewidth}{0pt}
\renewcommand{\footrulewidth}{0pt}
\cfoot{\large{Krakow 2008}}
}

\newcommand{\dedication}[2]{
\thispagestyle{empty}
\ 
\vskip 15cm
\begin{flushright}
\begin{minipage}{#1}
\begin{center}\textit{#2}\end{center}
\end{minipage}
\end{flushright}
\newpage
}

\newcommand{\term}[1]{\emph{#1}}
\newcommand{\refTitle}[1]{\emph{#1}}

\newtheoremstyle{myTheorem}{2ex}{2ex}{\itshape}{0pt}{\bfseries}{.}{ }%
		{\thmname{#1} \thmnumber{#2}\thmnote{\normalfont\ (#3)}}
\theoremstyle{myTheorem}

\newtheorem{theorem}{Theorem}[section]
\newtheorem{lemma}[theorem]{Lemma}
\newtheorem{corollary}[theorem]{Corollary}
\newtheorem{question}[theorem]{Question}
\newtheorem{conjecture}[theorem]{Conjecture}
\newtheorem{proposition}[theorem]{Proposition}

\newtheorem{problem}[theorem]{Problem}

\newtheoremstyle{myDefinition}{2ex}{2ex}{\normalfont}{0pt}{\bfseries}{.}{ }%
		{\thmname{#1} \thmnumber{#2}\thmnote{\normalfont\ (#3)}}
\theoremstyle{myDefinition}
\newtheorem{definition}[theorem]{Definition}
\newtheorem{remark}[theorem]{Remark}
\newtheorem{example}[theorem]{Example}

\numberwithin{equation}{section}

\newcommand{\hide}[1]{}		

\newcommand{\image}[2]{\raisebox{-#2}{\includegraphics{images/#1.eps}}}
\newcommand{\scimage}[3]{\raisebox{-#3}{\includegraphics[height=#2,keepaspectratio=1]{images/#1.eps}}}
\newcommand{\puthimage}[2]%
{\psfragscanon\includegraphics[width=#2,keepaspectratio=1]{images/#1.eps}\psfragscanoff}
\newcommand{\putvimage}[2]%
{\psfragscanon\includegraphics[height=#2,keepaspectratio=1]{images/#1.eps}\psfragscanoff}
\newcommand{\putimage}[1]{\psfragscanon\includegraphics{images/#1.eps}\psfragscanoff}

\newcommand{\fnt}[1]{\raisebox{-2pt}{\includegraphics[height=10pt,keepaspectratio=1]{images/fnt-#1.eps}}}
\newcommand{\tfnt}[1]{\raisebox{-1pt}{\includegraphics[height=10pt,keepaspectratio=1]{images/fnt-#1.eps}}}
\newcommand{\sfnt}[1]{\raisebox{0pt}{\includegraphics[height=6pt,keepaspectratio=1]{images/fnt-#1.eps}}}
\newcommand{\khfnt}[1]{\raisebox{-3pt}{\includegraphics[height=12pt,keepaspectratio=1]{images/fnt-Kh-#1.eps}}}

\newcommand{\PKauff}[1]{\langle#1\rangle}		
\DeclareMathOperator{\KhCube}{\mathcal{I}}	
\DeclareMathOperator{\KhCom}{\mathrm{Kh}}		
\newcommand{\KhBra}[1]{\llbracket#1\rrbracket}  
\DeclareMathOperator{\lk}{\mathrm{lk}}			

\newcommand{\kCr}{\mathrm{Cr}}		

\newcommand{\sgn}{\mathrm{sgn}}  

\newcommand{\cat}[1]{\textrm{\normalfont\bfseries #1}}		
\newcommand{\catP}[1]{\mathbb{P}\textrm{\normalfont\bfseries #1}}	
\newcommand{\catL}[1]{\cat{#1}_{/l}}								
\newcommand{\catH}[1]{\cat{#1}_{/h}}								
\DeclareMathOperator{\cone}{\mathrm{cone}}						

\DeclareMathOperator{\id}{\mathrm{id}}			
\DeclareMathOperator{\im}{\mathrm{im}}			
\DeclareMathOperator{\Ob}{\mathrm{Ob}}			
\DeclareMathOperator{\Mor}{\mathrm{Mor}}		
\DeclareMathOperator{\Mat}{\cat{Mat}}			
\DeclareMathOperator{\Kom}{\cat{Kom}}			
\DeclareMathOperator{\Cub}{\cat{Cub}}			

\newcommand{\cob}[3]{#1\colon #2\Rightarrow #3}				
\newcommand{\chcob}[4]{(#1,#2)\colon #3\Rightarrow #4}	
\newcommand{\chch}[3]{#1\colon #2\rightsquigarrow #3}		
\DeclareMathOperator{\tcob}{\sigma}						
\DeclareMathOperator{\Chron}{\mathrm{Chron}}			
\DeclareMathOperator{\crit}{\mathrm{Crit}}			
\DeclareMathOperator{\Diff}{\mathrm{Diff}}			

\DeclareMathOperator{\S1}{\mathbb{S}^1}		
\DeclareMathOperator{\D2}{\mathbb{D}^2}		
\DeclareMathOperator{\CPO}{\mathcal{CPO}}		
\DeclareMathOperator{\F}{\mathcal{F}}			
\newcommand{\Fun}[1]{\mathcal{#1}}				

\usepackage[pdfauthor={Krzysztof Putyra},pdftitle={Cobordisms with chronologies and a generalisation of the Khovanov complex}]{hyperref}
\usepackage[all]{hypcap}

\author{Krzysztof Putyra}
\title{Cobordisms with chronologies and a~generalisation of~the~Khovanov complex}

\begin{document}

\titlePagin

\begin{center}

\scimage{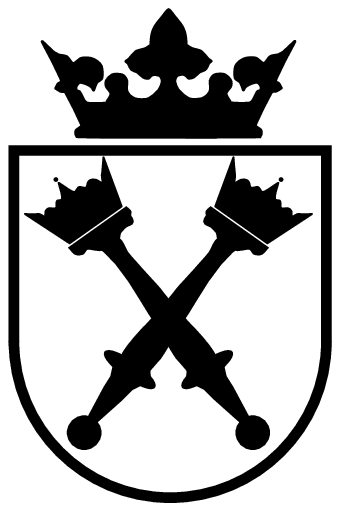}{60pt}{0pt}

\begin{Large}
Jagiellonian University

\vskip 2pt
Faculty of Mathematics%
\end{Large}

\vskip 5cm
\begin{Huge}
\bfseries Cobordisms with chronologies\\ and a~generalisation of\\ the~Khovanov complex

\end{Huge}

\begin{Large}
\vskip 1cm
Krzysztof Putyra

\vskip 2cm
Master's thesis supervised by

\vskip 2pt
dr hab. Klaudiusz W\'ojcik
\end{Large}

\end{center}

\clearpage
\plainPagin

\dedication{18em}{To my supervisor,\\ dr hab. Klaudiusz W\'ojcik,\\ for supervising my program of studies as~well~as for his support and favour during writing this thesis}
\dedication{20em}{To prof. Dror Bar-Natan\\ for his inestimable help and support during my stay at the University of Toronto, which~bore fruits with these results}
\dedication{18em}{To my Parents\\ for making it possible to me to take studies of my dreams and devoting uncountable many moments to my education and breeding}

\chead{\fancyplain{}{Preface}}
\cleardoublepage
\phantomsection
\chapter*{Preface}
The classical theory of knots deals with embeddings of circles into the~Euclidean space~$\mathbb{R}^3$.
Allowing more circles as well as closed intervals we obtain links and tangles respectively.
The~internal structure of these objects depends only on the~amount and type of components.
Therefore, the~embedding is the~main object of research.
It was G.~W.~Leibniz, who noticed in 1679 importance of embeddings,
introducing the~term \term{geometria situs}. First remarks on knots come from 
A.~T.~Vandermonde, whereas C.~F.~Gauss defined first invariants (see also~\cite{PrzytHist}).

The~end of 20th century is the~time of a~huge development of the~theory due to papers
of J.~H.~Conway and~V.~F.~R.~Jones~\cite{Jones-Poly}, which introduced very strong polynomial
invariants. Although the~result of Conway is mostly a~reformulation of a~definition of the~well-known
before Alexander polynomial, the~Jones polynomial $V_L(t)$ was a~new invariant, given by three simple conditions:
\begin{enumerate}[label=(\arabic*)]
\item $V_L(t)$ is a~link invariant
\item $V_{\S1}(t) = 1$,
\item $t^{-1}V_{\sfnt{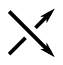}}(t) - tV_{\sfnt{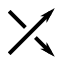}}(t) = (t^{1/2}-t^{-1/2})V_{\sfnt{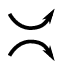}}(t)$.
\end{enumerate}
Several months later, L.~Kauffman constructed in~\cite{KauffmanBracket} the~Jones polynomial by a~bracket
$\PKauff{D}$, which can be~computed for non-oriented diagrams as a~state-sum of polynomials defined
for smoothed diagrams obtained from $D$ (see fig.~\ref{fig:tref-ex-sm}). On one hand such a~definition of the~Jones
polynomial resulted soon in proofs of almost hundred-years old Tait's conjectures, on the~other it gave
impetus to the~search for invariant defined in a~similar way in quantum algebras.

\begin{figure}
	\begin{center}
		$\image{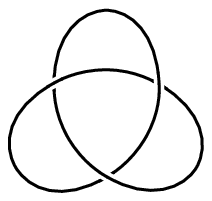}{26pt}\quad\longrightarrow\quad\image{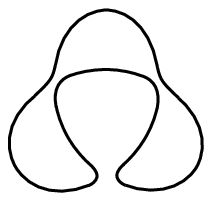}{26pt}$
	\end{center}
	\caption{A~diagram of a~trefoil and some of its smoothed versions.}\label{fig:tref-ex-sm}
\end{figure}

The~next step in understanding the~Jones polynomial is the~paper of M.~Khovanov~\cite{Khov-Jones}.
It contains a~description of a~graded homology groups $\mathcal{H}^*_{Kh}(L)$, with its Euler characteristic
being the~polynomial. In fact, it started the~search for other homology theories, which Euler characteristics
are other link invariants (such theories are called categorifications of these invariants).
There are several reasons, why this hunting is worth to spend time on it:
homology groups are often stronger invariants than their Euler characteristics,
properties of those invariants may have a~simpler explanation in terms of homology groups
and finally they can be naturally extended over a~larger class of objects, like tangles.
One of the~most recent results is the~paper of P.~Osv\'ath, J.~Rasmussen and~Z.~Szab\'o~\cite{Osv-Ras-Sz},
where they constructed homology groups $\mathcal{H}^*_{ORS}(L)$, called odd Khovanov homology.
They also categorify the~Jones polynomial, but are far different from the~one known before.
In particular, there exist links having isomorphic homology groups of one type but not isomorphic
of the~other type.

The~Khovanov's construction starts with a~commutative cube in the~category of cobordisms,
which is sent by a~functor to a~category of modules. In this category we can build a~complex
from this cube and compute its homology groups, which appear to be link invariants.
A~big step in understanding this construction is the~paper of D.~Bar-Natan~\cite{BarNatan-tangl},
where he constructed the~complex and proved its invariance in the~category of cobordisms.\footnote{\ 
To be more precise, in the~additive closure of cobordisms, i.e. the~category extended by formal
direct sums and formal sums of cobordisms
(see the~definition~\hyperref[def:cat-add]{\ref*{chpt:alg-hom}.\ref*{def:cat-add}})}
This approach gives a~natural extension over tangles as well as cobordisms between tangles
embedded in~$\mathbb{R}^3\times I$. Moreover, any functor $\F\colon\cat{2Cob}\to\cat{Mod}$
from the~category of cobordisms into the~category of modules, satisfying some additional conditions,
induces an~invariant complex and homology groups.

The~approach described above does not work for odd homology groups, because the~procedure
$\F_{\!ORS}\colon\cat{2Cob}\to\cat{Mod}$ is not a~functor. Indeed, it is defined only
up to a~sign. This raises a~question, if cobordisms can be enriched with some additional structure
so that the~construction of odd homology groups gets a~functorial description.
It would be great, if we could rewrite in this category the~construction of Bar-Natan.
The~answer is positive: it suffices to equip cobordisms with projections onto the~unit interval
$\tau\colon M\to I$ that have only non-degenerated critical points, all on different levels,
and define orientations of these critical points (visualised with arrows in~\cite{Osv-Ras-Sz}).
One of the~most important properties of this new category is breaking symmetries. For example, contrary
to the~usual cobordisms, there is no associativity law:

\medskip
\begin{center}
\putimage{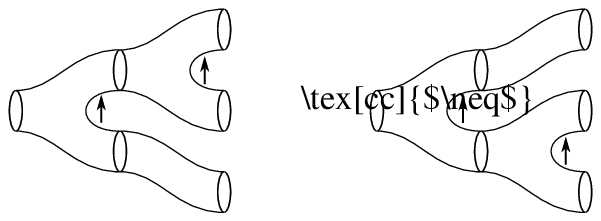}
\end{center}

A~motivation for cobordisms is a~topological quantum field theory (\term{TQFT}).
Cobordisms stand for space-times, whereas manifolds are spaces.
The~projection $\tau\colon M\to I$ can be seen as a~description of the~evolution of $\tau^{-1}(0)$
in time. Then each critical point denotes some special event, being a~qualitative change of the~space
(i.e. a~split into to spaces, a~creation of a~non-contractible loop, etc.).
We can say that $\tau$ keeps a~chronology of those events and therefore we call cobordisms with
such projections \term{chronological}. This may result in \term{a~chronological topological quantum
field theory}.

In this paper, we describe the~category of chronological cobordisms and construct a~complex in this category,
which is a~tangle invariant. Then, applying appropriate functors, we can recover both usual and odd
homology groups. Because the~chronological cobordisms are strongly non-symmetric, we introduce
changes of chronologies and connected to them relations -- multiplication by some coefficients from $R$.
As in the~case of odd theory we obtain a~cube that commutes up to invertible elements of $R$.
Finally, we prove the~cube can be fixed to be commutative in a~canonical way and that the~resulting complex
is a~tangle invariant.

The~paper consists of four chapters. The~first one is a~brief introduction to the~knot theory.
It contains basic definitions and theorems, constructions of the~Jones polynomial and the~Kauffman bracket
as well as basic facts on planar algebras, which generalises the~algebra of tangles.

The~chapter~\ref{chpt:chcob} is a~description of the~category of chronological cobordisms.
It starts with definitions and basic facts from the~theory of oriented cobordisms,
which are followed by the~definition and properties of chronologies.
The~section~\ref{sec:chcob-class} contains a~presentation of the~category of two-dimensional chronological
cobordisms in terms of generators and relations. Next pages introduce changes of chronologies
with explanation why the~quotient category by these relations is non-trivial.
This chapter ends with deliberation on cobordisms embedded in~$\mathbb{R}^3$,
which give a~natural framework for the~Khovanov cube and the~whole construction.

A~brief introduction to homological algebra is the~main part of the~chapter~\ref{chpt:alg-hom}.
We first define additive categories in which we can construct chain complexes and show
that every category can be extended to an~additive one. The~sections~\ref{sec:hom-cubes}
and~\ref{sec:hom-com-cub} deal with special constructions in such categories: cubes and cube complexes.
The~are crucial from the~view of the~Khovanov complex and the~proof of the~its invariance.

The~last chapter is devoted to the~Khovanov complex. The~construction is given in the~beginning,
then we prove its invariance. Next several properties are given: the~behaviour of the~complex under
reversing orientations of some components of a~link or mirroring the~tangle.
The~section~\ref{sec:khov-homol} contains examples of functors that can be used to compute
homology groups. In particular, we define a~functor $\F_{\!XYZ}$, which generalises both
$\F_{\!Kh}$ and~$\F_{\!ORS}$. It leads to the~notion of \term{a~chronological Frobenius algebra}.
The~last section shows that this new functor categorifies the~Jones polynomial.

\chead{\fancyplain{}{Preface to the~English version}}
\cleardoublepage
\phantomsection
\chapter*{Preface to the~English version}
This paper is a~translation of my Master's Thesis, originally written in Polish.
The~Thesis was defended in November 2008 and since that time several new things have been discovered:
dotted chronological cobordisms, extension of the~construction over cobordisms between tangles
up to invertible elements.
However, I tried to keep the~translation as close to the~original text as possible, making minor
changes only if necessary. The~chapters~\ref{chpt:knots}, \ref{chpt:khov} and most of \ref{chpt:alg-hom}
are faithful translations, modulo change of several symbols to avoid collisions.
The~second chapter has been changed in several places.
In the~section~\hyperref[sec:chcob-def]{\ref*{chpt:chcob}.\ref*{sec:chcob-def}}
the~definition of equivalent chronological cobordisms is simplified and examples of isotopies
of chronologies are included.
The~section~\hyperref[sec:chcob-class]{\ref*{chpt:chcob}.\ref*{sec:chcob-class}}
has been rewritten, due to gaps in the~original proof of the~classification theorem.

\tableofcontents

\chapter{A brief introduction to knots}\label{chpt:knots}
\chead{\fancyplain{}{\thechapter. A brief introduction to knots}}

This part contains basic definitions and facts on~knots and links:
existence of diagrams, Reidemester's theorem as well as the~constructions
of the~Kauffman bracket and the~Jones polynomial. For a~more comprehensive introduction
the~reader is referred to~\cite{CrowFox, Fox} (with no~results of Jones)
or~\cite{GilbPort, Przytycki, Roberts}.

\section{Basic definitions}\label{sec:knots-def}
Consider a standard oriented circle $\mathbb{S}^1 = \{z\in\mathbb{C}\ | z\bar{z}=1\}$.
A~disjoint sum of its $n$ copies will be denoted by $n\mathbb{S}^1$.

\begin{definition}\label{def:link}
A \term{knot} is a~smooth embedding of a~circle $\S1$ into the~oriented space $\mathbb{R}^3$.
An embedding of a~disjoint sum of $n$~circles is called a \term{link} and the embedded circles
are called \term{components}.
\end{definition}

\noindent A~knot is simply a~link with one component.
Some examples are provided in the~figure~\ref{fig:knots}.\footnote{\ 
The enumeration of knots used in this thesis comes from~\cite{Rolfsen}.}

\begin{remark}
One can also consider links in other space than $\mathbb{R}^3$, for instance
thickened surfaces $\Sigma_g\times I$ or fibre bundles over surfaces.
Non-trivial theories exist also for thin surfaces. For instance knots in a~torus are in
one-to-one correspondence with pairs of coprime numbers. (see~\cite{Rolfsen}, chapter 2).
\end{remark}

\begin{figure}[ht]
	\begin{center}
	  \hfill\hfill
	  \image{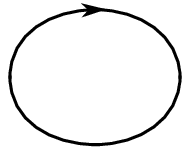}{-5pt}
	  \hfill
	  \putimage{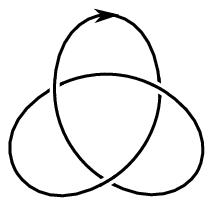}
	  \hfill
	  \putimage{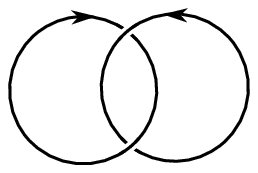}
	  \hfill
	  \putimage{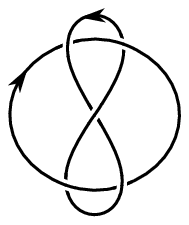}
	  \hfill\hfill\quad
	\end{center}
	\caption{Examples of knots and links: the unknot~$0_1$,
		the trefoil~$3_1$, the Hopf link and the Whitehead link.}
	\label{fig:knots}
\end{figure}

Links with $n$ components form a space
\begin{equation}
\mathcal{L}^n = \{L\colon n\mathbb{S}^1\hookrightarrow\mathbb{R}^3\ | L\textrm{ is a link}\} \subset C^\infty(n\mathbb{S}^1,\mathbb{R}^3)
\end{equation}
with the~open-compact Whitney topology.
Denote by $\mathcal{L}$ the~space of all links, being the~disjoint sum of the~spaces above.
We will identify links lying in the same component of $\mathcal{L}$ according to the~following definition.

\begin{definition}\label{def:links-equivalence}
Links $L_1$ and~$L_2$ are regarded as \term{equivalent}, if there is a smooth path
$\gamma\colon I\to\mathcal{L}$ such that~$\gamma(0) = L_1$ and $\gamma(1) = L_2$.
\end{definition}

This~definition is equivalent to the~existence of a~smooth isotopy
$H\colon n\mathbb{S}^1\times I\to\mathbb{R}^3$ from $L_1$ to $L_2$.
It has a~very geometrical meaning: links can be deformed so far as none of its component
is ,,torn'' nor their parts intersect each other.
Since now by links and knots we will mean their equivalence classes.

\begin{remark}
When forgetting orientation of circles, we get \term{non-oriented links}
in opposite to \term{oriented links} defined above.
The~equivalence of oriented links descends to non-oriented if we identify links
with different orientations.
\end{remark}

\begin{remark}
One can also consider knots and links without the~smoothness condition.
However, with no additional restrictions it leads to a~pathology called wild knots
(fig.~\ref{fig:wild-knot}).
To prevent from such situations a~link is usually assumed to be equivalent to a~sum of intervals
(so called \term{PL-knots} or~\term{combinatorial knots}).
However, smooth knots are equivalent to combinatorial ones (see~\cite{CrowFox}).

\begin{figure}[htbp]
	\begin{center}
		\scimage{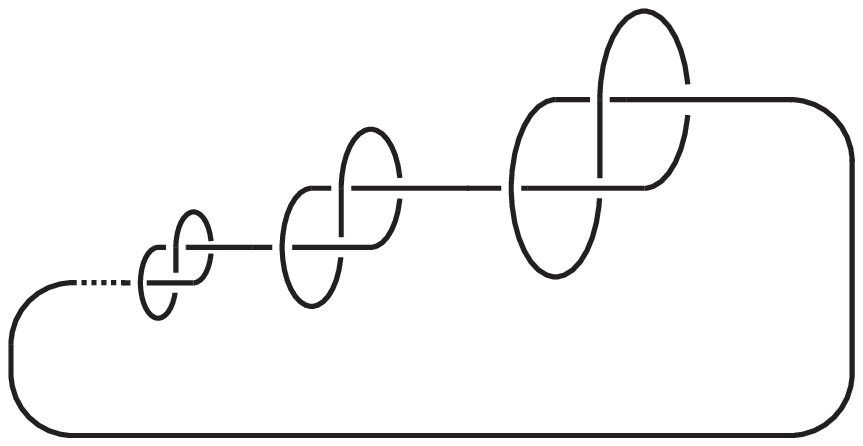}{4cm}{0pt}
	\end{center}
	\caption{A wild knot}
	\label{fig:wild-knot}
\end{figure}
\end{remark}

We distinguish \term{trivial links} as those equivalent to an embedding of circles
into a plane. They are denoted by $nU$, where $n$ stands for the~number of components.
In case $n=1$ we write simply $U$.
The main problem of knot theory is to determine if two given links are equivalent or not,
especially whether a given link is trivial (if it can be ,,untied'').

The~are two basic operations on links:
\begin{itemize}
\item \term{mirroring}: $L^* := S\circ L$, where $S(x,y,z) = (x,y,-z)$ is the symmetry of $\mathbb{R}^3$ along the~XY-plane,
\item \term{reversion}: $-L := L\circ(A,\dots,A)$, where $A(z)=\bar{z}$ is the symmetry of $\mathbb{S}^1$ along the~0X-axis.
\end{itemize}
They commute and preserve the~equivalence class of a~link
(compose the~path $\gamma$ in the~definition of equivalence with an appropriate operation).
$L^*$ is called \term{the~mirror} to $L$ and $-L$ \term{the~reversed} link.

\begin{figure}[htbp]
	\begin{center}
		\putimage{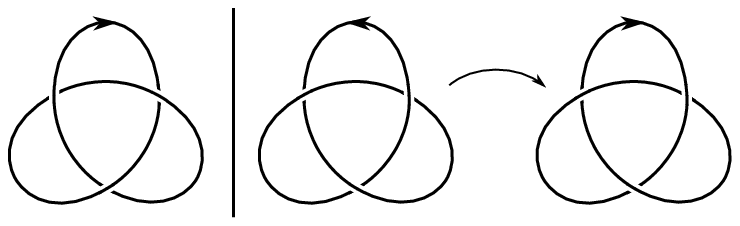}
	\end{center}
	\caption{Left- and right-handed trefoils are mirror knots.
	They are chiral and reversible.}
	\label{fig:left-right-3-1}
\end{figure}

\begin{definition}
A link $L$ is said to be \term{amphichiral} if $L^* \sim L$ (otherwise it is called
\term{chiral}) and~\term{reversible} if $-L \sim L$.
\end{definition}

It is usually difficult to show chirality of a~link. In particular,
the~fundamental group of the~knot complement does not change after mirroring.
In many cases the~Jones polynomial defined in the~section~\ref{sec:knots-Jones} is can be used.
The~figure eight knot $4_1$ is an example of an~amphichiral knot, whereas the~trefoil $3_1$
is chiral.

Even more challenging is to show non-reversibility, since most of known link invariants
do not depend on orientations of link components. Among knots with up to eight crossings
$8_{17}$ is the only non-reversible one. Another example is $K=9_{32}$, producing
four different oriented knots: $K, K^*, -K, -K^*$. We say $K$ is fully asymmetric.

\begin{figure}[htbp]
	\begin{center}\hfill\hfill\hfill
		\putimage{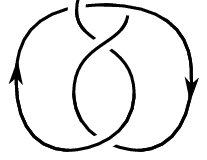}$4_1$\hfill
		\putimage{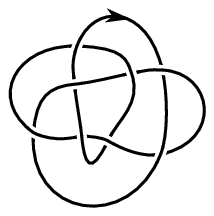}$8_{17}$\hfill
		\putimage{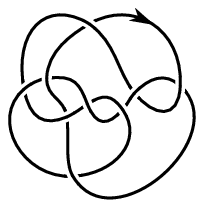}$9_{32}$
		\hfill\hfill\hfill\quad
	\end{center}
	\caption{$8_{17}$ (in the~middle) is non-reversible but equivalent to its mirror.
            $9_{32}$ (on the~right) is fully asymmetric, whereas $4_1$ (on the~left) is fully symmetric.}
	\label{fig:8-17-9-32}
\end{figure}

\section{Link diagrams}\label{sec:knots-diag}
Although links can be examined in a~purely topological approach (using the~fundamental group
of their complements or covering spaces),
one of the most efficient invariants arise from combinatorial approach based on diagrams of links.
Consider a~projection on a~plain of a~given link, satisfying the~following conditions:
\begin{enumerate}
\item each point is an~image of at most two points of the~link
\item there are only finitely many double points
\item all intersections are transverse
\end{enumerate}
Such a~projection is called \term{regular}.
\begin{figure}[hbt]
	\begin{center}
		\putimage{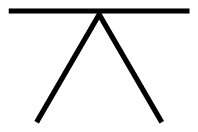}\qquad
		\putimage{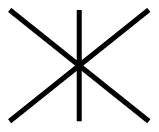}\qquad
		\putimage{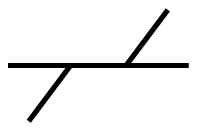}
	\end{center}
	\caption{Situations not allowed in regular projections}
	\label{fig:diag-bad}
\end{figure}

One can ask whether such projections exist. The Whitney's embedding theorem
implies a~set of immerse projections is dense. Then transversality theorems
gives a~positive answer to the~question.\footnote{\ 
There is an analogous theorem for combinatorial knots, see \cite{Duda, Przytycki}.}

\begin{theorem}\label{thm:reg-proj-dense}
The~subspace of regular projections of a~given link is dense.
\end{theorem}

A~regular projection does not fully describe a~link --- some information on double points
(\term{crossings}) has to be included: which part of the~link passes over (\emph{bridge})
and which under (\emph{tunnel}).
It is noted by breaking the~tunnel. A~segment between two tunnels is called an~\term{arc}.

\begin{definition}\label{def:link-diagram}
A~\term{diagram} of a~link $L$ is a~regular projection on a~plain modified 
at each double point by~breaking the~part of the~link that is closer to the~plain.
\end{definition}

Let us distinguish two types of diagrams. We will see later, that links possessing such diagrams
have very interesting properties.

\begin{definition}\label{def:link-diag-type}
A~diagram $D$ is \term{reduced} if each crossing meats four different regions.
$D$ is \term{alternating} if moving along any its component one goes alternately through
tunnels and bridges.
By an~\term{alternating link} we mean a~link with an~alternating diagram.
\end{definition}

Any link has a~reduced diagram. Indeed, any not reduced crossing $c$
can be removed by rotating part of the~diagram.
The first occurrence of a non-alternating knot is $8_{19}$.

\begin{figure}[htbp]
	\begin{center}
		\putimage{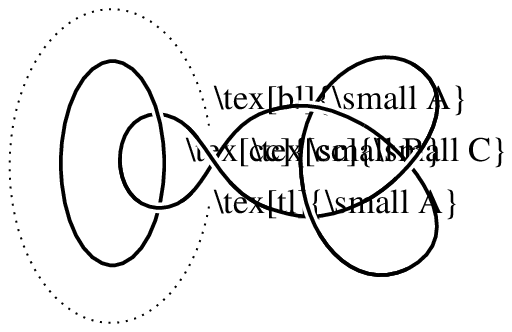}
	\end{center}
	\caption{An example of an~alternating non-reduced diagram.}
	\label{fig:red-nonalter}
\end{figure}

According to the~theorem~\ref{thm:reg-proj-dense} each link has a~diagram.
On the~other hand, a~diagram encodes enough information to rebuild the~link
up to equivalence. Obviously, a~link has a~wide range of diagrams,
but there exists a~pretty elegant theorem classifying in some sense all of them.

\begin{theorem}[Reidemeister, 1927]\label{thm:reidem}
Let~$L_1$ and~$L_2$ be links with diagrams $D_1$ and~$D_2$ respectively.
Then $L_1$ and~$L_2$ are equivalent if and only if $D_2$ can be obtained
from~$D_1$ by applying isotopies of a~plain and moves $R_1$ -- $R_3$ (fig.~\ref{fig:Reid-moves}).
\end{theorem}
\begin{figure}[htbp]
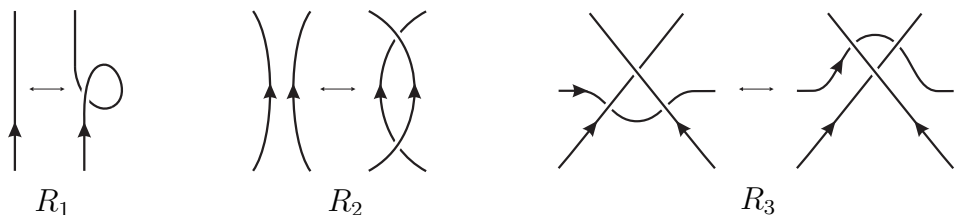

	\begin{center}
		\putvimage{orient-Reidem-1}{12ex}\qquad\qquad
		\putvimage{orient-Reidem-2}{12ex}\qquad\qquad
		\putvimage{orient-Reidem-3}{12ex}
	\end{center}

	\medskip
	\caption{Reidemeister moves. Only one orientation for each case is shown.
				In~the~theorem~\ref{thm:reidem} all possible orientations should be considered.}
	\label{fig:Reid-moves}
\end{figure}
\noindent Changes $R_1$ -- $R_3$ are called \term{Reidemeister moves}.
Elementary, though very technical proof of the~above theorem can be found%
\footnote{\ Those proofs are given for combinatorial knots.
However, which some effort they can be translated into the~language of differential topology.}
in~\cite{Przytycki,Roberts}.

The~number of crossings in a~diagram is not a~link invariant,
because it the~moves $R_1$ and~$R_2$ does not preserve it.
Define the~\term{crossing number} $c(L)$ of a~link as the~minimal number of crossings over all diagrams.
Despite the simplicity of the~definition, this number is very hard to compute.
Any diagram bounds it from above, whereas homology groups defined in chapter~\ref{chpt:khov}
give estimation from below.

\begin{figure}[htb]
	\begin{center}
		\putimage{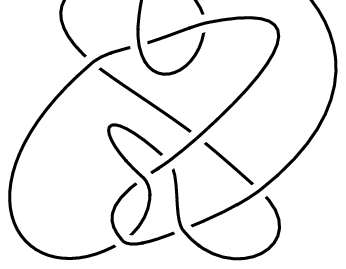}
	\end{center}
	\caption{A~demon diagram.}
	\label{fig:demon-knot}
\end{figure}
Theorem~\ref{thm:reidem} gives no efficient method to check whether two diagrams
represent the~same link or not. For example, there exist diagrams of the~unlink such that
any~move increases its~number of crossings (so called demons, see fig.~\ref{fig:demon-knot}).
Moreover, a~single crossing may decide whether a~link is trivial or not:
by~changing any crossing of the~trefoil to its mirror we get the~unknot,
although the~trefoil itself is non-trivial.
One can ask if any link can be untied by changing some crossings.
The~answer is positive: imagine a~descending point moving over the~given diagram.
Change the~crossings of this diagram such that the~projection of the~point is on a~bridge
at first pass through any~crossing. The~new diagram obtained in this way represents the~unknot
(fig.~\ref{fig:unknoting}).
In case of links repeat the~procedure for each component.

\begin{figure}[htbp]
	\begin{center}
		\putimage{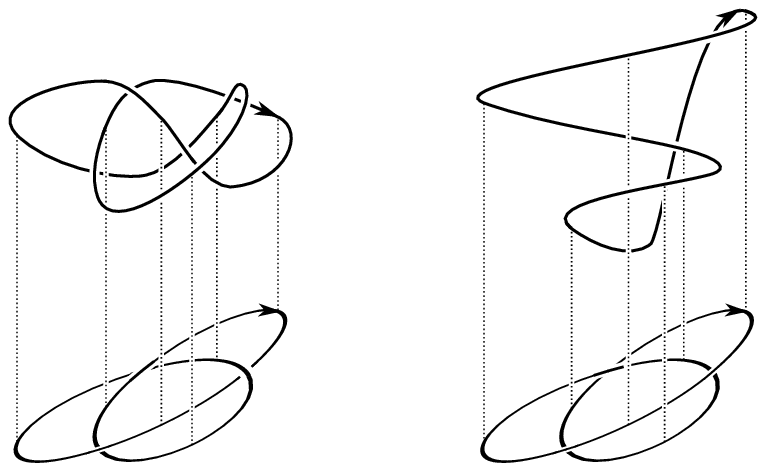}
	\end{center}
	\caption{A descending point moving over a~knot diagram shows how it can be untied.}
	\label{fig:unknoting}
\end{figure}

Define the~\term{unknotting number} $u(D)$ as the~minimal number of~crossings
trivialising the~diagram when changed into mirrors.
It can be arbitrary large and is not a~link invariant.

\begin{theorem}[K.~Taniyama, 2008]
A~non-trivial link $L$ has a~diagram with an~arbitrary large unknotting number.
\end{theorem}

\noindent Analogously to the~crossing number, define the~\term{unknotting number}
of a~link $L$ as the~minimum over all diagrams of~$L$.

There are two types of crossings in~oriented diagrams: positive and~negative (fig.~\ref{fig:cross-sgn}).
\term{The~writhe number} or \term{the~Tait number} $w(D)$ of a~diagram $D$ is the~sum of signs
over all crossings of $D$.
It is changed by the~first Reidemester move, but is preserved by others:
the~third move leaves all signs unchanged whereas the second creates or deletes
two crossings with opposite signs.

\begin{figure}[htbp]
	\begin{center}
		\begin{tabular}{cp{4cm}c}
		\putimage{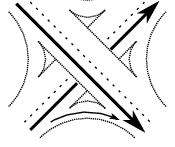} &\qquad &\putimage{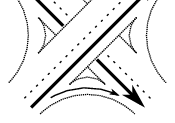}\\
		$\sgn = +1$ & & $\sgn = -1$
		\end{tabular}
	\end{center}
	\caption{Imagine components of a link are highways.
            Then the~sign of the~given crossing is determined by~the vertical direction
				in turning right: positive, when moving up and negative when moving down.}
	\label{fig:cross-sgn}
\end{figure}

Sings of crossings can be used to define another invariant. Let $L$ be a~link, $L_1$ and $L_2$
its components. A~\term{linking number} of $L_1$ and $L_2$ is the sum of signs over
crossings of $L_1$ with $L_2$ divided by two:
\begin{equation}
\lk(L_1,L_2) = \frac{1}{2}\sum_{c\in L_1\cap L_2} \sgn(c)
\end{equation}
In particular, one can compute the~linking number of any component $L_1$ with the~rest
of a~link $L\backslash L_1$.
In opposite to the~writhe number, the~linking number is a~link invariant,
since move $R_1$ affects only one component.

\section{The planar algebra of tangles}\label{sec:knots-tangles}
Links are compact sets. Hence, we can see them as embedded in a~standard ball $\mathbb{D}^3$ instead of $\mathbb{R}^3$.
This definition has the advantage that it can be extended over embeddings of intervals.

\begin{definition}\label{def:tangles}
A~\term{tangle} is a~neat embedding\footnote{\ 
An embedding of a~manifold $M\hookrightarrow N$ is called neat,
if~$M\cap\partial N = \partial M$ and $M$ is transverse to~$\partial N$.
} of a~disjoint sum of circles and closed intervals into a~standard ball $\mathbb{D}^3$.
Two tangles are called \term{equivalent}, if there exists an isotopy of the ball
constant on the boundary, moving one tangle to the~other.
\end{definition}

\begin{figure}[htb]
	\begin{center}
		\putimage{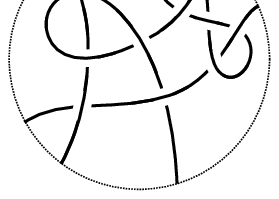}
		\caption{A~diagram of a~tangle.}\label{fig:diag-tangle}
	\end{center}
\end{figure}

Analogously to links, a~tangle has diagrams and two diagrams of a~given tangle are related
by Reidemeister moves. Denote by $\mathcal{T}(B)$ the~space of all tangles with endpoints
in a~set $B\subset\partial\mathbb{D}^2$.

Consider now a~disk from the~figure~\ref{fig:diag-plan}. One can put into its holes small tangles
obtaining a bigger one. In this way we can construct from tangles any link.
Obviously, instead of tangles one can put into holes other disks.
This results in an~algebraic structure, called the~planar algebra of tangles.
Now we will give a~formal definition.

\begin{definition}\label{def:diag-plan}
A \term{planar diagram} $D$ with~$s$ inputs is a~disk $\mathbb{D}^2$
missing smaller disks $\mathbb{D}^2_i$ for~$i=1,\dots,s$, together with a~neat embedding
of disjoint circles and closed intervals.
Say $D$ is \term{oriented} if the embedded circles and intervals are oriented.
Both oriented and non-oriented planar diagrams are considered up to planar isotopies
constant on boundary of $D$.
\end{definition}

\begin{figure}[htb]
	\begin{center}
		\putimage{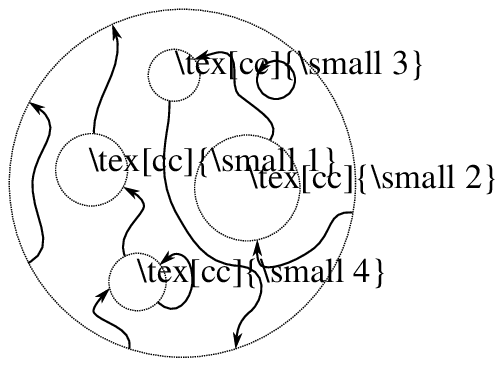}
	\end{center}
	\caption{An oriented planar diagram with four inputs.}\label{fig:diag-plan}
\end{figure}

Denote by~$\mathcal{T}^0(B)$ the set of all tangle diagrams
with the set of endpoints equal $B\subset\S1$.
Each planar diagram $D$ induces a~map
\begin{equation}
D\colon\mathcal{T}^0(B_1)\times\cdots\times\mathcal{T}^0(B_s)\to\mathcal{T}^0(B)
\end{equation}
for some sets $B,B_1,\dots,B_s$.
This gives us a~structure in $\mathcal{T}^0$ called the planar algebra of tangle diagrams.
In a~similar way, oriented tangle diagrams $\mathcal{T}^0_+(B)$ with oriented
planar diagrams create the~oriented planar algebra $\mathcal{T}^0_+$.
Due to locality of Reidemeister moves, both structures descend to tangles,
resulting in planar algebras of non-oriented tangles $\mathcal{T}$ and oriented ones $\mathcal{T}_+$.

Among all planar diagrams we will distinguish the~\term{radial} ones as those with one input in the~middle,
and only radial intervals. They induce identities in the above algebras.

Given a decomposition of a~link into tangles and a~planar diagram
one may think how replacing any tangle by another one may affect the~equivalence class of the~link.

\begin{definition}\label{def:mutation}
Let $D$ be a~diagram of a~link $L$ and $T$ its fragment being a~tangle diagram
with four ends in a~corner of some square. A~\term{mutation} of a~link $L$ is a~change
given by rotating $T$ by~$180^\circ$ along one of the~following axis of the~square:
vertical, horizontal or perpendicular to the~plane containing the~diagram.
The~new link obtained by this change is called a~\term{mutant} of $L$.
\end{definition}

Many of known invariants cannot distinguish mutants, also the Jones polynomial defined later.

\begin{figure}[htb]
	\begin{center}
		\putimage{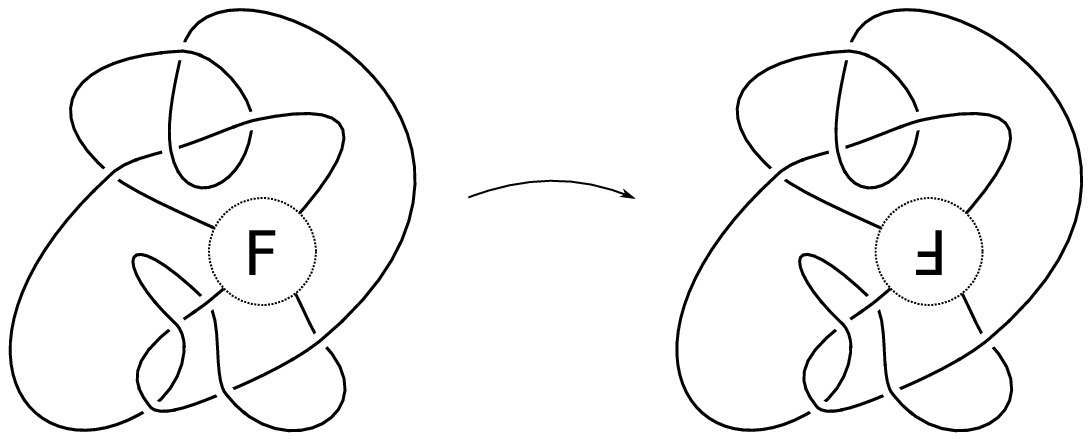}
	\end{center}
	\caption{An~example of a mutation.}\label{fig:mutation}
\end{figure}

In next chapters we will find similar algebraic structures in other categories.
Thus we will give now an abstract definition of a~planar algebra (compare with~\cite{Jones}).

\begin{definition}\label{def:alg-plan}
A~\term{planar algebra} (\term{oriented}) $\mathcal{P}$ is a~collection of sets $\mathcal{P}(B)$
defined for finite subsets $B\subset\S1$ (oriented) together with an~operator
\begin{equation}
D\colon\mathcal{P}(B_1)\times\cdots\times\mathcal{P}(B_s)\to\mathcal{P}(B)
\end{equation}
defined for each planar diagram (oriented)~$D$, such that their~composition is associative
and radial diagrams correspond to identities.
\end{definition}

Elements of $\mathcal{P}(\emptyset)$ are \term{closed}
and operators taking values in this set --- \term{closure operators}.
Denote by $\CPO$ the set of all such operators. It splits into sets $\CPO(B)$
consisting of operators with domain in~$\mathcal{P}(B)$. In general, let
\begin{equation}
\mathcal{P}(B_1,\dots,B_s;B) := \{D\colon\mathcal{P}(B_1)\times\dots\times\mathcal{P}(B_s)\to\mathcal{P}(B) \}
\end{equation}
Then $\CPO(B) = \mathcal{P}(B;\emptyset)$.

\begin{definition}\label{def:alg-plan-mor}
A~\term{morphism} of planar algebras $\Phi\colon\mathcal{P}_1\to\mathcal{P}_2$
is a~collection of maps $\Phi_B\colon\mathcal{P}_1(B)\to\mathcal{P}_2(B)$
commuting with planar operators:
\begin{equation}
D\circ(\Phi_{B_1},\dots,\Phi_{B_s}) = \Phi_B\circ D.
\end{equation}
\end{definition}

Sending a~tangle diagram to the~tangle itself is an~example of
a~morphism of planar algebras. Other examples will be given in the~chapter~\ref{chpt:chcob}.

\section{The Kauffman bracket}\label{sec:knots-Kauff}
Consider now non-oriented links.
Each crossing has two resolutions -- type 0 and type 1 (fig.~\ref{fig:cr-smooth}).
Unless it leads to confusion, denote by \tfnt{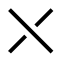}, \tfnt{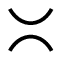} and \tfnt{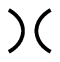}
diagrams which differ in a~single crossing as is presented by the~symbols.

\begin{figure}[htbp]
	\begin{center}
	\image{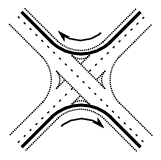}{0.7cm} $\quad\stackrel{0}\longleftarrow\quad$
	\image{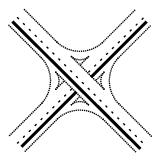}{0.7cm}  $\quad\stackrel{1}\longrightarrow\quad$
	\image{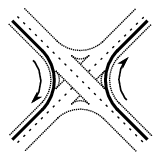}{0.7cm}
	\end{center}
	\caption{A~type of a~resolution, similarly to a~sign of~a~crossing, has a~simple
				interpretation. Again, considering link components as highways,
				any resolution can be compared to a~change of direction. Here a~type
				of a~resolution describe the~level to be left (zero -- lower, one -- upper).}
	\label{fig:cr-smooth}
\end{figure}

\noindent Define inductively a~polynomial $\PKauff{D}$ in variables $A,B,d$ by the~following equations:
\begin{enumerate}[label=(K\arabic*)]
\item\label{it:Kauff-unknot}$\PKauff{U} = 1$,
\item\label{it:Kauff-extra-circle}$\PKauff{U\sqcup D} = d\PKauff{D}$,
\item\label{it:Kauff-smoothings}$\PKauff{\fnt{cr-nor-p}} = A\PKauff{\fnt{cr-nor-h}} + B\PKauff{\fnt{cr-nor-v}}$.
\end{enumerate}

\begin{lemma}\label{lem:Kauff-loop}
Let $B=A^{-1}$ and~$d = -(A^2+A^{-2})$.
Then $\PKauff{D}$ is an~invariant under II and~III Reidemeister moves as well as
\begin{align}
\PKauff{\fnt{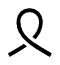}} &= -A^{3}\PKauff{\fnt{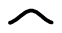}}\label{eq:Kauff-pos-loop}\\
\PKauff{\fnt{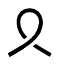}} &= -A^{-3}\PKauff{\fnt{R1-h-no-circle}}\label{eq:Kauff-neg-loop}
\end{align}
\end{lemma}
\begin{proof}
To show (\ref{eq:Kauff-pos-loop}) let us delete a~crossing using~\ref{it:Kauff-smoothings}:
$$
\PKauff{\fnt{R1-x}} = A\PKauff{\fnt{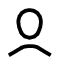}} + A^{-1}\PKauff{\fnt{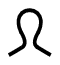}} = -A^3\PKauff{\fnt{R1-h-no-circle}}
$$
In a~similar way we obtain the~second equality.
They imply invariance under II Reidemeister move:
$$
\PKauff{\fnt{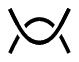}} = A\PKauff{\fnt{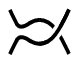}}+A^{-1}\PKauff{\fnt{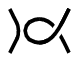}} =
                     A^2\PKauff{\fnt{cr-nor-v}}+\PKauff{\fnt{cr-nor-h}}-A^2\PKauff{\fnt{cr-nor-v}} = \PKauff{\fnt{cr-nor-h}}
$$
The~proof of invariance under III Reidemeister move does not need relations among $A$, $B$ and~$d$.
It is derived directly from invariance under II move:
\begin{align*}
\PKauff{\fnt{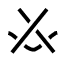}} = A\PKauff{\fnt{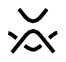}} + B\PKauff{\fnt{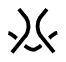}}
                     = A\PKauff{\fnt{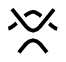}} + B\PKauff{\fnt{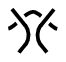}} = \PKauff{\fnt{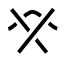}}
\end{align*}
\end{proof}

The~polynomial\footnote{\ 
Formally, $\PKauff{D}$ is a~Laurent polynomial, i.e.~an~element of a~ring $\mathbb{Z}[A,A^{-1}]$.%
} $\PKauff{D}$ in variable $A$ is called the~\term{Kauffman bracket}.
It was introduced by L.~Kauffman in~\cite{KauffmanBracket}.

As a consequence of \ref{it:Kauff-smoothings}, the~Kauffman bracket of a~link
can be computed by summing up polynomials of trivial links.
Define a~\term{Kauffman state} $S$ of a~diagram $D$ as a~sequence of resolutions of all crossings:
\begin{equation}
S\colon\kCr(D)\to\{0,1\}
\end{equation}
where $\kCr(D)$ is the~set of crossings of the~diagram $D$.
Denote by $S(D)$ the~set of all states of $D$.
Each state describes a~collection $|S|$ of disjoint circles in a~plain.
Let $n_1(S), n_0(S)$ be the~amounts of resolutions of type 1 and 0
accordingly and put $\tau(S)=n_1(S)-n_2(S)$.

\begin{theorem}\label{twr:Kauff-state-sum}
Let $D$ be a~link diagram. Then:
\begin{equation}\label{eq:Kauff-state-sum}
\PKauff{D} = \sum_{s\in S(D)} (-1)^{|S|-1}A^{\tau(S)}(A^{-2}+A^2)^{|S|-1}.
\end{equation}
\end{theorem}
\begin{proof}
For trivial links the~equality goes directly from \ref{it:Kauff-unknot} and~\ref{it:Kauff-extra-circle}.
Other cases are done by induction on the~number of crossings,
using relation~\ref{it:Kauff-smoothings}.
Indeed, picking a~crossing $c$ we have a~bijection
$$
S(\fnt{cr-nor-p}) = S_0(\fnt{cr-nor-p}) \cup S_1(\fnt{cr-nor-p}) \approx S(\fnt{cr-nor-h}) \sqcup S(\fnt{cr-nor-v})
$$
where $S_{\alpha}$ is the~set of states satisfying $S(c)=\alpha$. Then
\begin{align*}
\PKauff{\fnt{cr-nor-p}} &= A\PKauff{\fnt{cr-nor-h}} + A^{-1}\PKauff{\fnt{cr-nor-v}} = \\
                 &= \sum_{S\in S(\sfnt{cr-nor-h})} (-1)^{|S|-1}A^{\tau(S)+1}(A^{-2}+A^2)^{|S|-1} +
						  \sum_{S\in S(\sfnt{cr-nor-v})} (-1)^{|S|-1}A^{\tau(S)-1}(A^{-2}+A^2)^{|S|-1} = \\
                 &= \sum_{S\in S_0(\sfnt{cr-nor-p})} (-1)^{|S|-1}A^{\tau(S)}(A^{-2}+A^2)^{|S|-1} +
						  \sum_{S\in S_1(\sfnt{cr-nor-p})} (-1)^{|S|-1}A^{\tau(S)}(A^{-2}+A^2)^{|S|-1} = \\
                 &= \sum_{S\in S(\sfnt{cr-nor-p})} (-1)^{|S|-1}A^{\tau(S)}(A^{-2}+A^2)^{|S|-1}
\end{align*}
\end{proof}

After appearance of M.~Khovanov's paper~\cite{Khov-Jones}, O.~Viro introduced in~\cite{Viro}
the~notion of an~enhanced Kauffman state, adding orientations to circles.
The~sum over these states has a~simpler form --- it is a~sum of monomials.

\begin{definition}\label{def:en-Kauff-state}
An~\term{enhanced Kauffman state} $S$ is a~map which associates to each crossing
a~resolution and orientation to each circle in the~smoothed diagram described by the~resolutions.
\end{definition}
The~set of all enhanced states will be denoted by $ES(D)$.
Let $d_{+}(S)$, $d_{-}(S)$ be the~amounts of positively and negatively oriented circles
accordingly and~put $\sigma(S)= d_{+}(S) - d_{-}(S)$. Then
\begin{equation}
(A^{-2}+A^2)^{|S|} = \sum_{i=0}^{|S|}{|S|\choose i}A^{2(|S|-i)-2i} = \sum_{S'}A^{2\sigma(S')}
\end{equation}
where the~last sum is taken over all enhanced states equal $S$,
when we forget the~orientations of circles. As a~result we have the~following statement.

\begin{theorem}\label{twr:Kauff-enh-state-sum}
Let $D$ be a~link diagram. Then
\begin{equation}\label{eq:Kauff-enh-state-sum}
\PKauff{D} = \sum_{S\in ES(D)}(-1)^{|S|-1}A^{\tau(S)+2\sigma{S}}
\end{equation}
\end{theorem}

\section{The Jones polynomial}\label{sec:knots-Jones}
Here we will define a~polynomial that is an~invariant of non-oriented links.
Firstly, notice we have already defined for an~oriented diagram $D$ two objects,
which are preserved under II and III Reidemeister move:
\begin{itemize}
\item the~Kauffman bracket $\PKauff{D}$ (when the~orientation of $D$ is forgotten),
\item the~writhe number $w(D)$.
\end{itemize}
Define a~new polynomial for an~oriented diagram $D$, evaluating $\PKauff{D}$ at $t^{-1/4}$
and multiplying it by $(-t)^{-\frac{3}{4}w(D)}$:
\begin{equation}\label{eq:jones-def}
V_D(t) = (-t)^{-\frac{3}{4}w(D)}\PKauff{D}_{A=t^{1/4}}
\end{equation}

\begin{proposition}\label{prop:Jones-axioms}
The~polynomial $V_D(t)$ defined by~(\ref{eq:jones-def}) is a~link invariant
and the~following holds:
\begin{enumerate}[label=(J\arabic*)]
\item\label{it:Jones-unknot}$V_U(t) = 1$,
\item\label{it:Jones-skein}$t^{-1}V_{\sfnt{cr-or-p}}(t) - tV_{\sfnt{cr-or-n}}(t) = (t^{1/2}-t^{-1/2})V_{\sfnt{cr-or-h}}(t)$.
\end{enumerate}
\end{proposition}
\begin{proof}
Invariance under II and~III Reidemeister move is due to invariance of the~writhe number and the~Kauffman bracket.
By lemma~\ref{lem:Kauff-loop} I move also preserves $V_D$:
$$
V_{\sfnt{R1-x}}(t) = (-A)^{-3w(\sfnt{R1-x})}\PKauff{\fnt{R1-x}} =
                       (-A)^{-3(w(\sfnt{R1-h-no-circle})+1)}(-A^3)\PKauff{\fnt{R1-h-no-circle}} =
                       (-A)^{-3w(\sfnt{R1-h-no-circle})}\PKauff{\fnt{R1-h-no-circle}} = V_{\sfnt{R1-h-no-circle}}(t)
$$
and similarly for the~second loop.

The~normalisation condition~\ref{it:Jones-unknot} is satisfied by the~definition of $V_D$.
To show \ref{it:Jones-skein} notice that \ref{it:Kauff-smoothings} implies:
\begin{align*}
     A\PKauff{\fnt{cr-nor-p}} &= A^2\PKauff{\fnt{cr-nor-h}}+\PKauff{\fnt{cr-nor-v}}\\
A^{-1}\PKauff{\fnt{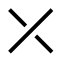}} &= \PKauff{\fnt{cr-nor-v}}+A^{-2}\PKauff{\fnt{cr-nor-h}}\\
A\PKauff{\fnt{cr-nor-p}}-A^{-1}\PKauff{\fnt{cr-nor-n}} &= (A^2-A^{-2})\PKauff{\fnt{cr-nor-h}}
\end{align*}
and the~last equality can be written for an~oriented link as:
$$
-A^4(-A)^{-3w(\sfnt{cr-or-p})}\PKauff{\fnt{cr-or-p}} + A^{-4}(-A)^{-3w(\sfnt{cr-or-n})}\PKauff{\fnt{cr-or-n}} =
(A^2-A^{-2})(-A)^{-3w(\sfnt{cr-or-h})}\PKauff{\fnt{cr-or-h}}
$$
The~change of powers at the~left hand side appears because of the~difference in writhe numbers:
diagrams at the~left side have one more crossing than the~one at the~right side.
To end the~proof put $t=A^{-4}$.
\end{proof}

Due to the~above theorem $V_L(t)$ is a~Laurent polynomial in $t^{1/2}$
with integer coefficients, called the~\term{Jones polynomial} of a~link $L$.
It was discovered before the~Kauffman bracket and defined by axioms from proposition~\ref{prop:Jones-axioms}.
The~equality~\ref{it:Jones-skein} is called \term{a~skein relation} and with the~normalisation
condition~\ref{it:Jones-unknot} can be used to compute the~polynomial without using Kauffman states.

\begin{example}\label{ex:Jones-trivial}
Consider the~following three diagrams:
\begin{center}
  \puthimage{two-circles}{24ex}
\end{center}
The first two of them represent the~unknot, hence by the~skein relation:
\begin{equation}
(t^{1/2}-t^{-1/2})V_{\sfnt{cr-or-h}}(t) = t^{-1}V_{\sfnt{cr-or-p}}(t) - tV_{\sfnt{cr-or-n}}(t) = t^{-1} - t
\end{equation}
and as a result $V_{2U}(t) = -t^{-1/2}-t^{1/2}$.
By induction:
\begin{equation}
V_{nU}(t) = (-t^{1/2}-t^{-1/2})^{n-1}.
\end{equation}
Obviously, it is the~same as when computed from Kauffman states.
\end{example}

\begin{example}\label{ex:Jones-3-1}
We will compute now the~polynomial for the~left-handed trefoil using the~computing tree
from figure~\ref{fig:3-1-compute}.
Vertices $3_1$ and $X$ describe the~equalities:
\begin{align}
t^{-1}V_U(t) - tV_{3_1}(t) &= (t^{1/2}-t^{-1/2})V_X(t) \\
t^{-1}V_{2U}(t) - tV_X(t) &= (t^{1/2}-t^{-1/2})V_U(t)
\end{align}
Having already computed the~polynomials of trivial links, we get:
\begin{align*}
V_{3_1}(t) =& t^{-2} V_U(t) + (t^{-3/2}-t^{-1/2}) \left( t^{-2} V_{2U}(t) + (t^{-3/2}-t^{-1/2})V_U(t) \right)\\
\notag  =&   -t^{-4} + t^{-3} + t^{-1}.
\end{align*}
\begin{figure}[hbt]
	$$\xy
	\place( 600,1300)[\putimage{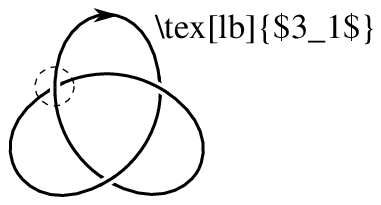}]
	\place(   0, 600)[\putimage{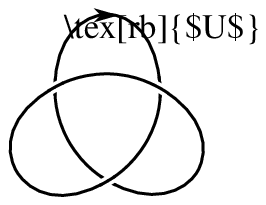}]
	\place(1200, 600)[\putimage{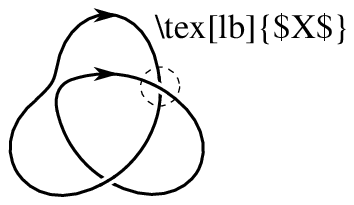}]
	\place( 700,-100)[\putimage{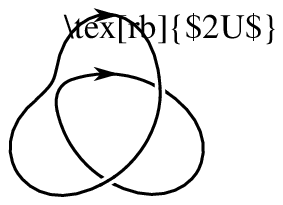}]
	\place(1700,-100)[\putimage{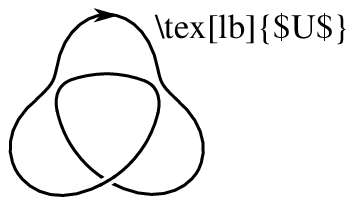}]

	\morphism(400,1300)<-250,-300>[`;\textrm{\normalsize change}]
	\morphism(800,1300)< 250,-300>[`;\textrm{\normalsize smoothe}]

	\morphism(1000,600)<-200,-240>[`;\textrm{\normalsize change}]
	\morphism(1400,600)< 200,-240>[`;\textrm{\normalsize smoothe}]
	\endxy$$
	\caption{The~computing tree for the~trefoil.}
	\label{fig:3-1-compute}
\end{figure}
\end{example}

\begin{remark}
The~polynomial for the~right-handed trefoil is
\begin{equation}
V_{3_1^*}(t) = -t^4+t^3+t
\end{equation}
so the~trefoil is chiral.
\end{remark}

\begin{remark}
Reversing all link components does not affect signs of crossings.
Hence writhe $w(D)$ is preserved and
\begin{equation}
V_L(t) = V_{-L}(t)
\end{equation}
Thus the~Jones polynomial cannot distinguish a~link from its reversion.
In particular, it is well-defined for non-oriented knots.
\end{remark}

The~example~\ref{ex:Jones-3-1} shows a~general method how to compute the~Jones polynomial
using only the~skein relation. Indeed, we can build a~computing tree for any link.
Recall that for each diagram $D$ there is a~sequence of crossings $c_1,\dots,c_k$,
such that $D$ becomes trivial when all $c_i$'s are changed into mirrors (see~fig.~\ref{fig:unknoting}).
Denote by:
\begin{itemize}
\item $D_i$ a diagram~$D$ with changed $c_1,\dots,c_i$ into mirrors,
\item $D_i^0$ a~diagram~$D$ with changed $c_1,\dots,c_{i-1}$ into mirrors and $c_i$ smoothed.
\end{itemize}
Build a~tree according to the~following rules:
\begin{enumerate}
\item $D$ is the~root
\item vertices are given by $D_i, D_i^0$
\item branches go from $D_i$ to both $D_{i+1}$ and $D_{i+1}^0$ for $i<k$
\end{enumerate}
This tree has $D_k$ and $D_i^0$'s in leaves. The~first diagram is trivial and the~rest
have less crossings than $D$, so we can build inductively a~computing tree for each of them.
Eventually, we end with a tree having trivial links in leaves.
Having computed polynomials for trivial links (example~\ref{ex:Jones-trivial}),
we can calculate polynomials for vertices consecutively starting from leaves and ending
in the~root $D$. As a~result we have

\begin{theorem}\label{twr:Jones-unique}
The Jones polynomial is the~unique invariant polynomial satisfying \ref{it:Jones-unknot} and~\ref{it:Jones-skein}.
\end{theorem}

A~spectacular application of the~Jones polynomial was proving three conjectures,
stated in the~second half of 19th century by P.~G.~Tait, who tried to classify knots.

\begin{theorem}[Tait's conjectures]
Let $L$ be an alternating link. Then
\begin{enumerate}[label=T\arabic*:]
\item any reduced alternating diagram of $L$ has a~minimal crossing number,
\item any two reduced alternating diagrams of $L$ have the same writhe number,
\item any reduced alternating diagram of $L$ can be obtained from another one by local flips (fig.~\ref{fig:flypes}).
\end{enumerate}
\end{theorem}
\begin{figure}[hbt]
	\begin{center}
		\putimage{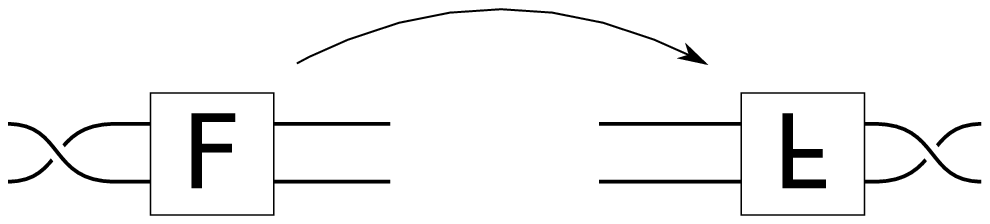}
	\end{center}
	\caption{Local moves classifying alternating diagrams.}
	\label{fig:flypes}
\end{figure}

First two statements have been proved independently by
L.~H.~Kauffman~\cite{KauffmanBracket}, K.~Murasugi~\cite{Murasugi}
and M.~B.~Thistlethwaite~\cite{Thistl} during two years after appearance of the~paper of Jones.
The~third claim had been waiting to be proven till 1990's, when
papers of W.~M.~Menasco and~M.~B.~Thistlethwaite appeared (\cite{MenascoThistl1,MenascoThistl2}).

There are non-trivial links with Jones polynomial equal 1
as well as different knots with the same Jones polynomial.
The~following theorem shows one way how to produce such pairs.

\begin{theorem}
Let $L_2$ be a~link obtained from $L_1$ by a~mutation.
Then $V(L_1)$ and $V(L_2)$ are equal.
\end{theorem}

The~proof can be found in \cite{Przytycki}. Even more is shown there: 
any invariant defined by a~skein relation%
\footnote{\ Such an~invariant is said to be of Conway type.}
is preserved by mutations.
However, it is still unknown whether the~Jones polynomial detects unknottedness or not.

\chapter{Cobordisms with chronologies}\label{chpt:chcob}
\chead{\fancyplain{}{\thechapter. Cobordisms and chronologies}}
The~standard Khovanov complex lives in the~category of oriented cobordisms,
described with details in~\cite{Abrams, Kock}.
The notion of a~chronology introduced in this chapter breaks some symmetries
what results in a richer category still having a~finite presentation.
However, such a rigid structure is unnecessary to~build a generalized complex.
Thus we will weaken it by allowing some changes of chronologies keeping control
over this process.

\section{Oriented cobordisms}\label{sec:cob-def}
Let $M$ be an~oriented manifold with a~fixed orientation on its boundary $\partial M$.
The~\term{input} $M_{in}$ is the~maximal boundary component with the~orientation equal
the~induced from $M$, whereas the~\term{output} $M_{out}$ is the~maximal boundary
component with the~opposite orientation.

\begin{figure}[htb]
	\begin{center}
		\includegraphics{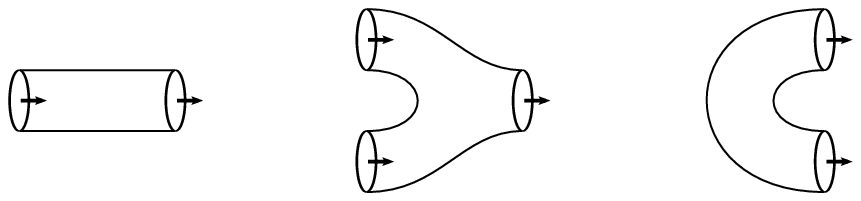}
	\end{center}
	\caption{Examples of cobordisms. Arrows denote inputs and outputs.
				They are usually omitted assuming the~input is on the~left-hand side
				of the~cobordisms and the~output on the~right-hand side.}\label{fig:cobordisms}
\end{figure}

The~definition is independent of the~choice of orientations in the~following sense:
reversing orientations of $M$ and $\partial M$ preserves both the~input and the~output of $M$.

\begin{definition}\label{def:cobordism}
An~\term{oriented cobordism} between oriented $n$-manifolds $\Sigma_{in},\Sigma_{out}$
is an~oriented pair $(M, \partial M)$ of dimension $(n\!+\!1)$
along with diffeomorphisms $\varphi_{in}\colon\Sigma_{in}\to M_{in}$
and~$\varphi_{out}\colon\Sigma_{out}\to M_{out}$ preserving the~orientations.
It is denoted by
$$
\Sigma_{in}\rightarrow M \leftarrow \Sigma_{out}\qquad\textrm{or}\qquad \cob{M}{\Sigma_{in}}{\Sigma_{out}}.
$$
\end{definition}

Diffeomorphic manifolds represent the~same cobordism, if the~diffeomorphism
agrees with both input and output.

\begin{definition}\label{def:cob-equiv}
A diffeomorphism $\psi\colon M\to M'$ is called an~\term{equivalence}
of~cobordisms $\cob M{\Sigma_0}{\Sigma_1}$ and~$\cob{M'}{\Sigma_0}{\Sigma_1}$,
if the following diagram commutes:
$$
\xy
  \morphism(0,250)|a|/{->}/<500,250>[\Sigma_0`\phantom{M};]
  \morphism(0,250)|a|/{->}/<500,-250>[\phantom{\Sigma_0}`\phantom{M'};]
  \morphism(1000,250)|a|/{->}/<-500,250>[\Sigma_1`\phantom{M};]
  \morphism(1000,250)|a|/{->}/<-500,-250>[\phantom{\Sigma_0}`\phantom{M'};]
  \morphism(500,500)|l|/{->}/<0,-500>[M`M';\psi]
\endxy
$$
\end{definition}
\noindent Equivalent cobordisms will be identified.

\begin{figure}[hbt]
	\begin{center}
		\image{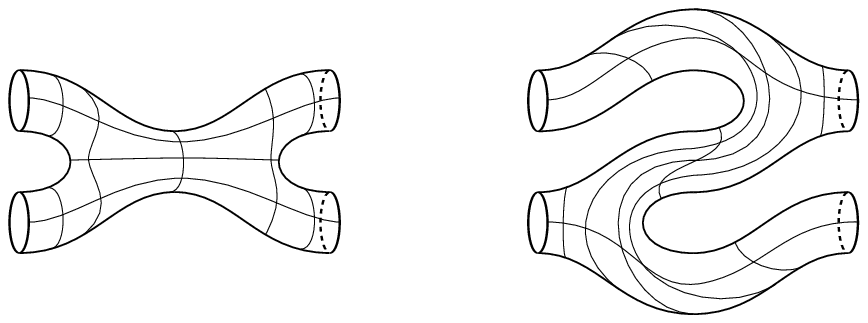}{0pt}
	\end{center}
	\caption{Equivalent cobordisms between disjoint sums of circles.
				The diffeomorphism is visualised by a~deformation of the net of the left cobordism
				on the right one.}\label{fig:equiv-cobo}
\end{figure}

\begin{example}\label{ex:cob-cylinder}
Let $\Sigma$ be an~$n$-manifold. Consider the~product $M=\Sigma\times I$ along with embeddings
\begin{equation}
\Sigma\approx \Sigma\times 0\hookrightarrow \Sigma\times I \hookleftarrow \Sigma\times 1 \approx \Sigma
\end{equation}
In this way we obtain a~cobordism with $\Sigma$ being as both the~input and the~output.
It is called the~\term{cylinder induced by $\Sigma$} and denoted by $C_\Sigma$ or simply $C$.

In general, an~orientation preserving diffeomorphism $\varphi\colon\Sigma_1\to\Sigma_0$ defines
a~cobordism
\begin{equation}
\Sigma_0\stackrel{\id}\longrightarrow\Sigma_0\times I \stackrel\varphi\longleftarrow\Sigma_1
\end{equation}
denoted by $C_{\varphi}$ and called the~\term{cylinder generated by $\varphi$}.
The~cobordism $C_\Sigma$ will be called the~\term{identity cylinder on $\Sigma$}.
\end{example}

$C_\varphi$ is not equivalent to the~identity cylinder, unless $\varphi$ is isotopic to the~identity.

\begin{lemma}\label{lem:cob-equiv-iso}
Cobordisms $C_{\varphi_1},C_{\varphi_2}$ are equivalent if and only if
the~diffeomophisms $\varphi_1,\varphi_2$ are isotopic.
\end{lemma}
\begin{proof}
Let $\psi\colon C_{\varphi_1}\to C_{\varphi_2}$ be an~equivalence of cobordisms.
The~desired isotopy is given by the~following composition:
$$
\Sigma_1\times I\stackrel{\varphi_1\times\id}\to\Sigma_0\times I\stackrel{\psi}\to\Sigma_0\times I\stackrel{\pi_1}\to\Sigma_0
$$
where~$\pi_1$ is the~projection on the~first variable.

Conversely, an~isotopy $H\colon\Sigma_1\times I\to\Sigma_0$ between $\varphi_1$ and $\varphi_2$
induces an~equivalence of cobordisms:
$$
\xy
\morphism(0,500)|a|/{->}/<600, 500>[\Sigma_0`;\id]
\morphism(0,500)|b|/{->}/<600,-500>[\Sigma_0`;\id]

\morphism(750,1000)|m|/{->}/<0,-500>[\Sigma_0\times I`\Sigma_1\times I;\varphi_1^{-1}\times\id]
\morphism(750, 500)|m|/{->}/<0,-500>[\Sigma_1\times I`\Sigma_0\times I;(H,\pi_2)]

\morphism(1500,500)|a|/{->}/<-600, 500>[\Sigma_1`;\varphi_1]
\morphism(1500,500)|b|/{->}/<-600,-500>[\Sigma_1`;\varphi_2]
\endxy
$$
where $\pi_2$ is the~projection on the~second variable.
\end{proof}

\begin{example}\label{ex:cob-perm}
Let $\Sigma_1$ and $\Sigma_2$ be $n$-manifolds. Take their disjoint sum
\begin{equation}
\Sigma_1\sqcup \Sigma_2 = (\Sigma_1\times 0) \cup (\Sigma_2\times 1)
\end{equation}
with the~natural smooth structure. Exchanging its components defines a~diffemorphism
\begin{equation}
s_{\Sigma_1,\Sigma_2}\colon \Sigma_1\sqcup \Sigma_2\to \Sigma_2\sqcup \Sigma_1
\end{equation}
In case $\Sigma_1=\Sigma_2=\Sigma$ the diffeomorphism $s_{\Sigma,\Sigma}$
is not isotopic to the~identity, hence $C_{s_{\Sigma,\Sigma}}$ is not an~identity cylinder.

\begin{figure}[htb]
	\begin{center}
		\image{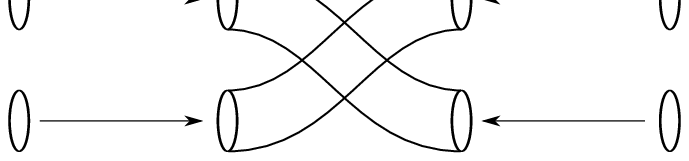}{0pt}
	\end{center}
	\caption{A~cobordism induced by a~permutation of components of a~disjoint sum is not
				equivalent to an~identity cylinder.}\label{fig:dsum-perm}
\end{figure}
\end{example}

\begin{example}\label{ex:cob-torus}
Consider a~family of diffeomophisms of a~two-dimensional torus:
\begin{equation}
\varphi_{m,n}(z,w) = (zw^m,wz^n),\qquad \gcd(m,n) = 1.
\end{equation}
They send a~curve $\gamma(t) = (e^{2\pi it}, e^{2\pi it})$ into non-homotopic curves,
thus the~diffeomorphisms are not isotopic.\footnote{\ 
Details on the~fundamental group and the~mapping class group of a~torus can be found
in~\cite{Rolfsen}.}
Hence all cobordisms $C_{\varphi_{m,n}}$ are different.

A similar observation gives for an~arbitrary manifold $\Sigma$ a~one-to-one correspondence
\begin{equation}
\Diff(\Sigma)/\!_\sim\ \ni [\varphi] \mapsto C_\varphi
\end{equation}
where $\Diff(\Sigma)/\!_\sim$ is the~group of isotopy classes of diffeomorphisms of $\Sigma$.
\end{example}

We will now define basic operations on cobordisms.
\begin{enumerate}
\item The~\term{reversion} of a~cobordism $\cob{M}{\Sigma_0}{\Sigma_1}$ is the~cobordism
		$\cob{M^*}{\Sigma_1}{\Sigma_0}$ given by reversing the~orientations of $M$ but preserving
		the~orientation of $\partial M$.
$$
  \image{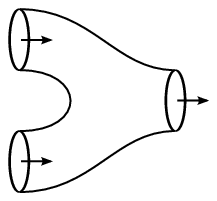}{25pt} \quad\stackrel{*}\longrightarrow\quad \image{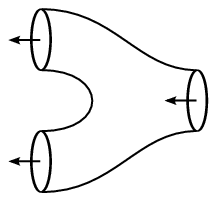}{25pt} \quad =\quad \image{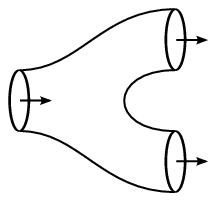}{25pt}
$$
\item The~\term{gluing} of cobordisms $\cob{M_1}{\Sigma_0}{\Sigma_1}$ and~$\cob{M_2}{\Sigma_1}{\Sigma_2}$
		is the~cobordism $\cob{M_1M_2}{\Sigma_0}{\Sigma_2}$ defined as the~sum of manifolds $M_1$ and~$M_2$
		along the~boundary $\Sigma_1$.

\item Say $\cob{M_1}{\Sigma_0}{\Sigma_1}$ and~$\cob{M_2}{\Sigma_1}{\Sigma_2}$ is a~\term{split}
		of $\cob{M}{\Sigma_0}{\Sigma_2}$, whenever $\cob{M}{\Sigma_0}{\Sigma_2}$ is a~gluing
		of $\cob{M_1}{\Sigma_0}{\Sigma_1}$ and~$\cob{M_2}{\Sigma_1}{\Sigma_2}$.
$$
  \image{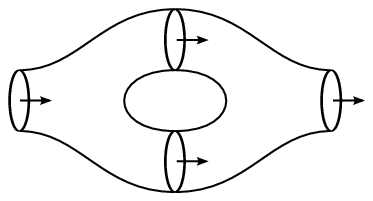}{25pt} \quad\longrightarrow\quad \image{bsplit-ar}{25pt} \quad \image{bmerge-ar}{24pt}
$$

\item The~\term{multiplication} of cobordisms $\cob{M}{\Sigma_0}{\Sigma_1}$ and~$\cob{M'}{\Sigma'_0}{\Sigma'_1}$
		is the~disjoint sum:
$$
\Sigma_0\sqcup\Sigma'_0\longrightarrow M\sqcup M'\longleftarrow\Sigma_1\sqcup\Sigma'_1
$$
      with induced input and output.
$$
  \image{}{7pt} \quad\sqcup\quad \image{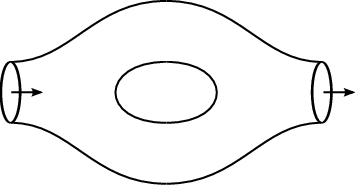}{24pt} \quad = \quad \image{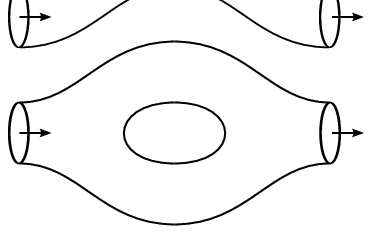}{40pt}
$$
The multiplication of $n$ copies of a~cobordism $M$ is denoted by $M^n$.
\end{enumerate}
All the~operations agree with the equivalence relation of cobordisms.
For details look in~\cite{Kock}.

\begin{lemma}\label{lem:cob-category}
The operations defined above have the following properties:
\begin{enumerate}
\item $(M_1M_2)M_3 = M_1(M_2M_3)$
\item $C_{\Sigma_0} M = M = MC_{\Sigma_1}$, where~$\cob{M}{\Sigma_0}{\Sigma_1}$
\item $M_1M_2 \sqcup N_1N_2 = (M_1\sqcup N_1)(M_2\sqcup N_2)$
\item $C_\Sigma \sqcup C_{\Sigma'} = C_{\Sigma \sqcup \Sigma'}$
\item $(M_1M_2)^* = M_2^* M_1^*$
\item $C_\Sigma^* = C_{\Sigma^*}$
\item $M^{**} = M$
\item $(M\sqcup N)^* = M^*\sqcup N^*$
\item\label{lbl:split-mult}
      if $K_1,K_2$ is a~split of $M\sqcup N$,
      then there exist splits $M=M_1M_2, N=N_1N_2$ such that $K_i = M_i\sqcup N_i$
\end{enumerate}
\end{lemma}
\begin{proof}
All equalities except~\eqref{lbl:split-mult} follow directly from definitions
of the~operations and topological properties of cobordisms. For~\eqref{lbl:split-mult}
take $M_i = K_i\cap M$ and~$N_i = K_i\cap N$.
\end{proof}

\noindent Due to the~first two points of the~lemma~\ref{lem:cob-category}
cobordisms form a~category $\cat{Cob}$ as follows:
\begin{itemize}
\item objects are closed oriented manifolds
\item morphisms are equivalence classes of oriented compact cobordisms
\item the~composition of morphisms is given by the~gluing of cobordisms: $M\circ N := NM$
\item the~identity $\id_\Sigma$ is given by the~cylinder $C_\Sigma$
\end{itemize}
Other points show $\cat{Cob}$ is a~symmetric monoidal category.

\begin{definition}\label{def:cat-monoid}
A~category $\cat{C}$ is called \term{monoidal}, if there exists a~functor
$\otimes\colon\cat{C}\times\cat{C}\to\cat{C}$ called \term{multiplication},
an~object $e\in\Ob(\cat{C})$ called the~\term{unit} and natural equivalences
$R_X\colon X\otimes e\to X$, $L_X\colon e\otimes X\to X$,
$A_{XYZ}\colon X\otimes(Y\otimes Z)\to (X\otimes Y)\otimes Z$ such that the~following diagrams commute:
$$
\xy
\morphism(   0,500)|a|/{->}/<1300,   0>%
	[X\otimes(Y\otimes(Z\otimes W))`(X\otimes Y)\otimes(Z\otimes W);A]
\morphism(1300,500)|a|/{->}/<1300,   0>%
	[(X\otimes Y)\otimes(Z\otimes W)`((X\otimes Y)\otimes Z)\otimes W;A]
\morphism(   0,500)|l|/{->}/<   0,-500>%
	[X\otimes(Y\otimes(Z\otimes W))`X\otimes((Y\otimes Z)\otimes W);\id\otimes A]
\morphism(   0,  0)|a|/{->}/<2600,0>%
	[X\otimes((Y\otimes Z)\otimes W)`(X\otimes(Y\otimes Z))\otimes W;A]
\morphism(2600,  0)|r|/{->}/<   0,500>%
	[(X\otimes(Y\otimes Z))\otimes W`((X\otimes Y)\otimes Z)\otimes W;A\otimes\id]
\endxy
$$
$$
\xy
\Vtriangle[X\otimes(e\otimes Y)`(X\otimes e)\otimes Y`X\otimes Y;A`\id\otimes L`R\otimes\id]
\endxy
$$
Furthermore, if there exists a~natural equivalence $S_{XY}\colon X\otimes Y\to Y\otimes X$
such that $S_{XY}\circ S_{YX} = \id, R_X = L_X\circ S_{Xe}$ and the~following diagram commutes
$$
\xy
\morphism(   0,500)|a|/{->}/<1000,0>[ X\otimes(Y\otimes Z)`(X\otimes Y)\otimes Z ;A]
\morphism(1000,500)|a|/{->}/<1000,0>[(X\otimes Y)\otimes Z` Z\otimes(X \otimes Y);S]
\morphism(2000,500)|r|/{->}/<0,-500>[ Z\otimes(X\otimes Y)`(Z\otimes X)\otimes Y ;A]

\morphism(   0,500)|l|/{->}/<0,-500>[ X\otimes(Y\otimes Z)` X\otimes(Z \otimes Y);\id\otimes S]
\morphism(   0,  0)|b|/{->}/<1000,0>[ X\otimes(Z\otimes Y)`(X\otimes Z)\otimes Y ;A]
\morphism(1000,  0)|b|/{->}/<1000,0>[(X\otimes Z)\otimes Y`(Z\otimes X)\otimes Y ;S\otimes\id]
\endxy
$$
the~monoidal category $\cat{C}$ is called \term{symmetric}.

A~functor $\F\colon\cat{C}\to\cat{D}$ is called \term{monoidal}, if it preserves the~multiplication:
\begin{equation}
\F\circ\otimes_\cat{C} = \otimes_\cat{D}\circ (\F,\F)\qquad\textrm{and}\qquad \F e_\cat{C} = e_\cat{D},
\end{equation}
and~agrees with equivalences: $\F\circ L_X^\cat{C} = L_{FX}^\cat{D}$ and similarly for~$R,A$.
If the~symmetry is also preserved, then~$\F$ is called a~\term{symmetric monoidal functor}.
\end{definition}

\begin{example}
The~category $\cat{Set}$ of sets possesses two monoidal structures given by
the~Cartesian product and by the~disjoint sum. Both of them are symmetric.
\end{example}

\begin{example}
The~category $\cat{Mod}_R$ of $R$-modules is symmetric monoidal with the~multiplication
given by the~tensor product. Another monoidal structure is given by the~exterior
product. This structure is symmetric with a~skew-linear permutation.
\end{example}

\begin{corollary}
The~category of cobordisms $\cat{Cob}$ with the~multiplication is symmetric monoidal.
Moreover, the~reversion is an~contravariant functor which is inverse to itself.
\end{corollary}
\begin{proof}
Due to points (3) -- (8) of the~lemma~\ref{lem:cob-category}, the~multiplication
of cobordisms is a~functor with an~identity $e = \emptyset$, 
whereas the~reversion is an~contravariant functor.
Equivalences $A_{\Sigma\Sigma'\Sigma''},L_\Sigma,R_\Sigma$ and~$S_{\Sigma\Sigma'}$ are given
by the~standard diffeomorphisms $a_{\Sigma\Sigma'\Sigma''},l_\Sigma,r_\Sigma,s_{\Sigma\Sigma'}$:
\begin{align*}
a_{\Sigma\Sigma'\Sigma''}\colon& (\Sigma\sqcup \Sigma')\sqcup \Sigma'' \approx \Sigma\sqcup (\Sigma'\sqcup \Sigma'') &
		\cob{A_{\Sigma\Sigma'\Sigma''}&}{(\Sigma\sqcup \Sigma')\sqcup \Sigma''}{\Sigma\sqcup (\Sigma'\sqcup \Sigma'')} \\
l_\Sigma\colon& \emptyset\sqcup \Sigma\approx \Sigma &
		\cob{L_\Sigma&}{\emptyset\sqcup \Sigma}\Sigma  \\
r_\Sigma\colon& \Sigma\sqcup\emptyset\approx \Sigma &
		\cob{R_\Sigma&}{\Sigma\sqcup\emptyset}\Sigma  \\
s_{\Sigma\Sigma'}\colon& \Sigma\sqcup \Sigma'\approx \Sigma'\sqcup \Sigma &
		\cob{S_{\Sigma\Sigma'}&}{\Sigma\sqcup \Sigma'}{\Sigma'\sqcup \Sigma}
\end{align*}
\end{proof}

Denote by~$\cat{nCob}$ the~subcategory of ($n-$1)-manifolds and~$n$-cobordisms.\footnote{\ 
The~empty set $\emptyset$ is considered as a~manifold of any dimension.
It is both an~object and a~morphism of each category $\cat{nCob}$.}
Hence,
\begin{equation}
\cat{Cob} = \displaystyle\bigcup_{n\in\mathbb{Z}_+}\cat{nCob}.
\end{equation}
and both the~multiplication and the~reversion preserve the~decomposition.

\begin{example}
Hilbert spaces and linear operations with a~tensor product form a~symmetric monoidal category.
A~monoidal functor $\F\colon \cat{nCob}\to\cat{Hilb}$ is called a~\term{topological quantum
field theory} (TQFT). Sometimes it is also assumed that reversion corresponds to conjugation:
$\F(M^*)= \F(M)^*$.
\end{example}

Classification of surfaces gives a~finite presentation of $\cat{2Cob}$.
We state the~theorem below without proof, which can be found in~\cite{Kock}.

\begin{theorem}\label{thm:cob-class}
Any (1+1)-cobordism is generated under the~composition and the~multiplication
by~the~following six cobordisms:

\medskip
\begin{center}
\begin{tabular}{*{11}{c}}
\image{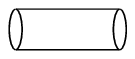}{0pt}  &\quad& \image{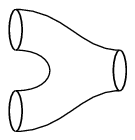}{0pt} &\quad&
\image{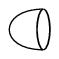}{0pt} &\quad& \image{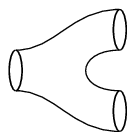}{0pt} &\quad&
\image{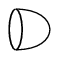}{0pt} &\quad& \image{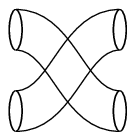}{0pt}	\\
cylinder							&& merge							&&
birth								&& split							&&
death								&& permutation					\\
\end{tabular}
\end{center}

\medskip
\noindent Moreover, any two decompositions define equivalent cobordisms if and only if
one can be obtained from the~other by the~following relations:

\begin{itemize}
\item permutation group relations:

\medskip\noindent
\puthimage{cob-rel-P1}{190pt}\qquad\puthimage{cob-rel-P2}{190pt}

\medskip
\item behaviour of a~birth and a~merge under a~permutation:

\medskip\noindent
\puthimage{cob-rel-BP}{190pt}\qquad\puthimage{cob-rel-PD}{190pt}

\medskip
\item behaviour of a~merge and a~split under a~permutation:

\medskip\noindent
\puthimage{cob-rel-MP}{190pt}\qquad\puthimage{cob-rel-PS}{190pt}

\medskip\noindent
\item associativity and coassociativity laws:

\medskip\noindent
\puthimage{cob-rel-MM}{190pt}\qquad\puthimage{cob-rel-SS}{190pt}

\medskip\noindent
\item commutativity and cocommutativity laws:

\medskip\noindent
\puthimage{cob-rel-PM}{190pt}\qquad\puthimage{cob-rel-SP}{190pt}

\medskip
\item the~unit and the~counit laws:

\medskip\noindent
\puthimage{cob-rel-BM}{190pt}\qquad\puthimage{cob-rel-SD}{190pt}

\medskip
\item the~Frobenius law:

\medskip\noindent
\center{\puthimage{cob-rel-MS}{190pt}}

\end{itemize}
\end{theorem}

\section{A chronology}\label{sec:chcob-def}
Let $M$ be a~cobordism and $\tau\colon M\to I$ its projection on the~unit interval.
It can be seen as a~deformation of a~space $\tau^{-1}(0)$ into $\tau^{-1}(1)$ in time.
Critical points corresponds to the~moments, when the modified space is not a~manifold
-- some critical event occurs (i.e.~in dimension two it can be a~merge, a~split, etc.).
Our task in this chapter is to enrich cobordisms, so we can keep track on such critical events.

\begin{definition}\label{def:chronology}
Let $\cob{M}{\Sigma_0}{\Sigma_1}$ be a~cobordism. A~\term{chronology} on $M$ is a~function
$\tau\colon M\to I$ such that
\begin{enumerate}
\item $\tau^{-1}(0) = M_{in}$
\item $\tau^{-1}(1) = M_{out}$
\item critical points of $\tau$ are non-degenerated
\item there is exactly one critical point for each critical level of $\tau$
\end{enumerate}
\end{definition}

A~cobordism $\cob{M}{\Sigma_0}{\Sigma_1}$ with a~chronology $\tau$ is denoted by $\chcob{M}{\tau}{\Sigma_0}{\Sigma_1}$
or $(M,\tau)$ and is called a~\term{cobordism with chronology} or a~\term{chronological cobordism}.
We will write $M$, if the~chronology is obvious from the~context. Denote by~$\Chron(M)$
the~space of all chronologies on~$M$.

\begin{figure}[thb]
	\begin{center}
		\putimage{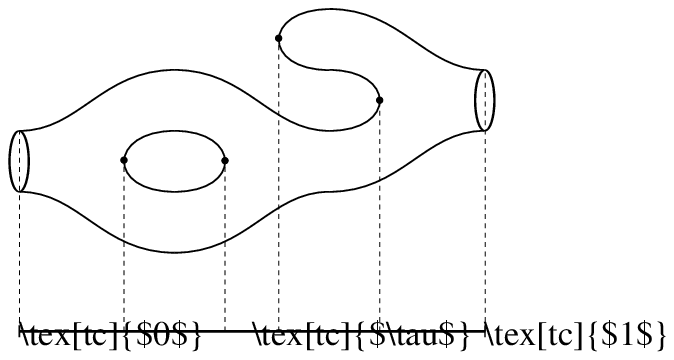}
	\end{center}
	\caption{A~chronology describes the~order of critical events on a~cobordism.}\label{fig:chronology}
\end{figure}

\begin{theorem}[compare \cite{Hirsch} theorem 6.1.2]\label{twr:chron-exist}
Let $\cob{M}{\Sigma_0}{\Sigma_1}$ be an~oriented cobordism. Then chronologies on $M$ form an~open-dense subset in
\begin{equation}
C^{\infty}(M,M_{we},M_{wy}; I,0,1) := \{f\in C^{\infty}(M;I)|\ f(M_{we}) = 0\land f(M_{wy}) = 1 \}
\end{equation}
\end{theorem}

Due to the~theorem, there exists a~chronology for any oriented cobordism.
If $\tau\colon M\to I$ is a~chronology with critical points $p_1,...,p_n$,
then a~diffeomorphism $\varphi\colon M\to M'$ induces on $M'$
a~chronology $\tau' = \tau\circ\varphi^{-1}$ with critical points $\varphi(p_1),...,\varphi(p_n)$.
However, the~critical levels are preserved, so the~structure it gives is too rigid.
We will soften it introducing a~sort of deformations, which preserve the~order of critical points.

\begin{definition}\label{def:chron-isot}
An~\term{isotopy of chronologies} is a~smooth homotopy
$H\colon M\times I\to I$
such that $H_t\colon M\to I$ is a~chronology for each $t\in I$.
\end{definition}

\begin{figure}[hbt]
	\begin{center}
		\image{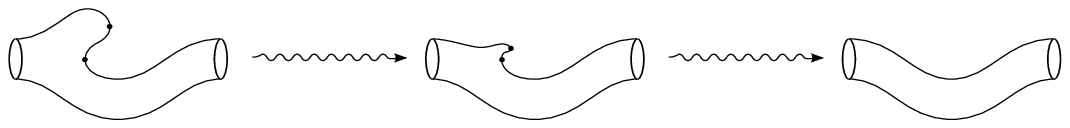}{0pt}
	\end{center}
	\caption{An~example of a~homotopy which is not an~isotopy of chronologies.}\label{fig:cob-homot}
\end{figure}

Recall that a~diffeotopy of a~manifold $\Sigma$ is
a~smooth map $\Phi\colon \Sigma\times I\to\Sigma$ such that $\Phi_0=\id$ and~$\Phi_t$
is a~diffeomorphism for each $t\in\ I$.

\begin{example}\label{ex:chron-isot-rep}
If $f_t\colon I\to I$ is a~diffeotopy of an~interval fixing the~endpoints, we can define an~isotopy
$H^f$ of a~chronology $\tau\colon M\to I$ by~the~following composition
\begin{equation}
H^f(p,t) = f_t(\tau(p))
\end{equation}
It is called a~\term{reparametrization} of $\tau$.
\end{example}

\begin{example}\label{ex:isot-chron-isot}
Let $\Phi\colon M\times I\to M$ be a~diffeotopy of $M$.
Then we can define an~isotopy $H^{\Phi}$ of a~chronology $\tau\colon M\to I$ as the~composition
\begin{equation}
H^{\Phi}(p,t) = \tau(\Phi_t^{-1}(p))
\end{equation}
\end{example}

The~examples above do not exhaust all types of isotopies, because the~first preserves levels
in the~sense that a~level of $H_0$ is a~level of $H_1$ and the~second preserves critical values.
Obviously one may take a~composition of these two, but it is unknown whether all isotopies
can be obtained in this way. The~problem can be reduced to getting a~smooth solution of some
smooth family of linear equations with highest rank but non-invertible matrices.

\begin{conjecture}
Any isotopy $H$ of a~chronology $\tau\colon M\to I$ is of the~form
$$
H(\Phi_t(p),t) = f_t(\tau(p))
$$
for some d $f_t\colon I\to I$ and $\Phi_t\colon M\to M$.
\end{conjecture}

\noindent A~partial result is given in the~lemma~\ref{lem:chcob-isot-rep}.

A~chronology~$\tau\in\Chron(M)$ induces a~linear order on the~set of critical points:
\begin{equation}
p < q \quad\Leftrightarrow\quad \tau(p) < \tau(q)
\end{equation}
In general, a~homotopy does not have to preserve neither the~number of critical points
nor their order (fig.~\ref{fig:cob-homot}).
It is not the case of isotopies and yet an~isotopy induces for~each critical point
a~path on $M$, giving a~natural isomorphism or ordered spaces for both chronologies.

\begin{lemma}\label{lem:chrono-crit}
Let $H\colon M\times I\ni(p,t)\to H_t(p)\in I$ be an~isotopy of chronologies on $M$.
Denote by $p_1 < \dots < p_n$ all critical points of $H_0$.
Then there exist paths $\gamma_i\colon I\to M$ for $i=1,...,n$ such that
$\gamma_i(0)=p_i, \gamma_i(t)\in\crit(H_t)$ and $\gamma_i(t) < \gamma_{i+1}(t)$ for each $t\in I$.
\end{lemma}
\begin{proof}
Critical points of a~chronology are non-degenerated, so:
\begin{equation}
\det\left(\frac{\partial^2 H}{\partial p^2}(p,t)\right) \neq 0,\quad\textrm{for } p\in\crit(H_t),
\end{equation}
and by the~implicit function theorem there exists a~unique smooth solution
for $i=1,\dots,n$ to the~equation:
\begin{equation}
\begin{cases}
\frac{\partial H}{\partial p}(\gamma_i(t), t) &= 0\\
\gamma_i(0)&=p_i
\end{cases}
\end{equation}
defined for~$t\in I$. Suppose for some $t\in I$ we have $H_t(\gamma_i(t)) \geqslant H_t(\gamma_j(t))$.
Due to continuity, there is $t' \leqslant t$ such that $H_{t'}(\gamma_i(t')) = H_{t'}(\gamma_j(t'))$,
so $\gamma_i(t') = \gamma_j(t')$ and by uniqueness of solutions $i=j$, what ends the proof.
\end{proof}

\begin{corollary}
Let $\tau_1,\tau_2$ be isotopic chronologies on~$M$.
Then there exists a~natural isomorphism $(\crit(\tau_1), <) \approx (\crit(\tau_2),<)$.
\end{corollary}

\begin{example}
Paths induced by the~isotopy from example~\ref{ex:isot-chron-isot} are very easy too see.
Indeed, it must be $\gamma_i(t) = \Phi_t(p_i)$, since critical values are preserved by diffeomorphisms.
\end{example}

Now we will define an~equivalence relation on cobordisms with chronologies using
isotopies. 

\begin{definition}\label{def:chcob-equiv}
An~\term{equivalence} of cobordisms $\chcob{M}{\tau}{\Sigma_0}{\Sigma_1}$
and~$\chcob{M'}{\tau'}{\Sigma'_0}{\Sigma'_1}$ is an~equivalence of oriented cobordisms
$\psi\colon M\to M'$ such that $\tau'$ and~$\tau\circ\psi^{-1}$ are isotopic.
\end{definition}

All the~operations defined for cobordisms can be lifted to the~chronological ones.
In case of multiplication some modification is necessary.
\begin{enumerate}
\item The~\term{reversion} of $\chcob{M}{\tau}{\Sigma_0}{\Sigma_1}$ is
		the~cobordism $\chcob{M^*}{\tau^*}{\Sigma_1}{\Sigma_0}$, where $\tau^*(p) = \tau(1-p)$.
$$
  \image{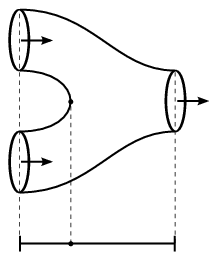}{42pt} \quad\stackrel{*}\longrightarrow\quad \image{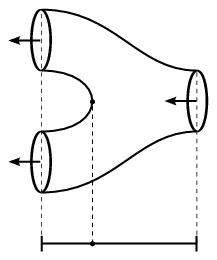}{42pt} \quad =\quad \image{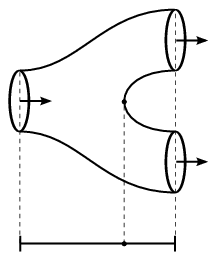}{42pt}
$$

\item The~\term{gluing} of cobordisms $\chcob{M_1}{\tau_1}{\Sigma_0}{\Sigma_1}$ and~$\chcob{M_2}{\tau_2}{\Sigma_1}{\Sigma_2}$
      is the~cobordism $\chcob{M_1M_2}{\tau_1\cdot\tau_2}{\Sigma_0}{\Sigma_2}$, where
$$
(\tau_1\cdot\tau_2)(p) = \begin{cases}
\frac{1}{2}\tau_1(p) & p\in M_1\\
\frac{1}{2}(\tau_2(p)+1) & p\in M_2
\end{cases}
$$
\item Say $\chcob{M_1}{\tau_1}{\Sigma_0}{\Sigma_1}$ and~$\chcob{M_2}{\tau_2}{\Sigma_1}{\Sigma_2}$
		is a~\term{split} of $\chcob M{\tau}{\Sigma_0}{\Sigma_2}$, whenever $\chcob M{\tau}{\Sigma_0}{\Sigma_2}$
		is the~gluing of $(M_1,\tau_1)$ and~$(M_2,\tau_2)$.
$$
\image{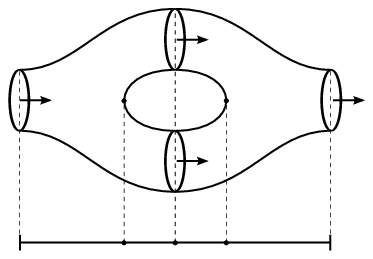}{42pt} \quad\longrightarrow\quad \image{bsplit-ar-chr}{42pt} + \image{bmerge-ar-chr}{42pt}
$$
\item The~\term{multiplication} of cobordisms $\chcob{M}{\tau}{\Sigma_0}{\Sigma_1}$ and~$\chcob{M'}{\tau'}{\Sigma'_0}{\Sigma'_1}$
		is their disjoint sum with shifted chronologies:
$$
(M,\tau)\sqcup (M',\tau') := ((M\sqcup C_{\Sigma'_0})(C_{\Sigma_1}\sqcup M'), (\tau\cdot\pi)\sqcup (\pi\cdot\tau'))
$$
where~$\pi\colon C_\Sigma\to I$ is a~canonical chronology on a~cylinder.
The~multiplication of $n$ copies of a~cobordism $M$ is denoted by $M^n$.
\end{enumerate}

There is no natural chronology for the~disjoint sum of cobordisms --- the~disjoint sum of two chronologies
might not be a~chronology (the~cobordisms may have same critical values).
Thus the~multiplication has to shift both cobordisms: the~first to the~left, and the~second
to the~right.\footnote{\ This can be seen as a~left multiplication, in opposite to the~right version,
where the~first cobordism is pushed to the~right, and the~second to the~left.}

The~multiplication can be generalized, by adding some information how critical points should
lie in the~result.
First notice every regular value $t\in I$ of a~chronology $\tau\colon M\to I$ defines a~split
$M = M_{[0,t]}M_{[t,1]}$.
Define for regular value $a<b$ the~cobordism $M_{[a,b]} = \tau^{-1}([a,b])$
between manifolds $\tau^{-1}(a)$ and~$\tau^{-1}(b)$ with a~chronology given by restriction
$\tau|_{[a,b]}$.
Next, denote by~$S_{m,n}$ the~set of all zero-one sequences of length $m+n$ with exactly $n$ ones:
\begin{equation}
S_{m,n} = \Big\{s\in \{0,1\}^{m+n}\ |\ \sum_{i=1}^{n+m}s_i = n\Big\}
\end{equation}
The~\term{conjugation} of a~sequence $s\in S_{m,n}$ is a~sequence $\bar s\in S_{n,m}$ such that
$\bar s_i = 1-s_i$, whereas the~\term{multiplication} of sequences $s_1\in S_{m_1,n_1}, s_2\in S_{m_2,n_2}$
is defined as a concatenation :
\begin{equation}
s_1s_2(i) = \begin{cases}s_1(i),\quad i \leqslant m_1+n_1 \\ s_2(i-m_1),\quad i > m_1+n_1 \end{cases}
\end{equation}
Let $\cat{ChCob}_0^m$ be the~set of chronological cobordisms with~exactly $m$ critical points.
A~sequence $s\in S_{m,n}$ gives a ,,stretched'' cobordism $(M,\tau)\in\cat{ChCob}_0^m$ as follows:
\begin{enumerate}
\item Let $\crit(\tau) = \{p_1 < \dots < p_m\}$ and $\tau(p_i) = t_i$.
\item Take mid-points of chronology: $q_0 = 0, q_{m+1}=1, q_i = \frac{1}{2}(t_{i+1}+t_i)$ for $0<i<m$.
		Let $\Sigma_i = \tau^{-1}(q_i)$.
\item Build the~cobordism $(M^s,\tau^s)$ due to the~following rules:
	\begin{itemize}
	\item if $s_j = 0$, append the~cylinder $C_{\Sigma_i}$
	\item if $s_j = 1$, append $M_{[q_i,q_{i+1}]}$, where~$i=\sum_{k=1}^j s_k$.
	\end{itemize}
\end{enumerate}
Such a~cobordism $(M^s,\tau^s)$ is obviously equivalent to $(M,\tau)$.
The~\term{generalized multiplication} of cobordisms $(M_1,\tau_1)\in\cat{ChCob}_0^m$
and~$(M_2,\tau_2)\in\cat{ChCob}_0^n$ along $s\in S_{n,m}$ is the~disjoint sum
\begin{equation}
(M_1,\tau_1) \sqcup_s (M_2,\tau_2) = (M_1^{\bar s},\tau_1^{\bar s})\sqcup (M_2^s,\tau_2^s)
\end{equation}
with induced input and~output.
\begin{figure}[htb]
	\begin{center}
	$$
	\image{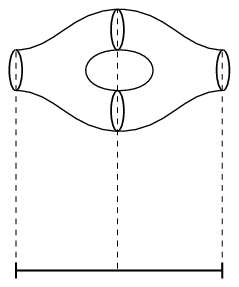}{57pt} \quad\sqcup\quad \image{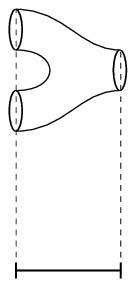}{57pt} \quad = \quad \image{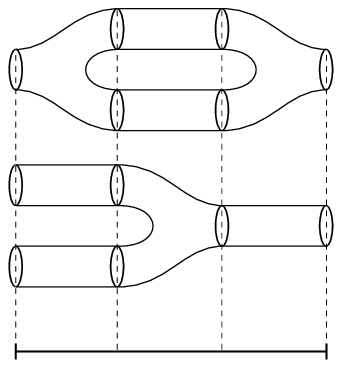}{57pt}
	$$
	\end{center}
	\caption{The~generalized multiplication of two cobordisms along $s=(0,1,0)$.}\label{fig:chron-dsum}
\end{figure}

Reversion, gluing and multiplication satisfy analogous properties to lemma~\ref{lem:cob-category},
except commutativity of multiplication and gluing. Therefore, chronological cobordisms form
a~category $\cat{ChCob}_0$, but the~multiplication is not a~functor.
However, when one of the~arguments is a~fixed cylinder, the~multiplication is a~functor
in the~other variable:
\begin{align}
MN\sqcup C &= (M\sqcup C)(N\sqcup C) \\
C\sqcup MN &= (C\sqcup M)(C\sqcup N)
\end{align}

\begin{definition}\label{def:cat-chron-monoidal}
Let $\cat{C}$ be a~category.
A~\term{chronological multiplication} in $\cat{C}$ is a~function
$\boxtimes\colon\cat{C}\times\cat{C}\to\cat{C}$ being a~half-functor: for each object $X$
the~functions of one variable $(\cdot)\boxtimes X$ and~$X\boxtimes(\cdot)$ are functors:
\begin{align}
(f\circ g)\boxtimes\id_X &= (f\boxtimes\id_X)\circ(g\boxtimes\id_X) \\
\id_X\boxtimes(f\circ g) &= (\id_X\boxtimes f)\circ(\id_X\boxtimes g)
\end{align}
with the~property that $f\boxtimes g = (\id\boxtimes g)\circ(f\boxtimes\id)$.
\end{definition}

\begin{remark}
The~following may be considered as a~\term{left multiplication} due to the~duality
of the~last equality in the~definition. Then a~\term{right multiplication} satisfies
\begin{equation}
f\boxtimes g = (f\boxtimes\id)\circ(\id\boxtimes g).
\end{equation}
\end{remark}

Replacing the~word ,,multiplication'' with ,,chronological multiplication'' in the~definition~\ref{def:cat-monoid}
and treating $L,R,A$ and~$S$ as transformation of half-functors (i.e.~when all arguments except one
are fixed) we obtain the~definition of a~\term{chronological monoidal category}
and~\term{symmetric chronological monoidal category}.
In particular, any monoidal category is chronological monoidal.

\begin{corollary}\label{wns:chcob-chron-monoid}
The~category $\cat{ChCob}_0$ is symmetric chronological monoidal.
It splits into subcategories of $(n-1)$-manifolds with~$n$-cobordisms
\begin{equation}
\cat{ChCob}_0 = \bigcup_{n\in\mathbb{Z}_+}\cat{nChCob}_0
\end{equation}
which are also symmetric chronological and monoidal. Moreover, the~reversion
is a~symmetric chronological and monoidal contravariant functor.
\end{corollary}

\section{Orientation of a~critical point}\label{sec:chcob-cpo}
Besides a~chronology, we will enrich cobordisms with orientation of critical points,
what will break commutativity laws.
Notice that a~chronology $\tau$ induces on~$M$ a~gradient flow $\phi^\tau$
given by a~vector field $\chi_\tau = \nabla\tau$.
Critical points of $\tau$ are exactly the~fixed points of $\phi^\tau$.

Let $p_0\in M$ be a~fixed point of $\phi^\tau$ with Morse index $\mu(p_0)$.
As $\tau$ is a~Morse function, $p_0$ is isolated and hyperbolic.
Choose its~isolating neighbourhood $U$ with orientation induced from~$M$.
Let $W_U^u \subset U$ be the~local unstable manifold (which is diffeomorphic to~$\mathbb{R}^{\mu(p_0)}$).
For $0<\mu(p_0)<n$ is has no natural orientation and we can choose it arbitrary.

\begin{figure}[htb]
	\begin{center}
		\putimage{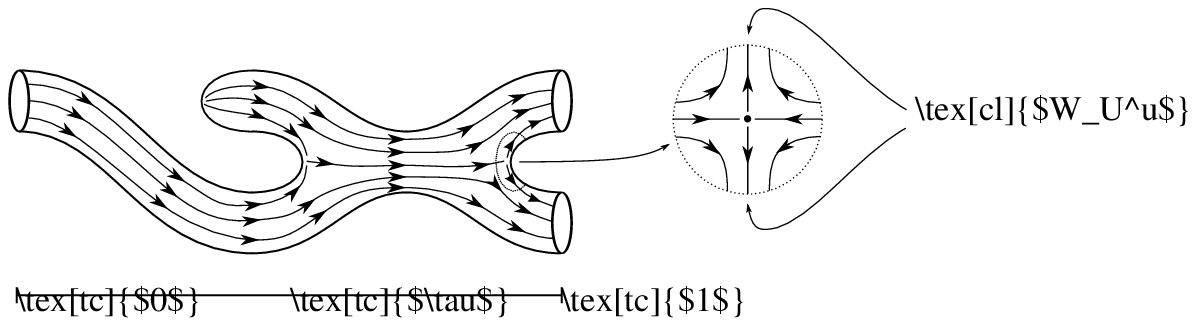}
	\end{center}
	\caption{A~local unstable manifold for a~gradient flow on a~cobordism.}\label{fig:crit-orient}
\end{figure}

Let $U, V$ be two isolating neighbourhoods of $p_0$.
Then $U\cap V$ is also an isolating neighbourhood and
\begin{equation}
W_{U\cap V}^u = W_U^u \cap V = W_V^u \cap U.
\end{equation}
Consider two oriented local unstable manifolds $W_U^u$ and~$W_V^u$ as equivalent,
whenever their orientations agree on the~intersection, i.e. both are equal on $W_{U\cap V}^u$.

\begin{definition}\label{def:crit-orient}
Let $\phi$ be a~gradient flow on $M$ with a~fixed point $p_0$.
\term{An~orientation of a~critical point} $p_0$ is an~equivalence class
of oriented local unstable manifolds.
\end{definition}

An~isotopy of chronologies $H\colon\tau_0\simeq\tau_1$ induces a~homotopy of vector fields
$\nabla H\colon\chi_{\tau_0}\simeq\chi_{\tau_1}$,
which defines a~homotopy of flows $\Phi\colon \phi^{\tau_0}\simeq\phi^{\tau_1}$,
where each $\Phi_t$ is a~flow. It carries orientations of critical points of $\tau$
to critical points of~$\tau'$. Called the~orientation induced by $H$.

\begin{definition}\label{def:crit-orient-eq}
Let $H\colon M\times I\to I$ be an~isotopy of chronologies.
Say the orientations of critical points $(M,H_0)$ and~$(M, H_1)$ \term{agree},
if the orientation of points of $(M,H_1)$ is equal to the~orientation induced by~$H$.
Cobordisms $(M,\tau)$ and~$(M',\tau')$ with oriented critical points are~\term{equivalent},
if there exists an~equivalence of chronological cobordisms $\psi\colon M\to M'$
preserving the~orientation of critical points, i.e. the~orientation of critical
points on $M'$ is equal to the~induced one by~$\psi$.
\end{definition}

\begin{remark}
All operations defined for chronological cobordisms lift to cobordisms with oriented critical points.
Hence, there is a~symmetric chronological monoidal category $\cat{ChCob}$ of cobordisms
with oriented critical points.
\end{remark}

\begin{remark}
An~orientation of the~local unstable manifold induces an~orientation of the~local stable
manifold, as a~complementary to the~orientation of $M$.
In dimension two it corresponds to a~rotation of an~arrow pointing the~chosen out-going
trajectory by $90^\circ$ clockwise.
\end{remark}

\begin{example}
In~the case of (1+1)-cobordisms, a~critical point with Morse index  $0$ or~$2$ has a~natural orientation.
For a~point with index $1$ an~orientation corresponds to a~choice of one of the~two out-going
trajectories and can be visualised by an~arrow.
The following two cobordisms are different, what shows that arrows are essential:

\medskip
\hfill\image{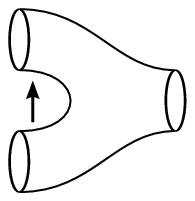}{25pt} \qquad $\neq$ \qquad \image{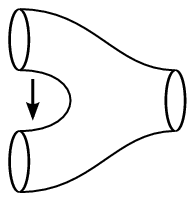}{25pt}\hfill\ 

\medskip
\noindent To show this it is sufficient to notice, that the~positive trajectory (pointed by
the~arrow) in the~left-hand side cobordism starts in the~first (upper) circle, whereas
in the~right-hand side cobordism is starts in the~second (lower) circle.
Hence, they must be~different, since an~equivalence preserve both orientations and the~order
of circles on input.
\end{example}

Unless it is stated differently, all chronological cobordisms in this paper are considered
to have oriented critical points.

\section{\texorpdfstring{A presentation of chronological $(1\!+\!1)$-cobordisms}{A presentation of chronological (1+1)-cobordisms}}\label{sec:chcob-class}
In the~section~\ref{sec:cob-def} we gave a~finite presentation of the~category $\cat{2Cob}$.
The~aim of this section is to give an~analogous description of chronological $(\!1+\!1)$-cobordisms.
We will restrict to a~full subcategory of $\cat{2ChCob}$, generated by a~standard circle.
There is no loss in generality, as for an~arbitrary cobordism $\chcob{M}{\tau}{\Sigma_0}{\Sigma_1}$
there are diffeomorphisms $\varphi\colon\Sigma_0\rightarrow n\S1$
and~$\psi\colon\Sigma_1\rightarrow m\S1$ together with a~cobordism $\chcob {M'}{\tau'}{n\S1}{m\S1}$
forming a~commutative diagram
$$
\xy
\square[X`n\S1`Y`m\S1;C_\varphi`M`M'`C_\psi]
\endxy
$$
\noindent so that $M = C_\varphi M' C_{\psi^{-1}}$.

\begin{theorem}\label{thm:chcob-pres}
The~category of chronological $(1\!+\!1)$-cobordisms $\chcob M{\tau}{\Sigma_0}{\Sigma_1}$ is generated
under composition and multiplication by the~following cobordisms:

\medskip
\begin{center}
\begin{tabular}{*{7}{c}}
\image{cob-birth}{0pt} &\quad& \image{cob-death}{0pt} &\quad&
\image{cob-tube}{0pt}  &\quad& \image{cob-perm}{0pt}    \\
a~birth                &\quad& a~death                &\quad&
a~cylinder             &\quad& a~permutation            \\
\\
\image{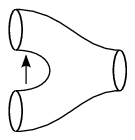}{0pt} && \image{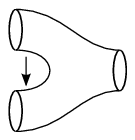}{0pt} &&
\image{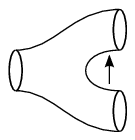}{0pt} && \image{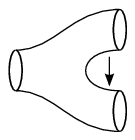}{0pt} \\
a~positive merge         && a~negative merge         &&
a~positive split         && a~negative split
\end{tabular}
\end{center}
modulo the following relations:
\begin{itemize}
\item permutation groups relations:

\medskip\noindent
\puthimage{cob-rel-P1}{190pt}\qquad\puthimage{cob-rel-P2}{190pt}

\medskip
\item behaviour of a~birth and a~merge under a~permutation:

\medskip\noindent
\puthimage{cob-rel-BP}{190pt}\qquad\puthimage{cob-rel-PD}{190pt}

\medskip
\item behaviour of a~merge and a~split under a~permutation:

\medskip\noindent
\puthimage{chcob-rel-MP}{190pt}\qquad\puthimage{chcob-rel-PS}{190pt}

\medskip\noindent
\item anticommutativity and anticocommutativity laws:

\medskip\noindent
\puthimage{chcob-rel-PM}{190pt}\qquad\puthimage{chcob-rel-SP}{190pt}

\medskip
\end{itemize}
\end{theorem}

In opposition to the~theory of classical cobordisms, there is no associativity law.
Due to anticommutativity laws we can reduce the~number of generators taking
only a~positive split and a~merge. Since now, if not stated differently,
all critical points of index $1$ have positive orientation.

The~proof of the~theorem~\ref{thm:chcob-pres} is divided into several steps.
First we will show that the~given set generates the~category. Next we will
show that any decomposition of a~given cobordism can be reduced to be of a~special
type called a~normal form. Finally, we will see that two normal forms are related
by the~relations listed above.

For the~first step we will use the~following results of the~theory of cobordisms.

\begin{theorem}[cf. \cite{Hirsch}, theorem 6.2.2]\label{thm:cob-no-crit}
Let $\chcob{M}{\tau}{\Sigma_0}{\Sigma_1}$ be a~cobordism with no critical points.
Then there exists a~diffeomorphism $\psi\colon\Sigma_0\times I\to M$ agreeing with~$\tau$,
i.e. the~following diagram commutes:
$$
\xy
\qtriangle(0,0)|abr|/{->}`{->}`{->}/<500,400>[\Sigma_0\times I`M`I;\psi`\pi`\tau]
\endxy
$$
\end{theorem}

\begin{theorem}[cf. \cite{Hirsch}, theorem 6.4.2]\label{thm:cob-one-crit}
Let $\chcob{M}{\tau}{\Sigma_0}{\Sigma_1}$ be a~connected $(\!1+\!1)$-chronological cobordism
with exactly one critical point of index $1$. Then $M$ is diffeomorphic with a~disk
missing two smaller disks inside.
\end{theorem}

Directly from the~theorem~\ref{thm:cob-no-crit} follows the~classification of
cobordisms with no~critical points.

\begin{corollary}\label{cor:chcob-no-crit}
Let $\chcob M{\tau}{n\S1}{m\S1}$ be a~cobordism with no critical points.
Then $n=m$ and~$(M,\tau)$ is a~cylinder generated by a~permutation of components of $n\S1$.
\end{corollary}
\begin{proof}
Let $\psi\colon n\S1\times I\to M$ be a~diffeomorphism given by the~theorem~\ref{thm:cob-no-crit}.
It induces an~embedding $\varphi\colon m\S1\to n\S1\times I$ and sets an~equivalence
of cobordisms $M$ and~$C_\varphi$, so $n=m$.
A~diffeomorphism preserving an~orientation of a~circle is isotopic to identity,
so by lemma~\ref{lem:cob-equiv-iso} we may assume $\varphi$ is a~permutation
of components of $n\S1$.
\end{proof}

The~second theorem gives us a~list of generators of $\cat{2ChCob}$.

\begin{lemma}\label{lem:chcob-gen}
Any chronological cobordism $\chcob M{\tau}{\Sigma_0}{\Sigma_1}$ decomposes
into generators listed in the~theorem~\ref{thm:chcob-pres}.
\end{lemma}
\begin{proof}
Use the induction on the number of critical points $n$.

If $(M,\tau)$ has no critical points, due to corollary~\ref{cor:chcob-no-crit}
it is generated be a permutation $\sigma$.
Taking its decomposition into transpositions $\sigma = t_1\cdot\ldots\cdot t_n$, $t_i = (i\ i+1)$, we obtain
\begin{equation}\label{eq:M-id-div}
M = C_{t_1}\cdot\dots\cdot C_{t_n}
\end{equation}
where each $C_{t_i}$ is a permutation of neighbouring circles.

Assume $n\geqslant 1$ and $t$ is a critical value of $\tau$ at $p$.
Take $\varepsilon > 0$ such that $M_{[t-\varepsilon, t+\varepsilon]}$ has exactly one critical point.
It gives us a decomposition
\begin{equation}\label{eq:M-div}
M = M_{[0,t-\varepsilon]}M_{[t-\varepsilon, t+\varepsilon]}M_{[t+\varepsilon, 1]}
\end{equation}
Let $N$ be a connected component of $M_{[t-\varepsilon, t+\varepsilon]}$ with the critical point $p$.
Then one of the following occurs:
\begin{itemize}
\item $p$ is a local minimum and~$N$ is a birth
\item $p$ is a local maximum and~$N$ is a death
\item $p$ is a saddle and due to theorem~\ref{thm:cob-one-crit}
      $N$ is a merge or a split (according to the~value of $\tau$ on boundary components of $N$)
\end{itemize}
The~other components of $M_{[t-\varepsilon, t+\varepsilon]}$ have no critical points,
hence they are generated by a~permutation and cylinders.
The~inductive hypothesis gives the decomposition of the~other two terms in~(\ref{eq:M-div}),
what ends the~proof.
\end{proof}

To prove the~second part of the~theorem, we will distinguish some special decompositions of cobordisms.
Then we will show that an~arbitrary decomposition can be reduced to a~special one, using relations
listed in the~theorem.

\begin{definition}
Say $M = P_1M_1P_2\dots P_nM_nP_{n+1}$ is in \term{a~normal form} if
\begin{itemize}
\item each $P_i$ is a~permutation of circles
\item each $M_1$ decomposes as $N_i\sqcup C_{\S1}^k$, where $N_i$ is the~unique component with a~critical point
\item if $N_i$ is a~split or a~merge, the~critical point has a~positive orientation
\end{itemize}
\end{definition}

\begin{lemma}\label{lem:chcob-norm-form-ex}
An arbitrary decomposition $M = P_1M_1\dots P_nM_nP_{n+1}$ can be reduced to a~normal form.
\end{lemma}
\begin{proof}
Notice that the relations from the theorem~\ref{thm:chcob-pres} implies $M = G\cup C^n$
is equal to $P(G\sqcup C^n)P'$, where $G$ is a~generator and both $P$ and $P'$ are generated
by permutations. We give below such a reduction for a split:

\medskip
\begin{center}
	\image{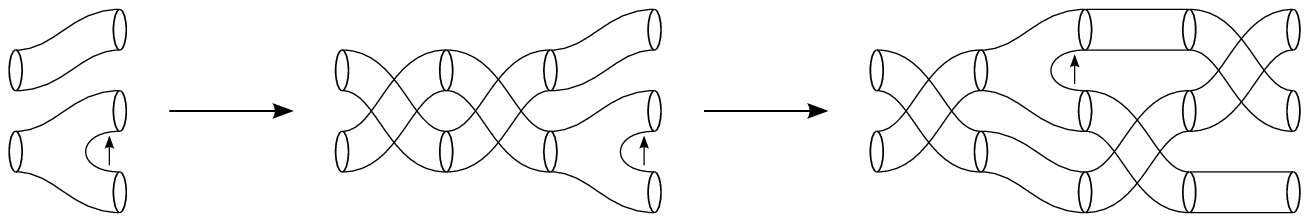}{0pt}
\end{center}
\end{proof}

This was the~easier part. The~harder is to show that any two normal decompositions
are related. To obtain this result we need three technical lemmas.

\begin{lemma}\label{lem:chcob-isot-rep}
Let $H\colon M\times I\to I$ be an~isotopy of chronologies such that $H_0$ and $H_1$
has the~same critical values. Then there exists an~isotopy $f_t\colon I\to I$ such that
a~map $(p,t)\mapsto f_t(H(p,t))$ has fixed critical values for all $t\in I$.
\end{lemma}
\begin{proof}
Let $t_1,\dots,t_n$ be the~critical values of $H_0$ and set $t_0 = 0, t_{n+1} = 1$.
Put
\begin{equation}
\varepsilon = \frac{1}{2}\min_{i=1,\dots,n}\left\{t_{i+1}-t_i\right\}
\end{equation}
and pick a~smooth increasing function $h\colon I\to I$ with all derivatives vanishing at $0$ and $1$
such that $h(0) = 0, h(1) = 1$. Using $h$, define a~homotopy $g_t\colon I\to I$ such that
\begin{align*}
g_t(H_t(\gamma_i(t))) = t_i - \varepsilon(H_t(\gamma_i(t)))
\end{align*}
Then $f_t(x) = g_t(x) + \varepsilon t$ is the~desired isotopy.
\end{proof}

\begin{lemma}\label{lem:chcob-diff-equiv}
Let $(M,\tau)$ and~$(M',\tau')$ be~equivalent chronological cobordisms.
Suppose $\varphi\colon M\to M'$ is a~diffeomorphism which agrees with chronologies and inputs of $M$ and~$M'$,
and preserves orientations of critical points. Then up to an~isotopy
$\varphi$ agrees also with outputs.
\end{lemma}
\begin{proof}
At first consider the~case $\tau$ and~$\tau'$ have no critical points.
Without loss of generality we may assume that
\begin{align*}
M\textrm{ has the form }&\Sigma_0\stackrel{\id}\longrightarrow\Sigma_0\times I \stackrel{f}\longleftarrow\Sigma_1\\
M'\textrm{ has the form }&\Sigma_0\stackrel{\id}\longrightarrow\Sigma_0\times I \stackrel{g}\longleftarrow\Sigma_1
\end{align*}
and $\varphi|_{\Sigma_0\times 0} = \id$. Obviously, $\varphi|_{\Sigma_0\times 1}\simeq\varphi|_{\Sigma_0\times 0}$.
Since $M$ and~$M'$ are equivalent, $f\simeq g$ and the~following triangle commutes up to an~isotopy
$$
\xy
  \morphism(1000,250)|a|/{->}/<-500,250>[\Sigma_1`\phantom{M};f]
  \morphism(1000,250)|a|/{->}/<-500,-250>[\phantom{\Sigma_0}`\phantom{M'};g]
  \morphism(500,500)|l|/{->}/<0,-500>[M`M';\varphi]
\endxy
$$
what proves the~hypothesis.

Suppose now that $\tau$ and $\tau'$ have critical points.
Decompose $M$ and $M'$ into terms without critical points and terms with exactly one critical point.
We may assume the~latter are of the~form $G\sqcup C$, where $C$ is an~identity cylinder and $G$ is a~generator.
Therefore, it remains to show the~hypothesis for generators with one critical point. It holds trivially for
all except a~split and in this case $\varphi$ maps the~positive output to the~positive one
and similar for the~negative one, because it preserves orientations of critical points.
\end{proof}

\begin{lemma}\label{lem:chcob-MN-M-N}
Let $(MN, \tau)$ and $(M'N', \tau')$ be two equivalent cobordisms with chronologies and
suppose that $M$ and $M'$ are also equivalent. Then $N\simeq N'$ as cobordisms with chronologies.
\end{lemma}
\begin{proof}
Without loss of generality we may assume $MN=M'N'$ as cobordisms without chronologies
and that $\tau'$ has same critical values as $\tau$ (since $M$ and $M'$ have
equal numbers of critical points and the~same holds for $N$ and $N'$, there is such
a~reparametrization of $\tau'$ fixing $\frac{1}{2}$).
Let $H\colon\tau\simeq\tau'$ be an~isotopy of chronologies.
Applying a~reparametrization from lemma~\ref{lem:chcob-isot-rep} we may assume
that critical values are fixed by $H$ for all $t\in I$.

Consider a~set
$$
L=H^{-1}(1/2)=\{(p,t)|\ H(p,t) = 1/2 \}\subset MN\times I
$$
Since $\frac{1}{2}$ is a~regular value of $\tau$, so is of $H$ and $L$ is a~cobordism
from $\tau^{-1}(1/2) = \Sigma_1$ to $\tau'^{-1}(1/2) = \Sigma_2$.
To show that $L$ is a~cylinder, consider a~projection $\pi\colon L\to I$, $\pi(p,t) = t$,
which is a~Morse function with no critical points.

$L$ induces an~isotopy $\Sigma_1\times I\to MN$, $(p,t)\mapsto H(p,t)$, which extends to a~diffeotopy
$\Phi\colon MN\times I\to MN$. We may assume that $\Phi$ is constant on boundary of $MN$.
To show that $\Phi_1|_N\colon N\to N'$ is the~desired equivalence we need to check that it
agrees with the~input and the~output of $N'$. The~second holds trivially and
the~first is guaranteed by the~lemma~\ref{lem:chcob-diff-equiv}.
Indeed, $M$ is equivalent to $M'$ and the~lemma assures $\Phi_1$ agrees with outputs
of $M$ and $M'$, which are at the~same time inputs of $N$ and $N'$.
\end{proof}

Now we are ready to prove the~last lemma, required for the~proof of the~main theorem in this section.

\begin{lemma}\label{lem:chcob-norm-form-eq}
Normal forms of equivalent cobordisms are equal up to the~relations listed in the~theorem~\ref{thm:chcob-pres}.
\end{lemma}
\begin{proof}
Let $(M,\tau)$ and $(M',\tau')$ be equivalent cobordisms with normal forms
\begin{equation}
M = P_1M_1 \dots P_nM_nP_{n+1}, \qquad M' = P_1'M_1' \dots P_n'M_n'P_{n+1}'
\end{equation}
For $n=0$ the~lemma holds due to corollary~\ref{cor:chcob-no-crit}.

Assume $n>0$. The~equivalence of cobordisms forces $M_1$ and $M_1'$ to be equal:
the~have the~same number of inputs, outputs and a~critical point of the~same index.
Suppose $M_1 = G\sqcup C^n$, where $G$ is a~merge (other cases are proven in the~same way).
Denote by $i,j,i',j'$ the~input components of $P_1$ and~$P_1'$ such that
\begin{equation}
P_1(i) = P_1'(i') = 1, \qquad P_1(j) = P_1'(j') = 2.
\end{equation}
Since $M$ and $M'$ are equivalent, $i=i'$ and~$j=j'$
(components sent by $P_1$ and $P_1'$ to inputs of $G$ are determined uniquely,
whereas the~orientation of the~critical point of $G$ distinguishes the~input circles).
Therefore there is a~permutation cobordism $S_1$ such that
$P_1' = P_1(C^2 \sqcup S_1)$. Finally
\begin{equation}
P_1'M_1 = P_1(C^2 \sqcup S_1)(G \sqcup C^{n-2}) = P_1(G\sqcup C^{n-2})(C \sqcup S_1) = P_1M_1S_1'
\end{equation}
and by the~lemma~\ref{lem:chcob-MN-M-N}
\begin{equation}
P_2M_2\dots P_nM_nP_{n+1} = (S_1'P_2')M_2'\dots P_n'M_n'P_{n+1}'
\end{equation}
The~inductive hypothesis ends and the~proof.
\end{proof}

After proving all those lemmas we can connect them to obtain the~presentation of $\cat{2ChCob}$.

\begin{proof}[Proof of the theorem~\ref{thm:chcob-pres}]
The~first part is given by the~lemma~\ref{lem:chcob-gen}.
For the~second part, notice that the~listed relations do not change the~equivalence class of a~cobordism.
To see the~list is complete, let $M$ and~$M'$ be equivalent cobordisms with decompositions
\begin{equation}
M = P_1M_1 \dots P_nM_nP_{n+1}, \qquad M' = P_1'M_1' \dots P_n'M_n'P_{n+1}'.
\end{equation}
Due to the lemma~\ref{lem:chcob-norm-form-ex} we may assume both decompositions are in normal forms.
Now use the~lemma~\ref{lem:chcob-norm-form-eq} to end the proof.
\end{proof}

\begin{remark}
Taking into account the~reversion of cobordisms, the~set of generators can be reduced
to a~birth, a~positive merge and a~permutation. Then the~set of relations can be
restricted to those dealing with births, positive merges and permutations.
\end{remark}

As an~application of the~theorem~\ref{thm:chcob-pres} we will introduce the~$2$-index
of a~chronological cobordism. Obviously, the~number of critical points of a~chosen type is fixed,
so the~following definition is independent on a~decomposition of a~cobordism.

\begin{definition}\label{def:cob-type}
Let $M$ be a~chronological cobordism and take its~decomposition.
Denote by letters~$m, b, s$ and~$d$ the~amounts of merges, births, splits and deaths
respectively. A~pair $\tcob(M) = (m-b, s-d)$ is called \term{the~2-index of a~cobordism} $M$.
If both numbers are equal to zero, we will write $\tcob(M)=0$.
\end{definition}

There is a~simply correspondence between 2-indices of two cobordisms, their composition
and multiplication. Furthermore, the~2-index of a~cobordism imposes conditions
on the~number of inputs and outputs, especially for cobordisms of type zero.

\begin{theorem}\label{thm:tcob-prop}
Let $\cob M{n\S1}{m\S1}$ and~$\cob N{m\S1}{k\S1}$ be $(\!1+\!1)$-chronological cobordisms
and put $\tcob(M) = (\alpha,\beta)$. Then
\begin{enumerate}
\item $\tcob(MN) = \tcob(M\sqcup N) = \tcob(M) +\tcob(N)$
\item $m - n = \beta - \alpha$
\item\label{thm-pt:tcob-0} if $\tcob(M) = 0$, then~$n=m\in\{0,1\}$
\item $\chi(M) = -\alpha-\beta$
\item $g(M) = 1 + \frac{1}{2}(\alpha+\beta-m-n)$ if $M$ is connected, where~$g(M)$ is the~genus of $M$
\end{enumerate}
\end{theorem}
\begin{proof}
The~first point holds due to the~definition of $\tcob$.
Obviously (2) holds for generators of $\cat{2ChCob}$
and other cases goes from (1). The~point (3) is a~special case of (2).
and (4) is the~formula for the~Euler characteristic:
\begin{equation}
\chi(M) = \Sigma_{x\in M}(-1)^{\mu(x)} = -\alpha-\beta
\end{equation}
where~$\mu(x)$ is the~Morse index of $x$.
Using the~relation between the~Euler characteristic, the~genus and the~number
of components of a~given cobordism:
\begin{equation}\label{eq:genus-euler}
\chi(M) = 2-2g(M)-(n+m)
\end{equation}
we obtain (5).
\end{proof}

\section{A change of a~chronology}\label{sec:chcob-chchc}
An~isotopy of chronologies preserves the~order of critical points as well as their characters
(i.e. merge is still a~merge etc.). Now we will allow some changes of chronologies
and introduce relations between two cobordisms differing by such a~change.

Let $H\colon M\times I\to I$ be a~smooth homotopy.
\term{A~critical moment} of $H$ is $t\in I$ such that $H_t$ is not a~chronology.
Assume $t_0$ is an~isolated critical moment of $H$ and one of the~following occurs:
\begin{enumerate}[label=(CHCH\arabic*)]
\item\label{chch-I} $H_{t_0}$ has two critical points at some level and for a~small $\varepsilon > 0$
		the~chronologies $H_{t_0-\varepsilon}$ and~$H_{t_0+\varepsilon}$ are not isotopic
\item\label{chch-II} $H_{t_0}$ has a~degenerated critical point and for a~small $\varepsilon > 0$
		$H_{t_0+\varepsilon}$ has two critical points more than $H_{t_0-\varepsilon}$
\item\label{chch-III} $H_{t_0}$ has a~degenerated critical point and for a~small $\varepsilon > 0$
		$H_{t_0+\varepsilon}$ has two critical points less than $H_{t_0-\varepsilon}$
\end{enumerate}
Regard such changes to be respectively of type I, II and~III.

\begin{definition}\label{def:chron-change}
Let $\chcob M{\tau}{\Sigma_0}{\Sigma_1}$ be a chronological cobordism.
\term{A~change of the~chronology} $\tau$ into~$\tau'$ is a~homotopy $H\colon M\times I\to I$
such that $H_0 = \tau$, $H_1 = \tau'$ and~$H_t$ has finitely many critical moments $t_1,\dots,t_n$,
each of type I, II or III.
If all critical moments of $H$ has the~same type, say $H$ is of this type.
A~change $H$ with exactly one critical moment is called \term{an~elementary change of a~chronology}.
\end{definition}

A~change of a~chronology $H$ from~$\tau$ to~$\tau'$ will be denoted by $\chch{H}{\tau}{\tau'}$.
If $\tau_i$ and~$\tau'_i$ are isotopic for~$i=0,1$, then regard changes
$\chch{H}{\tau_0}{\tau_1}$ and~$\chch{H'}{\tau'_0}{\tau'_1}$ as \term{equivalent}.
In particular, a~change $\chch H{\tau}{\tau'}$ is \term{trivial}, if $\tau$ and~$\tau'$ are isotopic.

\begin{figure}[hbt]
	\begin{center}
		\includegraphics{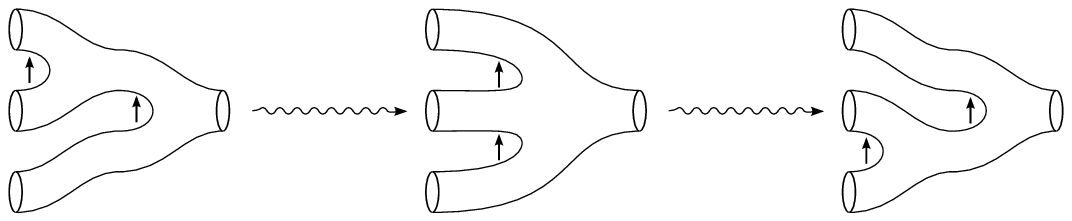}
	\end{center}
	\caption{An~elementary change of a~chronology of type I. The~middle state visualise the~critical moment.}
	\label{fig:chron-change}
\end{figure}

\begin{remark}
If $H_{t_0}$ has a~degenerated critical point, then more than two critical points can be created.
For instance, $H$ can create at the~same time a~split, a~merge, a~birth and a~death.
Thus in (CHCH2) and~(CHCH3) the~condition on the~number of critical points is essential.
In case of (CHCH1) the~condition on non-isotopicity of chronologies guarantees
that an~elementary change of type~I is non-trivial.
\end{remark}

Given two changes of chronologies $\chch{H}{\tau_0}{\tau_1}$ and~$\chch{H'}{\tau_1}{\tau_2}$
define their \term{composition} $\chch{H\cdot H'}{\tau_0}{\tau_2}$, as a~change of a~chronology
given as follows:
$$
H\cdot H'(p,t) = \begin{cases}H(p,2t),&t\leqslant\frac{1}{2}\\ H'(p,2t-1),&t\geqslant\frac{1}{2}\end{cases}
$$
and smoothed near $t=1/2$ if necessary.
Directly from the~definition, this operation agrees with the~equivalence relation
of changes of chronologies and up to equivalence is associative, has neutral elements
(trivial changes) and every change $\chch{H}{\tau_0}{\tau_1}$ has an~inverse
$\chch{H^{-1}}{\tau_1}{\tau_0}$.

\begin{remark}\label{rm:chch-cat}
Isotopy classes of chronologies on a~cobordism $M$ with equivalence classes of changes
of chronologies form a~category $\cat{Chron}(M)$.
Notice, that every morphism is an~isomorphism.\footnote{\ Such a~category is called a~groupoid.}
\end{remark}

Elementary changes of chronologies affect the~set of critical points in a~way
that can be easily described.
Indeed, the~techniques from the~proof of the~lemma~\ref{lem:chrono-crit}
can be used to show the~following three results.

\begin{lemma}\label{lem:chrono-crit-I-type}
Let $H\colon M\times I\to I$ be an~elementary change of a~chronology of type~I
and $p_1 < \dots < p_n$ be all critical points of $H_0$.
Then there exist $1 \leqslant i_0 \leqslant n$ and paths $\gamma_i\colon I\to M$ such that
$\gamma_i(0) = p_0, \gamma_i(t)\in\crit(H_t)$ and
$\gamma_i(t) < \gamma_j(t)$ for each $t\in I$ and $i<j, (i,j)\neq(i_0,i_0+1)$.
Moreover, $\gamma_{i_0}(1) > \gamma_{i_0+1}(1)$.
\end{lemma}

\begin{lemma}\label{lem:chrono-crit-II-type}
Let $H\colon M\times I\to I$ be an~elementary change of a~chronology of type~II
and $p_1 < \dots < p_n$ be all critical points of $H_0$.
Then there exist paths $\gamma_i\colon I\to M$ such that
$\gamma_i(0) = p_0, \gamma_i(t)\in\crit(H_t)$ and
$\gamma_i(t) < \gamma_{i+1}(t)$ for each $t\in I$ and $1\leqslant i < n$.
\end{lemma}

\begin{lemma}\label{lem:chrono-crit-III-type}
Let $H\colon M\times I\to I$ be an~elementary change of a~chronology of type~III
and $p_1 < \dots < p_n$ be all critical points of $H_0$.
Then there exist $1 \leqslant i_0 \leqslant n$ and paths
$\gamma_i\colon I\to M$ for $i\neq i_0, i_0+1$ such that
$\gamma_i(0) = p_0, \gamma_i(t)\in\crit(H_t)$ and
$\gamma_i(t) < \gamma_j(t)$
for each $t\in I$ and $i<j$.
Moreover $\crit(H_1) = \{\gamma_i(1)\ |\ i\neq i_0,i_0+1\}$.
\end{lemma}

The~lemmas guarantee that the~description of an~elementary change of a~chronology given
by a~pair of critical points which are permuted (type~I), created (type~II) or deleted
(type~III) is unambiguous.

\begin{remark}\label{rm:chron-change-triv}
There exist nontrivial changes of chronologies preserving the~cobordism.
This is because cobordisms are considered up to equivalence which is stronger
that just an~isotopy of chronologies.
The~following pictures show two such non-trivial changes in dimension two,
which preserve cobordisms up to the~orientation of critical points
(the~normal form is given on the~right-hand side).

\medskip
\begin{enumerate}[label=(T\arabic*)\qquad]
\item \image{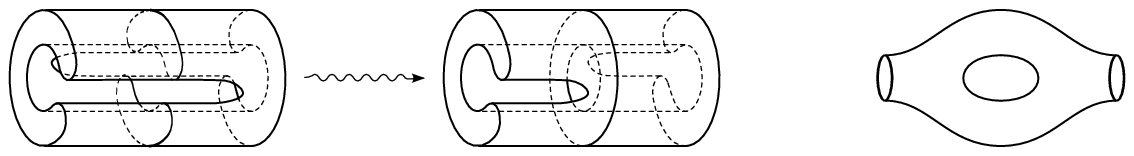}{18pt}

\medskip
\item \image{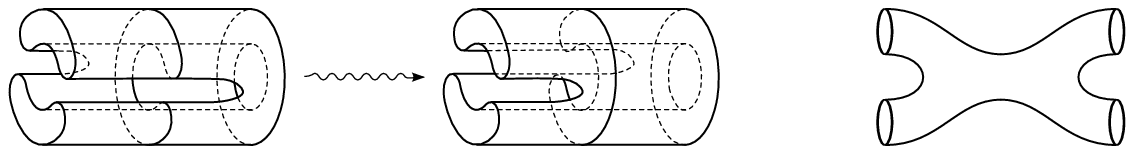}{18pt}
\end{enumerate}

\medskip\noindent
The~change (T1) preserves both orientations of the~merge and the~split,
whereas (T2) changes orientation of one point.
Furthermore, reversing the~orientation of the~merge or the~split in the~left-hand side cobordism,
after applying the~change, results in both cases in reversing the~orientation of the~other point.

In a~change of a~chronology which preserves the~cobordism characters of permuted critical points
must be changed. The~Morse index, as a~local property, is preserved (see the~lemma~\ref{lem:chrono-crit-I-type}),
so in dimension two such a~chronology has to permute a~merge and a~split lying in the~same
component of the~cobordism. Depending on the~beginning order of these point, there are exactly two
such elementary changes:
\begin{enumerate}
\item a~split of a~circle in one point and then a~merge in another point -- (T1)
\item a~merge of two circles in on point and then a~split in another point -- (T2)
\end{enumerate}
\end{remark}

Every change of a~chronology is equivalent to a~composition of homogeneous changes
(i.e. in which each critical moment is of the~same fixed type).

\begin{theorem}\label{thm:ch-ch-nform}
For a~given change of a~chronology $H$ there exist changes $P,C,D$ of types I, II and~III
respectively such that $H\sim C\cdot P\cdot D$.
\end{theorem}
\begin{proof}
At first notice, that a~creation of two points can be pulled back before any other change,
i.e. for an~elementary change $H$ and a~change $H_c$ of type~II there is a~change $H'$
such that
$$
H\cdot H_c \sim H_c\cdot H'
$$
Indeed, let $\chch{H_c}{(M,\tau)}{(M',\tau')}$ creates critical points $p_1$ and~$p_2$
on a~cylinder $M_{[t_0,t_1]}$. The one of the~following occurs:
\begin{itemize}
\item $H$ is trivial on~$M_{[t_0,t_1]}$. We may assume it is constant on this region
		and take as $H'$ a~constant change on $M_{[t_0,t_1]}$ and equal to $H$ out of this region.
\item $H$ carries some critical points through $M_{[t_0,t_1]}$. Then correct it to $H'$
		by adding appropriate permutations.
\item $H$ creates points in $M_{[t_0,t_1]}$ and sends them beside this region.
		Then take as $H'$ a~homotopy creating points $q_1,q_2$ of the~same types
		between $p_1$ and~$p_2$, then permuting appropriate points.
\item $H$ deletes points in $M_{[t_0,t_1]}$, which are beside this region in time $t=0$.
		Then take as $H'$ an~inverse change to the~one describe in the~previous case.
\end{itemize}
Hence, we may assume $H = C\cdot H'$, where $H'$ has only critical moments of type I and~III
and $C$ has only critical moments of type II.
In a~similar way any change $H_d$ of type~III can be pushed to the~end of $H$, what ends the~proof.
\end{proof}

If $C\cdot P\cdot D$ is a~trivial change, then $C$ and~$D$ has to create and delete
the~same amount of critical points at each component of $M$.
Therefore, before deleting points $p_1$ and~$p_2$ one can permute them simultaneously with others
(for example just after a~permutation of $p_1$ and~$q$ there is a~permutation of $p_2$ and~$q$)
and this procedure preserves the~equivalence class of the~change.
Hence, one can assume $D$ deletes points in the~same places, where points are created by $C$.

\begin{corollary}\label{wns:ch-ch-triv-nform}
If $C\cdot P\cdot D$ is a~trivial change, then we may assume $D = C^{-1}$.
In this situation $P$ is a~trivial change.
\end{corollary}

Due to the~lemma~\ref{lem:chrono-crit-I-type} a~change $H$ of type~I induces
a~permutation of critical points $\sigma_H$. Obviously $\sigma_H=\id$ for a~trivial
change $H$. However, it is not the~case of a~change which preserves the~cobordism,
as it was shown in the~remark~\ref{rm:chron-change-triv}.
In dimension two even preserving types of critical points is not enough:
\begin{equation}\label{fig:chron-nontriv-change-torus}
\putimage{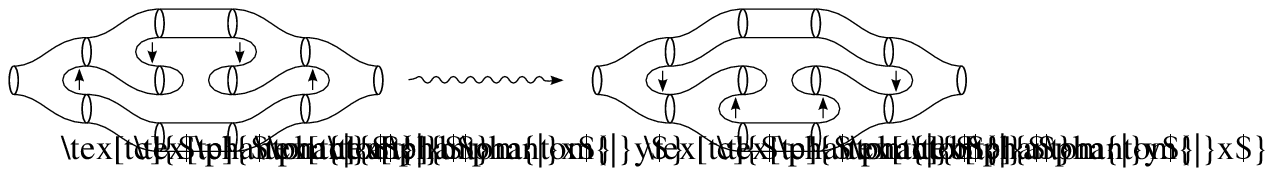}
\end{equation}

\medskip
Changes of chronologies reduce the~category $\cat{2ChCob}$ into $\cat{2Cob}$.
Indeed, the~relations of (co)commutativity, (co)associativity and~(co)unity given
in the~theorem~\ref{thm:cob-class} can be described by changes of chronologies.
We want to allow changes of chronologies but to avoid such a~reduction.
It can be done by colouring cobordisms and consider a~change of a~chronology
as a~change of colours.

\begin{definition}\label{def:chcob-color}
Let $G$ be an~Abelian group.
\term{A~coloured chronological cobordism} is a~pair $(M,g)$,
where~$M$ is a~chronological cobordism and $g\in G$.
The~colour of a~composition or a~multiplication is given by the~multiplication
of colours of both cobordisms:
\begin{align}
\label{eq:col-cob-comp}(M,g)(N,h) &:= (MN,gh)\\
\label{eq:col-cob-prod}(M,g)\sqcup(N,h) &:= (M\sqcup N,gh)
\end{align}
\end{definition}

\begin{remark}
The~multiplication in $G$ is associative, so is the~operation~(\ref{eq:col-cob-prod})
and we obtain a~symmetric chronological monoidal category of coloured chronological cobordisms
$G\cat{ChCob}$.
\end{remark}

For a~change of a~chronology $\chch{H}{M}{M'}$ we want to introduce a~relation
\begin{align}
\label{eq:chchcond-coef}M' = r_HM,&\quad r_H \in G
\end{align}
such that the~coefficients $r_H$ agree with composition of changes:
\begin{align*}
\tag{Ch1}r_{H\cdot H'} = r_Hr_{H'}.
\end{align*}
Moreover, the~quotient category should not be trivial. In particular,
the~following non-degeneracy condition ought to be satisfied, where $G_i$ are generators:
\begin{align*}
\tag{Ch2}\label{eq:chchcond-non-triv} g(G_1\sqcup\dots\sqcup G_n) = (G_1\sqcup\dots\sqcup G_n)\qquad\Rightarrow\qquad g = 1.
\end{align*}

Label a~merge, a~birth, a~split and a~death by $M,B,S,D$ respectively.
Define coefficients $r_H$ for the~following changes of chronologies:
\begin{itemize}
\item if $H$ permutes points $p < q$ labeled by $\alpha$ and~$\beta$, and $H$ preserves types of these points,
		put $r_H = \lambda_{\alpha\beta}$
\item if $H$ creates or deletes points $p < q$, put $r_H = 1$, provided that
		the~birth or the~death is on the~positive side of the~merge or the~split (i.e. it is pointed
		by the~arrow denoting the~orientation of the~critical point)\footnote{\ 
			There is no loss of generality, if one put $r_H=1$ for these changes.
			Indeed, whenever these coefficients for creating points are equal $\mu_{ab}$,
			define the~isomorphism of categories $F\colon G\cat{2ChCob}\to G\cat{2ChCob}$
			by multiplying a~birth by~$\mu_{BM}$ and a~death by~$\mu_{SD}$.
			In the~target category both coefficients are equal $1$.}
\end{itemize}
Every change of a~chronology is invertible, thus due to~(\ref{eq:chchcond-non-triv}):
\begin{equation}\label{eq:chch-abba-1}
\lambda_{\alpha\beta}\lambda_{\beta\alpha} = 1.
\end{equation}

There are more relations for the~coefficients $\lambda_{ab}$.
A~simple analysis of elementary changes gives the~following necessary condition
on the~coefficients $\lambda_{\alpha\beta}$, such that the~non-degeneracy condition holds.

\begin{theorem}[the~change of a~chronology condition]\label{thm:ch-ch-cond}
There exist elements $X,Y,Z\in G$ such that:
\begin{equation}\label{eq:ch-ch-cond}
\setlength\arraycolsep{1pt}
\begin{array}{lllll}
\lambda_{MM} &= \lambda_{BB} &= \lambda_{MB} &= \lambda_{BM} &= X \\
\lambda_{SS} &= \lambda_{DD} &= \lambda_{SD} &= \lambda_{DS} &= Y \\
\lambda_{SM} &= \lambda_{DB} &= \lambda_{MD} &= \lambda_{BS} &= Z \\
\lambda_{MS} &= \lambda_{BD} &= \lambda_{DM} &= \lambda_{SB} &= Z^{-1}
\end{array}
\end{equation}
and~$X^2=Y^2=1$.
\end{theorem}
\begin{proof}
Consider the~following change of a~chronology:

\medskip\begin{center}\image{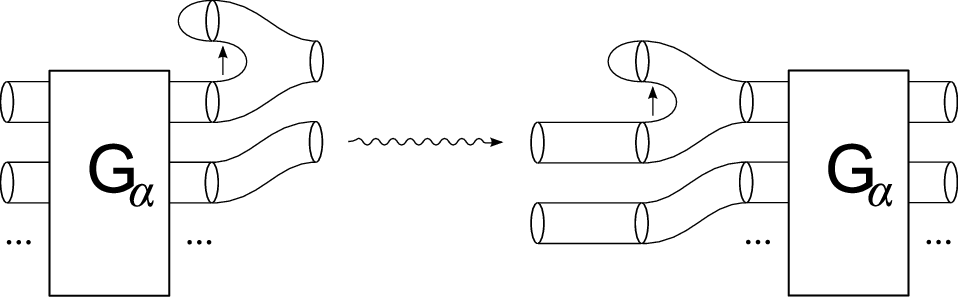}{0pt}\end{center}

\medskip\noindent where~$G_\alpha$ stands for a~cobordism with exactly one critical point of type~$\alpha$.
Adding a~creation of a~birth at the~beginning and deleting it at the~end, we get:
\begin{equation}
\lambda_{M\alpha}\lambda_{B\alpha}G_\alpha = G_\alpha
\end{equation}
and from~(\ref{eq:chchcond-non-triv}) we have $\lambda_{B\alpha}=\lambda_{M\alpha}^{-1}$.
In a~similar way $\lambda_{D\alpha} = \lambda_{S\alpha}^{-1}$.
Due to~(\ref{eq:chch-abba-1}) we have
\begin{equation}
\lambda_{M\alpha} = \lambda_{\alpha B}\quad\textrm{and}\quad\lambda_{S\alpha} = \lambda_{\alpha D}
\end{equation}
When replacing~$\alpha$ with $M,B,S,D$ we obtain the~equalities
\begin{equation}
\setlength\arraycolsep{1pt}
\begin{array}{lllll}
\lambda_{MM} &= \lambda_{BB} &= \lambda_{MB} &= \lambda_{BM} &= X\\
\lambda_{SS} &= \lambda_{DD} &= \lambda_{SD} &= \lambda_{DS} &= Y\\
\lambda_{SM} &= \lambda_{DB} &= \lambda_{MD} &= \lambda_{BS} &= Z\\
\lambda_{MS} &= \lambda_{BD} &= \lambda_{DM} &= \lambda_{SB} &= W\\
\end{array}
\end{equation}
Finally, \eqref{eq:chch-abba-1} implies $X^2 = Y^2 = ZW = 1$.
\end{proof}

\begin{remark}
The~orientation of critical points were not used in the~definition of
the~coefficients $\lambda_{\alpha\beta}$. In fact it plays no role, since there
are changes of chronologies which reverse the~orientation of a~single critical point:
\medskip
\begin{equation}\label{eq:chch-crit-orient}
\putimage{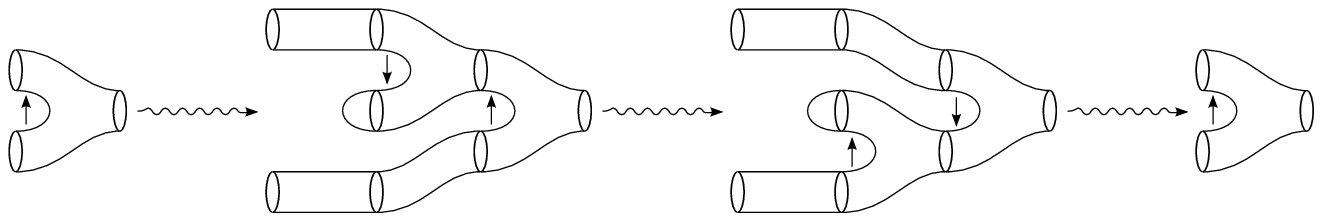}
\end{equation}

\medskip\noindent
and similarly for a~split.
Using such changes one can first set all orientations to be positive, then apply
required changes and at the~end restore the~original orientation.
\end{remark}

When calculating coefficients for changes~(\ref{eq:chch-crit-orient}) one gets
$r_H = X$ for a~merge and $r_H = Y$ for a~split.
Now we can complete the~list of coefficients $r_H$, taking in account also the~changes
from the~remark~\ref{rm:chron-change-triv} as well as the~other type of deletions
and creations (with the~opposite orientation). All of them\footnote{\ 
	In (T1) both cobordisms $M_1$ and~$M_2$ have a~positive genus.
	Therefore, using the~change from fig.~\ref{fig:chron-nontriv-change-torus}
	we get $XYM_i =M_i$, so the~coefficient is defined only up to
	the~factor $XY$. The~convention used in the~table, in which
	changes (T1) with different orientations of critical points
	are distinguished, is due to the~fact that when the~orientation
	of a~merge in $M_1$ is reversed, the~split in $M_2$ gets a~reversed
	orientation. Hence the~change (T1) between cobordisms $M'_1$
	and~$M'_2$ with opposite orientations of the~merge and the~split
	can be described as a~change from $M_1=XM'_1$ to~$M_2=YM'_2$
	with agreeing orientations. Different coefficients for these two
	types of (T1) will play a~crucial role in proving the~uniqueness
	of a~coefficient for a~change of a~chronology in case of
	cobordisms embedded in~$\mathbb{R}^3$ in the~next paragraph.}
are contained in the~table~\ref{tbl:chchcond}.

\begin{table}
	\begin{center}
	\begin{tabular}{|cccc|c|}
	\hline
	\multicolumn{4}{|c|}{A~description of a~change $H$} & The~coefficient $r_H$\\
	\hline
	\multicolumn{5}{|c|}{Permutations preserving the~types of the~points}\\
	\hline
	$MM\rightsquigarrow MM$ & $MB\rightsquigarrow BM$ & $BB\rightsquigarrow BB$ && $X$\\
	$SS\rightsquigarrow SS$ & $SD\rightsquigarrow DS$ & $DD\rightsquigarrow DD$ && $Y$\\
	$SM\rightsquigarrow MS$ & $MD\rightsquigarrow DM$ & $BS\rightsquigarrow SB$ & $DB\rightsquigarrow BD$ & $Z$\\
	\hline
	\multicolumn{5}{|c|}{Permutations changing the~types of the~points}\\
	\hline
	\multicolumn{2}{|c}{$S^+M^+\rightsquigarrow S^+M^+$}&\multicolumn{2}{c|}{$S^-M^-\rightsquigarrow S^-M^-$} & 1\\
	\multicolumn{2}{|c}{$S^+M^-\rightsquigarrow S^-M^+$}&\multicolumn{2}{c|}{} & $XY$\\
	\multicolumn{2}{|c}{$M^+S^+\rightsquigarrow M^-S^+$}&\multicolumn{2}{c|}{$M^-S^-\rightsquigarrow M^+S^-$} & $X$\\
	\multicolumn{2}{|c}{$M^+S^+\rightsquigarrow M^+S^-$}&\multicolumn{2}{c|}{$M^-S^-\rightsquigarrow M^-S^+$} & $Y$\\
	\hline
	\multicolumn{5}{|c|}{Deleting and creating critical points}\\
	\hline
	\multicolumn{4}{|c|}{$BM^+\rightsquigarrow \emptyset$} & $1$\\
	\multicolumn{4}{|c|}{$BM^-\rightsquigarrow \emptyset$} & $X$\\
	\multicolumn{4}{|c|}{$S^+D\rightsquigarrow \emptyset$} & $1$\\
	\multicolumn{4}{|c|}{$S^-D\rightsquigarrow \emptyset$} & $Y$\\
	\hline
	\end{tabular}	
	\end{center}
	\caption{The~coefficients $r_H$ for elementary changes of chronologies.
				The~table presents coefficients only for changes in one direction --- for the~opposite
				one take the~inverse of the~appropriate element. Signs ,,$+$'' and~,,$-$''
				denote orientations of critical points if they are important.
				Elements~$X,Y\in G$ has to be of order $2$.}\label{tbl:chchcond}
\end{table}

\begin{remark}\label{rm:cob-chch-cat}
The~change of a~chronology relations are defined locally, so they are compatible with
operations on cobordisms. Therefore, there exists a~quotient category $G\cat{2ChCob}/_{\!XYZ}$
of cobordisms modulo changes of chronologies, which is symmetric chronological and monoidal.
Reversion of a~cobordism interchanges the~role of $X$ and~$Y$ (it gives the~dual relations),
so there is a~contravariant chronological monoidal functor
\begin{equation}
*\colon G\cat{2ChCob}/_{\!XYZ}\to G\cat{2ChCob}/_{\!YXZ}.
\end{equation}
Every homomorphism $h\colon G\to G'$ induces a~chronological monoidal functor
\begin{equation}\label{def:funct-F-h}
\Fun{F}_h\colon G\cat{2ChCob}/_{\!XYZ}\to G'\cat{2ChCob}/_{\!h(X)h(Y)h(Z)}.
\end{equation}
In particular, these categories are isomorphic, provided $h$ is an~isomorphism.
\end{remark}

Instead of a group one can take a~ring $R$ and pick the~coefficients
from the~group of units $U(R)$. In the~chapter~\ref{chpt:khov} we will define a~symmetric
chronological monoidal functor $\F_{\!XYZ}\colon R\cat{2ChCob}/_{\!XYZ}\to\cat{Mod}_R$,
which maps every generator to a~non-zero linear map between free modules,
what gives that $r\F_{\!XYZ}(G_1\sqcup\dots\sqcup G_n) = 0$ implies $r=0$.
In particular, for the~group ring $R=\mathbb{Z}[G]$ we will get the~following result.

\begin{corollary}[the~change of a~chronology condition]
The~coefficients of changes of chronologies given in the~table~\ref{tbl:chchcond}
satisfies the~non-degeneracy condition~(\ref{eq:chchcond-non-triv}).
\end{corollary}

Let $M$ and~$N$ be chronological cobordisms. Due to the~theorem~\ref{thm:ch-ch-cond}
a~coefficient of a~change of a~chronology from $(M\sqcup C)(C\sqcup N)$ to~$(C\sqcup N)(M\sqcup C)$
depends only on 2-indices of cobordisms $\tcob(M)$ and~$\tcob(N)$.
Indeed, by direct calculations
\begin{equation}\label{eq:ch-ch-tkob}
(M\sqcup C)(C\sqcup N) = X^{ac}Y^{bd}Z^{bc-ad}(C\sqcup N)(M\sqcup C)
\end{equation}
where~$\tcob(M)= (a,b)$ and~$\tcob(N) = (c,d)$.
Moreover, directly from the~theorem~\ref{thm:tcob-prop}, cobordisms with vanishing 2-index
are central.

\begin{corollary}
Let $\cob{M}{\Sigma_0}{\Sigma_1}$ and~$\cob{N}{\Sigma_1}{\Sigma_0}$ be chronological cobordisms
and 2-indices of components of $M$ are zero. Then~$MN = NM$.
\end{corollary}
\begin{proof}
If the~2-index of a~connected cobordism $M'$ is zero, then there is a~change of chronology
between $M'$ and a~cylinder. Hence, we have
\begin{equation}
MN = \lambda CN = \lambda NC = NM
\end{equation}
\end{proof}

\section{\texorpdfstring{Cobordisms embedded in~$\mathbb{R}^3$}{Cobordisms embedded in R\textasciicircum{}3}}\label{sec:chcob-emb}
This section deals (1+1)-cobordisms embedded into~$\mathbb{R}^3$.
Comparing to abstract cobordisms, the~embedded ones have a~richer structure,
which can be used to introduce a~chronology and equivalence of cobordisms
in a~more delicate way. In particular, there exists a~planar algebra of cobordisms
as well as for a~certain type of changes of chronologies the~coefficient
depends only on the~equivalence class of such a~change.

\begin{definition}\label{def:cob-emb}
Let $\Sigma_0$ and~$\Sigma_1$ be compact one-dimensional submanifolds of a~plane $\mathbb{R}^2$ with no boundary.
An~\term{embedded cobordism in~$\mathbb{R}^3$} from~$\Sigma_0$ to~$\Sigma_1$ is a~surface
$S\subset\mathbb{R}^2\times I$ such that
\begin{equation}
S\cap(\mathbb{R}^2\times\{1\}) = \Sigma_0\qquad\textrm{and}\qquad S\cap(\mathbb{R}^2\times\{0\}) = \Sigma_1.
\end{equation}
We write $\cob S{\Sigma_0}{\Sigma_1}$.
\end{definition}

We will consider the~embedded cobordisms up to ambient isotopies, i.e. diffeotopies of $\mathbb{R}\times I$.

\begin{definition}\label{def:cob-emb-equiv}
Regard two embedded cobordisms $\cob S{\Sigma_0}{\Sigma_1}$ and~$\cob Q{\Sigma_0}{\Sigma_1}$ as
\term{equivalent}, if there exists an~ambient isotopy $\Psi_t$ of $S$
constant on the~boundary of $\mathbb{R}\times I$ such that $\Psi_1(S) = Q$.
\end{definition}

Embedded cobordisms form a~category $\cat{Cob}^3$ in a~natural way: the~composition
$Q\circ S = SQ$ is given by placing $S$ on the~top of $Q$ (see fig.~\ref{fig:cob-emb-comp}).
Moreover, there is a~reversion defined as a~symmetry along a~plane
$\mathbb{R}^2\times\left\{\frac{1}{2}\right\}$.

\begin{figure}
	\begin{center}
		\putimage{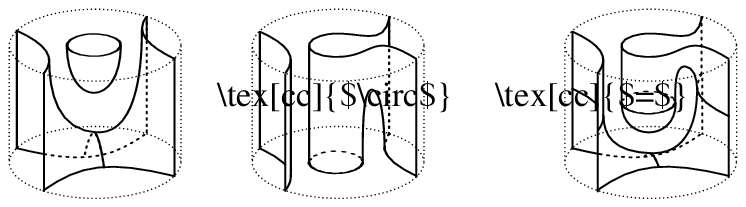}
	\end{center}
	\caption{The~composition of embedded cobordisms in given by putting the~second
				on the~first.}\label{fig:cob-emb-comp}
\end{figure}

The~definition of an~embedded cobordism can be extended over manifolds with boundary.
Indeed, take an~arbitrary cobordism $\cob S{\Sigma_0}{\Sigma_1}$ and cut it with a~cylinder $\D2\times I$.
Assume $S\cap(\partial\D2\times I) = B\times I$ for some finite $B\subset\partial\D2$.
Then $S'=S\cap (\D2\times I)$ is a~surface between $\Sigma_0'$ and $\Sigma_1'$,
where $\Sigma_i'=\Sigma_i\cap\D2$.
Elements of the~set $B\times\{0,1\}$ are called \term{corners} of the~surface $S'$
and elements of $B$ are called \term{endpoints} of $\Sigma_i$.

\begin{definition}\label{def:cob-emb-corn}
A~surface $S'\subset\D2\times I$ constructed above is called a~\term{cobordisms with corners}
from $\Sigma_0'\subset\D2$ to~$\Sigma_1'\subset\D2$.
Two cobordisms with corners $S$ and~$Q$ are said to be \term{equivalent}, if there exists
an~ambient isotopy $\Psi_t$ of $S'$ constant on the~boundary of $\D2\times I$
such that $\Psi_1(S) = Q$.
\end{definition}

Directly from the~definition, cobordisms with corners form a~category,
which extends the~category $\cat{Cob}^3$. We will denote it in the~same way,
putting $\cat{Cob}^3(\emptyset)$ for cobordisms with no corners.

For a~given cobordism $\cob S{\Sigma_0}{\Sigma_1}$ we have $\partial\Sigma_0=\partial\Sigma_1$.
Thus, there is a~well defined subcategory $\cat{Cob}^3(B)$ consisting of all submanifolds
$\Sigma\subset\D2$ with boundary $\partial\Sigma=B$ and cobordisms with corners in $B$.
Hence, the~category $\cat{Cob}^3$ decomposes as
\begin{equation}
\cat{Cob}^3 = \bigcup_{B\subset\partial\D2}\cat{Cob}^3(B)
\end{equation}

Objects of $\cat{Cob}^3(B)$ can be seen as diagrams of trivial tangles from $\mathcal{T}^0(B)$
(see the~section~\hyperref[sec:knots-tangles]{\ref*{chpt:knots}.\ref*{sec:knots-tangles}}),
so~they inherits a~structure of a~planar algebra. The~structure lifts to cobordisms.
Indeed, a~planar diagram $D$ represents a~three-dimensional curtain $D\times I$,
which acts on cobordisms by putting them inside cylindrical holes (fig.~\ref{fig:cob-plan}).
Clearly, each planar diagram $D$ lifts to a~functor
\begin{equation}
D\colon\cat{Cob}^3(B_1)\times\dots\times\cat{Cob}^3(B_s)\to\cat{Cob}^3(B)
\end{equation}
which, for simplicity, will be denoted by the~same letter. This defines a~planar algebra
structure on the~category~$\cat{Cob}^3$.
\begin{figure}[ht]
	\begin{center}
		\image{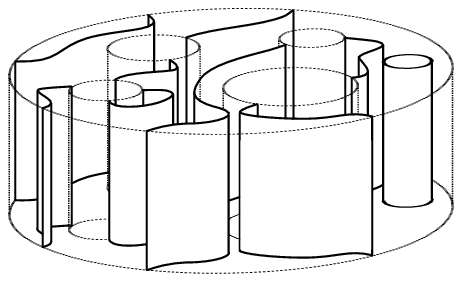}{0pt}
	\end{center}
	\caption{The~planar operator on the~set of~cobordisms with corners given by the~planar diagram
				from picture~\ref{chpt:knots}.\ref{fig:diag-plan} is a~three-dimensional curtain.}\label{fig:cob-plan}
\end{figure}

We will now add a~chronology to our new framework.
Notice there is a~natural projection for a~given cobordism $\cob S{\Sigma_0}{\Sigma_1}$:
\begin{equation}
\pi\colon S\ni(x,y,z)\to z\in I.
\end{equation}

\begin{definition}
A~cobordism $\cob S{\Sigma_0}{\Sigma_1}$ is called \term{chronological}, if the~projection $\pi\colon S\to I$
is a~Morse function with at~most one critical point at each level.
An~ambient isotopy $\Psi_t$ of $S$  is called \term{chronological}, if it is constant
on the~boundary of $\D2\times I$ and $\Psi_t(S)$ is a~chronological cobordism for each $t\in I$.
Embedded chronological cobordisms $S$ and~$Q$ are said to be~\term{equivalent},
if there exists an~ambient chronological isotopy $\Psi_t$ of $S$ such that $\Psi_1(S)=Q$.
\end{definition}

Like in the~case of abstract cobordisms, the~projection $\pi$ induces a~gradient flow on~$S$
and we~can introduce orientations of critical points. The~category of embedded chronological
cobordisms with oriented critical points will be denoted by $\cat{ChCob}^3$.
In analogous to~$\cat{Cob}^3$ there exists a~decomposition into subcategories
\begin{equation}
\cat{ChCob}^3 = \bigcup_{B\subset\partial\D2}\cat{ChCob}^3(B)
\end{equation}

\begin{remark}\label{rm:chcob3-2chcob}
A~classification of $\cat{ChCob}^3$ is at least as difficult as to~classify knots
(tubular neighbourhoods of non-equivalent knots give non-equivalent cobordisms).
However, there is a~forgetful functor
\begin{equation}\label{eq:chcob3-2chcob}
\mathcal{U}\colon \cat{ChCob}^3(\emptyset) \to \cat{2ChCob}
\end{equation}
assigning to an~embedded cobordism its equivalence class in the~sense of definition~\ref{def:chcob-equiv}.
Hence, any functor from the~category $\cat{2ChCob}$ lifts
to a~functor from the~category $\cat{ChCob}^3(\emptyset)$.
\end{remark}

When we try to define planar operators for embedded chronological cobordisms, we meet
the~same problem which arose in defining the~multiplication of chronological cobordisms
--- using the~naive definition we may get a~cobordism which is not chronological.
It can be overcome in the~same way as before: a~planar diagram $D$ induces an~operator
\begin{equation}
D\colon\Mor(\cat{ChCob}^3(B_1))\times\cdots\times\Mor(\cat{ChCob}^3(B_s))\to\Mor(\cat{ChCob}^3(B))
\end{equation}
which puts the~$i$-th cobordism into the~$i$-th cylindrical hole in $D\times I$ in such a way,
that all its critical points project onto $(\frac{i-1}{s},\frac{i}{s})$.
Hence, $D(S_1,\dots,S_s)$ has at first critical points of $S_1$, then critical points $S_2$, etc.
However, functoriality is lost, as the~following does not hold in general:
\begin{equation}\label{eq:chcob-D-no-fun}
D(S_1Q_1,\dots,S_sQ_s) = D(S_1,\dots,S_s)D(Q_1,\dots,Q_s)
\end{equation}
But it is the~case when $S_i$ and~$Q_i$ are cylinders for all indices except one.
For instance, operators given by diagrams with one input, especially closure operators,
are functorial.

\begin{remark}\label{rm:chcob3-all-2chcob}
A~functor $F\colon\cat{ChCob}^3(\emptyset)\to\cat{C}$
might not extend naturally on the~hole $\cat{ChCob}^3$.
But it can be extended over the~category of sequences in $\cat{C}$,
by taking values of $F$ on all closure operators:
\begin{equation}\label{eq:chcob3-all-2chcob}
F|_{\cat{ChCob}^3(B)}(S) := \{F(DS)\ |\ D\in\CPO(B)\}
\end{equation}
\end{remark}

We have already defined in the~section~\ref{sec:chcob-chchc} change of a~chronology
relations in~$G\cat{2ChCob}$. It can be directly applied to $\cat{ChCob}^3(\emptyset)$.
However, some modifications are necessary for cobordisms with corners, due to the~fact that
a~single saddle after applying to it a~closure operator can become either a~split or a~merge.

Let $S\subset\D2\times I$ be a~chronological cobordism and $\Psi$ its ambient isotopy,
not necessary chronological. Call $t\in I$ \term{a~critical moment} of $\Psi$,
if $\Psi_t(S)$ is not a~chronological cobordism.

\begin{definition}
An~ambient isotopy $\Psi$ of $S$ constant on the~boundary of $\D2\times I$
is called \term{a~change of a~chronology}, if it has finitely many critical moments $0<t_1<\dots<t_k<1$
and each $\Psi_{t_i}$ satisfies one of the~conditions \ref{chch-I} -- \ref{chch-III}.
\end{definition}

A~change of a~chronology from $S$ to~$Q$ will be denoted by $\chch{\Psi}SQ$.
We will identify changes $\chch{\Psi}SQ$ and $\chch{\Psi'}{S'}{Q'}$
if $S\simeq S'$ and $Q\simeq Q'$. Notice that changes (T1) and (T2) from
the~remark~\ref{rm:chron-change-triv} do not preserve embedded cobordisms.

Like in the~paragraph~\ref{sec:chcob-chchc} we can define a~composition of changes
of chronologies, obtaining a~category of embedded cobordisms and~changes of~chronologies
$\cat{Chron}$, which decomposes into subcategories $\cat{Chron}(B)$ generated by
cobordisms from $\cat{ChCob}^3(B)$.

\begin{proposition}\label{thm:chch-alg-plan}
The~category $\cat{Chron}$ has a~structure of a~planar algebra with functorial
planar operators.
\end{proposition}
\begin{proof}
Let $D$ be a~planar diagram with~$s$ inputs and $\chch {\Psi^i}{S_i}{Q_i}$
be changes of chronologies of appropriate types for $i=1,\dots,s$.
Each $\Psi^i$ induces a~change of a~chronology
\begin{equation}
\chch {\Psi^i_D}{D(Q_1,\dots,Q_{i-1},S_i,S_{i+1},\dots,S_s)}{D(Q_i,\dots,Q_{i-1},Q_i,S_{i+1},\dots,S_s)}
\end{equation}
which is constant beyond the~subset $\D2\times(\frac{i-1}{s},\frac{i}{s})$, on which
is equal to $\Psi^i$.
A~change $D(\Psi^1,\dots,\Psi^s)$ is defined as the~composition of the~induced changes:
\begin{equation}\label{eq:D-chch}
D(\Psi^1,\dots,\Psi^s) := \Psi^1_D\cdot \ldots \cdot \Psi^s_D.
\end{equation}
Functoriality holds, because all $\Psi^i_D$ act on disjoint regions.
\end{proof}

The~functor from the~remark~\ref{rm:chcob3-2chcob} carries the~change of a~chronology
relations from $G\cat{2ChCob}$ to $\cat{ChCob}^3(\emptyset)$ and further to $\cat{ChCob}^3$,
using the~construction described in the~remark~\ref{rm:chcob3-all-2chcob}.
The~details are provided below.

\begin{definition}\label{def:cob-emb-color}
Let $G$ be an~Abelian group.
\term{A~coloured chronological embedded cobordism} is a~pair $(S,g)$
consisting of a~chronological embedded cobordism $S$ with corners in $B$ and
a~function $g\in G^{\CPO(B)}$ from the~set of closure operators $\CPO(B)$ into
the~group $G$. The colour of the~composition is given by the~multiplication in $G$:
\begin{equation}
(S,g)(Q,h) := (SQ,gh).
\end{equation}
\end{definition}

As before, we will usually write $gS$ instead of $(S,g)$.

Coloured cobordisms form a~category denoted by $G\cat{ChCob}^3$.
The~set of functions $G^X$ is an~Abelian group for any set $X$,
so this category is naturally equipped with a~structure of a~planar algebra:
\begin{equation}
D(g_1S_1, \dots, g_sS_s) := g_1\cdot\ldots\cdot g_sD(S_1,\dots,S_s).
\end{equation}

A~change of a~chronology $\chch{\Psi}SQ$ of cobordisms in~$\cat{ChCob}^3(\emptyset)$
induces a~change of a~chronology of abstract cobordisms with some coefficient $r_{\Psi}$
defined as in the~table~\ref{tbl:chchcond}. In the~category $\cat{ChCob}^3(B)$
the~theorem~\ref{thm:chch-alg-plan} gives for a~change of a~chronology $\chch{\Psi}SQ$
a~function $r_{\Psi}\colon\CPO(B)\to G$ such that for every closure operator $D$ the~following holds:
\begin{equation}
DQ = r_{\Psi}(D)DS,\qquad r_{\Psi}(D) = r_{D\Psi}.
\end{equation}
It is called \term{a~coefficient of a~change of a~chronology of an embedded cobordism}.
The~next proposition goes directly from the~definition.

\begin{proposition}\label{prop:chch-coef-alg-plan}
Coefficients of changes of chronologies form a~planar algebra, where a~planar operator $D$
acts as follow:
\begin{equation}\label{eq:chchc-alg-plan}
D(r_{\Psi^1},\dots,r_{\Psi^s}) := r_{D(\Psi^1,\dots,\Psi^s)}.
\end{equation}
Moreover, the~mapping $\Psi\mapsto r_{\Psi}$ is a~morphism of planar algebras.
\end{proposition}

A~priori the~coefficient $r_{\Psi}$ may be different for two changes of a~chronology
$\Psi\colon S\to Q$ and $\Psi'\colon S\to Q$. An example is provided
by the~change~\ref{fig:chron-nontriv-change-torus} in the~previous section.
We will end this chapter showing that for some class of changes of chronologies
the~coefficient $r_{\Psi}$ is well-defined (i.e. depends only on $S$ and~$Q$).

Let $D$ be a~planar diagram with $k$ inputs. Then a~permutation $\sigma\in S_k$
induces a~planar diagram $D^\sigma$ which differs from $D$ only in the~order of inputs:
the~$i$-th input of $D^\sigma$ is the~$\sigma(i)$-th input of $D$.
Let $S_1,\dots,S_k$ be cobordisms such that $S_i$ fits into the~$i$-th input of $D$.
Directly from the~definition, cobordisms $D(S_1,\dots,S_k)$ and~$D^\sigma(S_{\sigma(1)},\dots,S_{\sigma(k)})$
differ by some change of a~chronology of type~I, constant on the~cylinder $D\times I$.

\begin{definition}\label{def:D-chch}
A~change of a~chronology induced by a~planar diagram $D$ is called
\term{adapted to~$D$} or simply \term{a~$D$-change}.
\end{definition}

\noindent Every elementary change of a~chronology is adapted to some planar diagram $D$.

\begin{lemma}\label{lem:chch-diag-plan}
Let $S$ be a~cobordism with two critical points and $\chch{\Psi}SQ$ be an~elementary
change of a~chronology of type~I. Then there exists a~planar diagram $D$
with two inputs, such that $\Psi$ is a~$D$-change.
\end{lemma}
\begin{proof}
Let $t_0$ be the~unique critical moment of $\Psi$.
Then there exists a~unique critical level of $\Psi_{t_0}(S)$
\begin{equation}
P = \Psi_{t_0}(S)\cap(\D2\times\{c\}).
\end{equation}
It contains two critical points $x_1$ and~$x_2$. Let~$U_1$ and~$U_2$ be their neighbourhoods
such that $U_i\cap P$ is a~point (if $\mu(x_i)\neq 1$) or a~cross (if $\mu(x_i)=1$).
Then $D = P\backslash(U_1\cup U_2)$ is a~planar diagram with two inputs, inducing
a~change of a~chronology equivalent to $\Psi$.
\end{proof}

Due to the~above lemma we can narrow our interest to planar diagrams.
All elementary changes of chronology of type~I between saddles are listed in the~figure~\ref{fig:chch-cob-emb}.
Circles describe inputs, while saddles are shown as thick arcs --- the~output of a~cobordism
can be obtained by connecting circles with the~boundary of bands attached along the~arcs.
As we need to put the~arcs in order to have a~well-defined chronology, permuting this order
corresponds to the~change of a~chronology.

\begin{figure}[tb]
	\begin{center}
	\begin{tabular}{ccccc}
		X & \qquad\qquad & \multicolumn{3}{c}{Y}\\
		\quad& \\
		\image{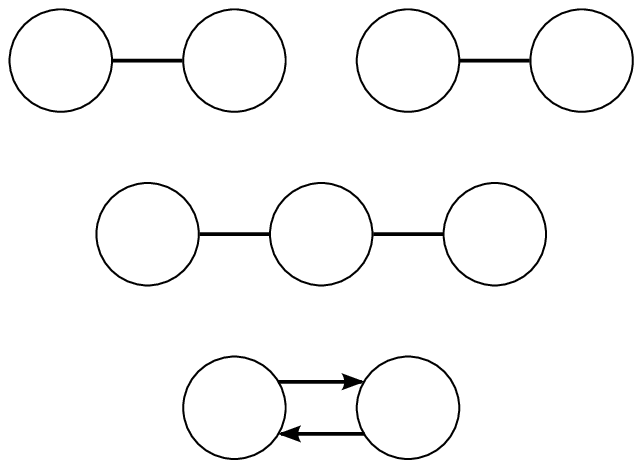}{0pt} &  & \multicolumn{3}{c}{\image{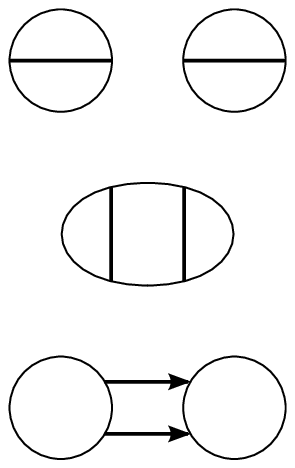}{0pt}}\\
		\quad& \\
		Z & \qquad\qquad & I & \qquad & R\\
		&\qquad &&&\\
		\image{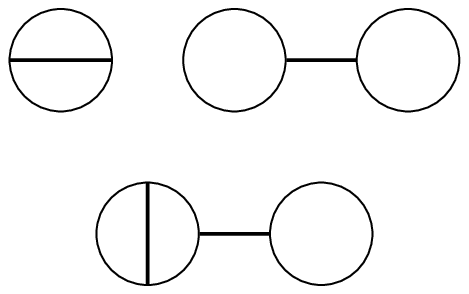}{50pt} && \image{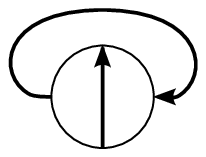}{10pt} && \image{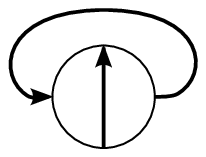}{10pt}\\
	\end{tabular}
	\end{center}
	\caption{Permutations of points of type $M$ and~$S$ can be split into five groups
				$X$, $Y$, $Z$, $I$ and~$R$. The~names of the~first three corresponds
				to the~coefficients they generate. Changes denoted by $I$ and~$R$ are the~changes
				(T1) from the~remark~\ref{rm:chron-change-triv}, where in $I$ orientations
				of the~merge and~the~split agrees but in $R$ they are opposite. Hence, we get
				coefficients $1$ and~$XY$. Thinner curves represents circles and thicker arcs
				describe saddles. If it matters, arrows describe orientations of
				critical points.}\label{fig:chch-cob-emb}
\end{figure}

A~composition of two changes adapted to the~same planar diagram $D$ is also a~$D$-change
and there is a~subcategory $\cat{Chron}_D$ of cobordisms and $D$-changes.
Moreover, each $D$-change $\Psi$ is given uniquely up to a~composition with an~isotopy of chronologies
by a~permutation of critical points $\sigma_{\Psi}$ and a~decomposition of $\sigma_{\Psi}$
into elementary transpositions (i.e. transpositions of the~form $(i\ i+1)$) gives a~decomposition
of $\Psi$ into elementary changes.

\begin{theorem}\label{thm:D-chch-unique-coef}
Let $\chch{\Psi}SQ$ be a~trivial change of a~chronology adapted to some planar diagram $D$.
Then $r_{\Psi} = 1$.
\end{theorem}
\begin{proof}
It is sufficient to show the~theorem for $\cat{ChCob}^3(\emptyset)$,
since the~uniqueness of coefficients for changes of chronologies of cobordisms with
no corners implies the~uniqueness in the~general case.

We will show, that the~coefficient $r_{\Psi}$ is independent on a~decomposition of $\sigma_{\Psi}$.
Denote by~$\tau_i = (i\ i+1)$ an~elementary transposition of elements $i$ and~$(i+1)$.
Any two decompositions of the~same permutation are related by the~following three relations:
\begin{align*}
\tag{S1}\label{eq:perm-ij-ji}\tau_i\tau_j &= \tau_j\tau_i,\qquad |i-j|>1,\\
\tag{S2}\label{eq:perm-ii}\tau_i^2 &= 1,\\
\tag{S3}\label{eq:perm-iji-jij}\tau_i\tau_{i+1}\tau_i &= \tau_{i+1}\tau_i\tau_{i+1}.
\end{align*}
(\ref{eq:perm-ij-ji}) corresponds to commutativity of changes of chronologies acting on
disjoint levels, whereas (\ref{eq:perm-ii}) to the~composition of a~change with its inverse.
It remains to show invariance under the~last relation.

Let $D$ be a~planar diagram with three inputs and pick three cobordisms $S_1,S_2,S_3$ with
one critical point each, fitting into the~inputs of $D$. Given a~permutation $\sigma\in S_3$ write
$S^\sigma = D^\sigma(S_{\sigma(1)},S_{\sigma(2)},S_{\sigma(3)})$.
All these cobordisms form a~hexagonal diagram:
\begin{equation}\label{eq:chch-hexagon}
\xy
\morphism(   0,300)/->/<500, 300>[S^{\id}`S^{(1\ 2)};r_1]
\morphism(   0,300)/<-/<500,-300>[S^{\id}`S^{(2\ 3)};r_6]
\morphism( 500,600)/->/<600,   0>[S^{(1\ 2)}`S^{(1\ 2\ 3)};r_2]
\morphism( 500,  0)/<-/<600,   0>[S^{(2\ 3)}`S^{(1\ 3\ 2)};r_5]
\morphism(1100,600)/->/<500,-300>[S^{(1\ 2\ 3)}`S^{(1\ 3)};r_3]
\morphism(1100,  0)/<-/<500, 300>[S^{(1\ 3\ 2)}`S^{(1\ 3)};r_4]
\endxy
\end{equation}
where arrows represent changes of a~chronologies with coefficients $r_i$.
It suffices to show that for each diagram $D$ and surfaces $S_1,S_2,S_3$
the~product of $r_i$'s is trivial:
\begin{equation}\label{eq:chch-hex-eq}
r_1\cdot\ldots\cdot r_6=1.
\end{equation}

The case of disconnected cobordisms is simple. Let the~point $p_1$ lie on a~different
components than $p_2$ and~$p_3$. Then in the~diagram~(\ref{eq:chch-hexagon}) we have
equalities $r_3 = r_6^{-1}$ (a~permutation of $p_2$ and~$p_3$) and $r_1r_2 = r_5^{-1}r_4^{-1}$
(a~permutation of $p_1$ with $p_2$ and~$p_3$), what gives~(\ref{eq:chch-hex-eq}).

In the case of connected cobordisms all critical points are saddles.
The~figure~\ref{fig:chch-2-cob-emb} shows all possible situations.
Likewise in figure~\ref{fig:chch-cob-emb} thinner curves represent inputs
of cobordisms and thicker arcs describe saddles. Each diagram is equipped with
sequences of six numbers, equal to amounts of elementary changes with coefficients
respectively $X,Y,Z,Z^{-1},1$ and~$XY$ which appear in the~hexagon diagram (more than one
for some diagrams, as different orientations of critical points may lead to different changes).
For each such a~sequence $(c_1,\dots,c_6)$ we have $X^{c_1+c_6}Y^{c_2+c_6}Z^{c_3-c_4} = 1$,
so the~relation~(\ref{eq:perm-iji-jij}) preserves in each case the~coefficient of the~change
of a~chronology $r_{\Psi}$.
For instance, the~first diagram leads to the~following hexagon
$$
\xy
\morphism(0,400)|m|//<2000,0>[`;\image{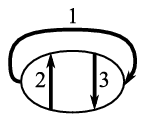}{0pt}]

\morphism(   0,400)|a|/->/<600, 400>[S^{\id}`S^{(1\ 2)};\image{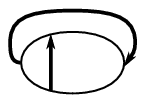}{0pt}]
\morphism(   0,400)|b|/<-/<600,-400>[S^{\id}`S^{(2\ 3)};\image{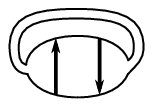}{0pt}]
\morphism( 600,800)|a|/->/<800,   0>[S^{(1\ 2)}`S^{(1\ 2\ 3)};\image{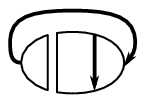}{0pt}]
\morphism( 600,  0)|b|/<-/<800,   0>[S^{(2\ 3)}`S^{(1\ 3\ 2)};\image{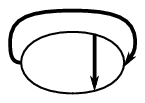}{0pt}]
\morphism(1400,800)|a|/->/<600,-400>[S^{(1\ 2\ 3)}`S^{(1\ 3)};\image{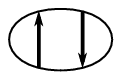}{0pt}]
\morphism(1400,  0)|b|/<-/<600, 400>[S^{(1\ 3\ 2)}`S^{(1\ 3)};\image{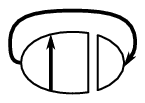}{0pt}]

\endxy
$$
where the~diagram inside shows the~chosen enumeration of inputs and orientations of critical
points, whereas diagrams along arrows describe changes of chronologies. Notice, each
of them is obtained by either forgetting one arrow (when the~first two points are permuted)
or applying a~surgery along one arrow (when the~second and the~third point are permuted).
In this example the~coefficients are equal to $1, Z^{-1}, Y, Z, XY$ and~$X$.
\end{proof}

\begin{figure}
\begin{center}
	\image{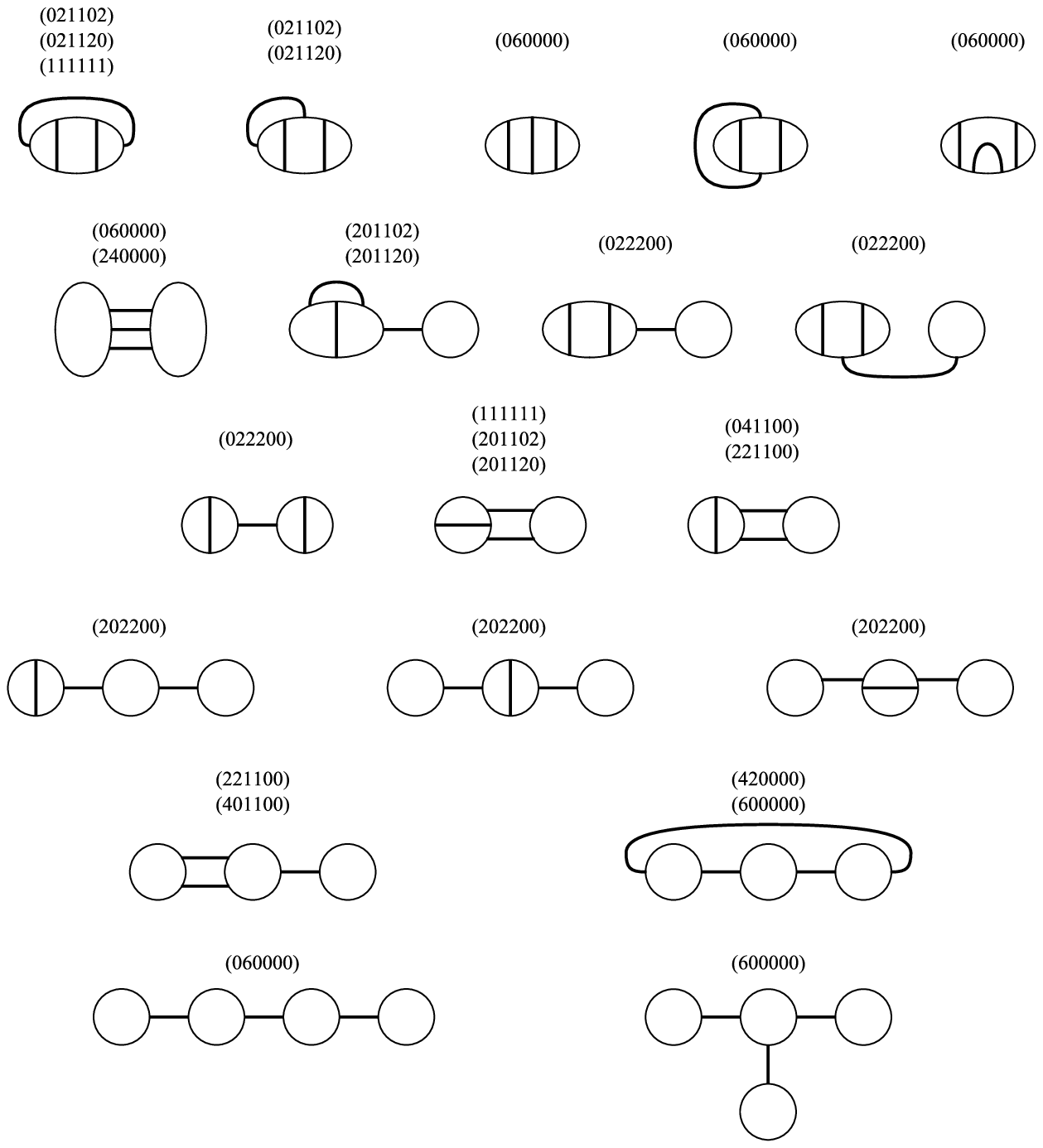}{0pt}
\end{center}
\caption{$D$-changes of three critical points. Each diagram corresponds to a~hexagon
			of elementary changes. Numbers in brackets count the~amounts of changes
			of type respectively $X, Y, Z, Z^{-1}, I$ and~$R$ which appear in the~hexagon.
			Different sequences correspond to different choices of orientations of thicker arcs.}
\label{fig:chch-2-cob-emb}
\end{figure}

\begin{corollary}\label{cor:equiv-chch-equiv-coef}
If changes $\chch{\Psi}SQ$ and~$\chch{\Psi'}SQ$ are adapted to the~same planar diagram $D$,
then $r_{\Psi} = r_{\Psi'}$.
\end{corollary}

\chapter{Elements of homological algebra}\label{chpt:alg-hom}
\chead{\fancyplain{}{\thechapter. Elements of homological algebra}}
The~object of interest of homological algebra are Abelian categories,
for instance the~category of modules over a~commutative ring.
Among its properties the~following seem to be the~most important:
an~additive structure of homorphisms and existence of direct sums, 
kernels, images and quotient modules.
In this chapter we will introduce some machinery which can be used
to~carry over several constructions from the~category of modules into an~arbitrary category.
This will allow us to construct the~generalised Khovanov complex in the~next chapter
in the~category of cobordisms and prove its invariance on that level.

\section{Additive categories}\label{sec:hom-cat-add}
Take an~arbitrary category $\cat{C}$ and pick two of its objects $A$ and~$B$.

\begin{definition}\label{def:cat-prod}
An~object $X$ together with morphisms $A\stackrel{\pi_A}\longleftarrow X \stackrel{\pi_B}\longrightarrow B$
is \term{a~product} or~\term{a~direct product} of objects $A$ and~$B$,
if for any object $D$ and morphisms $A\stackrel{f}\longleftarrow D \stackrel{g}\longrightarrow B$
there exists a~unique morphism $h\colon D\to X$ such that the~following diagram commutes:
$$\xy
\Atrianglepair/->`-->`->`<-`->/[D`A`X`B;f`h`g`\pi_A`\pi_B]
\endxy$$
The~object $X$ is denoted by $A\times B$ and morphisms $\pi_A$ and~$\pi_B$ are called
\term{projections}.
The~unique morphism $h\colon D\to A\times B$ is denoted by $(f,g)$.
\end{definition}

When we reserve the~arrows, we obtain a~dual construction.

\begin{definition}\label{def:cat-coprod}
An~object $X$ together with morphisms $A\stackrel{i_A}\longrightarrow X \stackrel{i_B}\longleftarrow B$
is \term{a~coproduct} or \term{a~direct sum} of objects $A$ and~$B$,
if for any object $D$ and morphisms $A\stackrel{f}\longrightarrow D \stackrel{g}\longleftarrow B$
there exists a~unique morphism $h\colon X\to D$ such that the~following diagram commutes:
$$\xy
\Atrianglepair/<-`<--`<-`->`<-/[D`A`X`B;f`h`g`i_A`i_B]
\endxy$$
The~object $X$ is denoted by $A\oplus B$ and morphisms $i_A$ and $i_B$ are called
\term{embeddings}.
The~unique morphism $h\colon A\oplus B\to D$ is denoted by $f+g$.
\end{definition}

A~product and a~coproduct, if exist, are unique up to an~isomorphism. Moreover,
for morphisms $f\colon A\to C$ and~$g\colon B\to D$ there are unique morphisms
$f\times g\colon A\times B\to C\times D$ and $f\oplus g\colon A\oplus B\to C\oplus D$
agreeing respectively with projections and embeddings. The~properties below follow
directly from the~definitions of products and coproducts.

\begin{proposition}
Let $A,B,C$ be objects of a~category $\cat{C}$. Then there exist natural isomorphisms
\begin{itemize}
\item $A\times B \cong B\times A$
\item $A\times (B\times C) \cong (A\times B)\times C$
\item $A\oplus B \cong B\oplus A$
\item $A\oplus (B\oplus C) \cong (A\oplus B)\oplus C$
\end{itemize}
as long as all the~objects are well defined.
\end{proposition}

\noindent Naturality of the~isomorphisms above means compliance with induced morphisms.
For instance, when $f\colon A\to A'$ and $g\colon B\to B'$, the~diagrams below commute
$$
\xy
\square(0,0)<700,500>[A\times B`B\times A`A'\times B'`B'\times A';\cong`f\times g`g\times f`\cong]
\square(1500,0)<700,500>[A\oplus B`B\oplus A`A'\oplus B'`B'\oplus A';\cong`f\oplus g`g\oplus f`\cong]
\endxy
$$
and similarly in other cases.

\begin{example}
The~category of sets $\cat{Set}$ has Cartesian products as products, and disjoint sums
as coproducts. In the~category of Abelian groups $\cat{Ab}$ Cartesian products
also play the~role of products, whereas coproducts are given by direct sums.
In case of the~full category of groups $\cat{Grp}$ products does not change,
but coproducts are given by free products.
\end{example}

The~definition of the~product and the~coproduct can be easily extended over any number of
objects: the~product of objects $\{A_\lambda\}_{\lambda\in\Lambda}$ is the~object $X$
along with projections $\{\pi_\lambda\colon X\to A_\lambda\}_{\lambda\in\Lambda}$ such that
given any object $D$ with morphisms $\{f_\lambda\colon X\to A_\lambda\}_{\lambda\in\Lambda}$
there exists a~unique morphism $h\colon D\to X$ making the~following diagram commute
for each $\lambda\in\Lambda$:
$$
\xy
\btriangle/-->`->`->/[D`X`A_\lambda;h`f_\lambda`\pi_\lambda]
\endxy
$$
The~product is called \term{finite} if $\Lambda$ is a~finite set.
In case $\Lambda=\emptyset$ we obtain \term{the~terminal object} $T$ and there exists
exactly one arrow to $T$ from any object $A$.
We can define the~general coproduct dually, obtaining for~$\Lambda=\emptyset$
\term{the~initial object} $I$.

\begin{proposition}\label{prop:prod-coprod-unit}
Let $\cat{C}$ be a~category. If $T$ and~$I$ are respectively the~terminal
and the~initial object in $\cat{C}$, then
\begin{enumerate}
\item $T\times A\cong A\times T\cong A$
\item $I\oplus A\cong A\oplus I\cong A$
\end{enumerate}
\end{proposition}
\begin{proof}
Isomorphisms are provided by the~definitions of products and coproducts.
We will show the~case $T\times A\cong A$ --- other proofs are mostly the~same.

Notice first, that $\pi_A\colon T\times A\to A$ is the~only morphism, for which the~following diagram commutes:
$$\xy
\Atrianglepair/->`-->`->`<-`->/[T\times A`T`A`A;\pi_T`\pi_A`\pi_A`a`\id_A]
\endxy$$
where~$a\colon A\to T$ is the~unique morphism to the~terminal object $T$.
Indeed, $a\circ\pi_A$ is a~morphism from $T\times A$ to $T$, same as $\pi_T$.
Since $T$ is terminal, these morphisms have to be equal.
But the~universal property of a~product gives also a~morphism $i_A\colon A\to T\times A$
which appears to be the~inverse of $\pi_A$. Hence, $A\cong T\times A$.
\end{proof}

In the~categories of Abelian groups $\cat{Ab}$ and~$R$-modules $\cat{Mod}_R$ 
all finite products and coproducts exist. Moreover, a~set of morphisms between any two
objects is an~Abelian group. This is a~motivation for the~following definition.

\begin{definition}\label{def:cat-add}
A~category $\cat C$ is said to be \term{additive}, if
\begin{enumerate}
\item each set $\Mor_\cat{C}(X,Y)$ is an~Abelian group
	and composition of morphisms is additive from both sides:
	$$
	(f+g)\circ h = f\circ h + g\circ h,\qquad h\circ(f+g) = h\circ f +h\circ g
	$$
\item finite products and coproducts exist
\end{enumerate}
A~category is called \term{preadditive} if the~first condition holds but not the~second.
\end{definition}

A~finite product of Abelian groups is also their coproduct.
In fact, this holds in any additive category.

\begin{theorem}\label{thm:cat-prod-coprod}
Let $\cat{C}$ be a~preadditive category and $A,B$ its objects.
Then the~product $A\times B$ exists if and only if there is the~coproduct $A\oplus B$.
Moreover, these two are equal and
\begin{equation}\label{eq:prod-coprod}
\pi_Ai_A = \id_A,\quad \pi_Bi_B =\id_B, \quad \pi_Ai_B = \pi_Bi_A = 0, \quad i_A\pi_A + i_B\pi_B = \id_{A\times B}.
\end{equation}
\end{theorem}
\begin{proof}
Assume the~product $A\times B$ exists. Then there is a~unique morphism
$i_A$ given by the~following diagram
$$\xy
\Atrianglepair/->`-->`->`<-`->/[A`A`A\times B`B;\id_A`i_A`0`\pi_A`\pi_B]
\endxy$$
and similarly we can define $i_B$. First fourth equalities in~\eqref{eq:prod-coprod} hold trivially.
To show the~last one consider the~following diagram with $m=i_A\pi_A + i_B\pi_B$:
$$\xy
\Atrianglepair/->`->`->`<-`->/[A\times B`A`A\times B`B;\pi_A`m`\pi_B`\pi_A`\pi_B]
\endxy$$
The diagram commutes, so $m=\id$ and the~triple $(A\times B, i_A, i_B)$ is a~coproduct
with the~induced morphism $h = f\pi_A + g\pi_B$ for any $f$ and~$g$.

To end the~proof we need to show that a~terminal object $T$ is also an~initial object.
In a~preadditive category there is a~zero morphism $0\colon A\to B$ for any
objects $A$ and $B$. In particular, $\id_T = 0$.
Hence, any morphism $f\colon T\to A$ is equal $f\circ\id_T = 0$, so $T$ is initial.

In the~same way one can prove that coproducts are products.
\end{proof}

\begin{corollary}
Categories $\cat{Set}$ and $\cat{Grp}$ cannot be extended to preadditive categories
preserving both products and coproducts.
\end{corollary}

In a~preadditive category the~initial object, which is also terminal, is called
\term{the~zero object} $0$. Any morphism having it as a~domain or a~codomain
has to be a~zero morphism.

Now we will give a~more general notion of an~additive category.

\begin{definition}
Let $R$ be a~commutative ring. Say a~category $\cat{C}$ is \term{$R$-preadditive},
if a~set of morphisms $\Mor_\cat{C}(X,Y)$ is an~$R$-module for any objects
$X$ and $Y$. A~category $\cat{C}$ is \term{$R$-additive}, if it is both additive and $R$-preadditive.
\end{definition}

A~preadditive category is $\mathbb{Z}$-preadditive.
It can be extended to a~$R$-preadditive one by tensoring it with $R$:
\begin{equation}
\Mor_{\tilde{\cat{C}}}(X,Y) = \Mor_\cat{C}(X,Y)\otimes R
\end{equation}

\begin{remark}
A~category $\cat{C}$ can be extended to an~$R$-preadditive category $R\cat{C}$
in the~following way:
\begin{itemize}
\item objects are preserved: $\Ob(R\cat{C}) = \Ob(\cat{C})$
\item the~set of morphisms $\Mor_{R\cat{C}}(X,Y)$ is the~free $R$-module
		generated by $\Mor_\cat{C}(X,Y)$, i.e. it consists of formal finite sums
	$$
		\sum_{i=1}^n r_if_i \colon X\to Y
	$$
	where $r_i\in R$ and $f_i\colon X\to Y$ are morphisms in $\cat{C}$
\item composition in $R\cat{C}$ is a~bilinear extension of a~composition in $\cat{C}$:
   \begin{align*}
		f\circ (rg+sh) &= r(f\circ g) + s(f\circ h)\\
		 (rf+sg)\circ h &= r(f\circ h) +s(g\circ h)
   \end{align*}
\end{itemize}
\end{remark}

\begin{remark}
Any preadditive category $\cat{C}$ can be extended to an~additive category $\Mat(\cat{C})$
as follows:
\begin{itemize}
\item objects of $\Mat(\cat{C})$ are formal direct sums $\bigoplus_{i=1}^n C_i$ of objects from $\cat{C}$
\item a~morphism $F\colon \bigoplus_{i=1}^n X_i\to\bigoplus_{j=1}^m Y_j$ is a~matrix of morphisms
$$
F = (F_{ji}\colon X_i\to Y_j)
$$
\item the~addition of morphisms in $\Mat(\cat{C})$ is defined as the~addition of matrices
\item the~composition of morphisms is defined as the~multiplication of matrices:
$$
F_{ij}\circ G_{jk} = H_{ik},\qquad\textrm{where }H_{ik} = \sum_j F_{ij}\circ G_{jk}
$$
\end{itemize}
\end{remark}

Obviously, if $\cat{C}$ is $R$-preadditive, its extension to an~additive category
is $R$-additive. The~category $\Mat(\cat{C})$ is called \term{the~category of matrices
over $\cat{C}$} or \term{the~additive closure} of~$\cat{C}$.
Objects can be represented by finite sequences of objects from $\cat{C}$,
while morphisms by bundles of morphisms from~$\cat{C}$ (fig.~\ref{fig:cat-mat-mor}).
In this view, the~component $(F\circ G)_{ik}$ is a~sum of all paths from $X_k$ to $Z_i$.

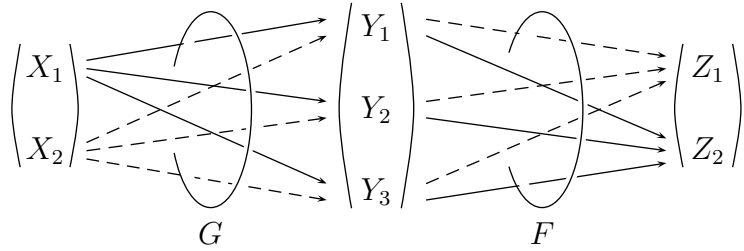
\begin{figure}
	\psset{unit=6ex}
	\begin{pspicture}(8,3)
	\psset{linewidth=0.5pt}
	\pscurve{-}(-0.3,0.8)(-0.4,1.5)(-0.3,2.3)
	\rput(0,2){$X_1$}
	\rput(0,1){$X_2$}
	\pscurve{-}(0.3,0.8)(0.4,1.5)(0.3,2.3)

	\psline{->}(0.5,2.1)(3.4,2.6)
	\psline{->}(0.5,2.0)(3.4,1.6)
	\psline{->}(0.5,1.9)(3.4,0.6)
	\psset{linestyle=dashed}
	\psline{->}(0.5,1.1)(3.4,2.4)
	\psline{->}(0.5,1.0)(3.4,1.4)
	\psline{->}(0.5,0.9)(3.4,0.4)
	\psset{linestyle=solid,linecolor=white,linewidth=3.5pt}
	\psellipticarc{-}(2.05,1.5)(0.5,1.2){240}{120}
	\psset{linecolor=black,linewidth=0.5pt}
	\psellipticarc{-}(2,1.5)(0.5,1.2){230}{130}
	\rput(2,0){$G$}

	\pscurve{-}(3.7,0.3)(3.55,1.5)(3.7,2.8)
	\rput(4,2.5){$Y_1$}
	\rput(4,1.5){$Y_2$}
	\rput(4,0.5){$Y_3$}
	\pscurve{-}(4.3,0.3)(4.45,1.5)(4.3,2.8)

	\psline{->}(4.6,2.4)(7.5,1.15)
	\psline{->}(4.6,1.4)(7.5,1.00)
	\psline{->}(4.6,0.4)(7.5,0.85)
	\psset{linestyle=dashed}
	\psline{->}(4.6,2.6)(7.5,2.15)
	\psline{->}(4.6,1.6)(7.5,2.00)
	\psline{->}(4.6,0.6)(7.5,1.85)
	\psset{linestyle=solid,linecolor=white,linewidth=3.5pt}
	\psellipticarc{-}(6.05,1.5)(0.5,1.2){240}{120}
	\psset{linecolor=black,linewidth=0.5pt}
	\psellipticarc{-}(6,1.5)(0.5,1.2){240}{120}
	\rput(6,0){$F$}

	\pscurve{-}(7.7,0.8)(7.6,1.5)(7.7,2.3)
	\rput(8,2){$Z_1$}
	\rput(8,1){$Z_2$}
	\pscurve{-}(8.3,0.8)(8.4,1.5)(8.3,2.3)

	\end{pspicture}
	\caption{The~composition of morphisms in the~additive closure of a~category.
				Solid lines denote the~component $(F\circ G)_{21}$.}\label{fig:cat-mat-mor}
\end{figure}

\begin{example}
If $\cat{C}$ is additive, $\Mat(\cat{C})$ is equivalent to $\cat{C}$. Indeed,
by uniqueness of coproducts there is a~natural isomorphism between a~formal direct sum
and the~internal direct sum in the~category $\cat{C}$. Thus, the~closure introduces
no essential objects.
\end{example}

\begin{example}
Let $\cat{R}$ be a~category consisting of a~unique object being a~commutative ring $R$
and $\Mor(\cat{R}) = \mathrm{End}(R)$. Then $\Mat(\cat{R})$ is the~category of free modules
over $R$. In particular, when $R=\mathbb{K}$ is a~field, we get the~category
$\cat{Vect}_\mathbb{K}$ of vector spaces over $\mathbb{K}$.
\end{example}

A~functor $F\colon\cat{C}\to\cat{D}$ between $R$-preadditive categories is called
\term{additive}, if it is $R$-linear on morphisms, i.e. $F(rf+sg) = rFf +sFg$.
Given a~product $A\times B$, an~additive functor $F$ induces morphisms
\begin{equation}
FA\two/<-`->/^{F\pi_A}_{Fi_A} F(A\times B)\two/->`<-/^{F\pi_B}_{Fi_B} FB
\end{equation}
satisfying equations~(\ref{eq:prod-coprod}). Uniqueness of the~product provides
$F(A\times B)\cong FA\times FB$ and the~following holds.

\begin{proposition}\label{prop:funct-add-prod}
An~additive functor $F\colon\cat{C}\to\cat{D}$ between preadditive categories
preserves all products and~coproducts.
\end{proposition}

We will end this section with the~notion of gradation.

\begin{definition}\label{def:cat-grad}
An~$R$-preadditive category $\cat{C}$ has \term{a~gradation}, if
\begin{enumerate}
\item for any objects $X,Y$ the~set $\Mor_\cat{C}(X,Y)$ is a~graded $R$-module
		with a~distinguished subset $\widetilde\Mor_\cat{C}(X,Y)$ consisting of functions
		called \term{homogeneous} such that $\id_A$ is homogeneous for any object~$A$
\item there is a~degree function $\deg\colon\widetilde\Mor(\cat{C})\to\mathbb{Z}$ defined for homogeneous functions,
		agreeing with the~composition, i.e. $\deg(f\circ g) = \deg f+ \deg g$
\item there is an~operation called \term{gradation shift}
		$$\Ob(\cat{C})\times\mathbb{Z}\ni(X,m)\to X\{m\}\in\Ob(\cat{C}),$$
		preserving morphisms, i.e.~$\Mor_\cat{C}(X\{m\},Y\{n\}) = \Mor_\cat{C}(X,Y)$,
		but changing its gradings: a~morphism $f\in\Mor_\cat{C}(X\{m\},Y\{n\})$ has degree
		$\deg f = d+n-m$, where~$d$ is the~degree of $f$ as an~element of~$\Mor_\cat{C}(X,Y)$
\end{enumerate}
\end{definition}

\begin{example}\label{ex:cat-sigma}
Any additive category $\cat{C}$ can be extended to a~graded category $\Sigma{\cat{C}}$ as follows:
\begin{enumerate}
\item objects are direct sums: $\bigoplus_{i\in\mathbb{Z}}X^i$
\item morphisms are the morphisms from $\cat{C}$ and $f\colon X\to Y$ is homogeneous
		of degree $\deg f = r$, if $f=\bigoplus_{i\in\mathbb{Z}}f^i$ is a~direct sum
		of morphisms $f^i\colon X^i\to Y^{i+r}$
\item the~gradation shift is defined by a~shift of indices: $X\{m\} = \bigoplus_{i\in\mathbb{Z}}X^{i+m}$
\end{enumerate}
\end{example}

Let $\cat{C}$ be an~arbitrary category. If there is a~function
$\deg\colon\Mor(\cat{C})\to\mathbb{Z}$ that is additive under compositions,
then the~category $\cat{C}$ can be extended to a~graded category $\tilde{\cat{C}}$
such that $\deg$ becomes a~degree map. It is done by adding formal objects $X\{m\}$
for all $X\in\Ob(\cat{C})$ and~$m\in\mathbb{Z}$.

\begin{example}
Define the~degree map in the~category $\cat{ChCob}^3(B)$ as an~Euler characteristic
corrected by half of the~number of endpoints $B$:
\begin{equation}
\deg M = \chi(M) - \frac{1}{2}|B|.
\end{equation}
This function satisfies first two points of the~definition~\ref{def:cat-grad}.
Hence, $\cat{ChCob}^3$ can be extended to a~graded category using the~procedure
described above.
\end{example}

\section{Chain complexes}\label{sec:hom-com}
In the~example~\ref{ex:cat-sigma} we have constructed a category $\Sigma\cat{C}$ from any
additive category $\cat{C}$. Denote by $(C,d)$ a pair consisting of an object from $\Sigma\cat{C}$
and a morphism $d\colon C\to C$ of degree $+1$.

\begin{definition}
Providing $d\circ d = 0$, a~pair $(C,d)$ is called \term{a~chain complex},
and a~morphism $d$ is \term{a~differential} of $C$.
A~complex $(C,d)$ is \term{bounded}, if $C^r = 0$ for sufficiently large and small $r$.
\term{A~chain map} $f\colon (C_1,d_1)\to (C_2,d_2)$ is a~morphism
$f\colon C_1\to C_2$ which commutes with differentials, i.e.~$f\circ d_1 = d_2\circ f$.
\end{definition}

Chain complexes with chain maps form a~graded category. The~subcategory consisting of
all chain complexes and morphisms of degree 0 is denoted by $\Kom(\cat{C})$.
Both categories $\cat{C}$ and~$\Sigma\cat{C}$ can be seen as its subcategories,
when zero objects and differentials are added, i.e. an object $X\in\Ob\cat{C}$ can be seen as a complex
\begin{equation}
\cdots\to 0\to^0 \underline X\to^0 0\to\cdots
\end{equation}
where the underlined term $X$ is in a~degree 0.

\begin{remark}
If $\cat{C}$ is graded, the~category of complexes $\Kom(\cat{C})$ gets another grading
($f$ is homogenous of degree $d$, if for every $r\in\mathbb{Z}$ $f^r$ is homogenous of degree $d$).
Thus we obtain two gradations:
\begin{itemize}
\item a complex gradation (exterior): $C[i]^r := C^{r+i}$
\item an induced gradation (interior): $C\{m\}^r := C^r\{m\}$
\end{itemize}
Since now, We will denote by $\deg$ the~degree connected to the~induced gradation
and differentials as assumed to have degree zero with respect to the~induced grading
(i.e.~$\deg d = 0$).
\end{remark}

\begin{remark}
The~category $\Kom(\cat{C})$ is additive with natural products:
\begin{equation}
(C_1,d_1)\oplus (C_2,d_2) := (C_1\oplus C_2, d_1\oplus d_2)
\end{equation}
where~$(C_1\oplus C_2)^r = C_1^r\oplus C_2^r$, $(d_1\oplus d_2)^r = d_1^r\oplus d_2^r$.
\end{remark}

Notice that in $\cat{Mod}_R$ each morphism $f\colon M\to M'$ has
a~kernel $\ker f = f^{-1}(0)$ and an~image $\im f = f(M)$, both being $R$-modules.
The~condition $d\circ d=0$ implies that~$\ker d \subset \im d$,
and we can create quotient modules
\begin{equation}
H^n(C) = \ker d^n/ \im d^{n-1}
\end{equation}
The~graded module $H^{\bullet}(C)$ is called \term{a~homology} of the~complex $(C,d)$.
$(C,d)$ is called \term{an~exact sequence}, if $H^{\bullet}(C) = 0$.

\begin{theorem}
Take an exact sequence of complexes of $R$-modules:
\begin{equation}
0\to A\to B\to C\to 0
\end{equation}
Then there exists a long exact sequence of homologies:
\begin{equation}
\cdots\to H^i(A)\to H^i(B)\to H^i(C)\to H^{i+1}(A)\to \cdots
\end{equation}
\end{theorem}

\begin{definition}
Let $f\colon(C_a,d_a)\to (C_b,d_b)$ and~$g\colon(C_a,d_a)\to (C_b,d_b)$ be chain maps.
\term{A~chain homotopy} from~$f$ to~$g$ is a~morphism $h\colon C_a\to C_b[-1]$ such that
\begin{equation}
f - g = hd + dh
\end{equation}
In this case $f$ and~$g$ are called \term{homotopic} and we write $f\sim_h g$.

Chain complexes $(C_a,d_a)$ and~$(C_b, d_b)$ are called \term{homotopic},
if there exist morphisms $f\colon(C_a,d_a)\to (C_b,d_b)$
and~$g\colon(C_b,d_b)\to (C_a,d_a)$ such that
\begin{equation}
f\circ g \sim_h \id\qquad g\circ f\sim_h \id
\end{equation}
Morphisms $f$ and $g$ are called \term{homotopy equivalences}.
\end{definition}

The~homotopy equivalence relation agrees with compositions of morphisms. Therefore
the~quotient category $\Kom_{/h}(\cat{C})$ is well defined.
Moreover, directly from the~definition, homotopic chain complexes has isomorphic
homologies.

A special kind of homotopy equivalences are \term{deformation retractions},
which are the~morphisms $g\colon(C_b,d_b)\to (C_a,d_a)$ having \term{a section}
$f\colon(C_a,d_a)\to (C_b,d_b)$ being its homotopy inverse:
$$
g\circ f = \id_{C_a}\qquad f\circ g\sim_h\id.
$$
If there exists a~homotopy $h$ such that~$h\circ f = 0$, then~$g$ is called
\term{a strong deformation retraction}, whereas $f$ is
\term{an inclusion in a strong deformation retract}.

The~next definition show how to construct a~new chain complex from a~chain map.

\begin{definition}\label{def:cat-cone}
Let $f\colon (C_0,d_0)\to (C_1,d_1)$ be a~chain map.
\term{A cone} of $f$ is a~complex $(\cone(f),\tilde{d})$ defined as follows:
\begin{equation}
\cone(f) = C_0\oplus C_1[-1],\qquad \tilde{d} = \left(
\begin{matrix}
-d_0 & 0\\
f    & d_1[-1]
\end{matrix}
\right)
\end{equation}
\end{definition}

A~commutative square induces a~morphism of cones. Indeed, directly from
the~definition~\ref{def:cat-cone} we have the~following result.

\begin{proposition}
Let the~following be a~commutative diagram of complexes:
$$
\xy
\square[C_{1a}`C_{1b}`C_{2a}`C_{2b};g_1`f_a`f_b`g_2]
\endxy
$$
Then the~morphism $g=g_1\oplus g_2[-1]\colon\cone(f_a)\to\cone(f_b)$
is a~chain map.
\end{proposition}

In particular, the~following diagram (zero objects are omitted):
$$
\xy
\morphism(500,500)|l|/{->}/<0,-500>[C_1`C_2;f]

\morphism(500,500)|a|/{->}/<500,0>[C_1`C_1;\id]
\morphism(  0,  0)|a|/{->}/<500,0>[C_2`C_2;\id]
\endxy
$$
induces morphisms $i\colon C_2\to\cone(f)$ and~$\pi\colon\cone(f)\to C_1$ which
form a~sequence:
\begin{equation}\label{eq:cone-seq}
0\to C_2[-1]\to^i \cone(f)\to^\pi C_1\to 0
\end{equation}
In a~category of $R$-modules the~above sequence is exact and having an~additive
functor $F\colon\cat{C}\to\cat{Mod}_R$ along with a~morphism $f\colon C\to D$
we get a~long exact sequence of homologies
\begin{equation}
\cdots\to H^i(FD)\to H^i(FC)\to H^i(\cone(Ff))\to H^{i+1}(FD)\to \cdots
\end{equation}
as due to the~proposition~\ref{prop:funct-add-prod} an~additive functor preserves cones.

We will end this section with a~theorem of invariance of cones under homotopies
when composed with strong deformation retracts. In many cases using this theorem
a~given complex can be simplified a~lot if an~appropriate retract is known.

\begin{theorem}\label{thm:cone-homot}
Let the~following be a~commutative diagram
$$
\xy
\square/{<-}`{->}`{}`{->}/[C_{0a}`C_{0b}`C_{1a}`C_{1b};f_0`F``g_1]
\endxy
$$
where $f_0$ is an~inclusion in a~strong deformation retract and $g_1$ is a~strong deformation retraction.
Then all the~cones $\cone(F), \cone(Ff_0), \cone(g_1F)$ are homotopic.
\end{theorem}
\begin{proof}
For~$f_0$ there is a~strong deformation retraction $g_0\colon C_{0a}\to C_{0b}$
and a~homotopy $h_0\colon C_{0a}\to C_{0a}[-1]$ such that
$$
g_0f_0 = \id,\qquad \id - f_0g_0 = dh_0+h_0d,\qquad h_0f_0 = 0
$$
Take the~following morphisms:
\begin{align*}
&\tilde{f}_0\colon\cone(Ff_0)\to\cone(F)      && \tilde{f}_0^r = \left(\begin{matrix}f_0 & 0\\ 0 & \id\end{matrix}\right) \\
&\tilde{g}_0\colon\cone(F)\to\cone(Ff_0)      && \tilde{g}_0^r = \left(\begin{matrix}g_0 & 0\\ Fh_0 & \id\end{matrix}\right) \\
&\tilde{h}_0\colon\cone(F)^*\to\cone(F)^{*-1} && \tilde{h}_0^r = \left(\begin{matrix}-h_0 & 0\\ 0 & \id\end{matrix}\right)
\end{align*}
They form a~commutative diagram:
$$
\xy
\place(-200,800)[\cone(Ff_0):]
\place(-200,0)[\cone(F):]

\morphism( 300,800)|a|/{->}/<700,0>[\cdots`C_{0b}^r\oplus C_{1a}^{r-1};]
\morphism(2000,800)|a|/{->}/<700,0>[C_{0b}^{r+1}\oplus C_{1a}^r`\cdots;]

\morphism(1000,800)|a|/{->}/<1000,0>[\phantom{C_{0b}^r\oplus C_{1a}^{r-1}}`\phantom{C_{0b}^{r+1}\oplus C_{1a}^r};\tilde{d}]

\morphism(1000,800)|r|/@{>}@<2pt>/<0,-800>[\phantom{C_{0b}^r\oplus C_{1a}^{r-1}}`\phantom{C_{0a}^r\oplus C_{1a}^{r-1}};\tilde{f}_0^r]
\morphism(1000,  0)|l|/@{>}@<2pt>/<0, 800>[\phantom{C_{0a}^r\oplus C_{1a}^{r-1}}`\phantom{C_{0b}^r\oplus C_{1a}^{r-1}};\tilde{g}_0^r]

\morphism(2000,800)|r|/@{>}@<2pt>/<0,-800>[\phantom{C_{0b}^{r+2}\oplus C_{1a}^r}`\phantom{C_{0a}^{r+1}\oplus C_{1a}^r};\tilde{f}_0^{r+1}]
\morphism(2000,  0)|l|/@{>}@<2pt>/<0, 800>[\phantom{C_{0a}^{r+2}\oplus C_{1a}^r}`\phantom{C_{0b}^{r+1}\oplus C_{1a}^r};\tilde{g}_0^{r+1}]

\morphism(1000,0)|a|/@{>}@<2pt>/< 1000,0>[\phantom{C_{0a}^r\oplus C_{1a}^{r-1}}`\phantom{C_{0a}^{r+1}\oplus C_{1a}^r};\tilde{d}]
\morphism(2000,0)|b|/@{>}@<2pt>/<-1000,0>[\phantom{C_{0a}^{r+2}\oplus C_{1a}^r}`\phantom{C_{0a}^r\oplus C_{1a}^{r-1}};\tilde{h}_0]

\morphism( 300,0)|a|/{->}/<700,0>[\cdots`C_{0a}^r\oplus C_{1a}^{r-1};]
\morphism(2000,0)|a|/{->}/<700,0>[C_{0a}^{r+1}\oplus C_{1a}^r`\cdots;]

\endxy
$$
Moreover, $\tilde{g}_0\tilde{f}_0 = \id$ and~$\id-\tilde{f}_0\tilde{g}_0 = \tilde{d}\tilde{h}+\tilde{h}_0\tilde{d}$.
Hence, complexes $\cone(F)$ and~$\cone(Ff_0)$ are homotopic. The~other equivalence is shown
in the~same way.
\end{proof}

\section{Cubes}\label{sec:hom-cubes}
Let $I^n$ be a~standard unit $n$-cube in~$\mathbb{R}^n$.
Its edges together with vertices form a~directed graph.
Vertices are labeled with zero-one sequences $\xi=(\xi_1,\dots,\xi_n)$ of length $n$.
Let $|\xi| = \xi_1+\cdots+\xi_n$.
Replacing the~$i$-th item with a~star $*$, we get a label of an~edge
$\zeta=(\zeta_1,\dots,*,\dots,\zeta_n)$ going from
$\zeta(0) =(\zeta_1,\dots,0,\dots,\zeta_n)$ to~$\zeta(1)=(\zeta_1,\dots,1,\dots,\zeta_n)$.

\begin{definition}\label{def:cube}
\term{A~cube diagram} of dimension $n$ or~\term{an $n$-cube} in a~category $\cat{C}$
is a~mapping $F\colon I^n\to\cat{C}$ which associates each vertex with an~object of $\cat{C}$
and each edge $\zeta$ with a~morphism $F\zeta\colon F\zeta(0)\to F\zeta(1)$.
\term{A morphism} $\eta\colon F\to G$ of $n$-cubes is a~collection
of morphisms $\{\eta_\xi\colon F\xi\to G\xi\}$.
\end{definition}

All $n$-cubes in a~given category~$\cat{C}$ form a~category $n\Cub(\cat{C})$ with
an~obvious composition. Denote by~$\Cub(\cat{C})$ the~category of cubes of any dimension.

\begin{remark}\label{rm:bij-cube-mor}
A~morphism $F_*\colon F_0\to F_1$ of $n$-cubes induces an $(n\!+\!1)$-cube $F$
given as follows:
\begin{align}
F(\xi,i) &= F_i(\xi) \\
F(\xi,*) &= (F_*)_\xi
\end{align}
Contrary, each $(n\!+\!1)$-cube $F$ produces a~cube morphism
$F_*\colon F_0\to F_1$, where $F_i = F(\cdot,i)$. The~above correspondence is clearly a~bijection.

More generally, for every~$(m\!+\!n)$-cube $F\in\Cub(\cat{C})$ we can construct an~$n$-cube
$F^{(m)}\in\Cub(m\Cub(\cat{C}))$ as follows:
\begin{align}
F^{(m)}(\xi) &= F_\xi := F(\cdot,\xi) \\
F^{(m)}(\zeta) &= F_\zeta := F(\cdot,\zeta)
\end{align}
Hence, every $(m\!+\!n)$-cube $F$ can be seen as an $n$-cube $F^{(m)}$
in a~category of $m$-cubes. Contrary, each $n$-cube $F$
in such a~category describes an $(m\!+\!n)$-cube in $\cat{C}$.
\end{remark}

Since now fix a~commutative ring $R$ and assume $\cat{C}$ is $R$-additive. Denote by $G=U(R)$
the~group of invertible elements in $R$.
\term{A~projectivization} of $\cat{C}$ is the~category $\catP{C} = \cat{C}/G$,
in which any two morphisms differing by an~invertible element are identified.
\term{A~projectivization of a~cube} $F\colon I^n\to\cat{C}$ is defined as a~composition
of $F$ with the~canonical projection: $\mathbb{P}F = \mathbb{P}\circ F\colon I^n\to\catP{C}$.

\begin{definition}\label{def:cube-type}
Choose any two dimensional face of a~cube $F\colon I^n\to \cat{C}$
\begin{equation}\label{eq:n-cube-face}
\xy
\square[FA`FB`FC`FD;Fa`Fc`Fb`Fd]
\endxy
\end{equation}
This face is:
\begin{itemize}
\item \term{commutative}, if $Fb\circ Fa = Fd\circ Fc$
\item \term{anticommutative}, if $Fb\circ Fa = -Fd\circ Fc$
\item \term{projective}, if $Fb\circ Fa = \lambda Fd\circ Fc$ for some~$\lambda\in G$
\end{itemize}
The~cube $F$ is called \term{commutative}, \term{anticommutative} or~\term{projective},
if all its faces are respectively commutative, anticommutative or projective.
\end{definition}

It follows from the~above definition that a~cube $F\colon I^n\to\cat{C}$ is projective
if and only if its projectivization $\mathbb{P}F\colon I^n\to\catP{C}$ is commutative.
Projective cubes with equal projectivizations will be called \term{$\mathbb{P}$-equivalent}.

\begin{definition}\label{def:cube-mor-type}
A~cube morphism $\eta\colon F\to G$ is \term{commutative}, \term{anticommutative}
or~\term{projective}, if for each edge~$\zeta$ the~following square is respectively
commutative, anticommutative or projective:
\begin{equation}\label{eq:cube-mor-def}
\xy
\square[F\xi`G\xi`F\xi'`G\xi';\eta_\xi`F\zeta`G\zeta`\eta_{\xi'}]
\endxy
\end{equation}
Two projective morphisms are called \term{$\mathbb{P}$-equivalent},
if their projectivizations are equal.
\end{definition}

Both commutative and projective morphisms are closed under compositions. Hence we obtain
three subcategories in $\Cub(\cat{C})$: commutative cubes with commutative morphisms $\Cub^c(\cat{C})$,
anticommutative cubes with commutative morphisms $\Cub^a('cat{C})$
and projective cubes with projective morphisms $\Cub^p(\cat{C})$.

\begin{remark}\label{rm:bij-cube-mor-subcat}
In analogous to $\Cub(\cat{C})$, each of the~subcategories described above
possesses a~bijection between ($m\!+\!n$)-cubes and $n$-cubes in the~category of $m$-cubes.
In particular, in case $m=1$ there is a~bijection between $(n\!+\!1)$-cubes and morphisms
of $n$-cubes, as every morphism is a $1$-cube (commutative, anticommutative and projective
at the~same time).
However, the~morphism generated by an~anticommutative $(n\!+\!1)$-cube $F$
is given by $\eta\colon -F_0\to F_1$ (otherwise we will get an~anticommutative one).
Moreover, $\mathbb{P}$-equivalence is preserved for projective cubes.
\end{remark}

Let $F$ be a~projective $n$-cube.
Denote its face from diagram~(\ref{eq:n-cube-face}) by~$S$ and assume that directions of
morphisms $Fa$ and~$Fb$ agrees with the~natural orientation of $S$.
A~cochain $\psi\in C^2(I^n; G)$ is \term{associated} to the~cube $F$, if
\begin{equation}
Fb\circ Fa = \psi(S) Fd\circ Fc
\end{equation}
for every face $S$.

\begin{definition}\label{def:CC-cube}
Say a~cube $F$ is \term{a $C\!C$-cube} or that is satisfies \term{the~cocycle condition},
if there exists an~associated cochain being a~cocycle.
\end{definition}

Notice that a~cochain $\varphi\in C^1(I^n; G)$ defines a~cube $\varphi_*F$
by multiplying edges of $F$ by the~values of $\varphi$:
\begin{align}
\varphi_*F(\xi) &= F\xi \\
\varphi_*F(\zeta) &= \varphi(\zeta)F\zeta
\end{align}
If $\psi_F$ is a~cochain associated to $F$, the~cochain
$\psi_{\varphi_*F} = d\varphi\psi_F$ is associated to $\varphi_*F$.
In particular, if $F$ is a~$C\!C$-cube, so is $\varphi_*F$.
Since $\mathbb{P}$-equivalent projective cubes differ only by a~cochain,
we have the~following result.

\begin{corollary}\label{cor:proj-cube-repr}
Let $F$ and~$G$ be $\mathbb{P}$-equivalent projective cubes.
Then $F$ is a~$C\!C$-cube if and only if $G$ is a~$C\!C$-cube.
\end{corollary}

\noindent Say a~commutative cube $F\colon I^n\to\mathbb{P}\cat{C}$ is \term{a~$C\!C$-cube},
if there exists (and due to the~corollary every) its representative being a~$C\!C$-cube.

A~cochain $\varphi\in C^1(I^n; G)$ is \term{a~positive} or~\term{a~negative edge assignment}
of $F$, if $\varphi_*F$ is respectively a~commutative or anticommutative cube.
Directly from the~definition of a~differential we get the~lemma below.

\begin{lemma}
An~edge assignment $\varphi$ of a~cube $F$ is positive (negative) if and only if
$d\varphi = \psi$ (respectively: $d\varphi = -\psi$) for some associated cochain $\psi$.
\end{lemma}

\noindent In the~above situation the~edge assignment $\varphi$ is said to be \term{of type $\psi$}.
Having two edge assignments $\varphi_1$ and~$\varphi_2$ of the~same type (positive or negative)
the~equality $d(\varphi_1\varphi_2^{-1}) = 1$ is satisfied and the following holds.

\begin{theorem}\label{thm:cubes-ass-isom}
A~cube $F$ has both a~positive and a~negative edge assignment if and only if $F$ is a~$C\!C$-cube.
Furthermore, two edge assignments (both positive or negative) of the~same type
induce isomorphic cubes (in the~sense of commutative isomorphisms).
\end{theorem}
\begin{proof}
If $d\psi = 1$, then $\psi = d\varphi$ is a~coboundary, since~$H^2(I^n; G) = 0$.
This shows the~existence part (for negative assignments notice\footnote{\ 
Notice we use multiplicative notation for the~group $G$, so $d(-\psi)\neq -d\psi$.
Instead, we have an~equality $d\psi^{-1} = (d\psi)^{-1}$.} that $d(-\psi) = d\psi$).
As $H^1(I^n; G) = 0$, having two edge assignments (both positive or negative) $\varphi_1$
and~$\varphi_2$ there is a~cochain $\eta\in C^0(I^n; G)$ such that
$\varphi_1 = (d\eta)\varphi_2$.
Thus there exists a~commutative morphism $f\colon (\varphi_1)_*F \to (\varphi_2)_*F$ given by:
$$
f_\xi = \eta(\xi)\id_\xi,
$$
which is an~isomorphism, since~each $\eta(\xi)\in G$ is invertible.
\end{proof}

\begin{corollary}\label{col:CC-cube-unique}
Up to cube isomorphisms,  a~$C\!C$-cube $F\colon I^n\to\mathbb{P}\cat{C}$ describes for
each associated cocycle $\psi$ a~unique commutative and a~unique anticommutative cube in $\cat{C}$.
\end{corollary}

According to the~corollary~\ref{col:CC-cube-unique} each projective $C\!C$-cube
is $\mathbb{P}$-equivalent to a~commutative one. A~morphism of projective cubes
is called \term{a~$C\!C$-morphism}, if it induces a~$C\!C$-cube.
Due to the~previous observations, every $C\!C$-morphism is $\mathbb{P}$-equivalent
to a~morphism of commutative cubes. In fact, a~stronger theorem holds.

\begin{theorem}\label{thm:proj-cube-mor}
Let $F_0$ and~$F_1$ be commutative $n$-cubes.
Then for every $C\!C$-morphism $\eta\colon\mathbb{P}F_0\to\mathbb{P}F_1$
there exists a~representative $\tilde\eta\colon F_0\to F_1$ that is a~morphism of commutative
cubes $F_0$ i~$F_1$. Moreover, every two such representatives $\tilde\eta_1$ and~$\tilde\eta_2$
differs by an~invertible element, i.e.~$\tilde\eta_1 = \lambda\tilde\eta_2$ for some $\lambda\in G$.
\end{theorem}
\begin{proof}
Let $\tilde\eta$ represent $\eta$.
It induces an~($n$+1)-cube $F$, which is not commutative in general.
We have to find a~positive edge assignment for $F$, equals one on $I^n\times\partial I$.

Pick $\psi$ a~cocycle associated to $F$. Then $\psi(S) = 1$ for each face $S$ in $F_0$ or~$F_1$,
so $\psi\in C^2(I^{n+1},I^n\times\partial I; G)$ is a~relative cochain.
As the second relative homology group vanishes
$$H^2(I^{n+1},I^n\times\partial I; G) = 0$$
there is a~cochain $\varphi\in C^1(I^{n+1},I^n\times\partial I; G)$ such that $d\varphi = \psi$.
It is equal one on edges of both $F_0$ and~$F_1$, so
$\varphi_*\tilde\eta$ is a~commutative morphism of cubes $F_0$ and~$F_1$, representing~$\eta$.

To show the~second part notice that every two representatives $\tilde\eta_1$ and~$\tilde\eta_2$
induces a~cocycle $\varphi\in C^1(I^{n+1},I^n\times\partial I; G)$
such that $\tilde\eta_2=\varphi_*\tilde\eta_1$.
Hence, for two edges $\zeta, \zeta'$ from $I\times 0$ to $I\times 1$ and a~face $S$
connecting them we have
$$
\varphi(\zeta,*)\varphi(\zeta',*)^{-1} = d\psi(S) = 1
$$
what gives $\varphi(\zeta,*) = \varphi(\zeta',*)$. Connectedness of a~cube provides
$\varphi$ is constant, what ends the proof.
\end{proof}

\begin{corollary}\label{cor:part-edge-assign}
Let $F$ be an~$(n\!+\!m)$-cube satisfying the~cocycle condition such that
each~$m$-cube $F_\xi$ is commutative. Then there exists an~edge assignment
$\varphi$ of a~cube $F$ such that $\varphi_*(F_\xi) = F_\xi$ for each edge $\xi\in I^n$.
Moreover, if $\psi$ is a~cochain associated to $F$ such that $\psi|_{F_\xi} = 1$ for each
edge $\xi\in I^n$, we may assume that $\varphi$ is of type~$\psi$.
\end{corollary}
\begin{proof}
The case $n=0$ is trivial.
Assume the hypothesis holds for~$n=k$ and take a~$(k\!+\!1\!+m)$-cube $F$.
By induction hypothesis there exists edge assignments $\varphi_0$ and~$\varphi_1$
of cubes $F_0$ and~$F_1$. Applying the~theorem~\ref{thm:proj-cube-mor} we find
an~edge assignment $\varphi'$ of a~morphism $F_*\colon F_0\to F_1$.
The~product $\varphi_0\varphi_1\varphi'$ is the~desired edge assignment.
\end{proof}

\begin{corollary}
Let $F$ and~$G$ be commutative cubes and let
$\mathbb{P}\eta\colon\mathbb{P}F\to\mathbb{P}G$ be a~$C\!C$-morphism.
Then having a~cocycle $\psi$ associated to $\eta$ there exists a~unique commutative representative
$\tilde\eta\colon F\to G$ such that $\tilde\eta(0,\dots,0) = \eta(0,\dots,0)$.
\end{corollary}

\begin{remark}
The~above results also hold when the words `commutative' are replaced with `anticommutative'.
\end{remark}

Since $\mathbb{P}$-equivalence agrees with compositions, a~composition of $C\!C$-morphisms
is still a~$C\!C$-morphism. Therefore, a~category of projective cubes contains a~subcategory
of $C\!C$-cubes $\Cub^{C\!C}(\cat{C})$.

\section{Cube complexes}\label{sec:hom-com-cub}
All categories $\Cub(\cat{C})$, $\Cub^c(\cat{C})$, $\Cub^a(\cat{C})$, $\Cub^p(\cat{C})$
and~$\Cub^{CC}(\cat{C})$ constructed in the~previous section are $R$-additive. Indeed,
the~action of $R$ on morphisms is induced from the~category $\cat{C}$,
whereas the~coproduct of $F$ and $G$ is given as follows:
\begin{align}
(F\oplus G)(\xi) &:= F\xi\oplus G\xi \\
(F\oplus G)(\zeta) &:= F\zeta\oplus G\zeta
\end{align}
The~zero element is the~zero cube $0(\xi) := 0_\cat{C}$.

We will now pass to anticommutative cubes, as they can be used to produce chain complexes
in an easy way.

\begin{definition}\label{def:com-cub}
Let $F$ be an~anticommutative cube. \term{A~cube complex induced by~$F$}
is the~complex $\Kom(F) = (C_F,d_F)$ given as follows:
\begin{align}
C_F^r &:= \bigoplus_{|\xi| = r}F\xi \\
d_F^r|_{F\xi} &:= \sum_{\zeta\colon\xi\rightarrow\xi'} F\zeta
\end{align}
\end{definition}

The condition $d\circ d= 0$ holds due to anticommutativity of $F$.
Notice that a~commutative morphism $\eta\colon F\to G$ of anticommutative cubes induces
in a~natural way a~chain map $\Kom(\eta)\colon C_F\to C_G$ and we get a~functor $\Kom$
from the~category $\Cub^a(\cat{C})$ to chain complexes $\Kom(\cat{C})$.
In particular, an~anticommutative ($n$+1)-cube $F$ induces morphisms $\eta\colon -F_0\to F_1$
and $\Kom(\eta)\colon \Kom(-F_0)\to\Kom(F_1)$.
Directly from the~definition of a~cone we obtain the~following

\begin{theorem}\label{thm:cubes-cone}
The~complex $\Kom(F)$ is equal to the~complex $\cone(\Kom(\eta))$.
\end{theorem}

The theorem~\ref{thm:cubes-cone} is the~first step to compute cube complexes partially.
Here, we can at first compute complexes $\Kom(F_0)$ and~$\Kom(F_1)$,
postponing computations of $\Kom(F)$ to the~next step. Now we will develop this approach.

At the~beginning let us extend $\Kom$ over categories of cubes of complexes,
such that we will use the inner structure of complexes.

\begin{definition}\label{def:gkom}
Let $F\colon I^n\to\Kom(\cat{C})$ be an~anticommutative cube.
\term{The~extended cube complex} of $F$ is the~complex $\Kom(F) = (C_F, d_F)$ defined as follows:
\begin{align}
C_F &:= \bigoplus_{\xi\in I^n}(F\xi)[-|\xi|] \\
d_F|_{(F\xi)[-|\xi|]} &:= d_{F\xi}[-|\xi|]+\sum_{\zeta\colon\xi\rightarrow\xi'} (F\zeta)[-|\xi|]
\end{align}
\end{definition}

\begin{figure}
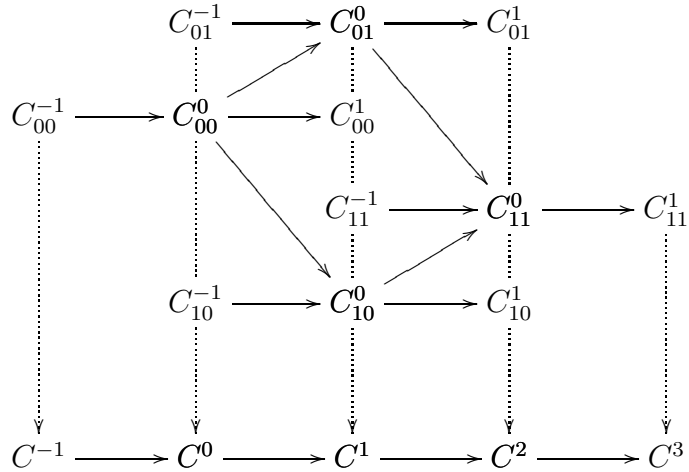

	$$\xy
	\morphism(   0,600)<500,0>[C_{00}^{-1}`C_{00}^0;]
	\morphism( 500,600)<500,0>[C_{00}^0`C_{00}^1;]

	\morphism( 500,900)<500,0>[C_{01}^{-1}`C_{01}^0;]
	\morphism(1000,900)<500,0>[C_{01}^0`C_{01}^1;]

	\morphism( 500,0)<500,0>[C_{10}^{-1}`C_{10}^0;]
	\morphism(1000,0)<500,0>[C_{10}^0`C_{10}^1;]

	\morphism(1000,300)<500,0>[C_{11}^{-1}`C_{11}^0;]
	\morphism(1500,300)<500,0>[C_{11}^0`C_{11}^1;]

	\morphism(500,600)<500, 300>[\phantom{C_{00}^0}`\phantom{C_{01}^0};]
	\morphism(500,600)<500,-600>[\phantom{C_{00}^0}`\phantom{C_{10}^0};]
	\morphism(1000,900)<500,-600>[\phantom{C_{01}^0}`\phantom{C_{11}^0};]
	\morphism(1000,  0)<500, 300>[\phantom{C_{10}^0}`\phantom{C_{11}^0};]

	\morphism(   0,-500)<500,0>[C^{-1}`C^0;]
	\morphism( 500,-500)<500,0>[C^0`C^1;]
	\morphism(1000,-500)<500,0>[C^1`C^2;]
	\morphism(1500,-500)<500,0>[C^2`C^3;]

	\morphism(   0,600)/..>/<0,-1100>[\phantom{C_{00}^{-1}}`\phantom{C^{-1}};]
	\morphism( 500,900)/../<0, -300>[\phantom{C_{01}^{-1}}`\phantom{C_{00}^0};]
	\morphism( 500,600)/../<0, -600>[\phantom{C_{00}^0}`\phantom{C_{10}^{-1}};]
	\morphism( 500,  0)/..>/<0, -500>[\phantom{C_{10}^{-1}}`\phantom{C^0};]
	\morphism(1000,900)/../<0, -300>[\phantom{C_{01}^0}`\phantom{C_{00}^1};]
	\morphism(1000,600)/../<0, -300>[\phantom{C_{00}^1}`\phantom{C_{11}^{-1}};]
	\morphism(1000,300)/../<0, -300>[\phantom{C_{11}^{-1}}`\phantom{C_{10}^0};]
	\morphism(1000,  0)/..>/<0, -500>[\phantom{C_{10}^0}`\phantom{C^1};]
	\morphism(1500,900)/../<0, -600>[\phantom{C_{01}^1}`\phantom{C_{11}^0};]
	\morphism(1500,300)/../<0, -300>[\phantom{C_{11}^0}`\phantom{C_{10}^1};]
	\morphism(1500,  0)/..>/<0, -500>[\phantom{C_{10}^1}`\phantom{C^2};]
	\morphism(2000,300)/..>/<0, -800>[\phantom{C_{00}^1}`\phantom{C^3};]

	\endxy$$
	\caption{The~extended cube complex takes care of the~inner structure of vertices}\label{fig:gkom}
\end{figure}

The~construction is visualised in the~figure~\ref{fig:gkom}.
Definitions~\ref{def:com-cub} and \ref{def:gkom} agree with respect to
the~canonical embedding $\cat{C}\to\Kom(\cat{C})$.
Moreover, treating $\eta\colon F_0\to F_1$ as a $1$-cube, we have the~equality
\begin{equation}\label{eq:gkom-cone}
\cone(\eta) = \Kom(\eta)
\end{equation}
hence~$\Kom$ generalizes the~notion of a~cone. Define now the~family of functors
\begin{equation}\label{eq:kom-gener}
\Kom^m\colon(m\!+\!n)\Cub(\Kom(\cat{C}))\to n\Cub(\Kom(\cat{C}))
\end{equation}
which computes partial cube complexes as follows:
\begin{align}
\Kom^m(F)(\xi) &:= \Kom(F_\xi)\\
\Kom^m(F)(\zeta) &:= \Kom(F_\zeta)
\end{align}
The~above means that a~vertex $\xi$ of a~cube $\Kom^m(F)$ contains a~complex
computed from the~restricted $m$-cube $F_\xi$, what explain why the~functor is called ,,partial''.
Obviously, for any $n$-cube $F$ we have $\Kom^n(F) = \Kom(F)$.
Moreover, direct calculation gives:
	
\begin{theorem}\label{thm:com-part-calc}
Let $F$ be an~anticommutative $k$-cube and~$m+n\leqslant k$. Then
\begin{equation}\label{eq:part-com-grp}
\Kom^m(\Kom^n(F)) = \Kom^{n+m}(F)
\end{equation}
\end{theorem}

Comparing equations~(\ref{eq:gkom-cone}) and~(\ref{eq:part-com-grp}) one can easily see
that the~theorem~\ref{thm:com-part-calc} generalises the~theorem~\ref{thm:cubes-cone}.

Perhaps the~main strength of partial computations is that all $\Kom^n$
preserve chain homotopies.

\begin{proposition}\label{thm:com-cub-homot}
Let $F$ and~$G$ be anticommutative $n$-cubes in a~category of complexes.
If $\eta,\nu,h\colon F\to G$ are cube morphisms such that for each edge $\xi$
$$
\eta_\xi-\nu_\xi = h_\xi d_{F\xi} + d_{G\xi}h_\xi
$$
then~$\Kom(h)$ is a~chain homotopy of induced morphisms $\Kom(\eta)$ and~$\Kom(\nu)$.
\end{proposition}

The~mapping $h$ in the~proposition above is called \term{a~cube homotopy},
whereas cubes $F$ and~$G$ are said to be \term{homotopic}.

\begin{corollary}\label{cor:part-com-homot}
Let $F$ and~$G$ be anticommutative cubes in a~category $\cat{C}$
of dimensions respectively $(m_1\!+\!n)$ and~$(m_2\!+\!n)$.
Then if $n$-cubes $\Kom^{m_1}(F)$ and~$\Kom^{m_2}(G)$ are homotopic,
so are $\Kom(F)$ and~$\Kom(G)$.
\end{corollary}
\begin{proof}
Pick two commutative cube morphisms
\begin{align*}
\eta\colon&\Kom^{m_1}(F)\to\Kom^{m_2}(G)\\
\nu\colon&\Kom^{m_2}(G)\to\Kom^{m_1}(F)
\end{align*}
together with cube homotopies $h_F\colon \nu\eta\simeq\id$ and~$h_G\colon\eta\nu\simeq\id$.
Due to the~proposition~\ref{thm:com-cub-homot} the~morphisms $h_F$ and~$h_G$ induce chain homotopies
\begin{align*}
\Kom(h_F)\colon& \Kom(\nu)\Kom(\eta)\simeq\id\\
\Kom(h_G)\colon& \Kom(\eta)\Kom(\nu)\simeq\id
\end{align*}
and due to the~equation~(\ref{eq:part-com-grp}) we have homotopies of complexes:
$$
\Kom(F) = \Kom(\Kom^{m_1}(F))\simeq\Kom(\Kom^{m_2}(G)) = \Kom(G).
$$
\end{proof}

\chapter{Khovanov complex}\label{chpt:khov}
\chead{\fancyplain{}{\thechapter. Khovanov complex}}
In this chapter we will construct of the~generalized Khovanov complex in the~spirit of Bar-Natan.
At first, we will define a cube and~a~complex in the~additive closure of $\cat{ChCob}^3$,
then we will prove an~invariance of the~latter up to chain homotopies and some relations.
Finally we will give examples of functors into Abelian categories, such that homology groups can be computed.
All of them will categorify the Jones polynomial.

\section{The construction of the complex}\label{sec:khov-def}
One picture is worth of tousand words, therefore we will describe the~generalized
Khovanov complex explaining the figure~\ref{fig:khov-cube}, which shows
the~complex $\KhBra{3_1}$ for the~trefoil.

\begin{figure}[t]
	\begin{center}
		\image{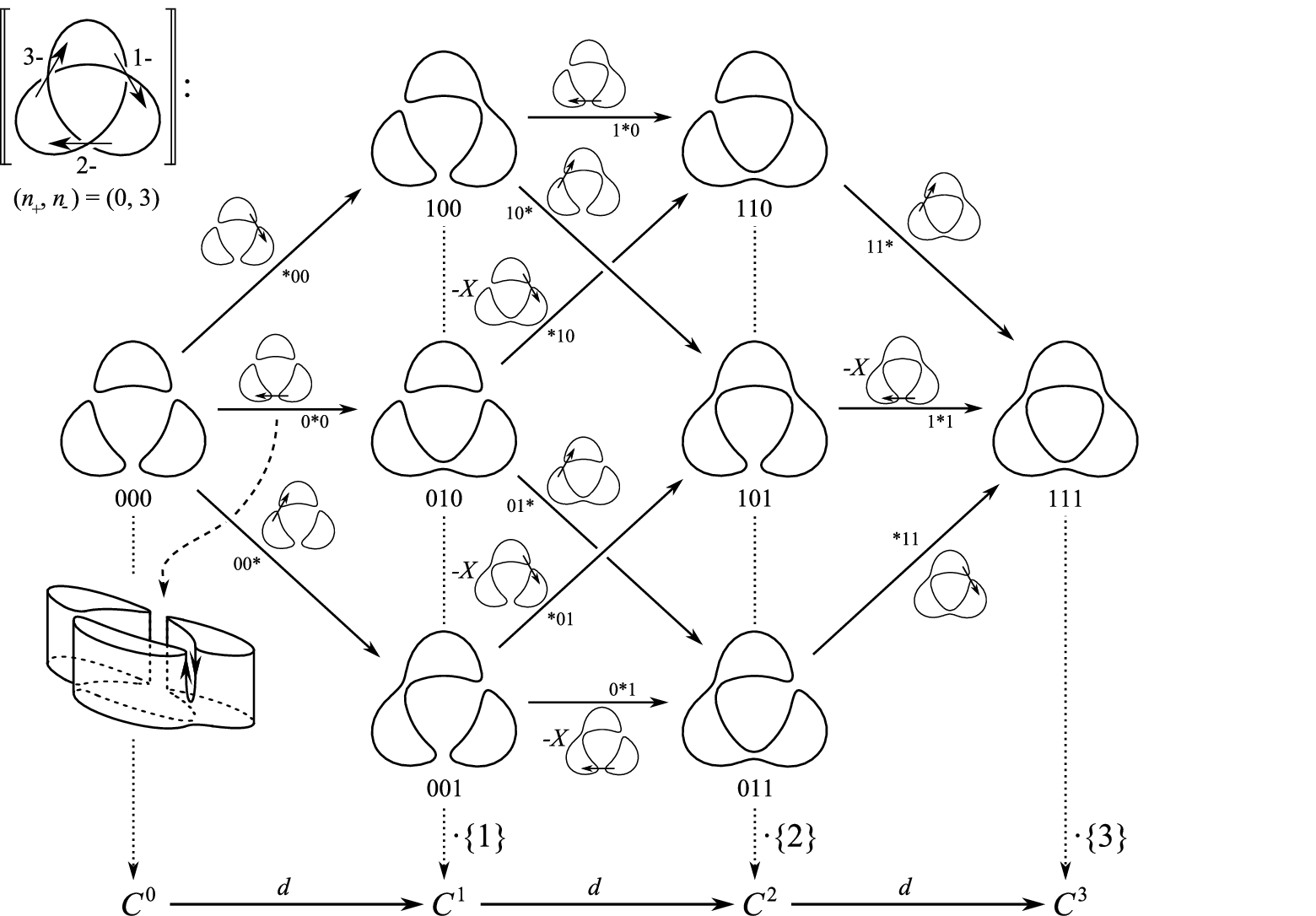}{0pt}
	\end{center}
	\caption{The generalised Khovanov complex for the~trefoil.}\label{fig:khov-cube}
\end{figure}

{\bfseries A knot.}
In the~left top corner we can see a~diagram $D$ of the~trefoil with enumerated crossings.
Minus signs stand for negative crossings. The~caption $(n_+,n_-)=(0,3)$ means the diagram
possesses three negative crossings and no positive ones. Moreover, each crossing
is equipped with an~arrow oriented in a way such that it connects the two arcs in the~type~$0$ resolution.
Notice there are two choices of the~arrow for each crossing.

{\bfseries Vertices.}
The main part of the picture consists of smooth diagrams placed in vertices of a~three-dimensional
cube $I^3$. The~diagram $D_\xi$ in a~vertex $\xi = (\xi_1,\xi_1,\xi_3)$ is obtained from $D$
by replacing $i$-th crossing with its resolution of type $\xi_i$.

{\bfseries Edges.}
Every edge is directed to the~diagram with more type~$1$ resolutions. Globally, the~arrows
give all possible paths from the~left-most diagram (all resolutions of type $0$)
to the~right-most diagram (all resolutions of type $1$) such that in each step one resolution is changed.

Pick any edge $\zeta\colon\xi\to\xi'$ and let $U$ be a~small neighbourhood of the~crossing,
which resolution is changed by this edge. This edge is labeled with a~cobordism in $\mathbb{R}^2\times I$,
being a~cylinder $(D\backslash U)\times I$ with a~saddle
\scimage{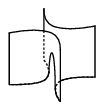}{30pt}{11pt} inserted over $U$. Orientation of the~critical point is given by the~arrow
in the~knot diagram in the~left-top corner. An example is given in the~left-bottom corner for the~edge $(0,*,0)$.

{\bfseries Anticommutativity.}
Fix a~commutative ring $R$ and units $X,Y,Z\in U(R)$ such that $X^2=Y^2=1$.
Consider the~change of chronology relations in the $R$-preadditive closure of $\cat{ChCob}^3$.
Above we have a~description of a~projective 3-cube $\KhCube_0(D)$ in~$\cat{ChCob}^3(\emptyset)/_{\!XYZ}$,
since each face corresponds to some change of~chronology.
Coefficients next to cobordisms form a~negative edge assignment $\varphi\in C^1(I^n,U(R))$.
Such a~modified cube will be denoted by~$\KhCube(D,\varphi)$ and called \term{the Khovanov cube}.
Since the~isomorphism class of the~cube is independent of the~edge assignment, we can write
also $\KhCube(D)$.

{\bfseries The complex.}
Due to the~previous chapter, we have a~complex $\Kom(\KhCube(D))$ in the~category of matrices
$\Mat(\cat{ChCob}^3(\emptyset))$ given by summing the~complex over diagonals $|\xi| = r$.
It is visualised by vertical dotted arrows.

{\bfseries Gradation.}
The~differential has degree $-1$ with respect to the~internal gradation of cobordisms. Therefore,
the~last step is to fix the~grading of $\Kom(\KhCube(D))$ by shifting the $r$-th term
by $r$. It is shown on the~picture by figures in brackets.
The~complex defined above will be called \term{the formal Khovanov bracket}\footnote{\ 
This definition differs from the one given by D.~Bar-Natan in~\cite{BarNatan-tangl},
since it includes partially gradation but lacks of the~horizontal shift. This is motivated
by the~interplay between the Kauffman bracket and the Jones polynomial and is more similar
to the~construction described in the~earlier paper~\cite{BarNatan-Khov}.}
and denoted by $\KhBra{D}_\varphi$.

\medskip
Let us make a~remark before describing the~general situation. A~planar diagram $D$
with $n$ inputs together with cobordisms $\cob {S_i}{\Sigma_i}{\Sigma_i'}$ forms a~projective $n$-cube
$D^S$ given as follows:
\begin{align}
\label{eq:Dcube-vert} D^S(\xi) &:= D(S_1^{\xi_1},\dots,S_n^{\xi_n}), &\textrm{where }& S_i^1 = \Sigma_i,\ S_i^0 = \Sigma_i'\\
\label{eq:Dcube-edge} D^S(\zeta) &:= D(S_1^{(\zeta_1)},\dots,S_n^{(\zeta_n)}),
&\textrm{where }& S_i^{(1)} = \Sigma_i\times I,\ S_i^{(0)} = \Sigma_i'\times I,\ S_i^{(*)} = S_i
\end{align}
Call it \term{a $D$-cube}. There is a~canonical associated cochain given by~chronology change relations.
Theorem~\hyperref[thm:D-chch-unique-coef]{\ref*{chpt:chcob}.\ref*{thm:D-chch-unique-coef}}
asserts the~cochain is a~cocycle, what is proven below.

\begin{proposition}\label{prop:Dcube-CC}
Any $D$-cube satisfies the~cocyle condition.
\end{proposition}
\begin{proof}
Let $F$ be a~$D$-cube obtained from a~planar diagram $D$.
We will show that the~canonical associated cochain $\psi$ given by chronology change relations
is a~cocycle. To do this pick any $3$-cube in $F$:

\medskip
\begin{center}
\puthimage{chcob-cochain-cube}{80pt}
\end{center}
\noindent and denote by~$M_{abc}$ the~cobordism given by the~path consisting of edges parallel to
$a$, $b$ and $c$ (in this order). Then
$$
M_{xyz} = \psi(S_t) M_{yxz} = \psi(S_f)\psi(S_t) M_{yzx} = \dots = d\psi(C) M_{xyz}
$$
where~$S_t$ and~$S_f$ stand respectively for the top and front faces of $C$.
Since a~coefficient of a~change of a~chronology does not depend on a~presentation of a~permutation
as a~composition of transpositions, $d\psi(C) = 1$, what ends the~proof.
\end{proof}

We can now go back to the~construction of the~Khovanov complex.
Let $T\in\mathcal{T}^0(B)$ be a~diagram with~$n$ crossings of a~tangle equipped with arrows over crossings.
As before, we can form an~$n$-cube $\KhCube_0(T)$ of resolutions of $T$ in $\cat{ChCob}^3(B)$.
Each face is a~change of a~chronology of type I, so the~cube is projective.
Moreover, $\KhCube_0(T)$ is a~$D$-cube and due to the~proposition~\ref{prop:Dcube-CC}
it has a~negative edge assignment $\varphi$.
Indeed, remove from $T$ small neighbourhoods of its crossings to get a~planar diagram $D$
(fig.~\ref{fig:T-diag-D-diag}). Each crossing of $T$ describes a~saddle with one critical point.
This saddles together with $D$ form a~$D$-cube being equal to $\KhCube_0(T)$.

\begin{figure}[htb]
\begin{center}
	\image{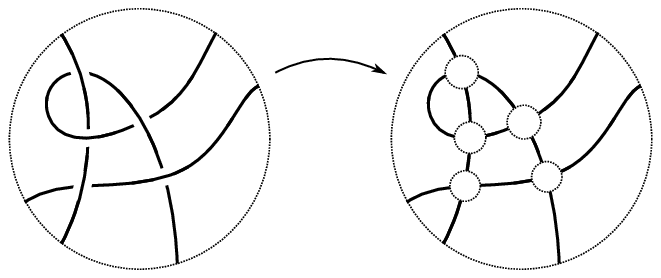}{0pt}
\end{center}
\caption{Having a~tangle diagram one creates a~planar diagram by removing small neighbourhoods of crossings.}\label{fig:T-diag-D-diag}
\end{figure}

\begin{definition}\label{def:kh-compl}
Let $T$ be a~tangle diagram with $n$ crossings equipped with arrows.
\term{The formal Khovanov bracket} of $T$ is the~complex $\KhBra{T}_\varphi$ given by
\begin{equation}
\KhBra{T}_\varphi^r := \Kom(\KhCube(T,\varphi))^{r}\{r\}
\end{equation}
where~$\KhCube(T,\varphi)$ is the~$D$-cube induced by $T$ with an~edge assignment $\varphi$.
If~$n_+$ and~$n_-$ are respectively the numbers of positive and~negative crossings in $T$,
define \term{the generalized Khovanov complex} $\KhCom(T,\varphi)$ as a~shift of the~formal bracket
\begin{equation}
\KhCom(T,\varphi) = \KhBra{T}_\varphi[-n_-]\{n_+-2n_-\}.
\end{equation}
\end{definition}

Directly from the~definition the~formal bracket $\KhBra{T}_\varphi$ is defined for unoriented tangles,
whereas the~complex $\KhCom(T,\varphi)$ for oriented.

To construct the~Khovanov complex we have enumerated crossings in the~diagram and assigned
an~arrow for each crossing. The~order of crossings is obviously irrelevant (a~permutation
of~crossings induces an~isomorphism of complexes). The~independence of the~choice of arrows
comes from existence of appropriate edge assignments as shown below.

\begin{lemma}\label{lem:kh-cube-arr}
Let $D_1, D_2$ be diagrams of a~tangle $T$ with~$n$ crossings,
which differ only in orientations of arrows over crossings.
Then for any edge assignment~$\varphi_1$ for $D_1$ there exists an~edge assignment
$\varphi_2$ of $D_2$ such that
$\KhCube(D_1,\varphi_1) = \KhCube(D_2,\varphi_2)$.
\end{lemma}
\begin{proof}
Without loss of generality we may assume $D_1$ and $D_2$ differ in an~orientation of exactly one arrow,
say over the~$i$-th crossing. Reversing the~arrow reverses orientations of critical points of cobodisms
assigned to edges $\zeta$ with $\zeta_i=*$. Let $\psi_i$ be the~canonical cocycle of the~cube $\KhCube_0(D_i)$.
For a~negative edge assignment $\varphi_1$ for $\KhCube_0(D_1)$ define
\begin{equation}
\varphi_2(\zeta) = \begin{cases}\phantom{\lambda_\zeta}\varphi_1(\zeta), & \zeta_i \neq *\\ \lambda_\zeta\varphi_1(\zeta),& \zeta_i=*\end{cases}
\end{equation}
where~$\lambda_\zeta$ is the~coefficient of reversing the~orientation of the~ciritical point
of the~cobordism assigned to $\zeta$. For a~face $S$ with the boundary
$\zeta_1\zeta_2\zeta_3^{-1}\zeta_4^{-1}$ we have
\begin{equation}
\psi_2(S) = \lambda_{\zeta_1}\lambda_{\zeta_2}\lambda^{-1}_{\zeta_3}\lambda^{-1}_{\zeta_4}\psi_1(S).
\end{equation}
Therefore $d\varphi_2(S) = \lambda_{\zeta_1}\lambda_{\zeta_2}\lambda^{-1}_{\zeta_3}\lambda^{-1}_{\zeta_4}d\varphi_1(S) = -\psi_2(S)$,
so $\varphi_2$ is the~desired edge assignment.
\end{proof}

\begin{remark}\label{rmk:kh-cube-well-def}
The~formal Khovanov bracket $\KhBra{D}$ and the~generalized Khovanov complex $\KhCom(D)$
are well defined up to an~isomorphism for any tangle diagram.
\end{remark}

The~formal Khovanov bracket has properties similar to the~ones of the~Kauffman bracket.
Directly from the~construction we have the~following.

\begin{proposition}\label{prop:Kh-bra-prop}
The~formal Khovanov bracket satisfies the following equations:
\begin{enumerate}[label=(Kh\arabic*)]
\item\label{it:Kh-norm}$\KhBra{U} = (0\to\underline{U}\to 0)$
\item\label{it:Kh-circle}$\KhBra{D\sqcup U} = \KhBra{D}\sqcup U$
\item\label{it:Kh-skein}$\KhBra{\fnt{cr-nor-p}} = \cone(\KhBra{\fnt{cr-nor-h}}\to^d \KhBra{\fnt{cr-nor-v}})$
\end{enumerate}
where in~\ref{it:Kh-circle} $\KhBra{D}\sqcup U$ stands for the~complex $\KhBra{D}$ with a~trivial tangle $U$
added to each item and the~identity $C_U$ added to the~differential.
\end{proposition}

\begin{remark}
The~morphism $d\colon\KhBra{\fnt{cr-nor-h}}\to\KhBra{\fnt{cr-nor-v}}$ in~\ref{it:Kh-skein}
is induced from the~cube morphism given by the~decomposition of $\KhCube(\fnt{cr-nor-p})$.
It is a~bunch of cobordisms
$$
M_\xi\colon D_{\sfnt{cr-nor-h}}(\xi)\to D_{\sfnt{cr-nor-v}}(\xi)
$$
with one critical point each, given by the~change of a~resolution of the~distinguished crossing.
\end{remark}

Similarly to the~case of the~Kauffman bracket, properties~\ref{it:Kh-norm}--\ref{it:Kh-skein}
determines the~formal Khovanov bracket uniquely.
We will come back to the~interplay between the~two brackets in the~section~\ref{sec:khov-jones}.

\section{The proof of invariance}\label{sec:khov-inv}
The~Khovanov complex $\KhCom(T)$ is not a~tangle invariant. For example it depends
on the~number of crossings in a~chosen diagram. To have an~invariant construction,
we will introduce three relations $S,T$ and~$4Tu$. The~theorem~\ref{thm:khov-inv}
says the~complex in the~quotient category is a~tangle invariant up to chain homotopies.

\newpage
\begin{floatingfigure}[r]{90pt}
\begin{flushright}
\mbox{\image{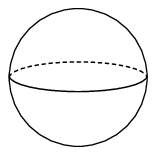}{0pt} \raisebox{17pt}{$= 0$}}
\end{flushright}
\end{floatingfigure}
The $S$ stands for a~sphere and means that a~sphere with two critical points is the~zero object.
Using the~monoidal structure of cobordisms we see that any cobordism $M$ with a~component being a~sphere
with exactly two critical points must be zero (the~two critical points must be consecutive in $M$).

\begin{floatingfigure}[r]{120pt}
\begin{flushright}
\mbox{\image{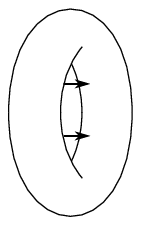}{0pt} \raisebox{27pt}{$= Z(X+Y)$}}
\end{flushright}
\end{floatingfigure}
The $T$ stands for a~torus and means that a~torus is equal to $Z(X+Y)$.
Again, by the~monoidal structure, having any cobordism $M$ with a~component being a~torus
with exactly four critical points which are consecutive in $M$ and agreeing orientations of the~split and the~merge,
we can remove this component and multiply the rest by $Z(X+Y)$. However, to be~consistent with
change of chronology relations, one has to use a~different coefficient if the~points are not consecutive
or the~toroidal component has more than four critical points (apply an appropriate change of a~chronology first).

\begin{floatingfigure}[r]{260pt}
\begin{flushright}
\mbox{\raisebox{27pt}{$Z$} \image{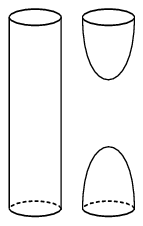}{0pt} \raisebox{27pt}{$+Z$} \image{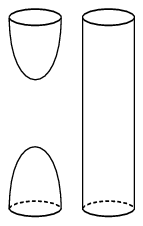}{0pt}%
\raisebox{27pt}{$=X$} \image{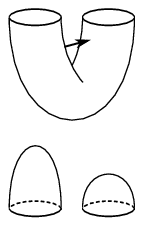}{0pt} \raisebox{27pt}{$+Y$} \image{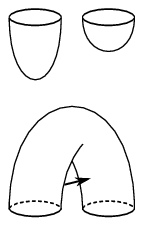}{0pt}}
\end{flushright}
\end{floatingfigure}
Finally, $4Tu$ stands for four tubes. It is best described locally.
Let $M$ be an~identity cylinder over $\S1\sqcup\S1$ with two components $M'$ and $M''$.
If we cut one of them and close the holes with a~birth and a~death, we get two cobordisms $M_1$ and $M_2$.
Construct $M_3$ by cutting both components and connecting together the~upper remaining parts.
Append deaths to the~lower parts so that the~higher death lies below the~negative part of the~split.
$M_4$ is constructed dually. Then $4Tu$ says that $Z\cdot M_1 +Z\cdot M_2 = X\cdot M_3 +Y\cdot M_4$.

Notice that both $T$ and $4Tu$ preserve the~$2$-index of a~cobordism.
Indeed, the~$2$-index of a~torus is zero and of the~$2$-index of each cobordism in $4Tu$ it is $(-1,-1)$.
Therefore, the~coefficient of the~change of a~chronology is well defined in the~quotient category
as well as the~degree of a~cobordism.

Let $\catL{ChCob}^3$ be the~category of embedded cobordisms modulo $S, T, 4Tu$ and chronology change relations.
By the above it is a~well-defined graded $R$-preadditive category and we can construct
the~category of double graded complexes in a~standard way. Denote by
$\cat{Kob}$, $\cat{Kob}(\emptyset)$ and~$\cat{Kob}(B)$ the~subcategories of cube complexes
in $\Kom(\Mat(\catL{ChCob}^3))$, $\Kom(\Mat(\catL{ChCob}^3(\emptyset)))$ and~$\Kom(\Mat(\catL{ChCob}^3(B)))$
respectively. The~corresponding homotopy categories will be denoted by
$\catH{Kob}$, $\catH{Kob}(\emptyset)$ and~$\catH{Kob}(B)$.

Now we are ready to construct homotopy equivalences between complexes of tangles
appearing in the~definitions of the~Reidemeister moves.

\begin{lemma}\label{lem:kh-inv-R1}
The~complex $\KhCom(\khfnt{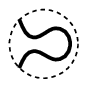})$ is a~strong deformation retract of $\KhCom(\khfnt{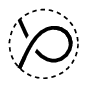})$.
\end{lemma}
\begin{proof}
We want to show that the~first of the~following complexes is a~retract of the~other:
\begin{align*}
\KhCom(\khfnt{R1-h}):&\qquad 0\to\underline{\khfnt{R1-h}}\to 0\\
\KhCom(\khfnt{R1-x}):&\qquad 0\to\underline{\khfnt{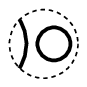}}\{1\}\to^d \khfnt{R1-h}\{2\}\to 0
\end{align*}
where~$d=\scimage{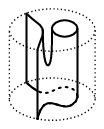}{20pt}{6pt}$. In both complexes we underlined the~item in degree zero.
\begin{figure}[ht]
	$$
	\xy
	\morphism(   0,0)|a|/@{>}@<2pt>/<1200,0>[\image{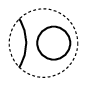}{8pt}`\image{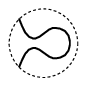}{8pt};d=\lambda\image{Kh-R1-d}{15pt}]
	\morphism(1200,0)|b|/@{>}@<2pt>/<-1200,0>[\phantom{\image{Kh-R1-T0}{8pt}}`\phantom{\image{Kh-R1-T1}{8pt}};h=-\lambda^{-1}\image{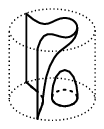}{15pt}]

	\morphism(0,   0)|l|/@{>}@<2pt>/<0, 1000>[\phantom{\image{Kh-R1-T0}{8pt}}`\phantom{\image{Kh-R1-T1}{8pt}};G^0=XZ^{-1}\image{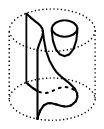}{15pt}]
	\morphism(0,1000)|r|/@{>}@<2pt>^(0.4){F^0 = XY\image{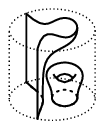}{15pt}-YZ\image{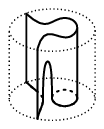}{15pt}}/<0,-1000>[\image{Kh-R1-T1}{8pt}`\phantom{\image{Kh-R1-T0}{8pt}};]

	\morphism(0,1000)|a|/{<->}/<1200,0>[\phantom{\image{Kh-R1-T1}{8pt}}`0;0]
	\morphism(1200,0)|r|/{<->}/<0,1000>[\phantom{\image{Kh-R1-T1}{8pt}}`\phantom{0};0]
	\endxy
	$$
	\caption{Invariance under the~$R_1$ move}\label{fig:R1-diag-cob}
\end{figure}
Construct maps $F\colon\KhCom(\khfnt{R1-h})\to\KhCom(\khfnt{R1-x})$
and~$G\colon\KhCom(\khfnt{R1-x})\to\KhCom(\khfnt{R1-h})$ as in the~figure~\ref{fig:R1-diag-cob}.
For simplicity the~arrows describing orientations of critical points are omitted --
assume all are directed to the~right. The~differential $d$ has to be a~merge for any closing operator,
whereas $F^0$ has to be a~split. Moreover, due to grading shifts, all morphisms have degree 0.
Directly from chronology change relations we have
$$
dF^0=XY\image{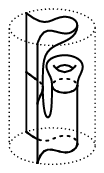}{20pt}-YZ\image{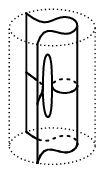}{20pt} =
Y\image{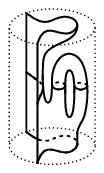}{20pt}-YZ\image{Kh-R1-dF2}{20pt} =
YZ\image{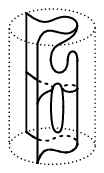}{20pt}-YZ\image{Kh-R1-dF2}{20pt} = 0
$$
so $F$ is a~chain map. Showing $G$ is a~chain map is trivial.
We will prove now, that $F$ and~$G$ are mutually inverse homotopy equivalences.
The~$T$ relation implies $GF = I$:
$$
G^0F^0=YZ^{-1}\image{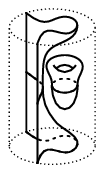}{20pt}-XY\image{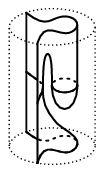}{20pt}=
Y(X+Y)\image{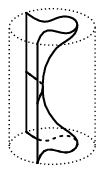}{20pt}-XY\image{Kh-R1-GF3}{20pt}=\image{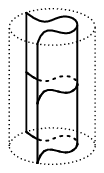}{20pt} = I
$$
Due to $4Tu$ we have $F^0G^0 - I = hd$:
\begin{align*}
0 &=\phantom{Y}Z\image{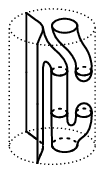}{20pt}\phantom{X} + Z\image{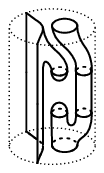}{20pt}\phantom{Z} - X\image{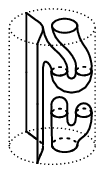}{20pt} - Y\image{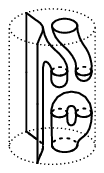}{20pt}\\
  &=YZ\image{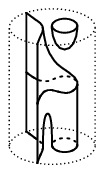}{20pt} + XZ\image{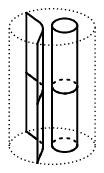}{20pt} - XZ\image{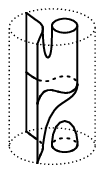}{20pt}\phantom{Y} - \image{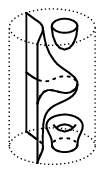}{20pt}
  = -XZ(F^0G^0 - I - hd)
\end{align*}
what together with $dh = -I$ (remove the~birth) gives $FG-I=hd+dh$.
Obviously $hF = 0$, so $\KhCom(\khfnt{R1-h})$ is a~strong deformation retract of $\KhCom(\khfnt{R1-x})$.
\end{proof}

\begin{lemma}\label{lem:kh-inv-R2}
The~complex $\KhCom(\khfnt{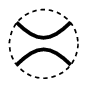})$ is a~strong deformation retract of $\KhCom(\khfnt{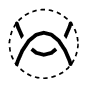})$.
\end{lemma}
\begin{proof}
This lemma is proven in the~same way as the~previous one.
Consider the~diagram in the~figure~\ref{fig:khov-inv-R2},
\begin{figure}
	\begin{center}
$$
\xy
\place(0,1500)[\KhCom(\khfnt{R2-T2}):]
\place(0, 300)[\KhCom(\khfnt{R2-xx}):]

\morphism(0,900)|r|/@{>}^(.6){G}/<0, 300>[`\mathrm{N};]
\morphism(0,900)|l|/@{>}^(.6){F}/<0,-300>[`\mathrm{S};]
\morphism(0,900)|r|/@{>}^(.6){d}/< 300,0>[`\mathrm{E};]
\morphism(0,900)|r|/@{>}^(.6){h}/<-300,0>[`\mathrm{W};]

\morphism(1000,1500)|a|/{->}/<1000,0>[0`\scimage{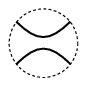}{30pt}{12pt};0]
\morphism(2000,1500)|a|/{->}/<1000,0>[\phantom{\scimage{Kh-R2-TT}{30pt}{12pt}}`0;0]

\morphism(1000,300)|a|/@{>}@<2pt>^(0.6){\image{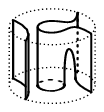}{15pt}}/< 800,-300>[\scimage{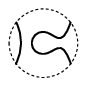}{30pt}{12pt}`\scimage{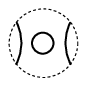}{30pt}{12pt};]
\morphism(1000,300)|a|/@{>}@<2pt>^(0.4){\image{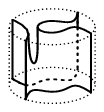}{15pt}}|(0.64)\hole|(0.66)\hole/<1200, 300>[\phantom{\scimage{Kh-R2-T00}{30pt}{12pt}}`\scimage{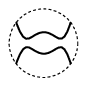}{30pt}{12pt};]
\morphism(1800,  0)|a|/@{>}@<2pt>^(0.5){\image{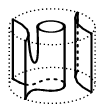}{15pt}}/<1200, 300>[\phantom{\scimage{Kh-R2-T01}{30pt}{12pt}}`\scimage{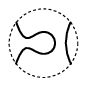}{30pt}{12pt};]
\morphism(2200,600)|a|/@{>}@<2pt>^(0.5){\lambda\image{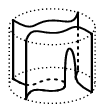}{15pt}}/< 800,-300>[\phantom{\scimage{Kh-R2-T10}{30pt}{12pt}}`\phantom{\scimage{Kh-R2-T11}{30pt}{12pt}};]

\morphism(1800,  0)|b|/@{>}@<2pt>/< -800, 300>[\phantom{\scimage{Kh-R2-T01}{30pt}{12pt}}`\phantom{\scimage{Kh-R2-T00}{30pt}{12pt}};-\image{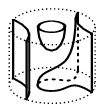}{15pt}]
\morphism(3000,300)|b|/@{>}@<2pt>/<-1200,-300>[\phantom{\scimage{Kh-R2-T11}{30pt}{12pt}}`\phantom{\scimage{Kh-R2-T01}{30pt}{12pt}};-Y\image{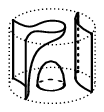}{15pt}]

\morphism(1000,300)|l|/{<->}/<0,1200>[\phantom{\scimage{Kh-R2-T00}{30pt}{12pt}}`\phantom{0};0]
\morphism(3000,300)|r|/{<->}/<0,1200>[\phantom{\scimage{Kh-R2-T11}{30pt}{12pt}}`\phantom{0};0]

\morphism(2200,600)|r|/{@{<->}@/_0.8em/}/<-200,900>[\phantom{\scimage{Kh-R2-T10}{30pt}{12pt}}`\phantom{\scimage{Kh-R2-TT}{30pt}{12pt}};I]
\morphism(1800,0)|l|/{@{>}@/^0.8em/_(0.6){-\image{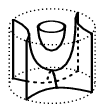}{15pt}}}/<200,1500>[\phantom{\scimage{Kh-R2-T01}{30pt}{12pt}}`\phantom{\scimage{Kh-R2-TT}{30pt}{12pt}};]
\morphism(2000,1500)|r|/{@{>}@<-4pt>@/_0.8em/_(0.4){-\lambda Y\image{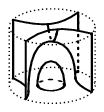}{15pt}}}/<-200,-1500>[\phantom{\scimage{Kh-R2-TT}{30pt}{12pt}}`\phantom{\scimage{Kh-R2-T01}{30pt}{12pt}};]

\place( 950, 130)[\textrm{\small 00}]
\place(1750,-170)[\textrm{\small 01}]
\place(2150, 430)[\textrm{\small 10}]
\place(2950, 130)[\textrm{\small 11}]

\endxy
$$
	\end{center}
	\caption{Invariance under the~$R_2$ move}\label{fig:khov-inv-R2}
\end{figure}
where, as before, the~omitted arrows for critical points are directed to the~right.
Notice that $F^0 = h^1d_{1\bullet}$ and~$G^0 = d_{\bullet 0}h^0$. Moreover, due to grading shifts,
all morphisms have degree 0.

Equalities $dF = 0$ and~$Gd = 0$ are either trivial or can be derived from chronology change
relations, so both $F$ and $G$ are chain maps.
The~$S$ relation implies $GF=I$ and $hF=0$. To end the~proof it remains to show
that $h$ is a~chain homotopy between $FG$ and an~identity. It is trivial in gradings $-1$ and $1$.
In the~zero grading we have to check the~matrix condition:
$$
\left(\begin{matrix}G^0F^0 & F^0\\ G^0 & I\end{matrix}\right) - \left(\begin{matrix}I & 0\\ 0 & I\end{matrix}\right) = 
\left(\begin{matrix}h^1d_{\bullet 1}+d_{0\bullet}h^0 & h^1d_{1\bullet}\\ d_{\bullet 0}h^0 & 0\end{matrix}\right)
$$
The~only non-trivial equality $F^0G^0 - I = h^1d_{x1}+d_{0x}h^0$ can be derived from $4Tu$:
\begin{align*}
0 &=\phantom{X}Z\image{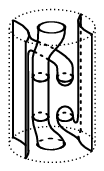}{20pt}\phantom{XY\lambda} + Z\image{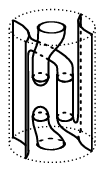}{20pt}\phantom{Z} - X\image{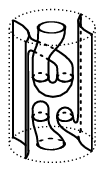}{20pt}\phantom{XZ} - Y\image{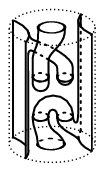}{20pt}\\
  &=XZ\image{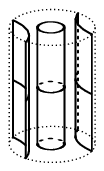}{20pt}\phantom{\lambda} + XYZ\image{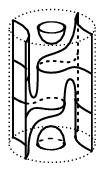}{20pt} - XZ\image{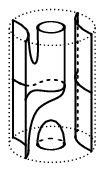}{20pt} - XYZ\image{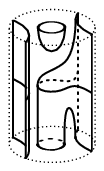}{20pt}\\
  &=XZ\image{Kh-R2-I}{20pt} - \lambda XYZ\image{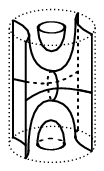}{20pt} - XZ\image{Kh-R2-hd}{20pt} - XYZ\image{Kh-R2-dh}{20pt}\\
  &= XZ(-F^0G^0+I+h^1d_{\bullet 1}+d_{0\bullet}h^0)
\end{align*}
When modifying the~second term, we first used chronology change relations, then anticommutativity
of the~lower square in~\ref{fig:khov-inv-R2} and finally the~expressions of $F$ and $G$ in terms of $h$ and $d$.
\end{proof}

The~case of the~third move is the~simplest one, although it deals with the~largest complex.
This is because it can be derived from the~invariance under the~second move, as it was in the~case
of the~Kauffman bracket.

\begin{lemma}\label{lem:kh-inv-R3}
The~complexes $\KhCom(\khfnt{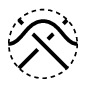})$ and~$\KhCom(\khfnt{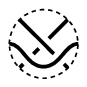})$ are chain homotopic.
\end{lemma}
\begin{proof}
First notice that the~lemma can be shown at the~level of formal Khovanov brackets.
This is because for both tangles $\khfnt{R3-ux}$ and~$\khfnt{R3-dx}$ the~Khovanov complexes
are the~formal brackets with the~same shifts.

Due to \ref{it:Kh-skein} the~complex $\KhBra{\khfnt{R3-ux}}$ is a~cone of the~chain morphism
$\Psi = \KhBra{\khfnt{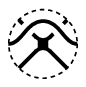}}\colon\KhBra{\khfnt{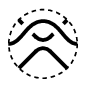}}\to\KhBra{\khfnt{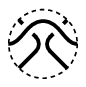}}$
given by the~following four morphisms:
$$
\xy
\morphism(  0,710)<595, 110>[\phantom{\image{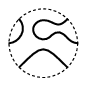}{8pt}}`\phantom{\image{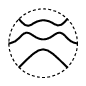}{8pt}};]
\morphism(  0,710)<405,-130>[\phantom{\image{Kh-R3-T000}{8pt}}`\phantom{\image{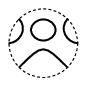}{8pt}};]
\morphism(595,820)<405,-130>[\phantom{\image{Kh-R3-T100}{8pt}}`\phantom{\image{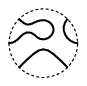}{8pt}};]
\morphism(405,580)<595, 110>[\phantom{\image{Kh-R3-T010}{8pt}}`\phantom{\image{Kh-R3-T110}{8pt}};]

\morphism(   0,710)<  0,-600>[\image{Kh-R3-T000}{8pt}`\image{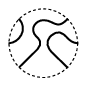}{8pt};]
\morphism( 405,570)<  0,-600>[\image{Kh-R3-T010}{8pt}`\image{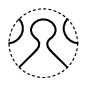}{8pt};]
\morphism( 595,830)/@{>}|(0.36)\hole/<  0,-600>[\image{Kh-R3-T100}{8pt}`\image{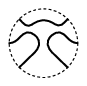}{8pt};]
\morphism(1000,690)<  0,-600>[\image{Kh-R3-T110}{8pt}`\image{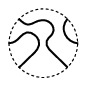}{8pt};]

\morphism(  0,110)/@{>}|(0.66)\hole/<595, 110>[\phantom{\image{Kh-R3-T001}{8pt}}`\phantom{\image{Kh-R3-T101}{8pt}};]
\morphism(  0,110)<405,-130>[\phantom{\image{Kh-R3-T001}{8pt}}`\phantom{\image{Kh-R3-T011}{8pt}};]
\morphism(595,220)<405,-130>[\phantom{\image{Kh-R3-T101}{8pt}}`\phantom{\image{Kh-R3-T111}{8pt}};]
\morphism(405,-20)<595, 110>[\phantom{\image{Kh-R3-T011}{8pt}}`\phantom{\image{Kh-R3-T111}{8pt}};]

\place(200,350)[\Psi\!\downarrow]
\endxy
$$
Now we can use the homotopy equivalence $F$ from the~proof of the~previous lemma.
Since it is an~embedding into a~strong deformation retract,
the~theorem~\hyperref[thm:cone-homot]{\ref*{chpt:alg-hom}.\ref*{thm:cone-homot}}
says $\KhBra{\khfnt{R3-ux}}$ is chain homotopic to the~cone of $\Psi_L = \Psi F$, which is presented
in the~figure \ref{fig:psi-cone}. For the~same argument $\KhBra{\khfnt{R3-dx}}$ is chain homotopic
to $\Psi_R$. Since \khfnt{R3-uv} and~\khfnt{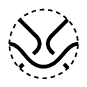} are isotopic,
$\Psi_L$ and~$\Psi_R$ are isomorphic, what gives the~invariance of the~Khovanov complex
under the~third Reidemeister move.
\begin{figure}[hb]
$$
\xy

\place(-200,400)[\Psi_L:]

\morphism(500,820)|l|/{@{>}@/_4pt/}/<-95,-820>[\image{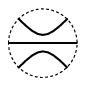}{8pt}`\phantom{\image{Kh-R3-T011}{8pt}};\scimage{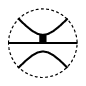}{15pt}{6pt}]
\morphism(500,820)|r|/{@{>}@/^4pt/}/< 95,-600>[\phantom{\image{Kh-R3-Vb}{8pt}}`\phantom{\image{Kh-R3-T101}{8pt}};\scimage{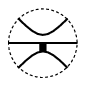}{15pt}{6pt}]

\morphism(  0,120)/@{>}|(0.67)\hole/<595, 100>[\image{Kh-R3-T001}{8pt}`\image{Kh-R3-T101}{8pt};]
\morphism(  0,120)<405,-120>[\phantom{\image{Kh-R3-T001}{8pt}}`\phantom{\image{Kh-R3-T011}{8pt}};]
\morphism(590,220)<405,-120>[\phantom{\image{Kh-R3-T101}{8pt}}`\phantom{\image{Kh-R3-T111}{8pt}};]
\morphism(410,  0)<595, 100>[\image{Kh-R3-T011}{8pt}`\image{Kh-R3-T111}{8pt};]

\place(1800,400)[\Psi_R:]

\morphism(2500,820)|l|/{@{>}@/_4pt/}/<-95,-820>[\image{Kh-R3-Vb}{8pt}`\phantom{\image{Kh-R3-T011}{8pt}};\scimage{Kh-R3-F01}{15pt}{6pt}]
\morphism(2500,820)|r|/{@{>}@/^4pt/}/< 95,-600>[\phantom{\image{Kh-R3-Vb}{8pt}}`\phantom{\image{Kh-R3-T101}{8pt}};\scimage{Kh-R3-F10}{15pt}{6pt}]

\morphism(2000,120)/@{>}|(0.67)\hole/<595, 100>[\image{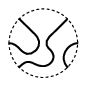}{8pt}`\image{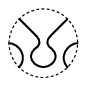}{8pt};]
\morphism(2000,120)<405,-120>[\phantom{\image{Kh-R3-TT001}{8pt}}`\phantom{\image{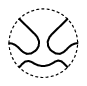}{8pt}};]
\morphism(2590,220)<405,-120>[\phantom{\image{Kh-R3-TT101}{8pt}}`\phantom{\image{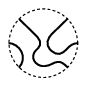}{8pt}};]
\morphism(2410,  0)<595, 100>[\image{Kh-R3-TT011}{8pt}`\image{Kh-R3-TT111}{8pt};]
\endxy
$$
\caption{Cones for tangles describing the~$R_3$ move}\label{fig:psi-cone}
\end{figure}
\end{proof}

The~next step is to show that the~homotopy equivalences built above extend for Reidemester moves
applied to any tangle diagram. Before, let us make some observations on planar algebras.

Planar operators in $\cat{ChCob}^3$ are not functors in general. The~exceptions are operators
with exactly one input, so that they can be naturally extended over the~categories of cubes or complexes.

Let~$D\colon\mathcal{T}^0(B_1)\times\mathcal{T}^0(B_2)\to\mathcal{T}^0(B)$ be a~planar diagram
and $T\in\mathcal{T}^0(B_1)$ be a~tangle diagram. Denote by~$D^T_\xi$ the~planar diagram obtained
from $D$ by inserting into the~first input the~diagram $T$ smoothed with respect to $\xi$.
For any diagram $T'\in\mathcal{T}^0(B_2)$ with~$m$ crossings we can define a~cube
\begin{equation}
\KhCube^{DT}(T') := \Kom^m\KhCube(D(T,T'))
\end{equation}
which has in a~vertex $\xi$ a~complex computed for the~tangle $D^T_\xi(T')$.
Moreover, a~morphism $f\colon\KhCom(T_1)\to\KhCom(T_2)$
lifts to a~morphism of cubes $f^{DT}\colon\KhCube^{DT}(T_1)\to\KhCube^{DT}(T_2)$ given by
\begin{equation}
f^{DT}(\xi) = D^T_\xi(f)
\end{equation}
Since each $D^T_\xi$ is a~functor, the~operation $(\cdot)^{DT}$ is functorial:
\begin{equation}
(fg)^{DT} = f^{DT}g^{DT}
\end{equation}
We will use this observation in the~proof of the~invariance theorem.

\begin{theorem}\label{thm:khov-inv}
The~Khovanov complex $\KhCom(T)$ is an~invariant of a~tangle $T\in\mathcal{T}(B)$ in the~category
$\catH{Kob}(B)$ up to an~isomorphism. In other words, Khovanov complexes computed for two diagrams of
a~given tangle $T$ are homotopic modulo relations $S, T, 4Tu$.
\end{theorem}
\begin{proof}
Let $T$ be a~tangle diagram and $T'$ be obtained from $T$ by applying a~Reidemeister move $R_i$.
Pick a~planar diagram $D$ with two inputs and tangle diagrams $T_1, T_2$ and~$T'_2$ such that
$T = D(T_1,T_2)$ and~$T'=D(T_1,T'_2)$, where $T_2$ and~$T'_2$ describe $R_i$.
The~theorem~\hyperref[thm:com-part-calc]{\ref*{chpt:alg-hom}.\ref*{thm:com-part-calc}}
together with existence of edge assignments
(see the~corollary~\hyperref[cor:part-edge-assign]{\ref*{chpt:alg-hom}.\ref*{cor:part-edge-assign}}) gives
\begin{align*}
\KhCom(T) &= \Kom\KhCube^{DT_1}(T_2) \\
\KhCom(T') &= \Kom\KhCube^{DT_1}(T'_2)
\end{align*}
and it remains to show that the~maps induced by~homotopy equivalences from the~lemmas above
are also homotopy equivalences.

{\bfseries The~move $R_1$.} Both $(F^0)^{DT}$ and~$(G^0)^{DT}$ are $D$-morphisms
and due to the~proposition~\ref{prop:Dcube-CC} there exist edge assignments for them
(notice that $F^0$ is homogeneous).
Furthermore, the~uniqueness of the~edge assignment assures that the~equality $G^0F^0 = I$ is preserved.
By the~definition, $h$ has the~opposite sign to $d$
($\sigma(h) = -\sigma(d)$). Hence $h^{DT}$ is a~morphism of anticommutative cubes
so it is a~cube homotopy between $FG$ and the~identity.

{\bfseries The~move $R_2$.} Same as before, the~homotopy $h$ induces a~cube homotopy $h^{DT}$.
Moreover, directly from the~definitions, $(F^0)^{DT}$ and~$(G^0)^{DT}$ are commutative,
since each of their components is either an~identity or a~composition of anticommutative morphisms:
$$
(F^0)^{DT} = h^{DT}d^{DT},\qquad (G^0)^{DT} = d^{DT}h^{DT}
$$

{\bfseries The~move $R_3$.} Since $(\cdot)^{DT}$ is functorial, $F^{DT}$ from the~previous paragraph
is still an~embedding into a~strong deformation retract and we can repeat the~proof of
the~lemma \ref{lem:kh-inv-R3} for any tangle.
\end{proof}

\section{Properties of the~complex}\label{sec:khov-prop}
In this section we will show basic properties of the~Khovanov complex.
All proofs are taken with minor modifications from~\cite{Khov-Jones}.

Let $T$ be an~oriented tangle with a~diagram $D$ and pick its component $T_0$ with linking number
$l = \lk(T_0, T-T_0)$. Reversing the~orientation of $T$ forms a~new tangle $T'$ with a~diagram $D'$
and $\lk(T_0, T'-T_0) = -l$. Since $\KhBra{D} = \KhBra{D'}$ and
\begin{equation}
n_+(D') = n_+(D) - 2l,\quad n_-(D') = n_-(D) + 2l,
\end{equation}
we have

\begin{proposition}\label{prp:kh-1-comp-rev}
Let $D$ and~$D'$ be the~diagrams given above. Then
\begin{equation}
\KhCom^r(D') = \KhCom^{r+2l}(D)\{2l\}
\end{equation}
\end{proposition}

Reversing the~global orientation (i.e. of all component) preserves signs of crossings.
Hence the~complex does not depend on the~global orientation.

\begin{proposition}\label{prp:kh-rev}
Let $T$ be a~tangle. Denote by $-T$ the~tangle $T$ with reversed orientation of all components. Then
\begin{equation}
\KhCom(T) = \KhCom(-T)
\end{equation}
\end{proposition}

Due to~\ref{it:Kh-skein} the~complex $\KhBra{\fnt{cr-nor-p}}$
is a~cone of $d\colon\KhBra{\fnt{cr-nor-h}}\to\KhBra{\fnt{cr-nor-v}}$.
Therefore the~sequence below is exact
\begin{equation}\label{eq:Kh-bra-exact-seg}
0\to\KhBra{\fnt{cr-nor-v}}[-1]\to\KhBra{\fnt{cr-nor-p}}\to\KhBra{\fnt{cr-nor-h}}\to 0
\end{equation}
and similar to the~Jones polynomial, we have the following

\begin{proposition}\label{prp:kh-exact-seq}
There is an~exact sequence of complexes:
\begin{equation}\label{eq:Kh-exact-seq}
0\to\KhCom(\fnt{cr-or-h})\{2\}\to\KhCom(\fnt{cr-or-n})[1]\{1\}\to\KhCom(\fnt{cr-or-p})[2]\to\KhCom(\fnt{cr-or-h})[2]\to 0
\end{equation}
\end{proposition}
\begin{proof}
First write down the~exact sequences~(\ref{eq:Kh-bra-exact-seg}) for diagrams $\fnt{cr-nor-p}$ and~$\fnt{cr-nor-n}$:
\begin{eqnarray*}
&0\to\KhBra{\fnt{cr-nor-v}}[-1]\to\KhBra{\fnt{cr-nor-p}}\to\KhBra{\fnt{cr-nor-h}}\to 0&\\
&0\to\KhBra{\fnt{cr-nor-h}}[-1]\to\KhBra{\fnt{cr-nor-n}}\to\KhBra{\fnt{cr-nor-v}}\to 0&
\end{eqnarray*}
We can combine them together and get an~exact sequence
\begin{equation}
0\to\KhBra{\fnt{cr-nor-h}}\to\KhBra{\fnt{cr-nor-n}}[1]\to\KhBra{\fnt{cr-nor-p}}[2]\to\KhBra{\fnt{cr-nor-h}}[2]\to 0
\end{equation}
If we orient the~diagrams and make the~necessary grading shifts, we will get the~desired sequence.
\end{proof}

Let $T^*$ be a~mirror tangle to $T$.
Recall that there is a~contravariant functor
\begin{equation}
*\colon\cat{ChCob}^3/_{\!XYZ}\to\cat{ChCob}^3/_{\!YXZ}
\end{equation}
It preserves all relations $S$, $T$ and~$4Tu$, so it induces a~functor between the~categories of complexes.
Directly from the~construction of the~Khovanov complex

\begin{proposition}\label{prp:kh-dual}
If $T$ a~mirror tangle to $T^*$ then
\begin{equation}
\KhCom_{XYZ}(T^*) = \KhCom_{YXZ}(T)^*
\end{equation}
\end{proposition}
\begin{proof}
First consider the~category $\cat{ChCob}^3$ with chronology change relations. Then
$$
\KhCube_0(T^*) = \KhCube_0(T)^*
$$
Moreover, any edge assignment $\varphi$ for $\KhCube_0(T)$ with chronology change relations
given by $X, Y, Z$ induces a~dual edge assignment $\varphi'$ for $\KhCube_0(T)^*$
with the~dual chronology change relations given by coefficients $Y, X, Z$, satisfying
$$
\varphi'(\zeta^*) = \varphi(\zeta),
$$
where~$\zeta^*$ is the~edge dual to~$\zeta$ (exchange zeros with ones).
\end{proof}

\section{Examples of homology groups}\label{sec:khov-homol}
The~complex $\KhCom(T)$ is an~invariant of a~tangle $T$, but it is hard (if even possible)
to make computations in $\catL{ChCob}^3$. For example, how to check that two complexes
are chain homotopic or not? Therefore, it is more convenient to use a~functor 
$\F\colon\catL{ChCob}^3\to\cat{A}$ into an~Abelian category such as modules or vector spaces,
where we can compute homology groups. Such a~functor can be additively extended
over $\Mat(\catL{ChCob}^3)$ and then to a~functor $\F\colon\cat{Kob}\to\Kom(\cat{A})$.
In this way we obtain a~complex $\F\KhCom(T)$ being a~tangle invariant up to chain homotopies.
Clearly, the~isomorphism classes of homology groups $H^{\bullet}(\F\KhCom(T))$ are tangle invariants.

If $\cat{A}$ is graded and $\F$ preserves degrees of morphisms,
$H^{\bullet}(\F\KhCom(T))$ is double graded with homological gradation the the~one induced by~$\F$.
If $\F$ is defined only on $\cat{ChCob}^3(\emptyset)$ we obtain a~priori only invariants of links.
However, it can be extended for tangles by the~construction described in
the~remark~\hyperref[rm:chcob3-all-2chcob]{\ref*{chpt:chcob}.\ref*{rm:chcob3-all-2chcob}}.
Notice, that the~result does not depend
on whether we first extend $\F$ over the~whole $\cat{ChCob}^3$ and then to the~category of complexes
or in the~different order: first to complexes $\cat{Kob}(\emptyset)$ and then over $\cat{Kob}$.

\begin{definition}
Pick $(\cat{A},\otimes,e,L,R,A,S)$ be a~symmetric monoidal subcategory of $R$-modules.
\term{A~Frobenius algebra} in~$\cat{A}$ is an~$R$-module $A\in\cat{A}$ together with operations
\begin{align*}
\mu\colon& A\otimes A\to A	&	\Delta\colon& A\to A\otimes A\\
\eta\colon& R\to A			&	\varepsilon\colon& A\to R
\end{align*}
called multiplication, comultiplication, unit and counit, equipping $A$ with the~structure
of (co)associative (co)commutative and (co)unital algebra and coalgebra
\begin{align}
\label{eq:frob-alg-mu}\mu\circ(\mu\otimes\id) &= \mu\circ(\id\otimes\mu) & \mu\circ S &= \mu & \mu\circ(\eta\otimes\id)&=\mu\\
\label{eq:frob-alg-delta}(\Delta\otimes\id)\circ\Delta &= (\id\otimes\Delta)\circ\Delta & S\circ\Delta &= \Delta & (\varepsilon\otimes\id)\circ\Delta &=\Delta
\end{align}
satisfying the~Frobenius equation:
\begin{equation}\label{eq:frob-alg-fc}
(\Delta\otimes\id)\circ(\id\otimes\mu) = \mu\otimes\Delta = (\id\otimes\Delta)\circ(\mu\otimes\id)
\end{equation}
\end{definition}

\begin{remark}\label{rmk:alg-frob-funct}
A~Frobenius algebra $(A,\mu,\Delta,\eta,\varepsilon)$ gives a~monoidal functor
$\F_{\!A}\colon\cat{2Cob}\to\cat{A}$ as follows:
\begin{align*}
\mathcal{F_A}(n\S1) &= A^{\otimes n} \\
\mathcal{F_A}\left(\scimage{cob-merge}{5ex}{1.9ex}\right) &= \mu &
\mathcal{F_A}\left(\scimage{cob-birth}{5ex}{1.9ex}\right) &= \eta &
\mathcal{F_A}\left(\scimage{cob-perm}{5ex}{1.9ex}\right)  &= S \\
\mathcal{F_A}\left(\scimage{cob-split}{5ex}{1.9ex}\right) &= \Delta &
\mathcal{F_A}\left(\scimage{cob-death}{5ex}{1.9ex}\right) &= \varepsilon
\end{align*}
Also the~opposite holds: any monoidal functor $F\colon\cat{2Cob}\to\cat{A}$ comes from a~Frobenius algebra
$(A,\mu,\Delta,\eta,\varepsilon)$, where the~module $A$ is given by the~value of $\F$ on a~circle
and the~operations are the~values of $\F$ on appropriate generators of $\cat{2Cob}$.
\end{remark}

When $R=\mathbb{Z}$ and~$X=Y=Z=1$, the~category $\cat{ChCob}^3$ reduces to
embedded cobordisms $\cat{Cob}^3$. Therefore, a~Frobenius algebra $A$
gives a~functor $\F_{\!A}\colon\cat{ChCob}^3(\emptyset)\to\cat{A}$,
and if it preserves $S, T$ and~$4Tu$ relations, then it may be used
to compute homology groups of the~Khovanov complex.

\begin{example}[M.~Khovanov, 1999]\label{ex:frob-alg-khov}
Let $A=\mathbb{Z}v_+\oplus\mathbb{Z}v_-$ be a~free graded module with two generators
$v_+$ and~$v_-$ in degrees respectively $+1$ and~$-1$. Define the~structure of a~Frobenius algebra as below:
\begin{itemize}
\item multiplication $\mu\colon A\otimes A\to A$:
\begin{align*}
\mu(v_+\otimes v_+) &= v_+ & \mu(v_-\otimes v_+) &= v_-\\
\mu(v_+\otimes v_-) &= v_- & \mu(v_-\otimes v_-) &= 0
\end{align*}

\item unit $\eta\colon\mathbb{Z}\to A$:
$$
\eta(1) = v_+
$$

\item comultiplication $\Delta\colon A\to A\otimes A$:
\begin{align*}
\Delta(v_+) &= v_-\otimes v_+ + v_+\otimes v_- & \Delta(v_-) &= v_-\otimes v_-
\end{align*}

\item counit $\varepsilon\colon A\to\mathbb{Z}$:
\begin{align*}
\varepsilon(v_+) &= 0 & \varepsilon(v_-) &= 1
\end{align*}

\end{itemize}
One can check that the~axioms (\ref{eq:frob-alg-mu})--(\ref{eq:frob-alg-fc}) holds.
Multiplication and comultiplication have degree\footnote{\ 
The~degree of $v_1\otimes\dots\otimes v_n$ is defined as the~sum of degrees: $\deg(v_1)+\dots+\deg(v_n)$.
} $-1$, whereas unit and counit have degree $1$. Hence we have a~degree preserving monoidal functor
$\F_{\!Kh}\colon\mathbb{Z}\cat{ChCob}^3(\emptyset)/_{111}\to\cat{Mod}_\mathbb{Z}$,
described for the~first time in~\cite{Khov-Jones}.
It preserves the~relations $S,T,4Tu$, so we can extend it to $\catL{ChCob}^3$
and compute the~standard Khovanov homology groups.\footnote{\ 
The~most general homology groups given in~\cite{Khov-Jones} are defined
over the~ring of polynomials $\mathbb{Z}[c]$. However, such a~functor does not preserve
the~relation $S$ nor $4Tu$. Our example is the~specialization to $c=0$.}
\end{example}

The~Frobenius algebra is not good from our point of view, because it forgets all information
encoded in chronologies. Therefore we will modify the~axioms so that we can use this additional
structure. It is not surprising, that the~new axioms correspond to chronology change relations
from~$\cat{ChCob}^3$.

Recall that a~chronological product in $\cat{C}$ is given by a~half-functor
$\boxtimes\colon\cat{C}\times\cat{C}\to\cat{C}$, i.e.
a~map sending objects to objects, morphisms to morphisms and functorial in one variable:
\begin{equation}
(f\circ g)\boxtimes\id = (f\boxtimes\id)\circ (g\boxtimes\id),\qquad
\id\boxtimes(f\circ g) = (\id\boxtimes f)\circ (\id\boxtimes g)
\end{equation}
Moreover there are natural isomorpisms $L,R$ and~$A$.

\begin{definition}\label{def:chron-frob-alg}
Let $(\cat{A},\boxtimes,e,L,R,A,S)$ be a~symmetric chronological monoidal subcategory of $R$-modules.
\term{A~chronological Frobenius algebra} is an~$R$-module $A$ together with operations
\begin{align*}
\mu\colon& A\boxtimes A\to A	&	\Delta\colon& A\to A\boxtimes A\\
\eta\colon& R\to A			&	\varepsilon\colon& A\to R
\end{align*}
called multiplication, comultiplication, unit and~counit,
satisfying chronology change relations with respect to invertible elements $X,Y,Z\in R$:
\begin{align}
(\id_A\boxtimes\mu)\circ(\mu\boxtimes\id_{A\boxtimes A}) &= X(\mu\boxtimes\id_A)\circ(\id_{A\boxtimes A}\boxtimes\mu) \\
(\id_A\boxtimes\eta)\circ(\mu\boxtimes\id_R) &= X(\mu\boxtimes\id_A)\circ(\id_{A\boxtimes A}\boxtimes\eta) \\
(\id_A\boxtimes\eta)\circ(\eta\boxtimes\id_R) &= X(\eta\boxtimes\id_A)\circ(\id_R\boxtimes\eta) \\
\notag\\
(\id_{A\boxtimes A}\boxtimes\Delta)\circ(\Delta\boxtimes\id_A) &= Y(\Delta\boxtimes\id_{A\boxtimes A})\circ(\id_A\boxtimes\Delta) \\
(\id_R\boxtimes\Delta)\circ(\varepsilon\boxtimes\id_A) &= Y(\varepsilon\boxtimes\id_{A\boxtimes A})\circ(\id_A\boxtimes\Delta) \\
(\id_R\boxtimes\varepsilon)\circ(\varepsilon\boxtimes\id_A) &= Y(\varepsilon\boxtimes\id_R)\circ(\id_A\boxtimes\varepsilon) \\
\notag\\
(\id_{A\boxtimes A}\boxtimes\mu)\circ(\Delta\boxtimes\id_{A\boxtimes A}) &= Z(\Delta\boxtimes\id_A)\circ(\id_A\boxtimes\mu)\\
(\id_R\boxtimes\mu)\circ(\varepsilon\boxtimes\id_{A\boxtimes A}) &= Z(\varepsilon\boxtimes\id_{A\boxtimes A})\circ(\id_A\boxtimes\mu)\\
(\id_{A\boxtimes A}\boxtimes\eta)\circ(\Delta\boxtimes\id_R) &= Z(\Delta\boxtimes\id_A)\circ(\id_A\boxtimes\eta)\\
(\id_R\boxtimes\eta)\circ(\varepsilon\boxtimes\id_R) &= Z(\varepsilon\boxtimes\id_A)\circ(\id_A\boxtimes\eta)
\end{align}
and equipping $A$ with the~structure of (co)associative (co)commutative and (co)unital algebra and coalgebra
in the~chronological sense:
\begin{align}
\label{eq:ch-frob-alg-mu}\mu\circ(\mu\boxtimes\id_A) &= X\mu\circ(\id_A\boxtimes\mu) & \mu\circ S &= X\mu & \mu\circ(\eta\boxtimes\id_A)&=\mu\\
\label{eq:ch-frob-alg-delta}(\Delta\boxtimes\id_A)\circ\Delta &= Y(\id_A\boxtimes\Delta)\circ\Delta & S\circ\Delta &= Y\Delta & (\varepsilon\boxtimes\id_A)\circ\Delta &=\Delta
\end{align}
satisfying the~chronological Frobenius equation:
\begin{equation}\label{eq:ch-frob-alg-fc}
(\Delta\boxtimes\id_A)\circ(\id_A\boxtimes\mu) = Z\mu\boxtimes\Delta = (\id_A\boxtimes\Delta)\circ(\mu\boxtimes\id_A)
\end{equation}€
\end{definition}

This choice of axioms gives an~analogous correspondence between chronological Frobenius algebras
and symmetric chronological monoidal functors $\F\colon\cat{2ChCob}\to\cat{A}$ to the~one
described in the~remark~\ref{rmk:alg-frob-funct}.
Indeed any such an~algebra $(A,\mu,\Delta,\eta,\varepsilon)$ gives a~functor
$\F_{\!A}\colon\cat{2ChCob}\to\cat{A}$ given below:
\begin{align*}
\mathcal{F_A}(n\S1) &= A^{\boxtimes n} \\
\mathcal{F_A}\left(\scimage{cob-merge-p}{5ex}{1.9ex}\right) &= \mu &
\mathcal{F_A}\left(\scimage{cob-birth}{5ex}{1.9ex}\right) &= \eta &
\mathcal{F_A}\left(\scimage{cob-perm}{5ex}{1.9ex}\right) &= S \\
\mathcal{F_A}\left(\scimage{cob-split-p}{5ex}{1.9ex}\right) &= \Delta &
\mathcal{F_A}\left(\scimage{cob-death}{5ex}{1.9ex}\right) &= \varepsilon
\end{align*}
In the~other direction, the~algebra $(A,\mu,\Delta,\eta,\varepsilon)$
is given by the~values of $\F$ on generators of $\cat{2ChCob}$.

\begin{example}[{P.~Ozsv\'ath}, J.~Rasmussen, {Z.~Szab\'o}, 2007]\label{ex:frob-alg-ORS}
Consider the~category of exterior algebras of free modules over~$\mathbb{Z}$.
Define the~chronological product to be the~exterior product:
\begin{equation}
(\Lambda^*\mathbb{Z}\langle v_1,\dots,v_n \rangle) \boxtimes (\Lambda^*\mathbb{Z}\langle w_1,\dots, w_m\rangle) :=
(\Lambda^*\mathbb{Z}\langle v_1,\dots,v_n,w_1\dots,w_m \rangle)
\end{equation}
with a~permutation $S\colon \Lambda^*\mathbb{Z}\langle v_1,v_2\rangle \to \Lambda^*\mathbb{Z}\langle v_1,v_2\rangle$
defined on generators as follows
\begin{align}
S(v_1) &= v_2 & S(v_2) &= v_1
\end{align}
Let $A = \Lambda^*\mathbb{Z}a_1$ be the~exterior algebra of a~free module with one generator.
Then its $n$-th power $A^{\boxtimes n}$ is the~exterior algebra on a~free module with $n$-generators $a_1,\dots,a_n$.
The~chronological Frobenius algebra on~$A$ is given by the~following operations:
\begin{itemize}
\item multiplication $\mu\colon A\wedge A\to A$ is given by identifying the~two generators and taking the~wedge product:
\begin{align*}
\mu(a_1) &= a_1 & \mu(1) &= 1\\
\mu(a_2) &= a_1 & \mu(a_1\wedge a_2) &= 0
\end{align*}
\item unit $\eta\colon\mathbb{Z}\ni \lambda\to \lambda 1\in A$ is the~standard embedding
\item comultiplication $\Delta\colon A\to A\wedge A$ is given by the wedge product with the~difference of generators:
\begin{align*}
	\Delta(1) &= a_1 - a_2 & \Delta(a_1) &= a_1\wedge a_2
\end{align*}
\item counit $\varepsilon\colon A\to\mathbb{Z}$ is the~dual to the~generator:
\begin{align*}
	\varepsilon(1) &= 0 & \varepsilon(a_1) &= 1
\end{align*}
\end{itemize}
One may check that all the~axioms of a~chronological Frobenius algebra are satisfied for $X=Z=1$ and~$Y=-1$.
Define the~degree in~$\Lambda^*\mathbb{Z}\langle a_1,\dots,a_n\rangle$ by
\begin{equation}
\deg(a_{i_1}\wedge\dots\wedge a_{i_k}) = n-2k
\end{equation}
Then multiplication and comultiplication have degree $-1$, whereas unit and counit have degree $1$.
The functor $\F_{ORS}\colon\mathbb{Z}\cat{ChCob}^3/_{1,-1,1}(\emptyset)\to\cat{Mod}_R$
obtained in this way was described for the~first time in~\cite{Osv-Ras-Sz}.
It preserves both the~grading and the~relations $S,T,4Tu$, so we can use it to compute
odd link homology groups.
\end{example}

Both constructions presented above are the~specific cases of a~more general one presented below.

\begin{example}\label{ex:frob-alg-XYZ}
Let $V=Rv_+\oplus Rv_-$ be a~free $R$-module on two generators $v_+$ and~$v_-$
in degrees  $+1$ and~$-1$. Pick invertible elements $X,Y,Z$ in $R$ such that $X^2=Y^2=1$
and define $S\colon V\otimes V\to V\otimes V$ as follows:
\begin{align*}
S(v_+\otimes v_+) &= \phantom{{}^{-1}}Xv_+\otimes v_+ & S(v_-\otimes v_+) &= Zv_+\otimes v_-\\
S(v_+\otimes v_-) &=             Z^{-1}v_-\otimes v_+ & S(v_-\otimes v_-) &= Yv_-\otimes v_-
\end{align*}
Since $S^n_k = \id^{\otimes(n-1)}\otimes S\otimes\id^{\otimes(k-n-2)}, k=1,\dots,n-1,$
satisfies the~relations of permutation groups, it gives us a~symmetry $S$
in a~monoidal subcategory generated by modules $V^{\otimes n}$.
It defines a~chronological product $\boxtimes$ as below:
\begin{align*}
X\boxtimes Y &:= X\otimes Y \\
f\boxtimes\id_Z &:= f\otimes\id_Z &
\id_Z\boxtimes f &:= S_{YZ}\circ(f\otimes\id_Z)\circ S_{ZX}
\end{align*}
where~$f\colon X\to Y$. Equip $V$ with a~structure of a~chronological Frobenius algebra
by the~following operations:
\begin{itemize}
\item multiplication $\mu\colon V\boxtimes V\to V$:
\begin{align*}
\mu(v_+\otimes v_+) &= v_+ & \mu(v_-\otimes v_+) &= XZv_-\\
\mu(v_+\otimes v_-) &= v_- & \mu(v_-\otimes v_-) &= 0
\end{align*}

\item unit $\eta\colon R\to V$:
$$
\eta(1) = v_+
$$

\item comultiplication $\Delta\colon V\to V\boxtimes V$:
\begin{align*}
\Delta(v_+) &= v_-\otimes v_+ + YZv_+\otimes v_- & \Delta(v_-) &= v_-\otimes v_-
\end{align*}

\item counit $\varepsilon\colon V\to R$:
\begin{align*}
\varepsilon(v_+) &= 0 & \varepsilon(v_-) &= 1
\end{align*}

\end{itemize}
As before one may check that all the~axioms of a~chronological Frobenius algebra are satisfied,
both multiplication and~comultiplication have degree $-1$ and both unit and counit have degree $1$.
Hence we have a~functor $\F_{\!XYZ}\colon R\cat{ChCob}^3(\emptyset)/_{XYZ}\to\cat{Mod}_R$
preserving the~grading and one can check that it preserves also the~relations $S,T,4Tu$.
It generalizes both functors described above. Indeed $\F_{\!Kh} = \F_{1,1,1}$ for~$R=\mathbb{Z}$
and for $\F_{\!ORS}$ notice first that if $R=\mathbb{Z}$ then there is an~isomorphism $A\cong V$ given by
\begin{equation}
1 \leftrightarrow v_+\qquad a_1 \leftrightarrow v_-
\end{equation}
Put $X=Z=1$ and~$Y=-1$. Then under this isomorphism the~permutation $S^A$ in the~algebra $A$
corresponds to the~permutation $S^V$ in $V$. The~same holds for other operations
and we have an~isomorphism of chronological Frobenius algebras
\begin{equation}
(A,\mu,\Delta,\eta,\varepsilon,S,\wedge) \cong (V,\mu,\Delta,\eta,\varepsilon,S,\boxtimes)
\end{equation}
so $\F_{\!ORS}$ is~equivalent to $\F_{1,-1,1}$.
\end{example}

Let $R_0 < R$ be a~subring of $R$ generated by the~coefficients $X,Y,Z$.
Directly from the~construction one can see that the~Khovanov complex is built in $R_0\catL{ChCob}^3$.
Moreover, the~functor $\F_h$ from
the~remark~\hyperref[rm:cob-chch-cat]{\ref*{chpt:chcob}.\ref*{rm:cob-chch-cat}}
induced by a~ring homomorphism $h\colon R\to R'$ agrees with relations $S,T,4Tu$,
hence it extends to a~functor between categories of complexes
$\F_{\!h}\colon\cat{Kob}_{XYZ}\to\cat{Kob}_{h(X)h(Y)h(Z)}$.
It is easy to check that $\F_{\!h(X)h(Y)h(Z)}\circ\F_{\!h} = \F_{\!h}\circ \F_{\!XYZ}$.

\begin{proposition}
Let $h\colon R\to R'$ be a~ring homomorphism and $L, L'$ be links.
If homology groups of the~links computed for $\F_{\!XYZ}$ are isomorphic,
so are the~ones computed for $\F_{\!h(X)h(Y)h(Z)}$.
In particular, if $h$ is an~isomorphism, $\F_{\!h}$ is an~isomorphism of categories
and both $\F_{\!XYZ}$ and~$\F_{\!h(X)h(Y)h(Z)}$ carry the~same amount of information.
\end{proposition}

\begin{remark}\label{rmk:frob-alg-XYZ-univ}
Let $R_U=\mathbb{Z}[x,y,z,z^{-1}]/(x^2=y^2=1)$ be a~reduced ring of polynomials,
and put $\F_{\!U} = \F_{\!xyz}$. Then for any functor $\F_{\!XYZ}$ we have
\begin{equation}
\F_{\!XYZ}\circ \F_{\!h} = \F_{\!h}\circ\F_{\!U}
\end{equation}
where~$\F_{\!h}$ is a~functor given by the~ring epimorphism $h\colon R_U\to R$ sending
$x,y,z\in R_U$ respectively to $X,Y,Z\in R$.
\end{remark}

The~functor $\F_U$ given above is the~generalization of both $\F_{Kh}$ and~$\F_{ORS}$.
In particular, if the~homology groups of two links computed for this functor are isomorphic,
neither the~standard Khovanov homology nor the~odd version can distinguish the~links.

Notice also that an~isomorphism $h\colon R\to R'$ induces an~isomorphism of chronological Frobenius algebras
$\F_{\!XYZ}(\S1)$ and~$\F_{\!h(X)h(Y)h(Z)}(\S1)$. For instance, homology groups computed for the~functors
\begin{equation}
\F_{\pm x,\pm y,\pm z},\qquad \F_{\pm y,\pm x,\pm z},\qquad \F_{\pm x,\pm y,\pm z^{-1}}, \qquad \F_{\pm y,\pm x,\pm z^{-1}}
\end{equation}
over $R_U$ are all isomorphic.

There is also a~natural choice for the~functor $\F$: the~module of morphisms from some fixed object.

\begin{example}\label{ex:frob-alg-univ}
Define the~tautological functor $\F_{\!X}\colon\catL{ChCob}^3\to\cat{Mod}_R$
for a~given object $X\in\catL{ChCob}^3$ as follows:
\begin{align}
\F_{\!X}(Y) &:= \Mor(X,Y) &
\F_{\!X}(S) &:= S\circ (\cdot)
\end{align}
It gives a~chronological Frobenius algebra $A_X$ in an~obvious way.
\end{example}

In the~case of chronological cobordisms without change of chronology relations,
the~functor $\F_{\!X}\colon\cat{ChCob}^3(\emptyset)\to\cat{Mod}_R$ is faithful
and any $\F\colon\cat{ChCob}^3(\emptyset)\to\cat{A}$ factors by it.

\begin{question}
Is the~tautological functor $\F_{\!X}$ faithful for a~given object $X\in\catL{ChCob}^3(\emptyset)$?
\end{question}

A~positive answer to the~question above for some object $X$ will imply
$\F_{\!X}$ is universal and if for some links $L, L'$ the~homology groups
$\F_{\!X}\KhCom(L), \F_{\!X}\KhCom(L')$ are equal, none homology groups described in this paper
can distinguish $L$ form $L'$.

\begin{question}
If $\F_{\!X}$ is a~faithful functor, are the~complexes $\F_{\!X}\KhCom(L)$ and~$\F_{\!X}\KhCom(L')$
homotopic if and only if the~complexes $\KhCom(L)$ and~$\KhCom(L')$ are homotopic?
Furthermore, does an~isomorphism of homology groups
$H^{\bullet}\F_{\!X}\KhCom(L) \cong H^{\bullet}\F_{\!X}\KhCom(L')$ give a~chain homotopy between
$\KhCom(L)$ and~$\KhCom(L')$?
\end{question}

\section{A~categorification of the~Jones polynomial}\label{sec:khov-jones}
We will now show the~connection between the~Jones polynomial and~homology groups
given by $\F_{\!XYZ}$. Let us first recall basic facts about the~Euler characteristic
of a~chain complex.

\begin{definition}\label{def:q-dim-Euler}
\term{A~graded rank} of a~graded $R$-module $M=\bigoplus_{i\in\mathbb{Z}}M_i$ is the~polynomial
\begin{equation}\label{eq:q-dim}
\dim_q M := \sum_{i\in\mathbb{Z}} q^i\dim M_i
\end{equation}
where~$\dim M_i$ stands for the~rank of $M_i$.
\term{The~Euler characterictic} of a~complex of graded $R$-modules $(C,d)$ is the~alternating sum
of graded dimensions of terms of $C$:
\begin{equation}\label{eq:q-Euler}
\chi_q(C) = \sum_{r\in\mathbb{Z}} (-1)^r\dim_q C^r
\end{equation}
\end{definition}

From the~basic linear algebra we know the~rank of a~quotient module is the~difference of
ranks of the~divided module and the~divisor. Therefore
\begin{equation}
\dim_q(M/N) = \dim_q(M) - \dim_q(N).
\end{equation}

\begin{corollary}
Let $(C,d)$ be a~complex of graded $R$-modules. Then the~homology groups $H^*(C)$ are also graded and
\begin{equation}
\chi_q(C) = \chi_q(H^{\bullet}(C))
\end{equation}
In particular the~Euler characteristic is preserved by chain homotopies.
\end{corollary}

\begin{corollary}
Pick a~finite exact sequence of graded complexes
\begin{equation}
0\to C_r\to C_{r+1}\to\dots\to C_s\to 0
\end{equation}
Then the~alternating sum of their Euler characteristics vanishes:
\begin{equation}
\sum_{i=r}^s (-1)^i\chi_q(C_i) = 0
\end{equation}
\end{corollary}

After these short remarks we are ready to show how to recover the Jones polynomial from the~Khovanov complex.
For this denote by $J_D(q)$ the~Euler characteristic of the~Khovanov complex of a~link diagram $D$
given by the~functor $\F_{\!XYZ}$:
\begin{equation}
J_D(q) := \chi_q(\F_{\!XYZ}\KhCom{D})
\end{equation}
We will show it is the~Jones polynomial up to normalisation.

\begin{theorem}\label{thm:q-jones}
The~polynomial $J_D(q)$ has the~following properties:
\begin{enumerate}[label=(qJ\arabic*)]
\item\label{it:q-Jones-unknot}$J_U(q) = q+q^{-1}$,
\item\label{it:q-Jones-skein}$q^{-2}J_{\sfnt{cr-or-p}}(q) - q^2J_{\sfnt{cr-or-n}}(q) = (q^{-1}-q)J_{\sfnt{cr-or-h}}(q)$
\end{enumerate}
Therefore we have an~equality
\begin{equation}\label{eq:Jones-vs-q-Jones}
V_L(t) = \frac{J_L(-t^{1/2})}{(-t^{1/2}-t^{-1/2})}
\end{equation}
\end{theorem}
\begin{proof}
The~point \ref{it:q-Jones-unknot} follows from the~definition of the~complex,
whereas \ref{it:q-Jones-skein} is a~consequence of existence of the~short
exact sequence~(\ref{eq:Kh-exact-seq}).
The~last equality is due to the~uniqueness theorem for the~Jones polynomial.
\end{proof}

\begin{remark}
The~connection between the~Jones polynomial and $J_D(q)$ can also be obtained more directly.
Let $\PKauff{D}_q$ be the~Euler characteristic of the~formal Khovanov bracket $\F_{\!XYZ}\KhBra{D}$.
From the~construction of the~bracket we can see directly the~following properties:
\begin{enumerate}[label=(qK\arabic*)]
\item\label{it:q-Kauff-empty}$\PKauff{\emptyset}_q = 1$
\item\label{it:q-Kauff-extra-circle}$\PKauff{U\sqcup D}_q = (q+q^{-1})\PKauff{D}_q$
\item\label{it:q-Kauff-smoothings}$\PKauff{\fnt{cr-nor-p}}_q = \PKauff{\fnt{cr-nor-h}}_q - q\PKauff{\fnt{cr-nor-v}}_q$
\end{enumerate}
and~obviously $J_D(q) = (-1)^{n_-(D)}q^{n_+(D)-2n_-(D)}\PKauff{D}_q$. The~bracket $\PKauff{D}_q$
has the~same meaning for $J_D(q)$ as the~Kauffman bracket for the~Jones polynomial.
In particular it can be expressed as a~state sum over smoothed diagrams:
\begin{equation}\label{eq:q-Kauff-states}
\PKauff{D}_q = \sum_{s\in S(D)}(-q)^{n_1(s)}(q+q^{-1})^{|s|}
\end{equation}
It is clear now that if we substitute $q=-A^{-2}$ and~$t=A^{-1/4}$ we get
\begin{equation}
(q+q^{-1})^{-1}J_D(q) = (-A)^{-3w(D)}\PKauff{A} = V_D(t)
\end{equation}
\end{remark}

As a~corollary from the~theorem~\ref{thm:q-jones} we have at hand several properties 
of the~Jones polynomial.

\begin{proposition}
Let $L$ be any link. Then
\begin{enumerate}
\item $J_{L^*}(q) = J_{L}(q^{-1})$
\item $J_{-L}(q) = J_L(q)$
\item $J_{L'}(q) = q^{2l}J_{L}(q)$, where~$L'$ is obtained form $L$ by reversing the~orientation of
its~component $L_0$ with the~linking number $\lk(L_0,L\backslash L_0) = l$.
\end{enumerate}
\end{proposition}

\chapter{Odds and ends}\label{chpt:ends}
\setcounter{section}{1}
\chead{\fancyplain{}{\thechapter. Odds and ends}}
The~mail goal of this paper was to find a~generalisation of both construction given by M.~Khovanov
and P.~Osv\'ath, J.~Rasmussen and~Z.~Szab\'o. We enriched the~category of oriented cobordisms
so that the~second got a~functorial description. Then we constructed a~complex in this category
and proved it was a~tangle invariant. Thanks to this we found a~common description for both homology theories.

One strange step in~\cite{Osv-Ras-Sz}, which does not appear in the~Khovanov's construction,
is looking for an~edge assignment for the~cube of resolutions.
Here we explained the~existence of the~assignment by the~fact that a~coefficient
of change of a~chronology is~independence of a~decomposition of the~change as a~permutation
of neighbouring critical points. In this way the~problem of existence of an~edge assignment
is reduced to the~problem of uniqueness of a~chronology change coefficient, which seems
to be more natural. However, the~prove given by us is still based on checking several cases.
Moreover, it is only a~minor modification of the~one given in~\cite{Osv-Ras-Sz}.

\begin{problem}
Why a~coefficient of a~chronology change adapted to some planar diagram is well-defined?
Is that true for a~larger class of changes of chronologies?
Is there a~simpler proof of
the~theorem~\hyperref[thm:D-chch-unique-coef]{\ref*{chpt:chcob}.\ref*{thm:D-chch-unique-coef}},
which is not based on checking different cases?
\end{problem}

The~next problem is the~lack of functoriality of planar operators in~$\cat{ChCob}^3$.
This is a~reason why we was unable to naturally define a~planar algebra of complexes in $\cat{Kob}$.
In the~case of classical cobordisms such a~structure gives automatically the~invariance of the~complex
for any tangle, provided the~invariance of elementary tangles in the~definitions of the~Reidemeister moves.
We overcame the~problem by computing complexes partially. In fact, this proof gives a~clue,
how we can restrict cube morphisms to have a~planar algebra.

We can define embedded cobordisms not only between tangles in a~disk, but also
in any compact two-dimensional submanifold of a~plain. In particular,
we have cobordisms between planar diagrams $\cob M{D_1}{D_2}$ and cubes
in the~category of planar diagrams and cobordisms between them.
Every such a~cobordism induces a~mapping
\begin{equation}
M\colon \cat{ChCob}^3(B_1)\times\dots\times\cat{ChCob}^3(B_s)\to\cat{ChCob}^3(B)
\end{equation}
which acts on objects $\Sigma_1,\dots,\Sigma_s$ by filling holes with cylinders
$C_{\Sigma_1},\dots,C_{\Sigma_s}$, whereas for cobordisms $\cob {S_i}{\Sigma_i}{\Sigma_i'}$ we have a~diagram
\begin{equation}
\xy
\square(0,0)/=>`=>`=>`=>/<1300,600>%
[D_1(\Sigma_1,\dots,\Sigma_s)`D_1(\Sigma_1',\dots,\Sigma_s')`D_2(\Sigma_1,\dots,\Sigma_s)`D_2(\Sigma_1',\dots,\Sigma_s');%
D_1(S_1,\dots,S_s)`M(\Sigma_1,\dots,\Sigma_s)`M(\Sigma_1',\dots,\Sigma_s')`D_2(S_1,\dots,S_s)]
\endxy
\end{equation}
which commutes up to invertible elements of $R$.
In case of classical cobordisms, $M$ is a~natural transformation of functors $D_1$ and~$D_2$.
This situation is similar, if we treat $D_1$ and~$D_2$ as functors of one variable (half-functors).
Therefore, we have an~induced action on the~category of cube complexes.

\begin{definition}
Say a~morphism of cube complexes $f\colon C\to D$ is \term{regular},
if for any $C\!C$-cube $\mathcal{I}$ in~$\CPO(\cat{ChCob}^3)$ the~induced
cube morphism $\mathcal{I}(f)$ is a~$C\!C$-cube.
\end{definition}

It turns out that the~category of cube complexes with regular morphisms
has a~natural structure of a~planar algebra and we in this framework the~proof
of the~theorem~\hyperref[thm:khov-inv]{\ref*{chpt:khov}.\ref*{thm:khov-inv}}
is a~bit shorter. However, one may ask
if the~category is natural in some sense or whether there is its~simpler definition.

\begin{problem}
Is there a~natural category with a~structure of a~planar algebra, containing cube complexes,
in which we can proof the~invariance of the~Khovanov complex?
\end{problem}

The~problems described so far are technical and do not bring much to the~mail goal of the~paper.
The~following two deal with possible constructions directly connected to homology groups.

All homology groups defined by functors $\F_{\!XYZ}$ categorify the~Jones polynomial.
In~\cite{BarNatan-tangl} D.~Bar-Natan showed how to recover the~polynomial directly from
the~complex $\KhCom(T)$. Unfortunately, it cannot be repeat in the~same way for chronological cobordisms,
because we do not have the~neck-cutting relation. In our case it has the~form
\begin{equation}
Z(X+Y)\ \image{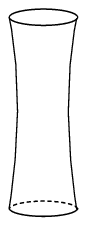}{27pt}\quad =\quad \image{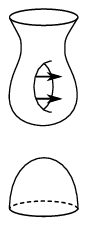}{27pt}\quad +\quad \image{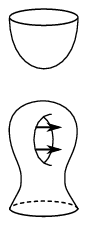}{27pt}
\end{equation}
For usual cobordisms, the~coefficient at the~left-hand side is equal $2$
and is invertible when we extend $\mathbb{Z}$ by a~fraction $\frac{1}{2}$.
In our case we can repeat it only if $X\neq Y$. In particular, we cannot do this
for odd theory. Moreover, the~existence of $(X+Y)^{-1}$ implies $X=Y$.
On the~other hand, each functor $\F_{\!XYZ}$ categorifies the~Jones polynomial.
This suggests we can obtain the~polynomial directly from the~complex $\KhCom(T)$.

Recall \term{a~trace} in an~$R$-additive category $\cat{C}$ is an~$R$-linear mapping
$\mathrm{Tr}\colon\mathrm{End}(\cat{C})\to G$, where~$\mathrm{End}(\cat{C})$ is the~class of
endomorphisms of the~category $\cat{C}$ and~$G$ is an~Abelian group, satisfying the~following condition:
\begin{equation}
\mathrm{Tr}(FG) = \mathrm{Tr}(GF)
\end{equation}
for any two morphisms $F\colon X\to Y$ and~$G\colon Y\to X$. Then we can define \term{the~dimension}
of an~object $X$ as a~trace of the~identity $\dim(X) = \mathrm{Tr}(\id_X)$
and we have an~Euler characteristic of a~complex given in a~usual way.
In particular, we can take for $G$ \term{the~trace group}
\begin{equation}
\Xi(\cat{C}) = \mathrm{End}(\cat{C}) / \langle FG - GF\ |\ F\colon X\to Y, G\colon Y\to X\rangle
\end{equation}
and~\term{the~universal trace} $\mathrm{Tr}_\star\colon\mathrm{End}(\cat{C})\to\Xi$.
It can be shown that any trace factorise by the~universal one.

\begin{problem}
Show the~connection between the~universal trace $\mathrm{Tr}_\star$ in $\cat{ChCob}^3$ and the~Jones polynomial.
\end{problem}

The~operation $\KhCom$ which associates a~complex in~$\cat{Cob}^3$ to a~tangle
induces chain maps between complexes for cobordisms between tangles.
In particular, a~cobordism $M$ between empty tangles (i.e. when $T_1=T_2=\emptyset$) is a~knotted surface
and $\KhCom(M)$ is a~multiplication by a~number. This gives invariants of surfaces.

In case of chronological cobordisms we can repeat the~proof from~\cite{BarNatan-tangl} with minor
modifications to show that the~naive definition gives a~chain map well-defined up to a~global invertible
element. We strongly believe that this the~whole construction can be fixed to produce well-defined
chain maps.

\begin{conjecture}
The~map $\KhCom$ extends functorially over cobordisms between tangles.
\end{conjecture}

\cleardoublepage
\phantomsection
\chapter*{Table of knots}\label{chpt:table}
\chead{\fancyplain{}{Table of knots}}
\begin{tabular}{ccccc}
\includegraphics[keepaspectratio=1,height=2cm]{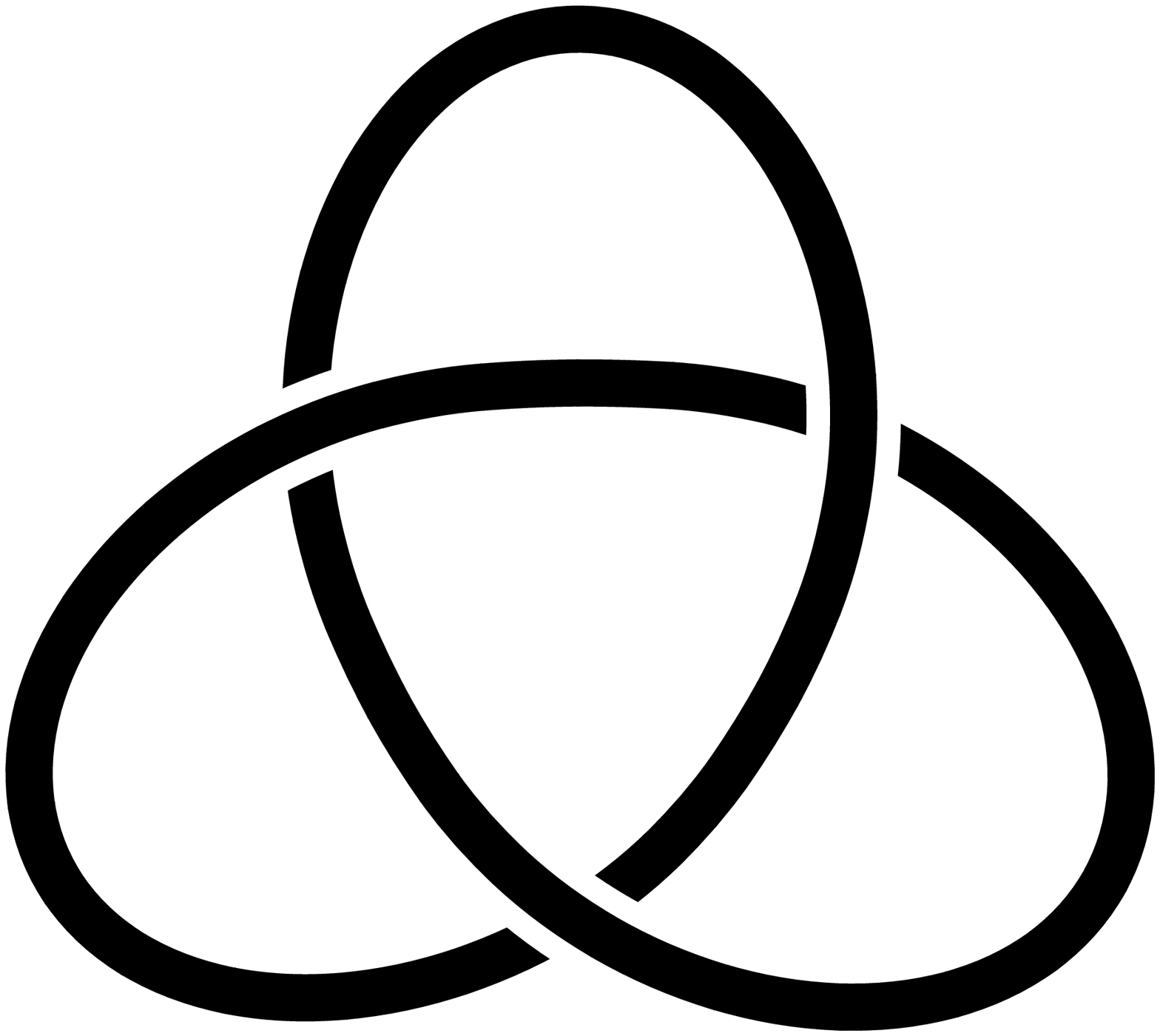}
$3_{1}$&
\includegraphics[keepaspectratio=1,height=2cm]{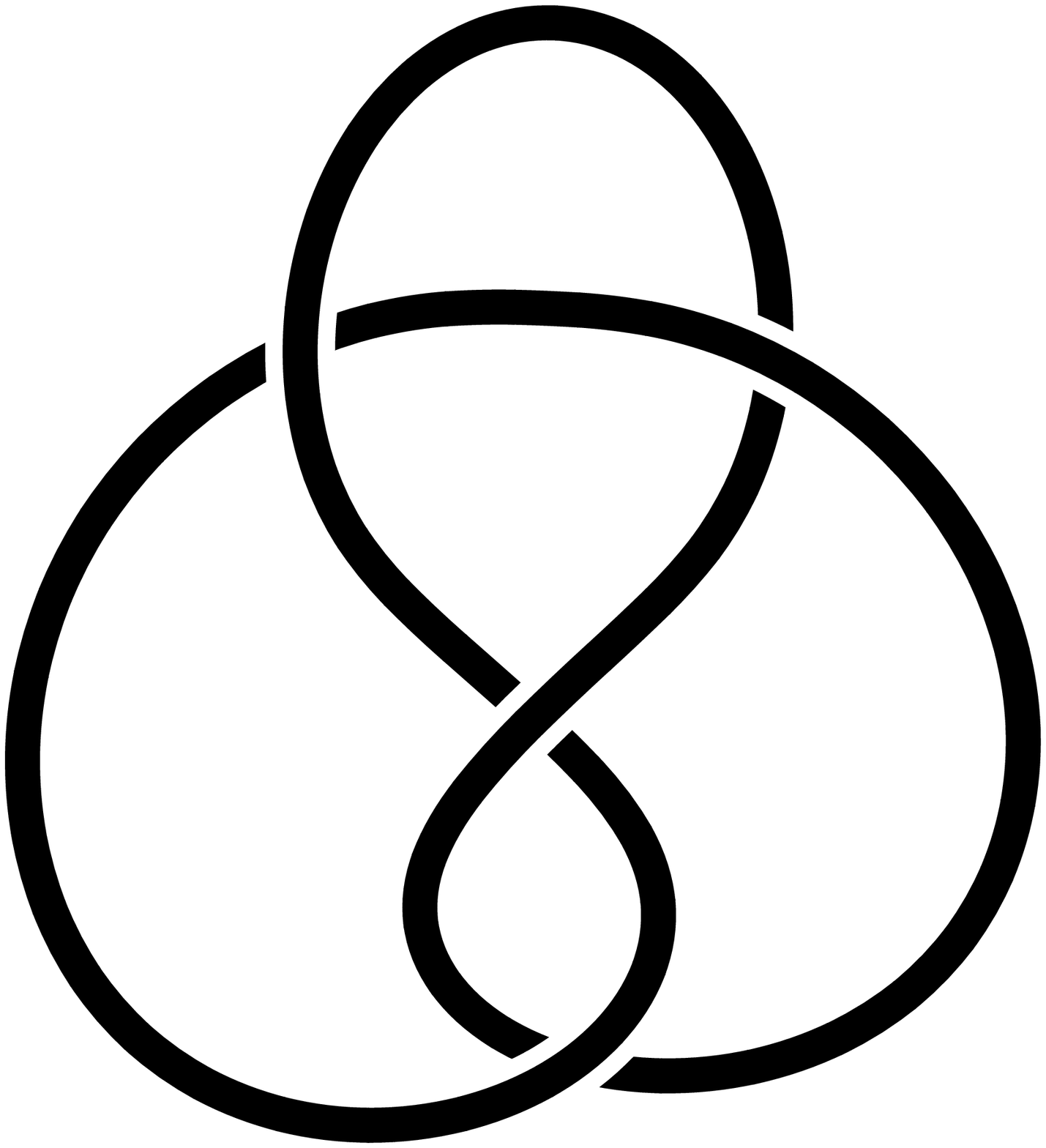}
$4_{1}$&
\includegraphics[keepaspectratio=1,height=2cm]{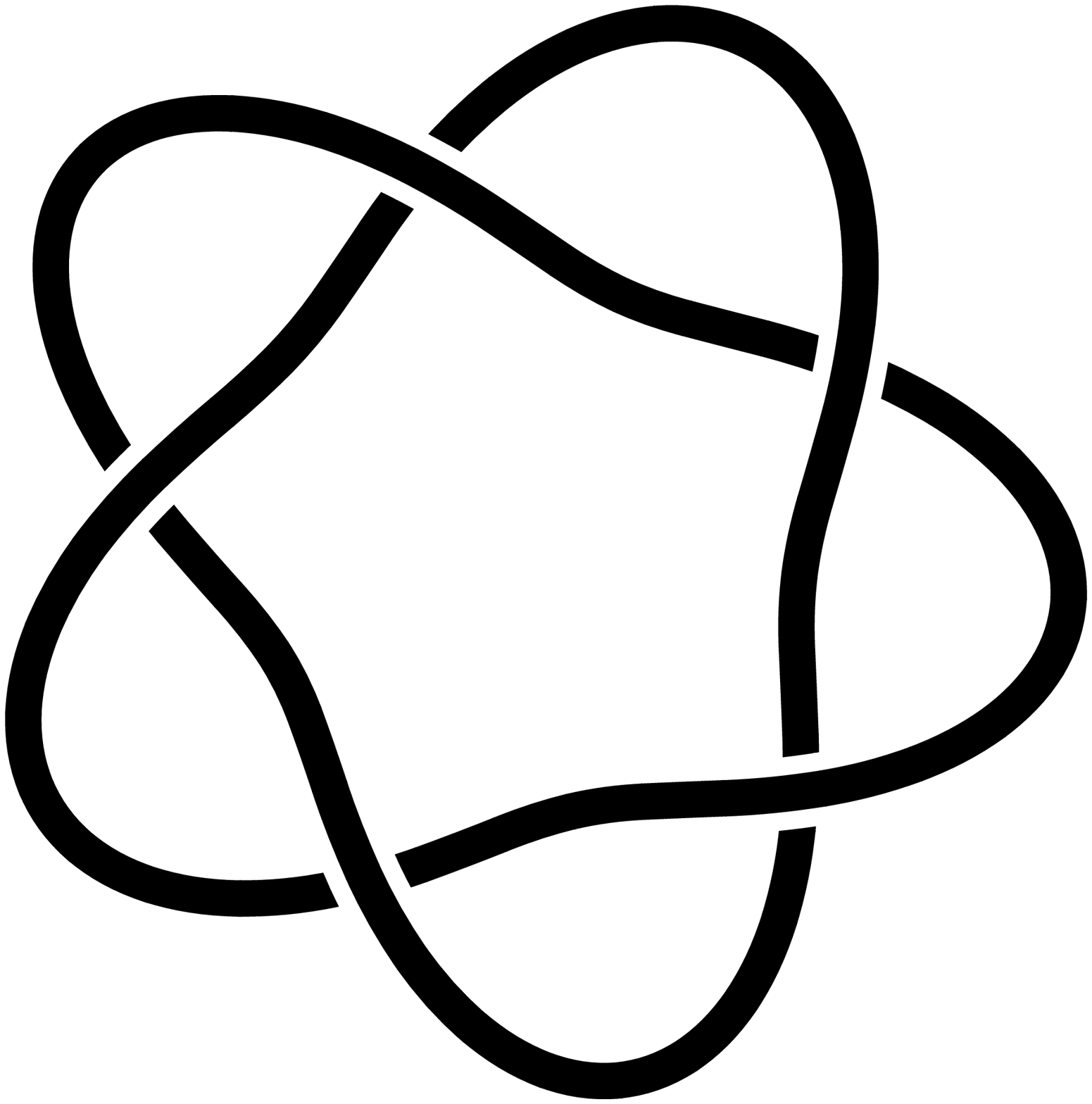}
$5_{1}$&
\includegraphics[keepaspectratio=1,height=2cm]{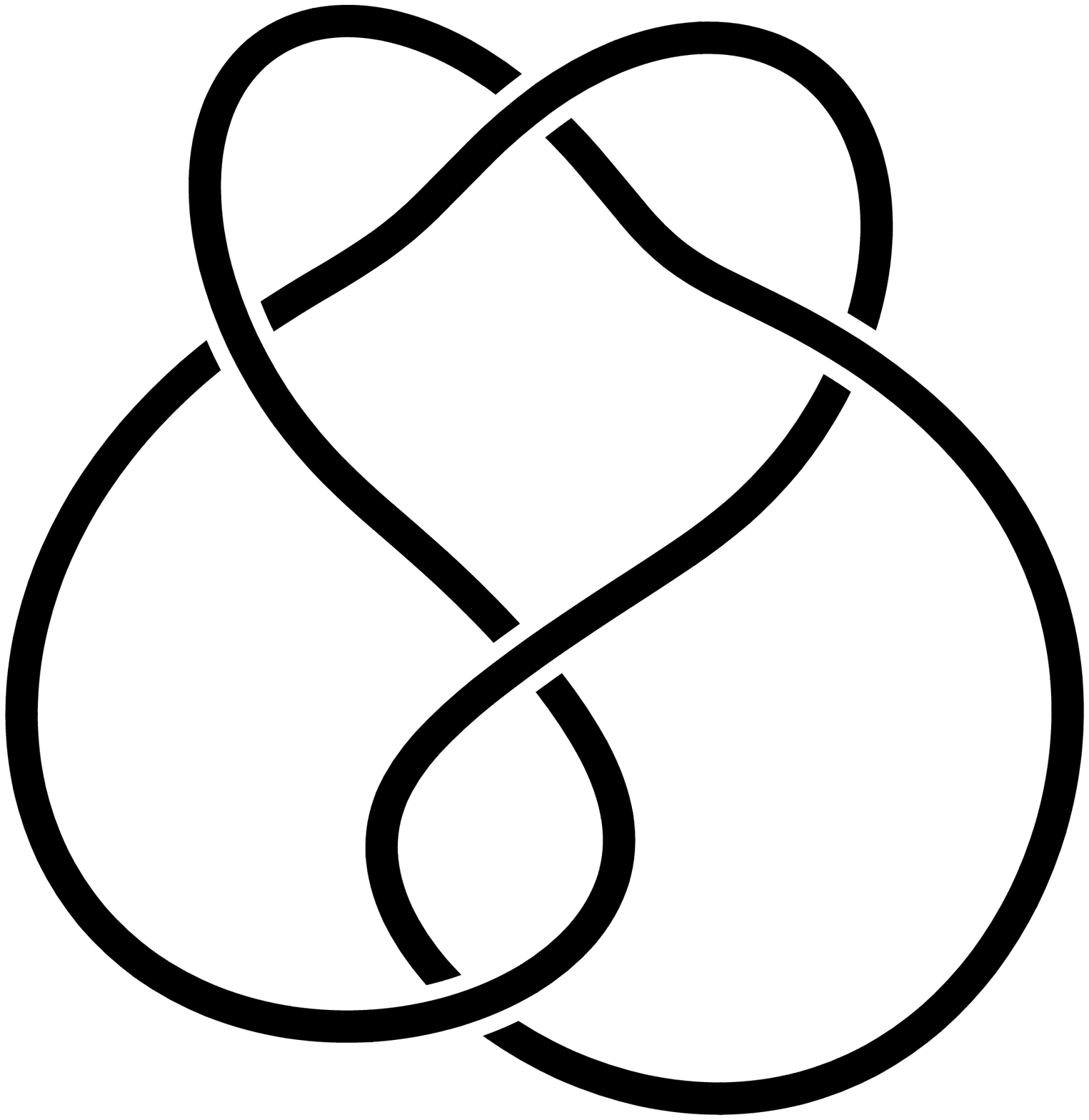}
$5_{2}$&
\includegraphics[keepaspectratio=1,height=2cm]{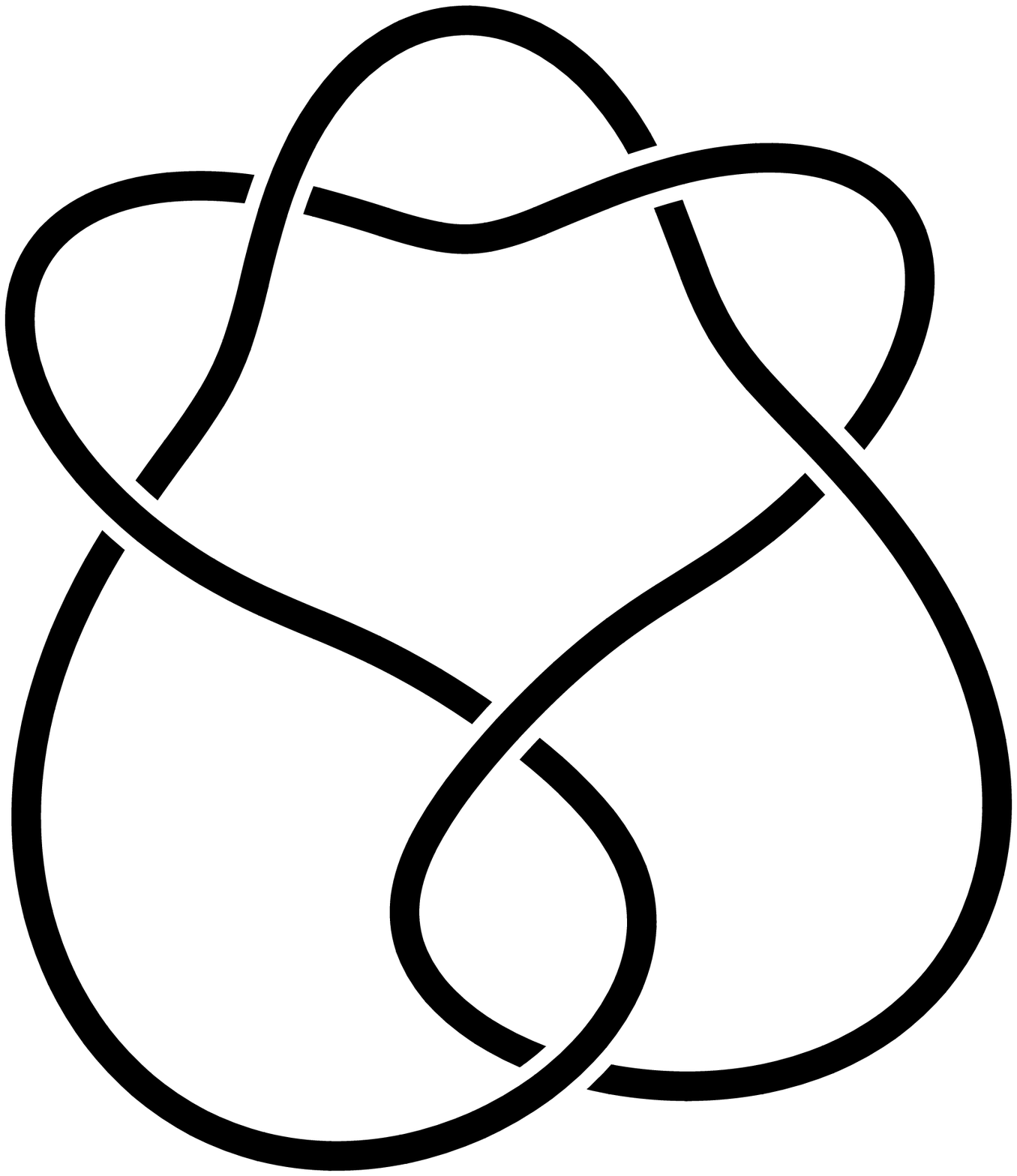}
$6_{1}$\\ \ &&&&\\
\includegraphics[keepaspectratio=1,height=2cm]{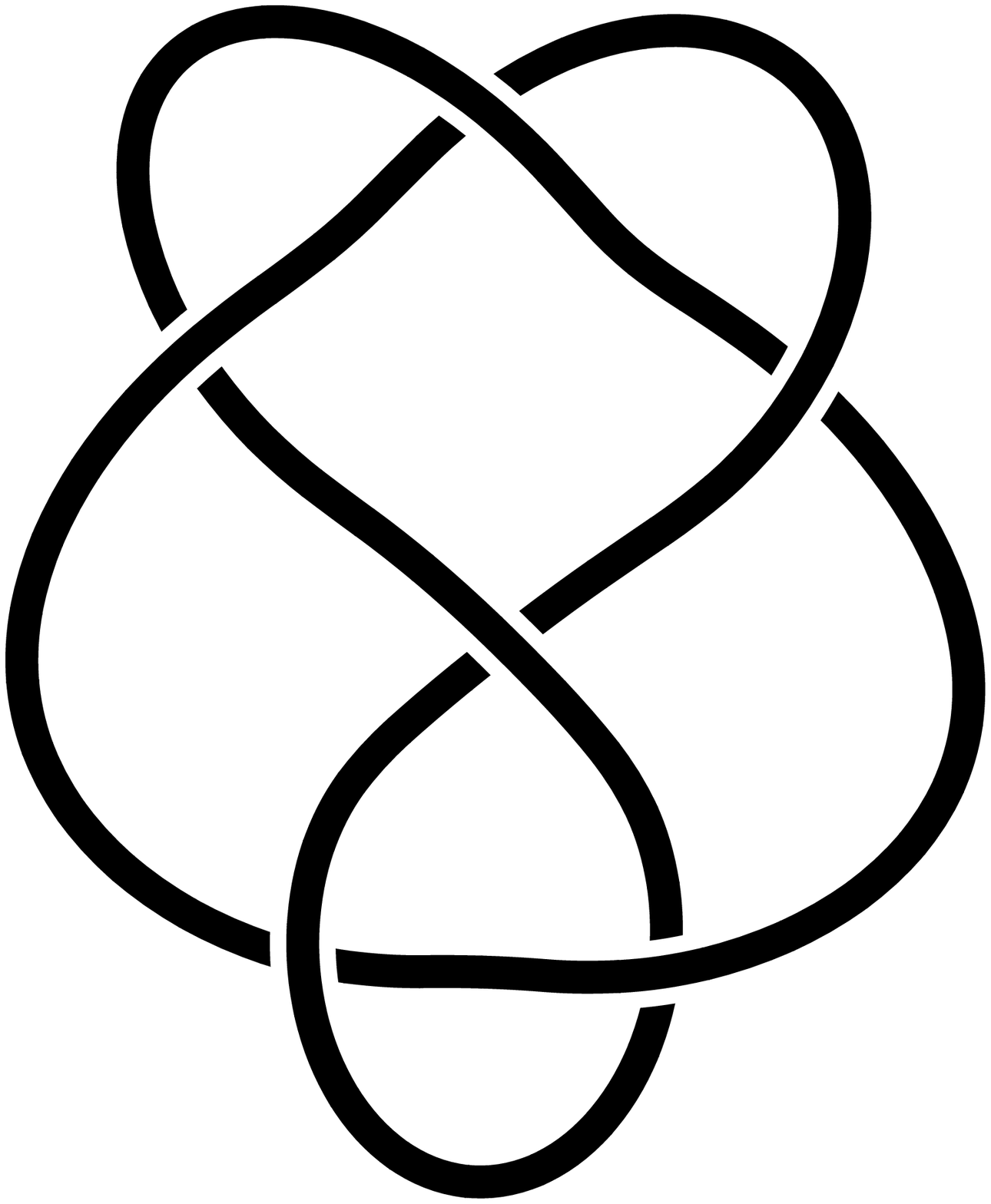}
$6_{2}$&
\includegraphics[keepaspectratio=1,height=2cm]{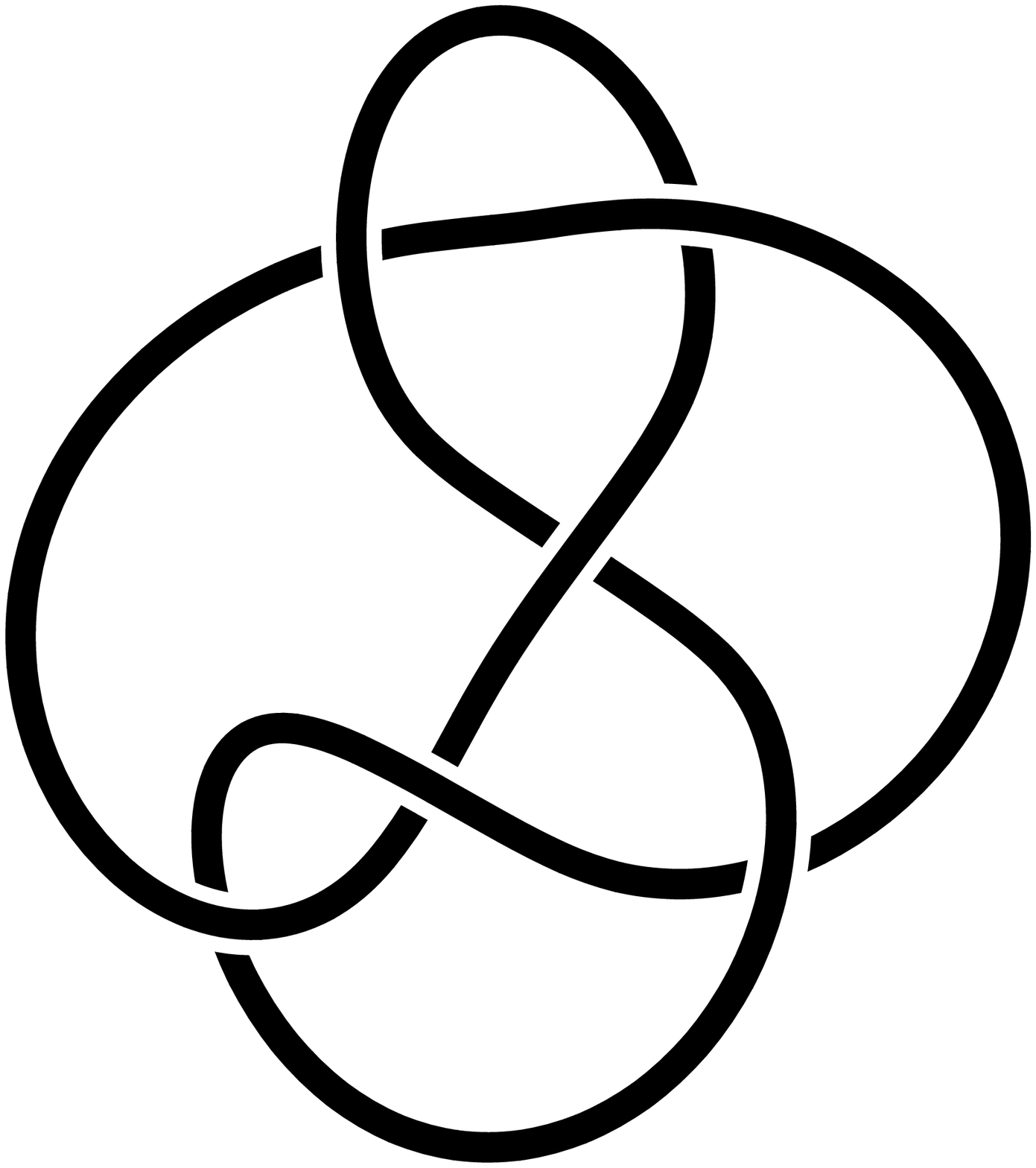}
$6_{3}$&
\includegraphics[keepaspectratio=1,height=2cm]{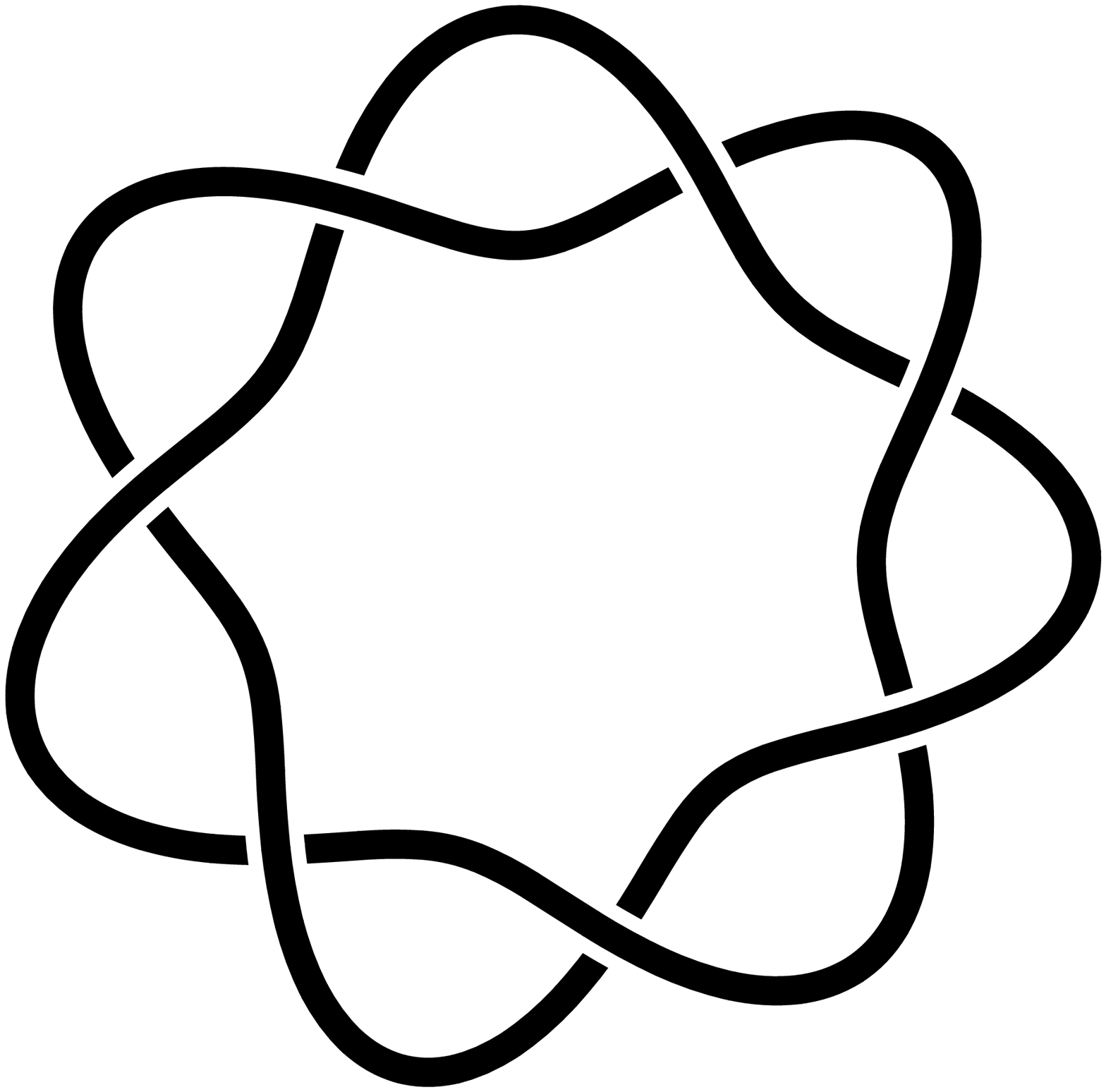}
$7_{1}$&
\includegraphics[keepaspectratio=1,height=2cm]{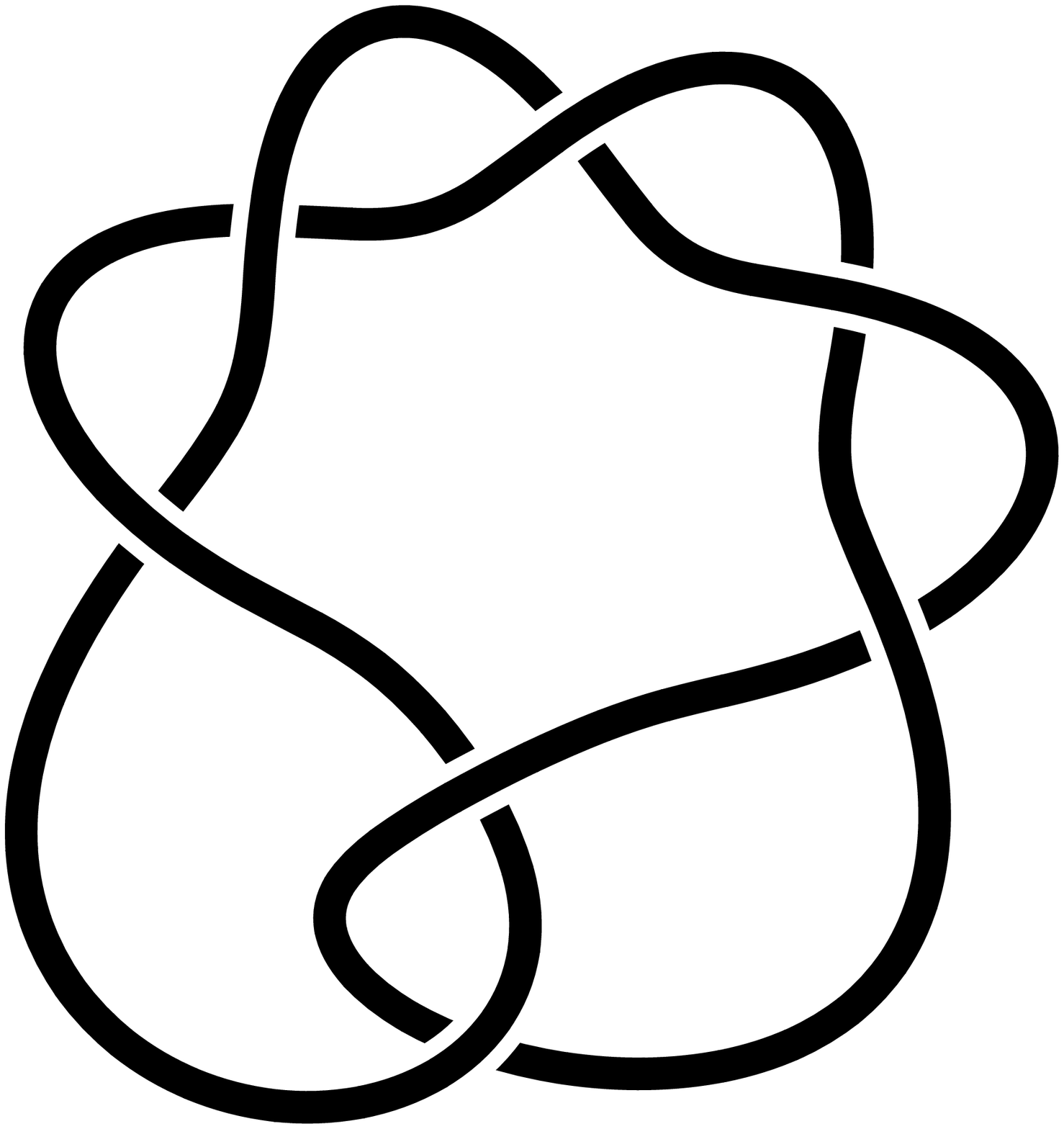}
$7_{2}$&
\includegraphics[keepaspectratio=1,height=2cm]{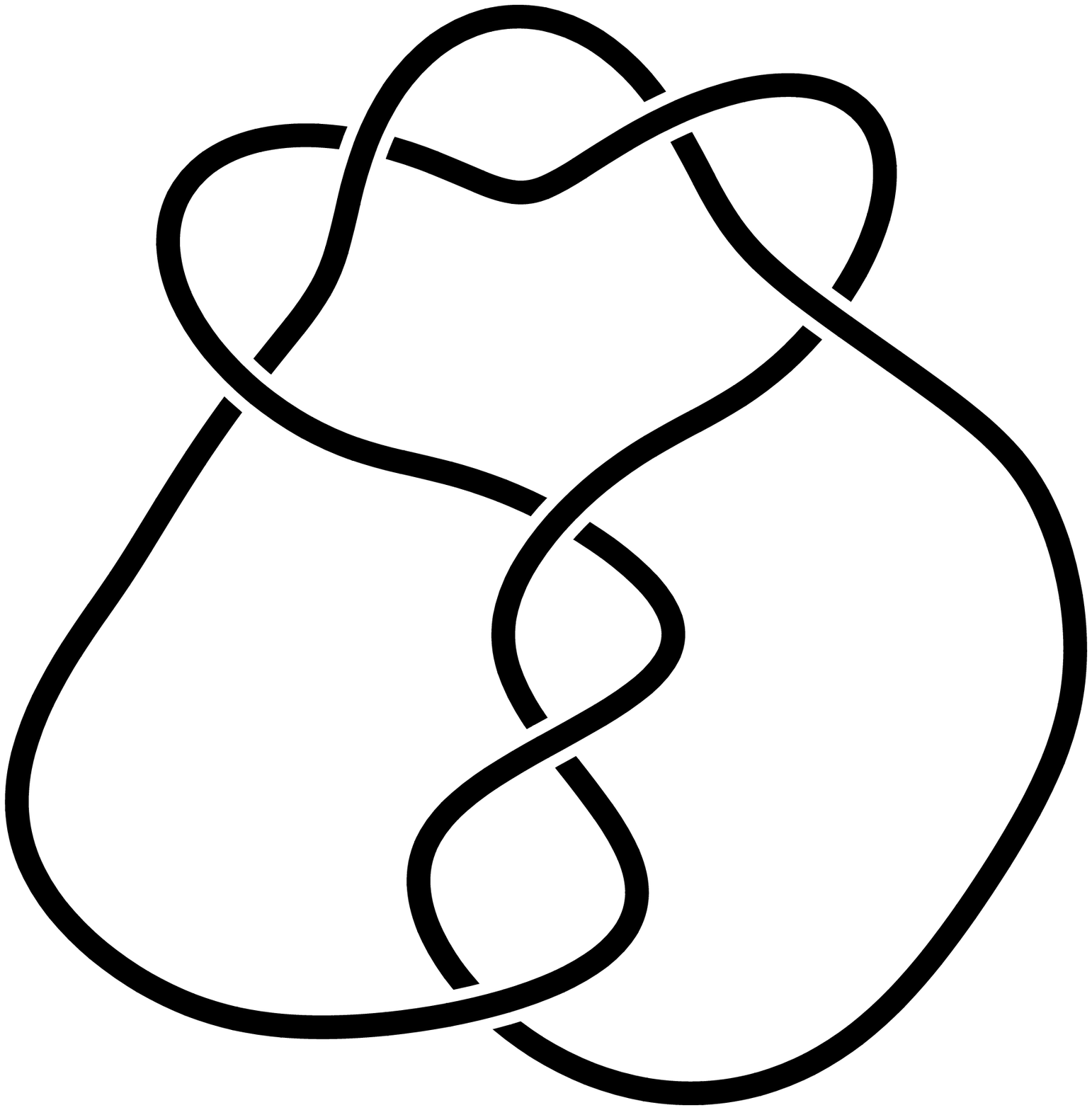}
$7_{3}$\\ \ &&&&\\
\includegraphics[keepaspectratio=1,height=2cm]{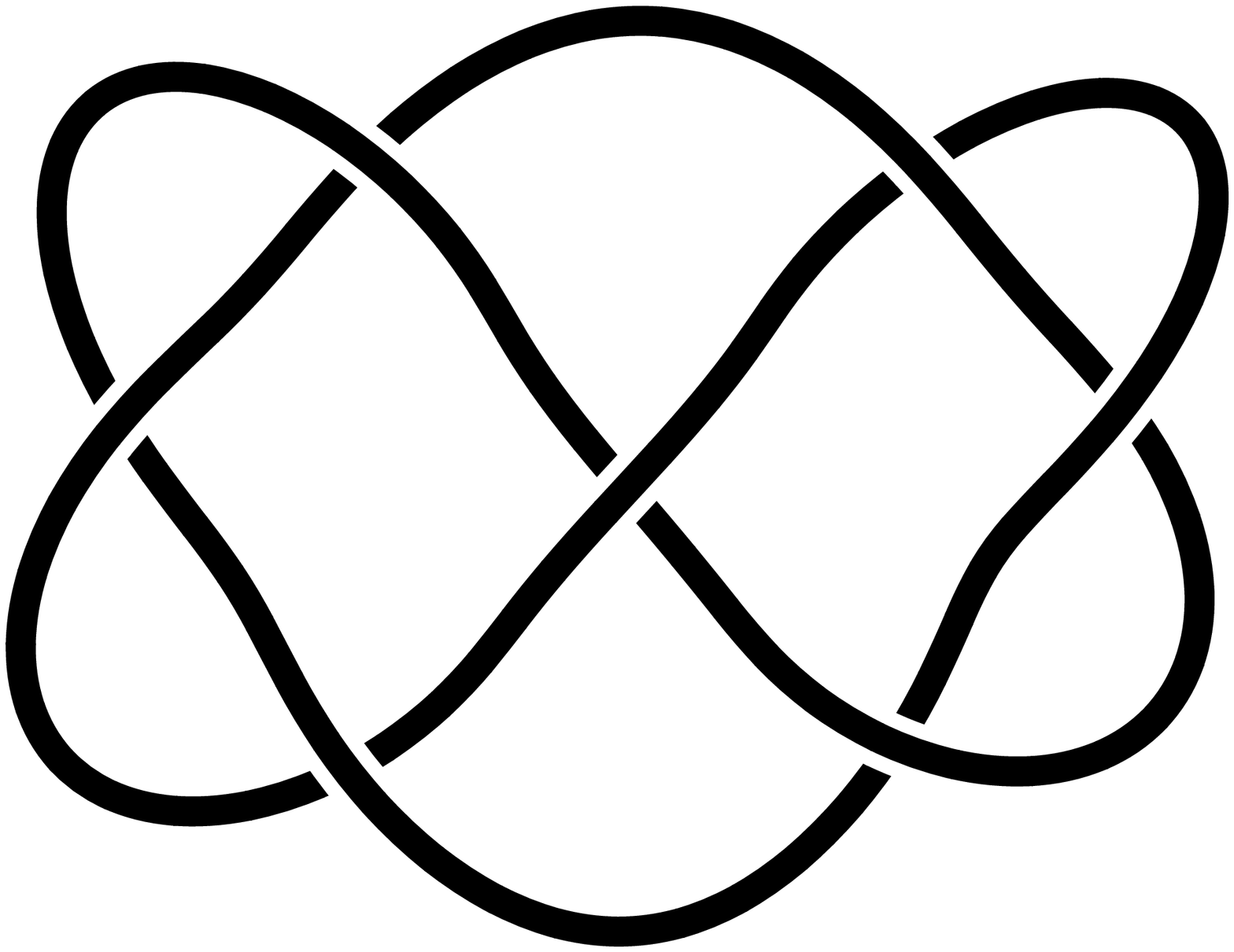}
$7_{4}$&
\includegraphics[keepaspectratio=1,height=2cm]{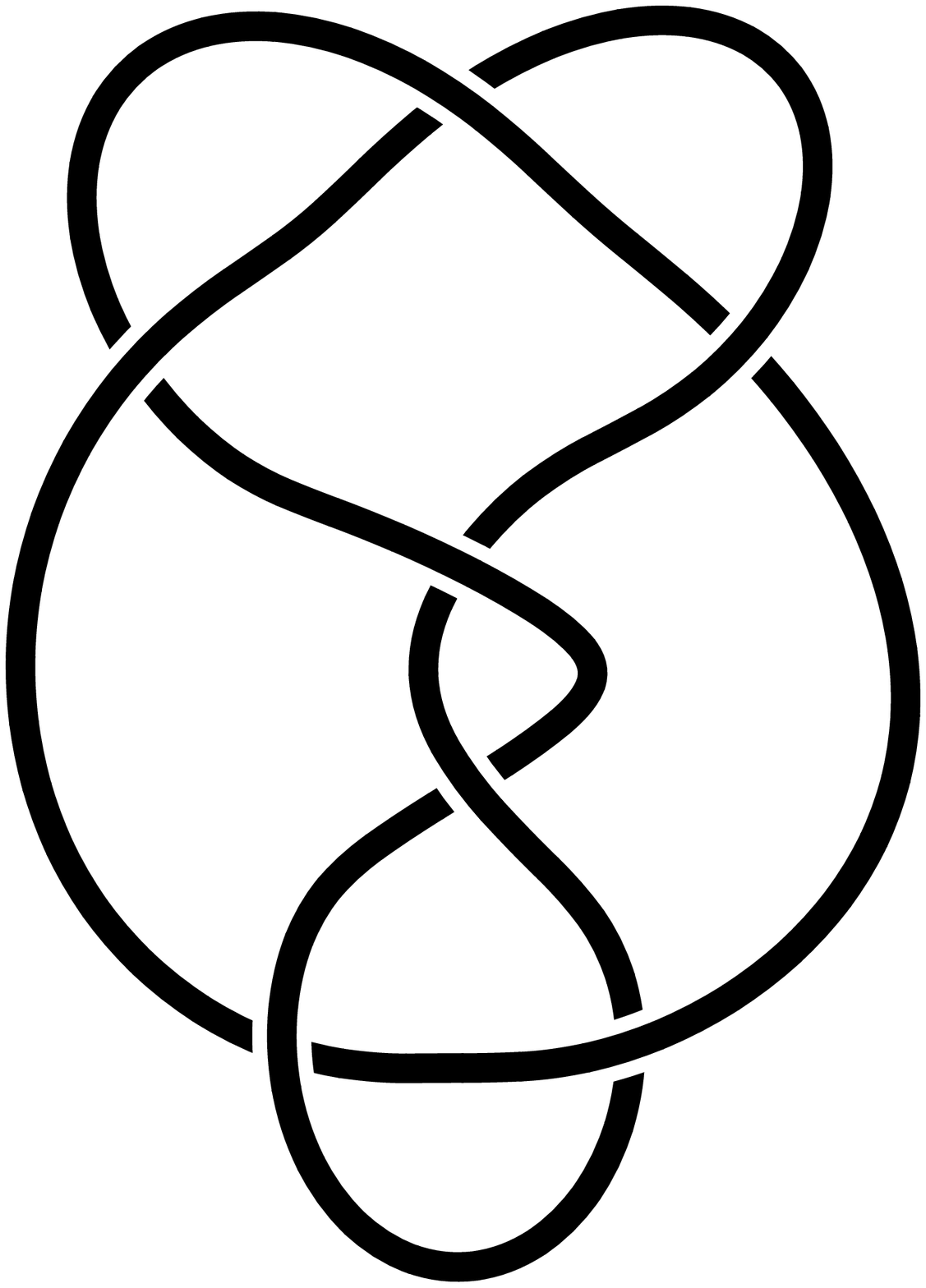}
$7_{5}$&
\includegraphics[keepaspectratio=1,height=2cm]{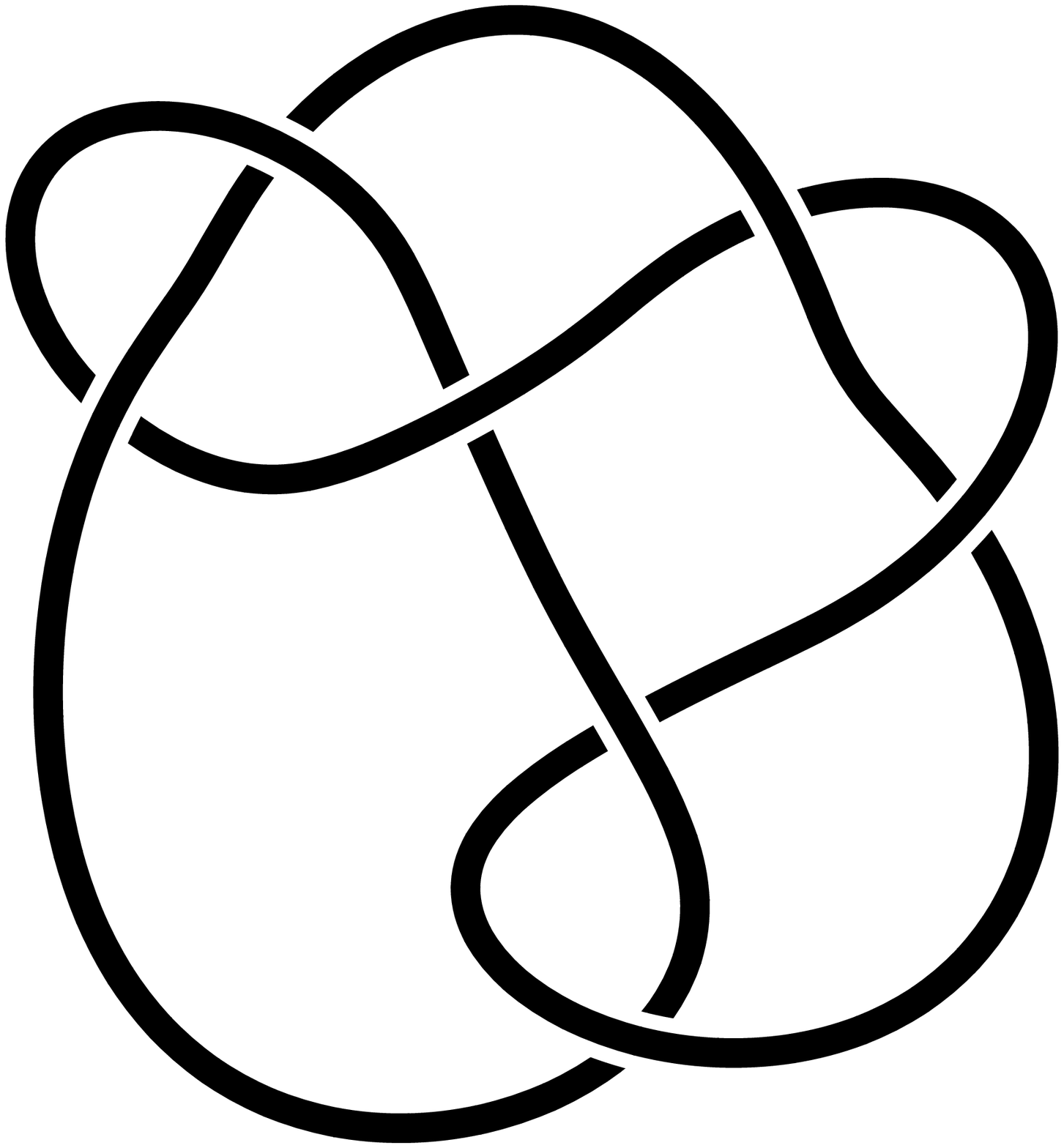}
$7_{6}$&
\includegraphics[keepaspectratio=1,height=2cm]{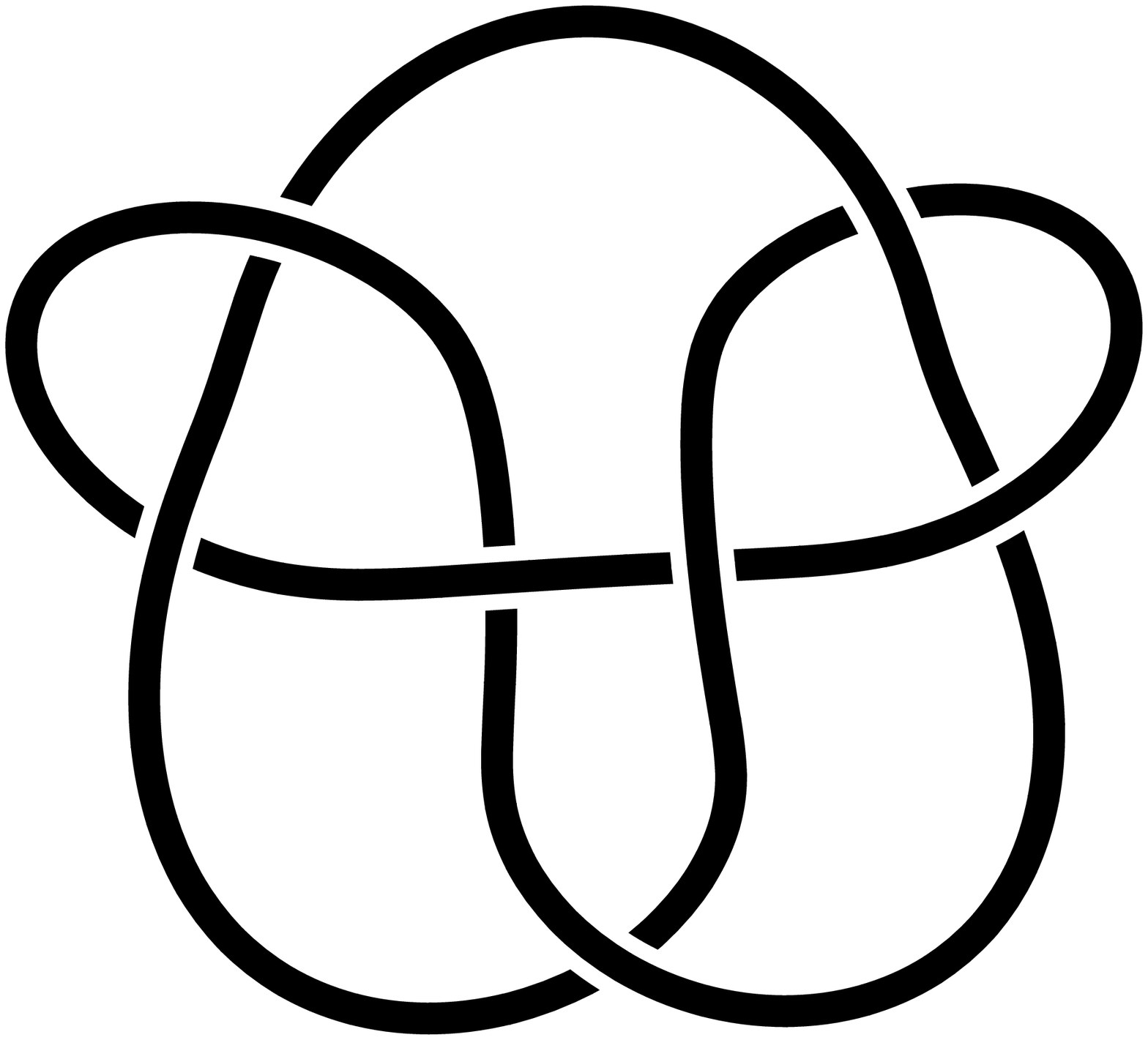}
$7_{7}$&
\includegraphics[keepaspectratio=1,height=2cm]{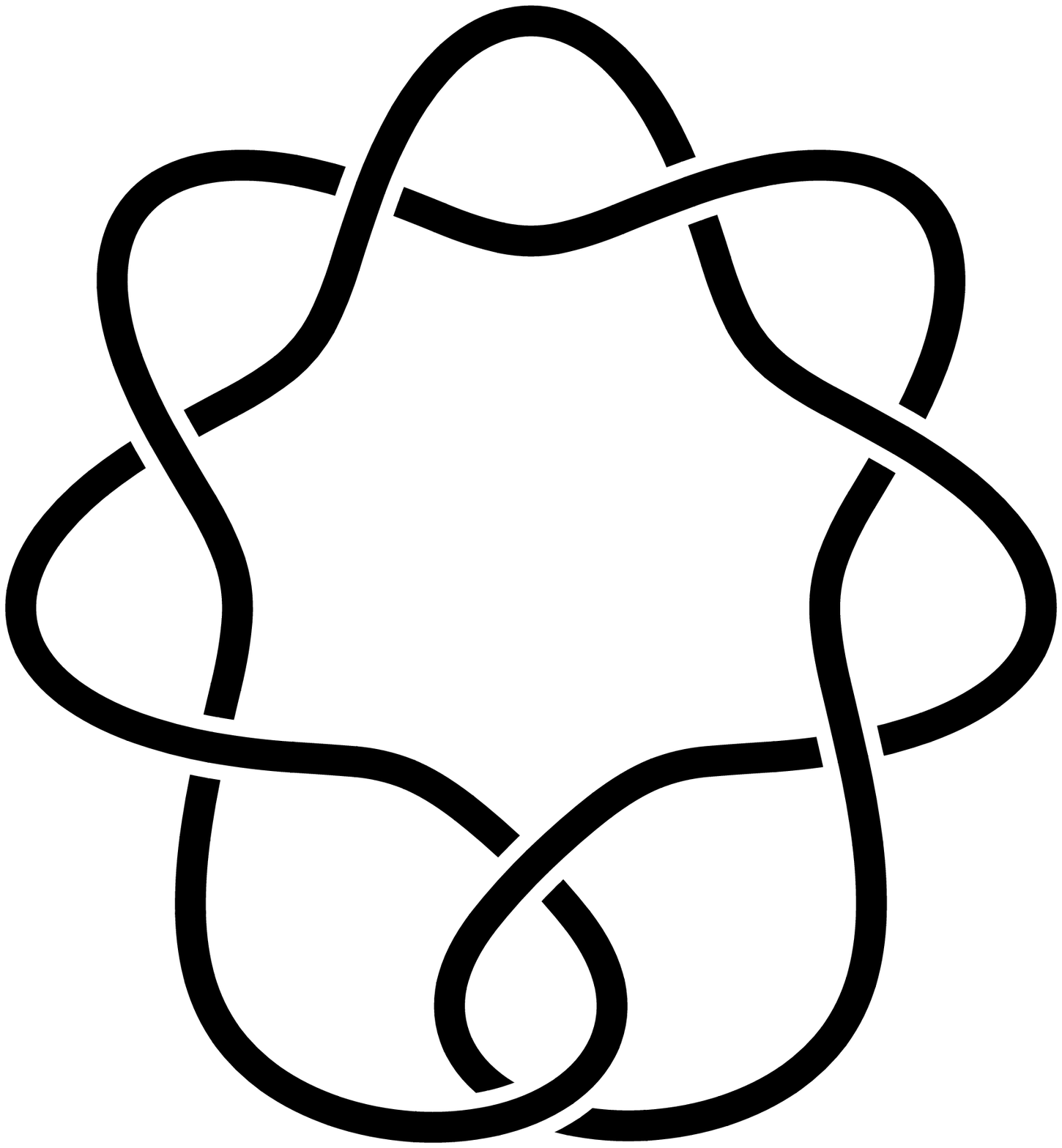}
$8_{1}$\\ \ &&&&\\
\includegraphics[keepaspectratio=1,height=2cm]{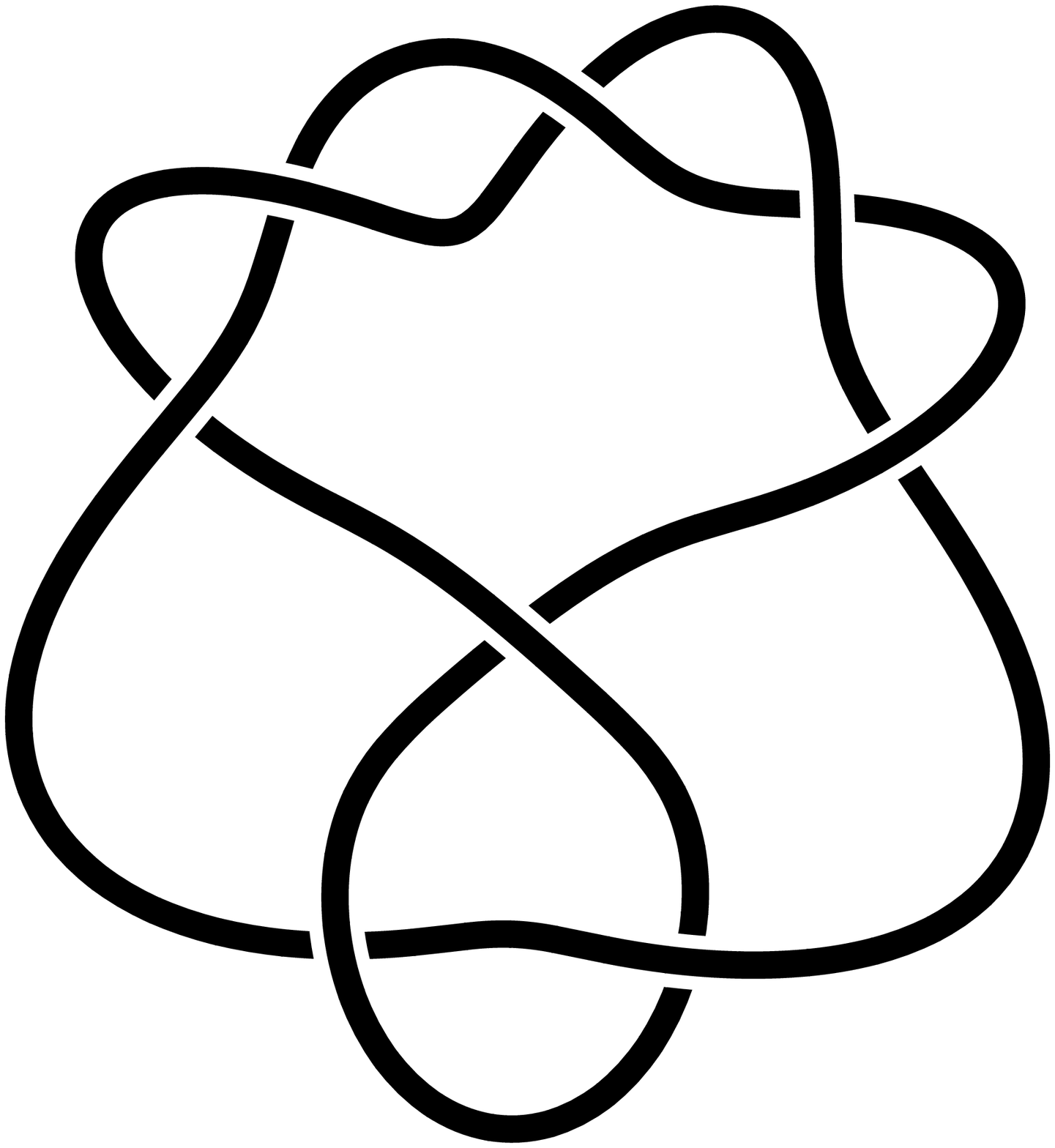}
$8_{2}$&
\includegraphics[keepaspectratio=1,height=2cm]{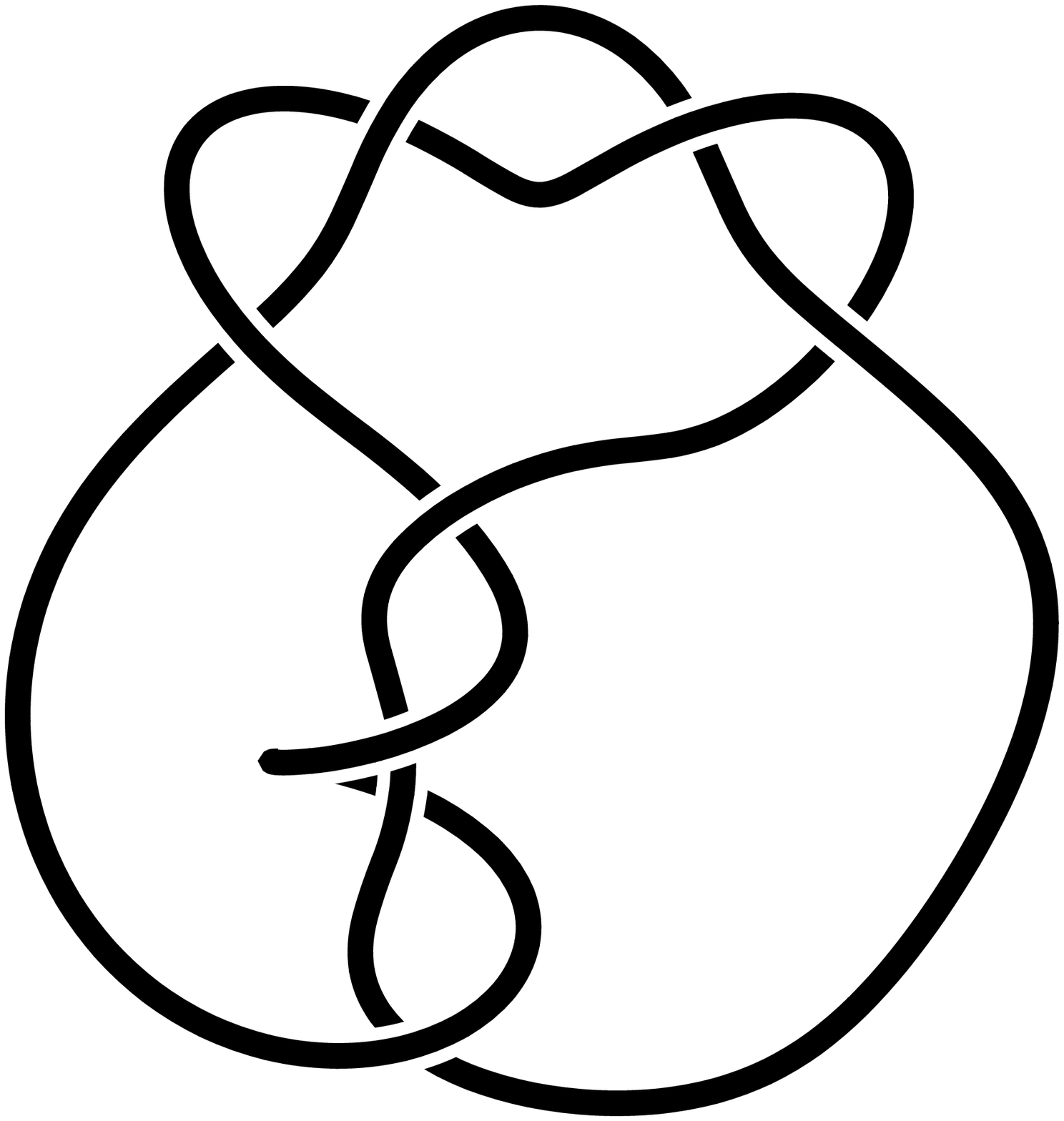}
$8_{3}$&
\includegraphics[keepaspectratio=1,height=2cm]{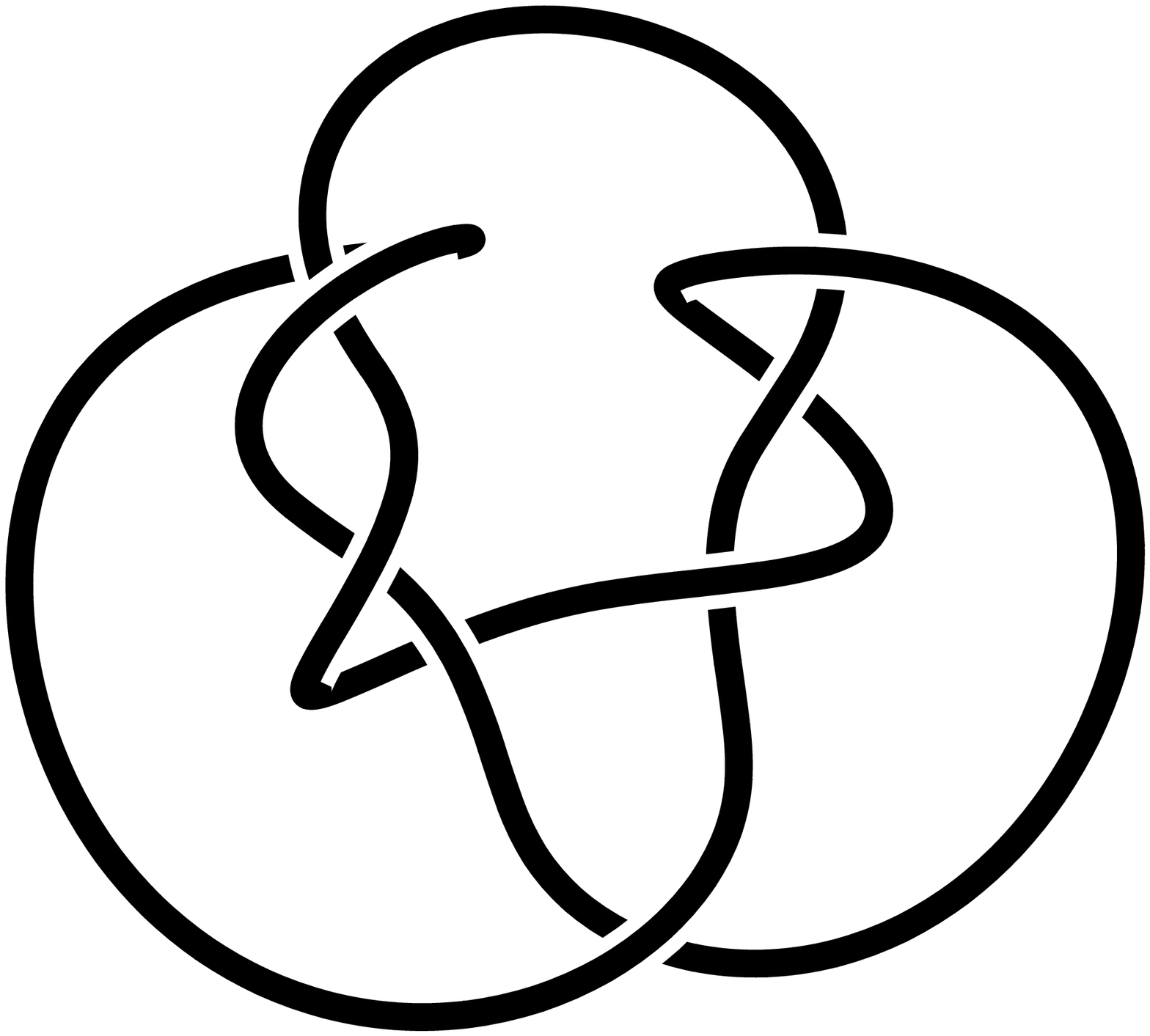}
$8_{4}$&
\includegraphics[keepaspectratio=1,height=2cm]{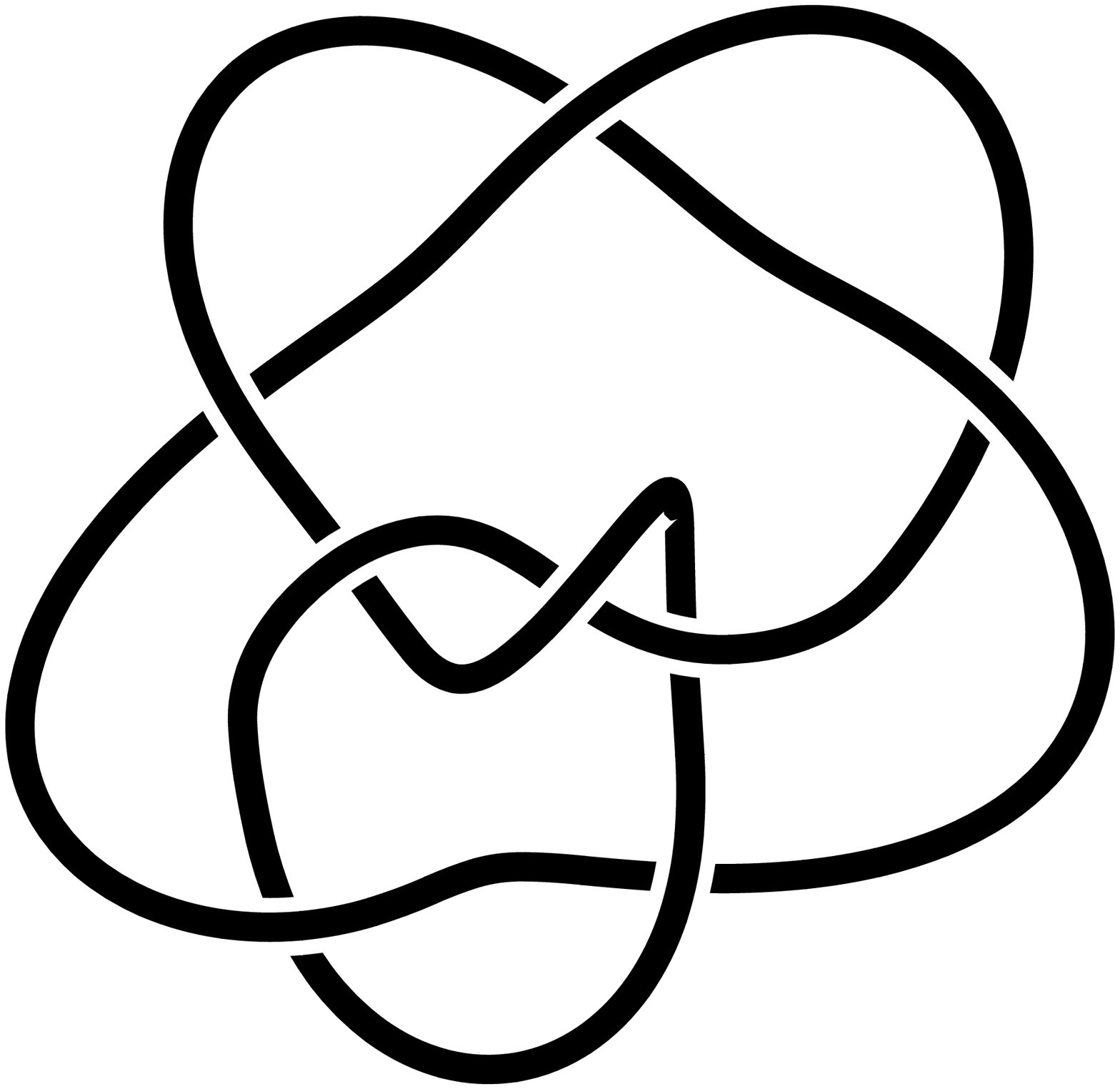}
$8_{5}$&
\includegraphics[keepaspectratio=1,height=2cm]{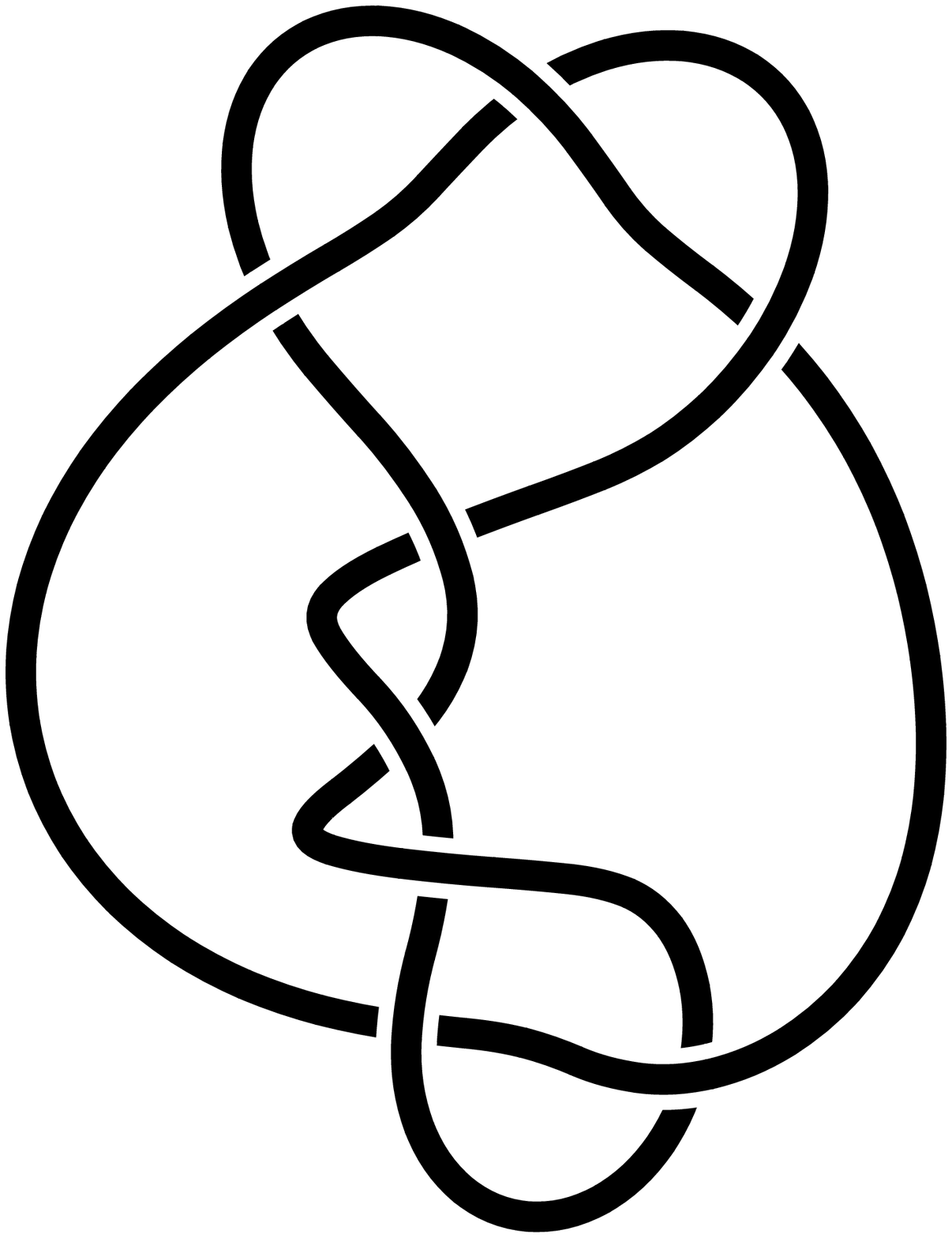}
$8_{6}$\\ \ &&&&\\
\includegraphics[keepaspectratio=1,height=2cm]{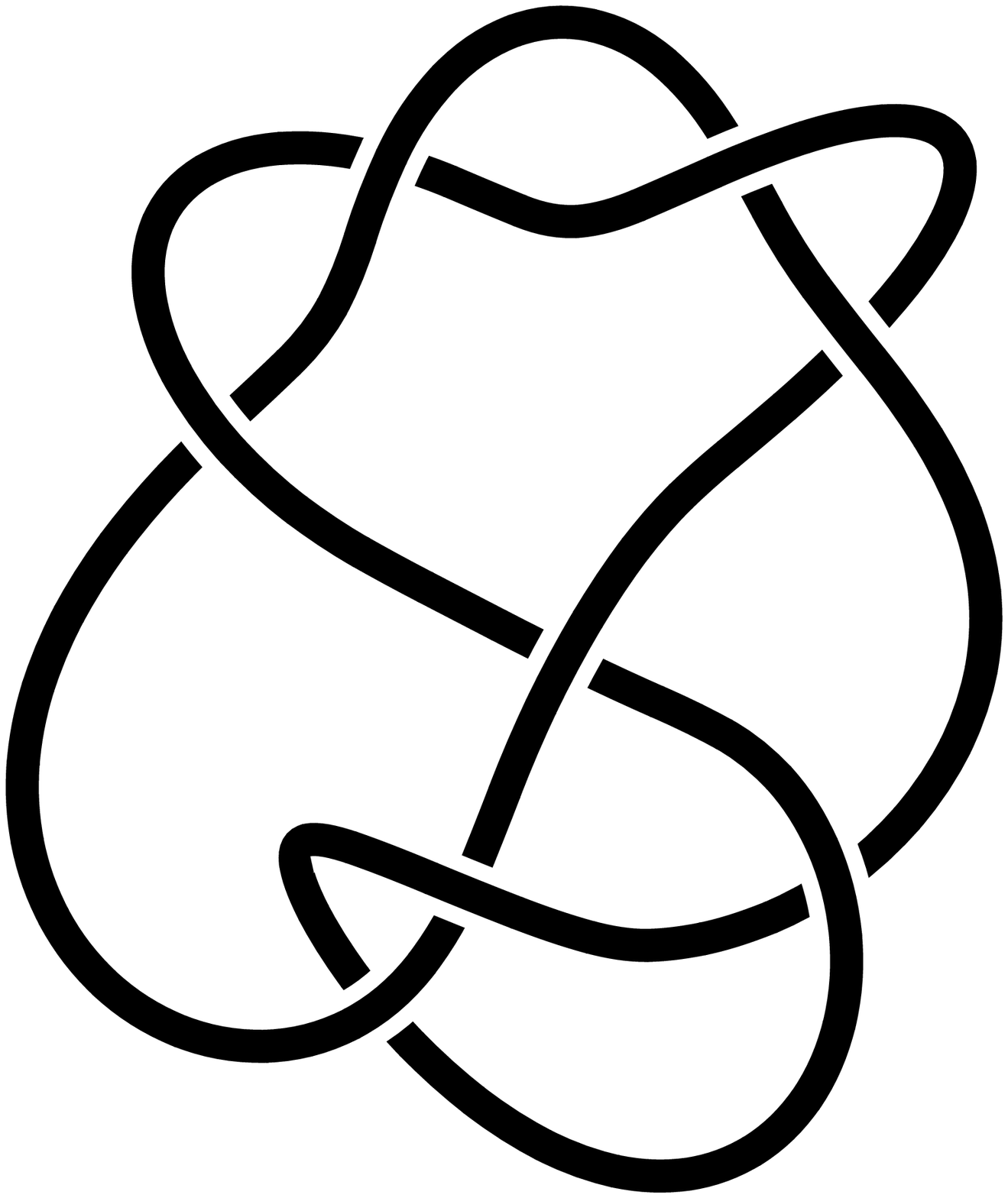}
$8_{7}$&
\includegraphics[keepaspectratio=1,height=2cm]{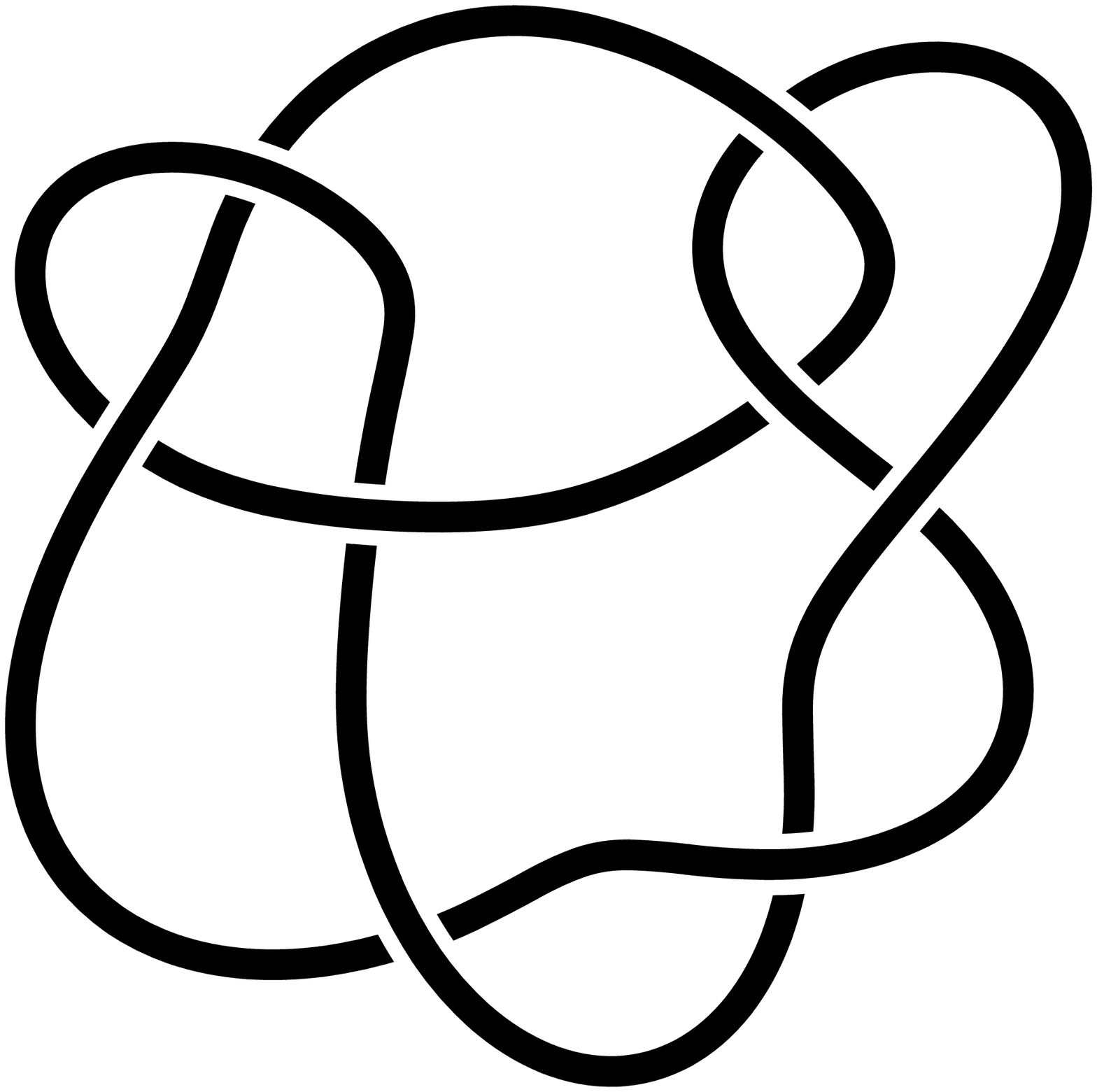}
$8_{8}$&
\includegraphics[keepaspectratio=1,height=2cm]{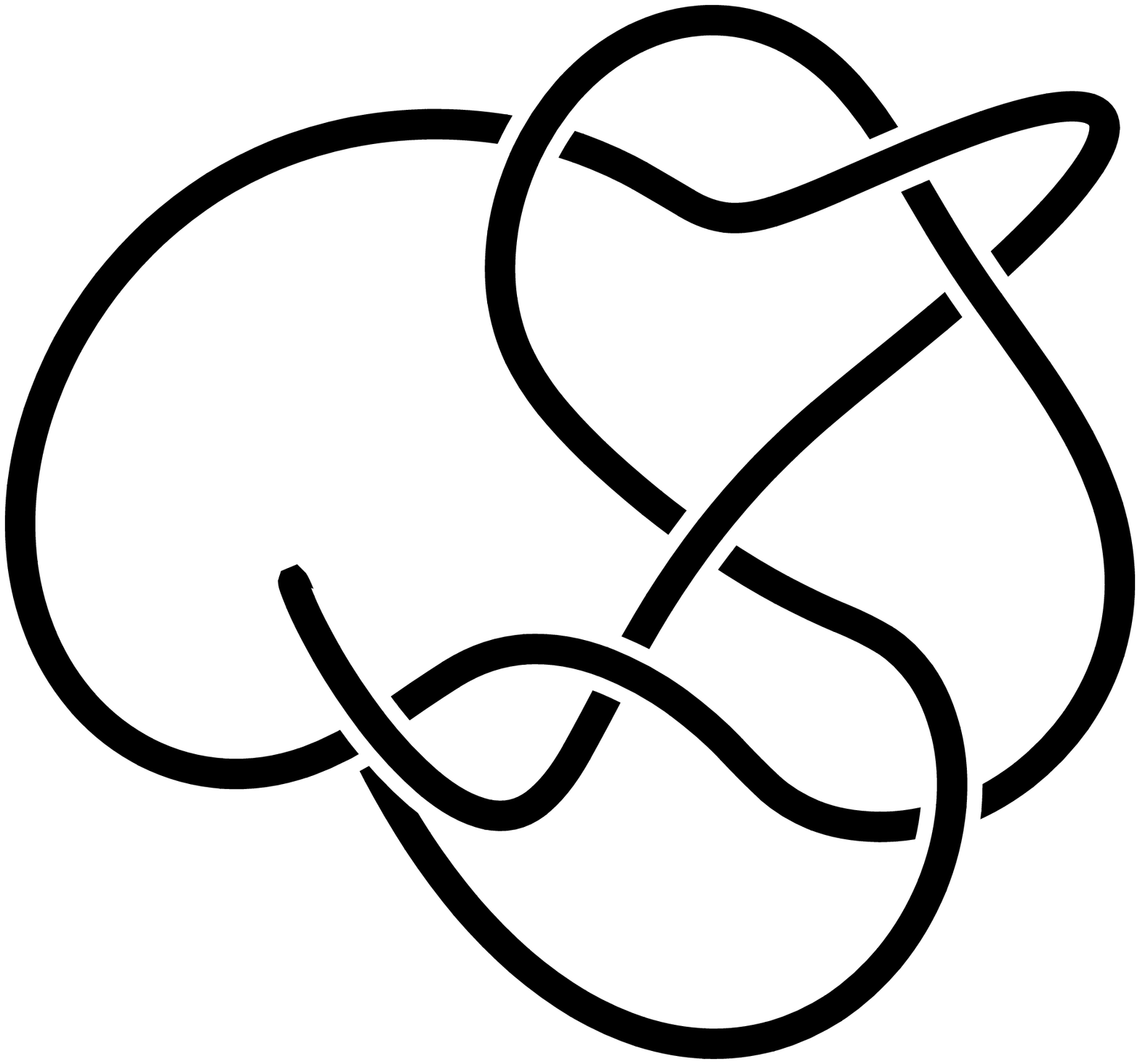}
$8_{9}$&
\includegraphics[keepaspectratio=1,height=2cm]{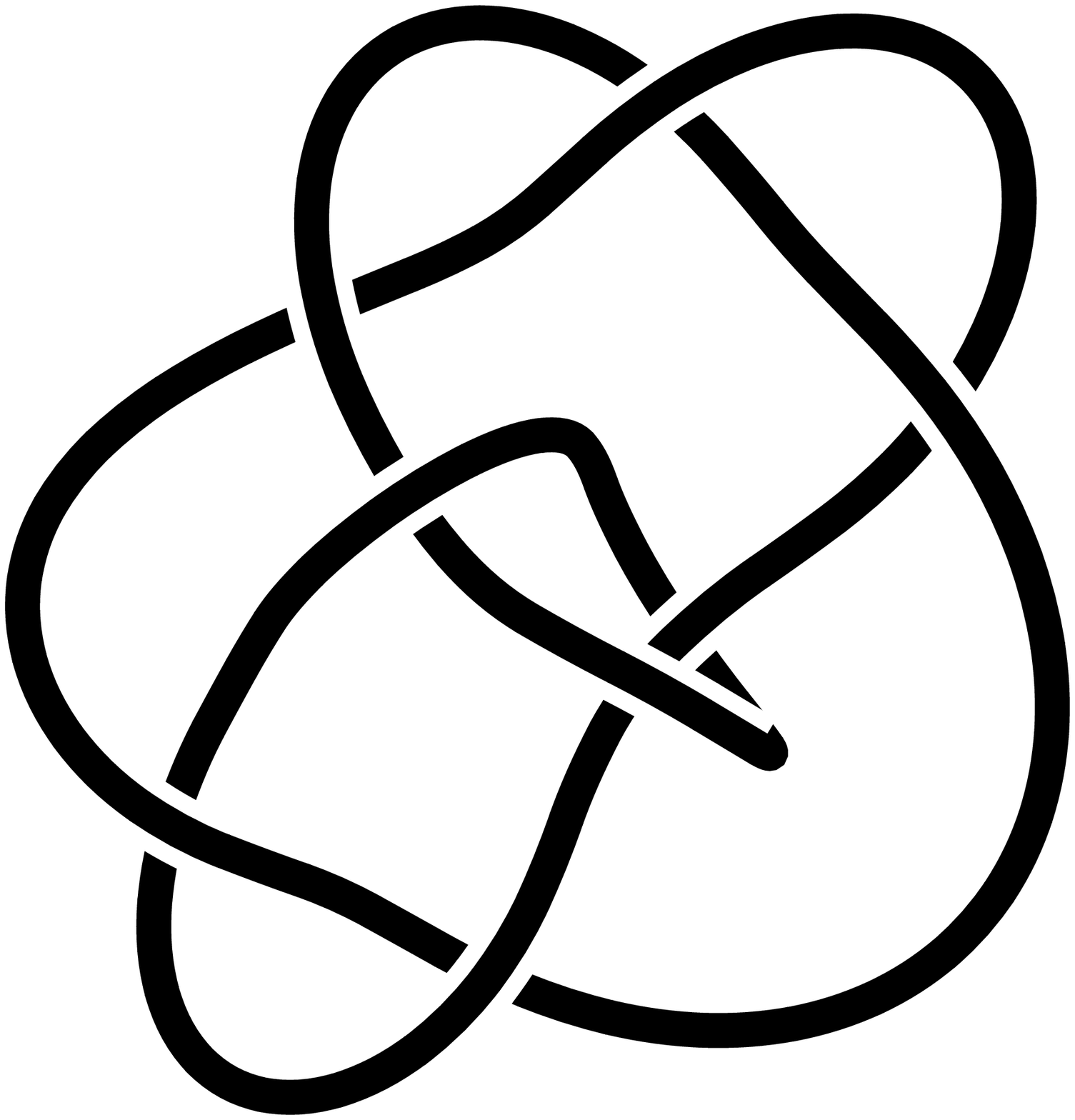}
$8_{10}$&
\includegraphics[keepaspectratio=1,height=2cm]{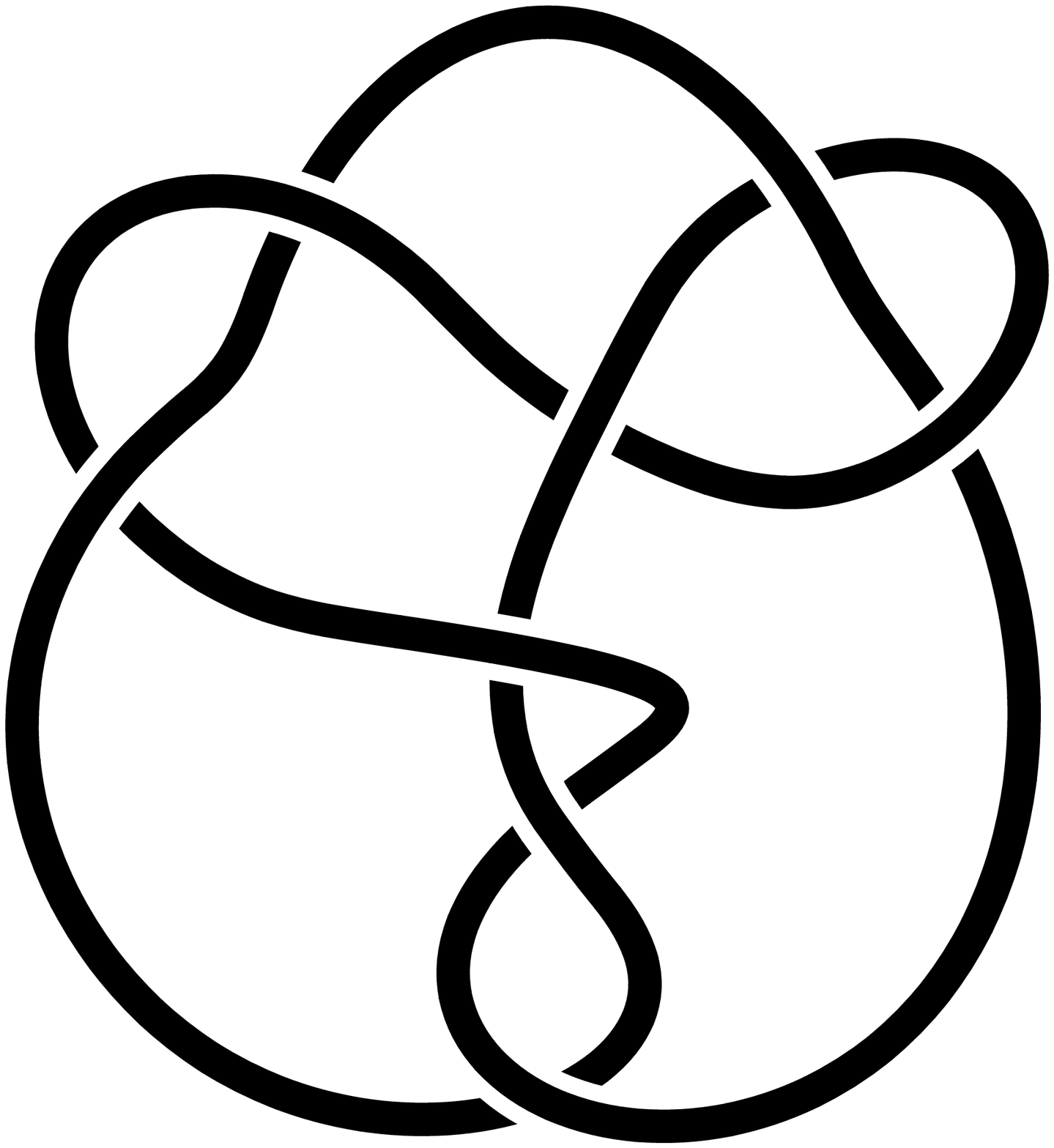}
$8_{11}$\\ \ &&&&\\
\includegraphics[keepaspectratio=1,height=2cm]{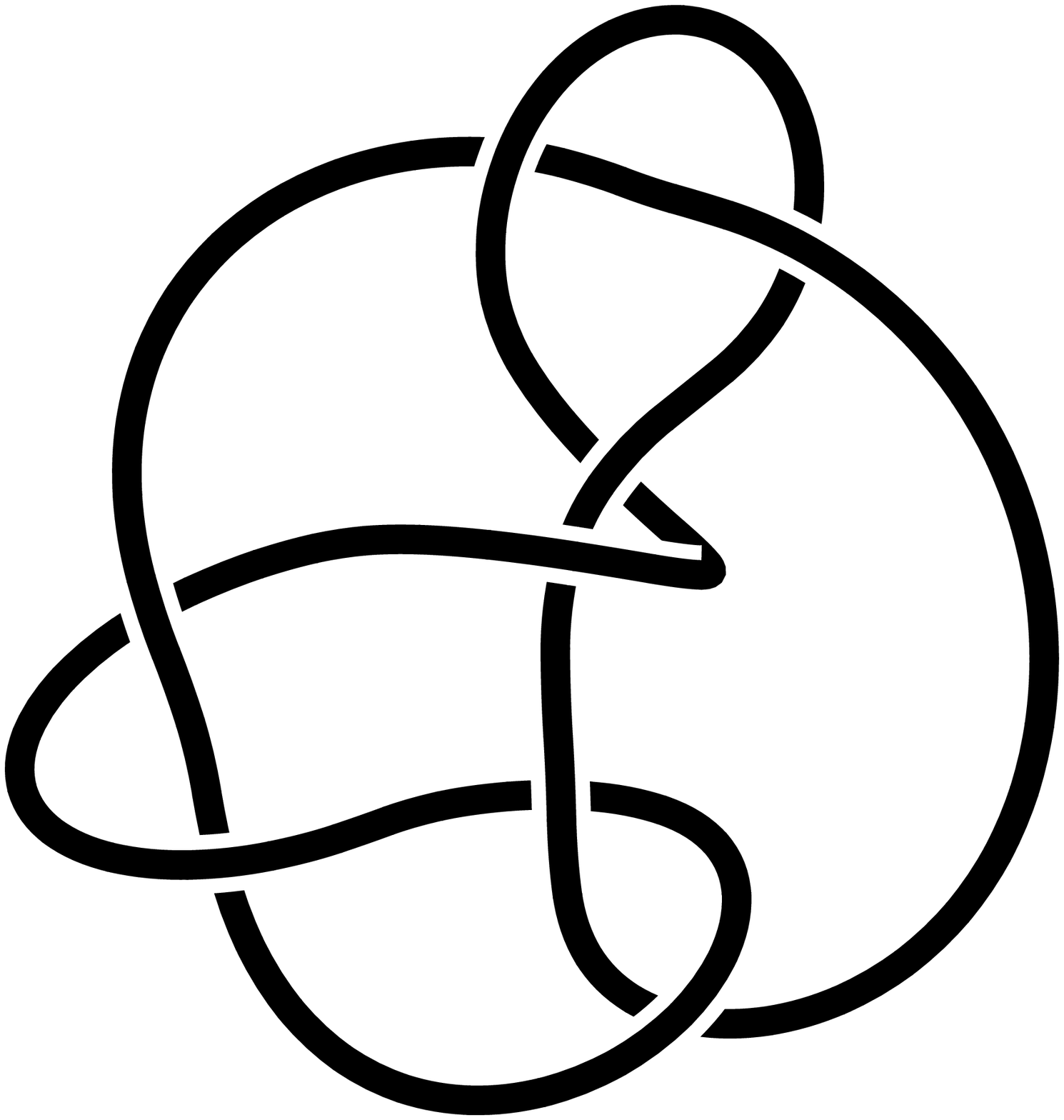}
$8_{12}$&
\includegraphics[keepaspectratio=1,height=2cm]{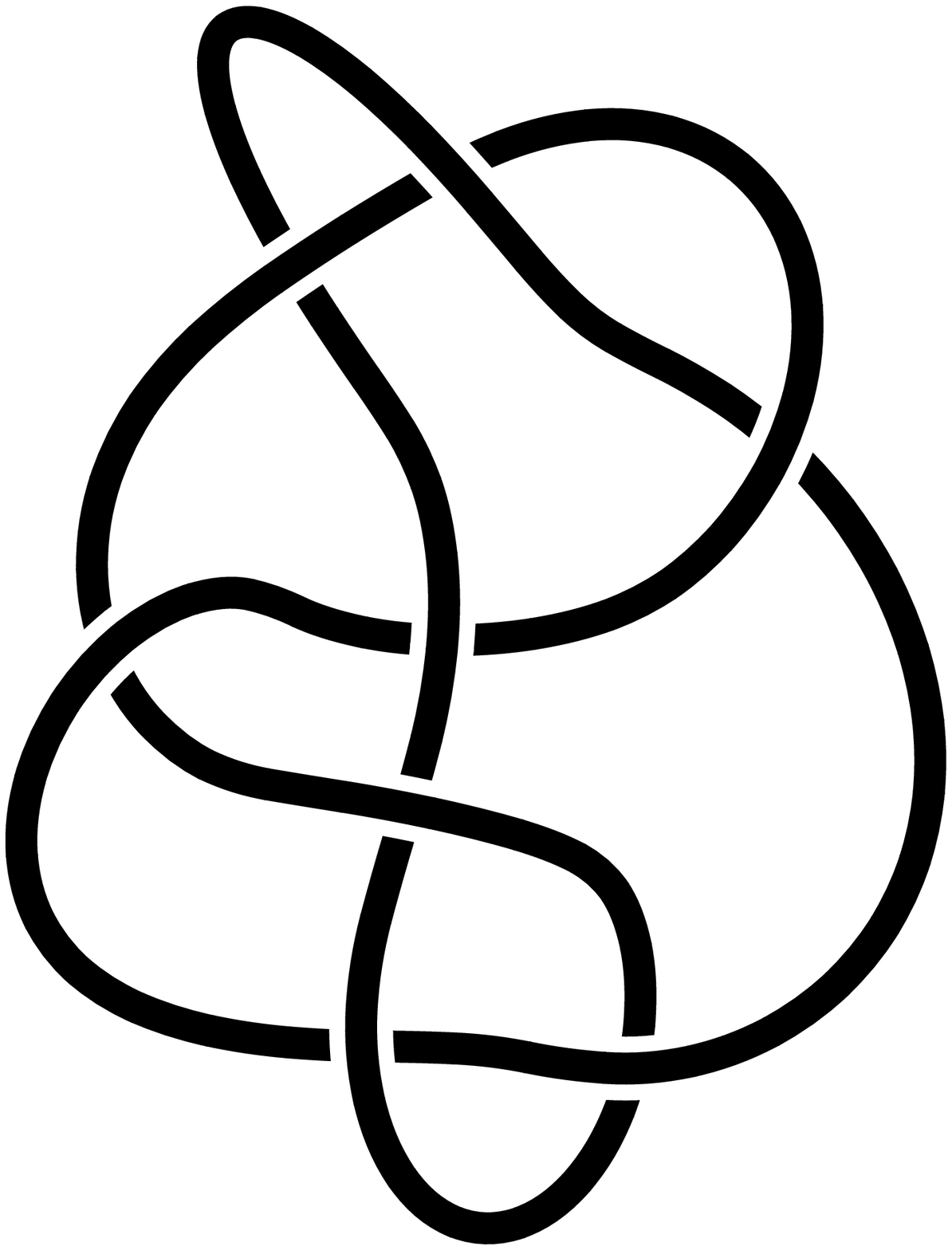}
$8_{13}$&
\includegraphics[keepaspectratio=1,height=2cm]{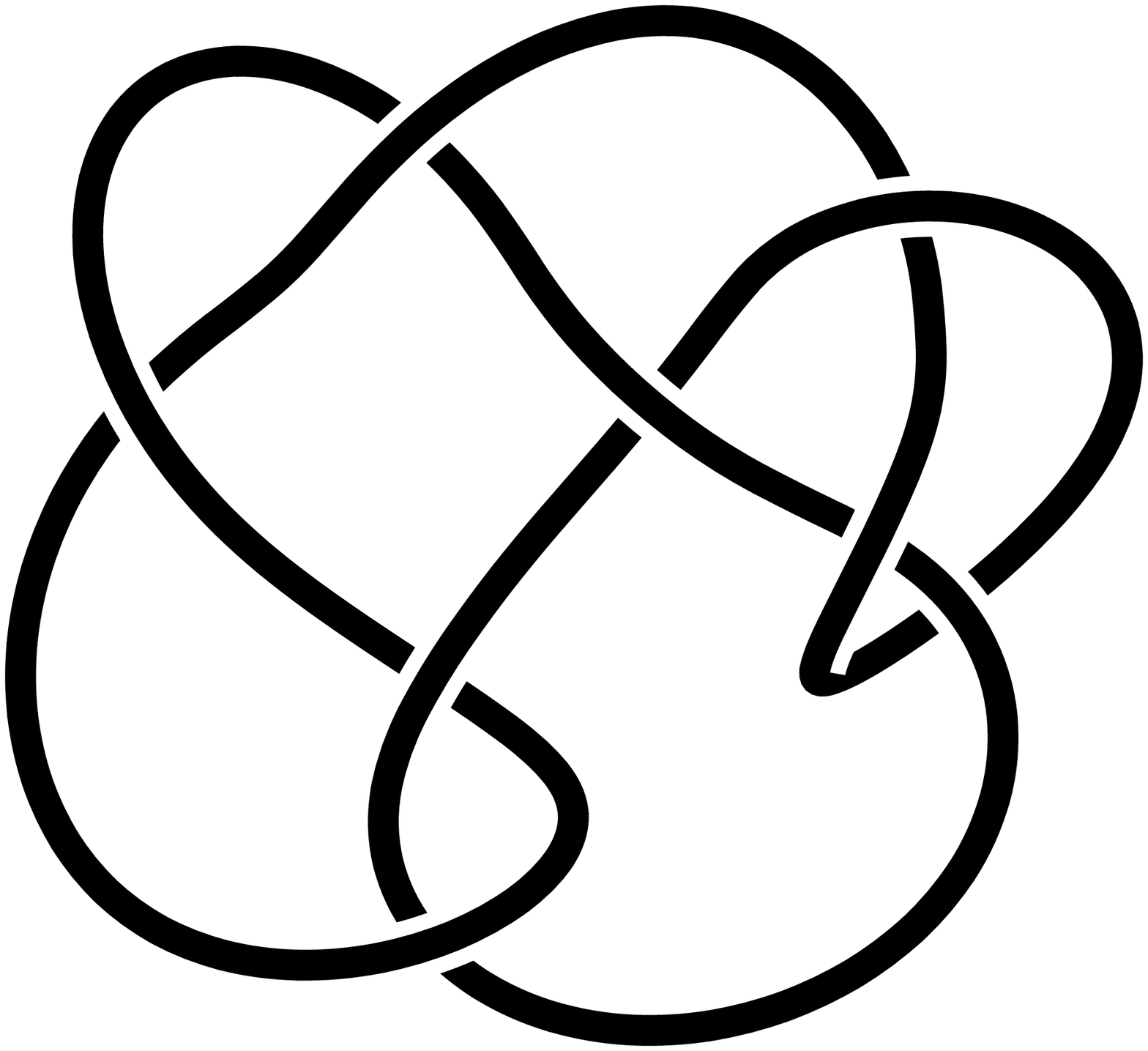}
$8_{14}$&
\includegraphics[keepaspectratio=1,height=2cm]{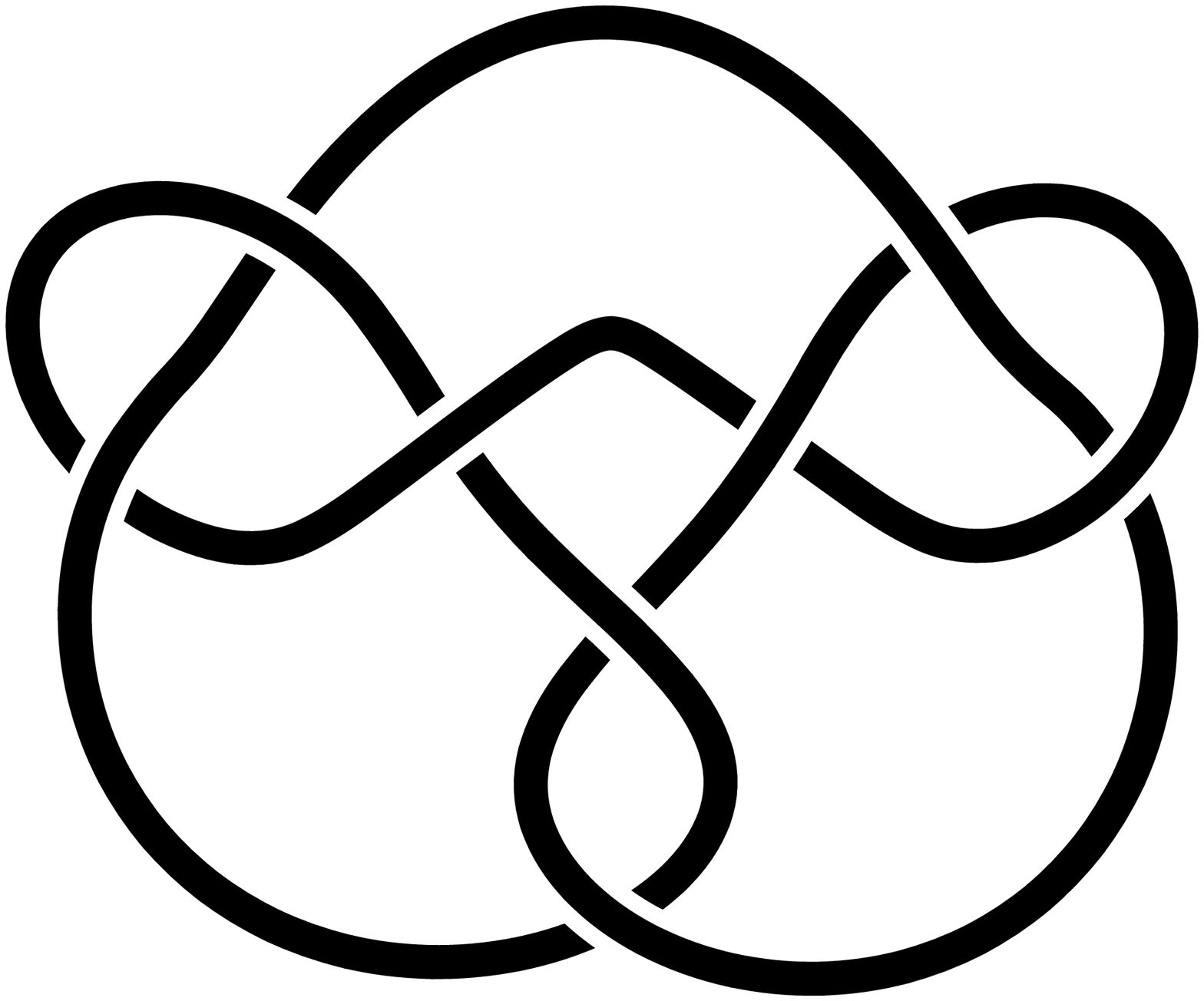}
$8_{15}$&
\includegraphics[keepaspectratio=1,height=2cm]{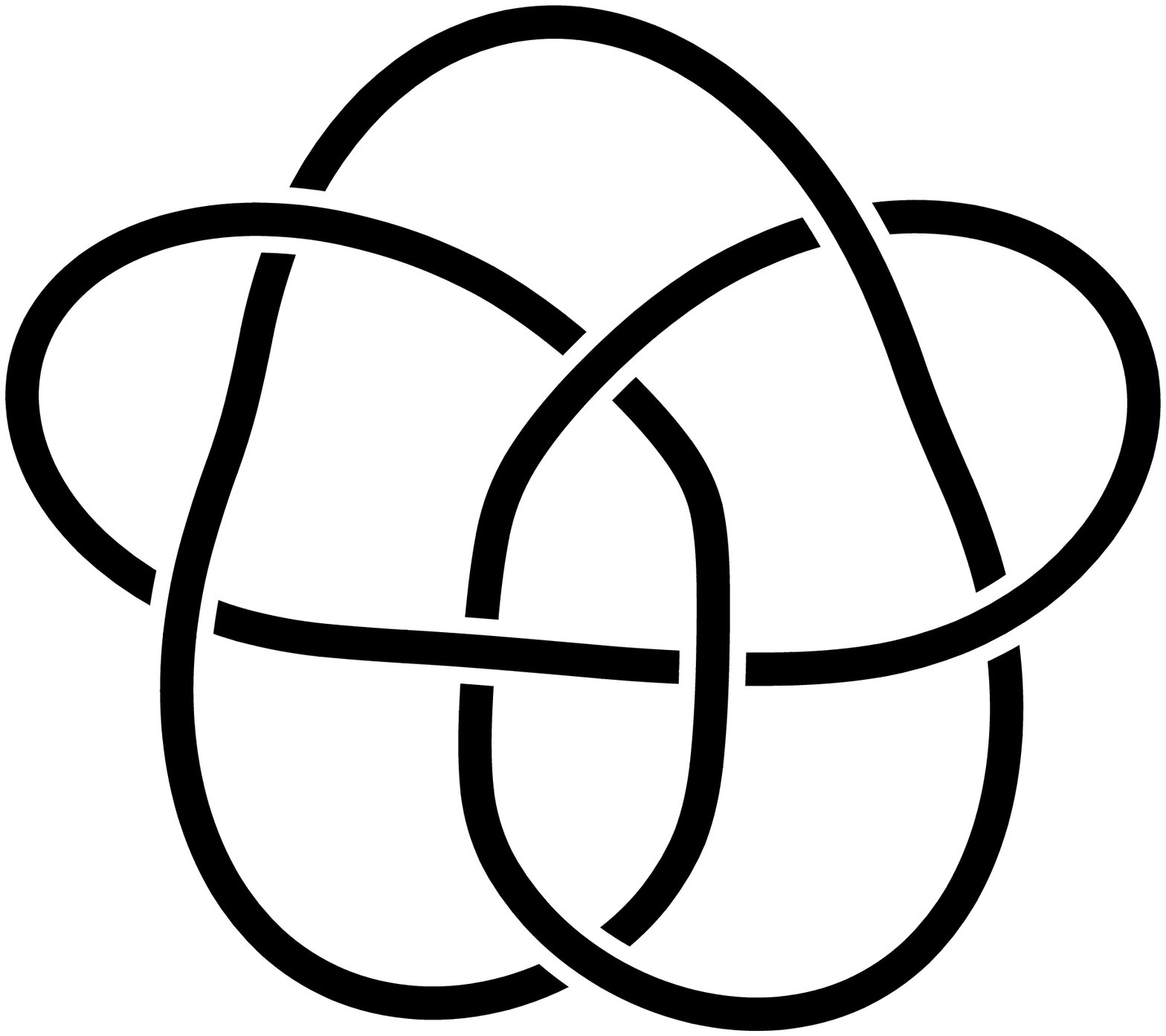}
$8_{16}$\\ \ &&&&\\
\includegraphics[keepaspectratio=1,height=2cm]{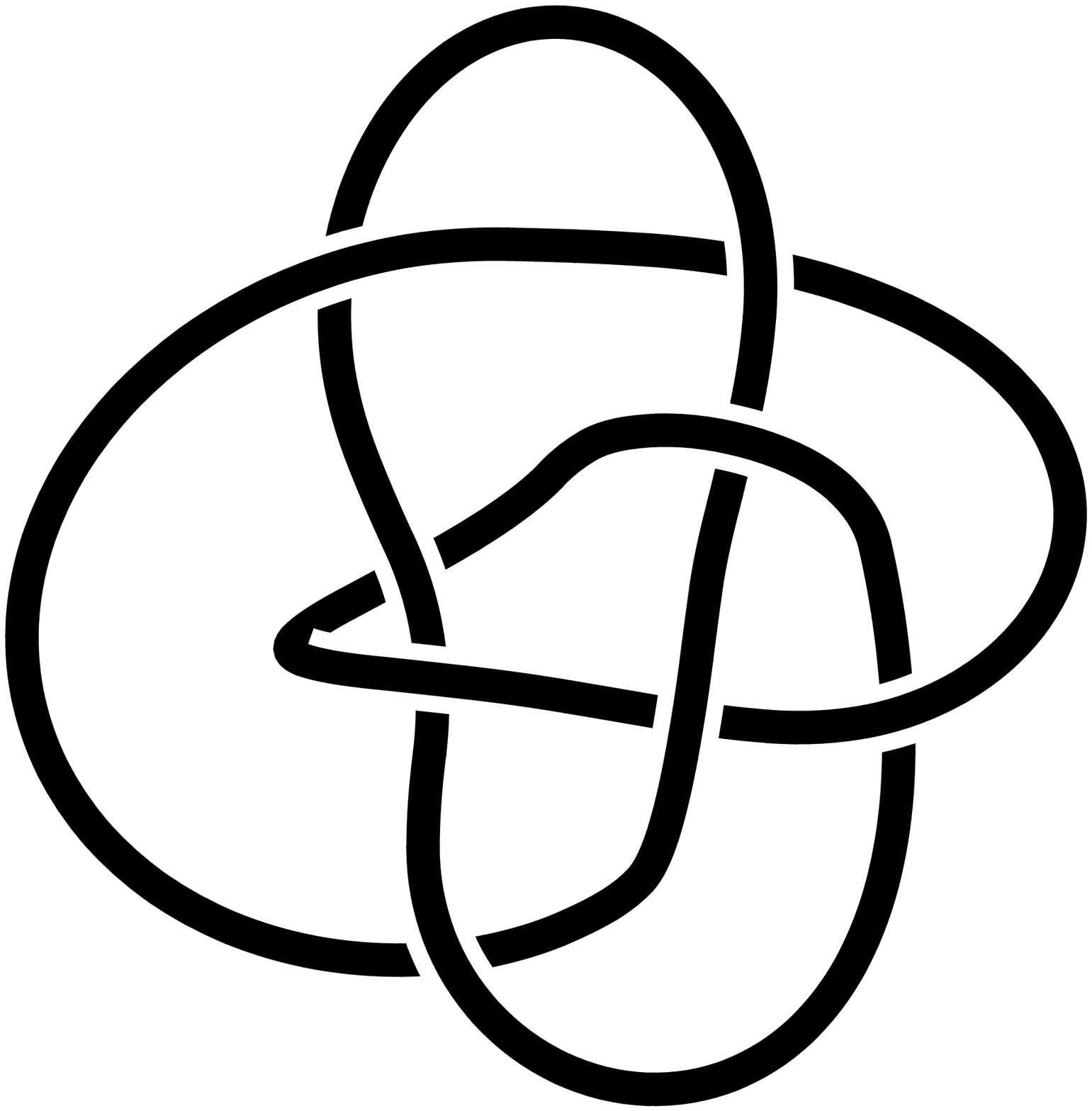}
$8_{17}$&
\includegraphics[keepaspectratio=1,height=2cm]{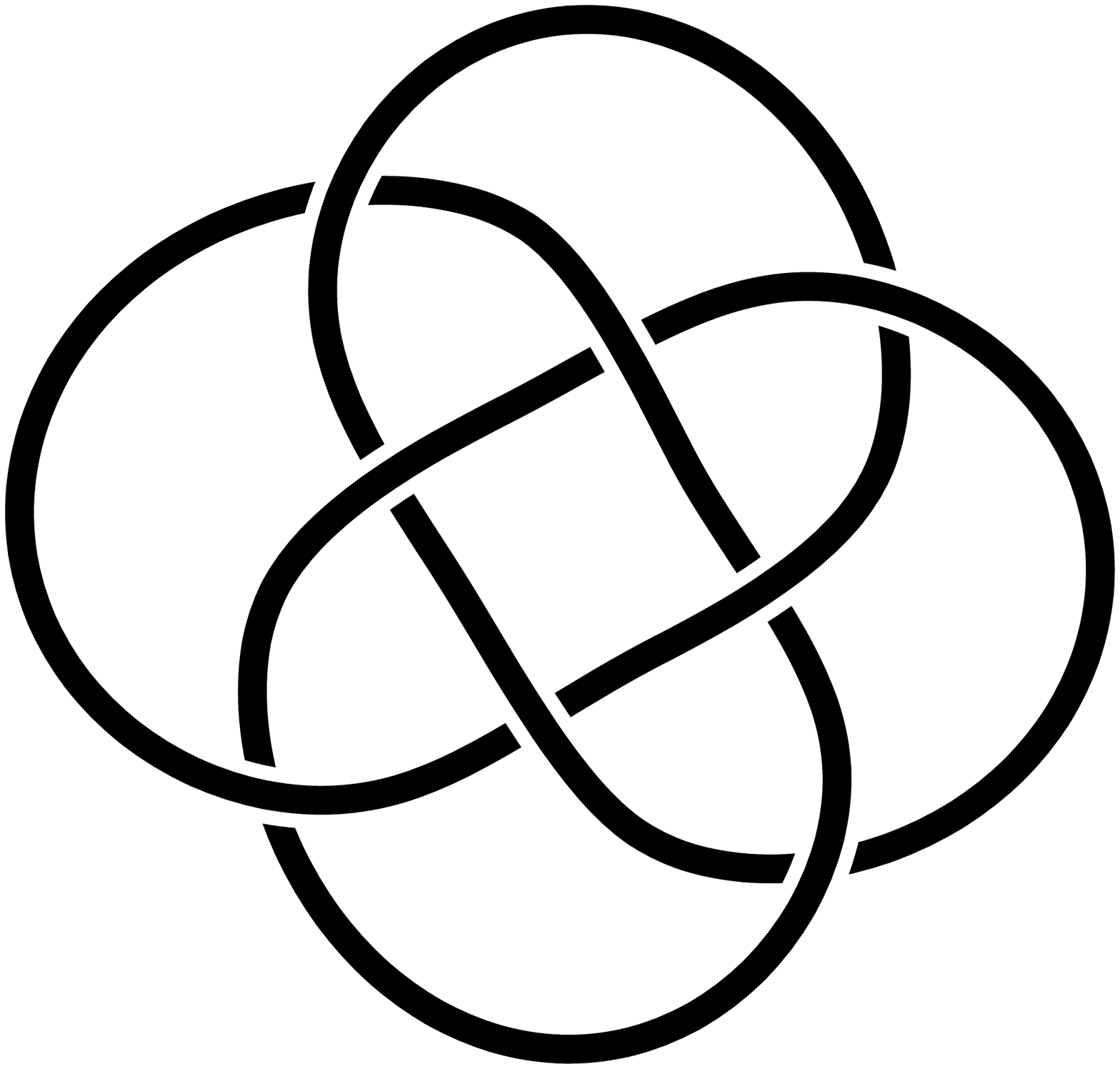}
$8_{18}$&
\includegraphics[keepaspectratio=1,height=2cm]{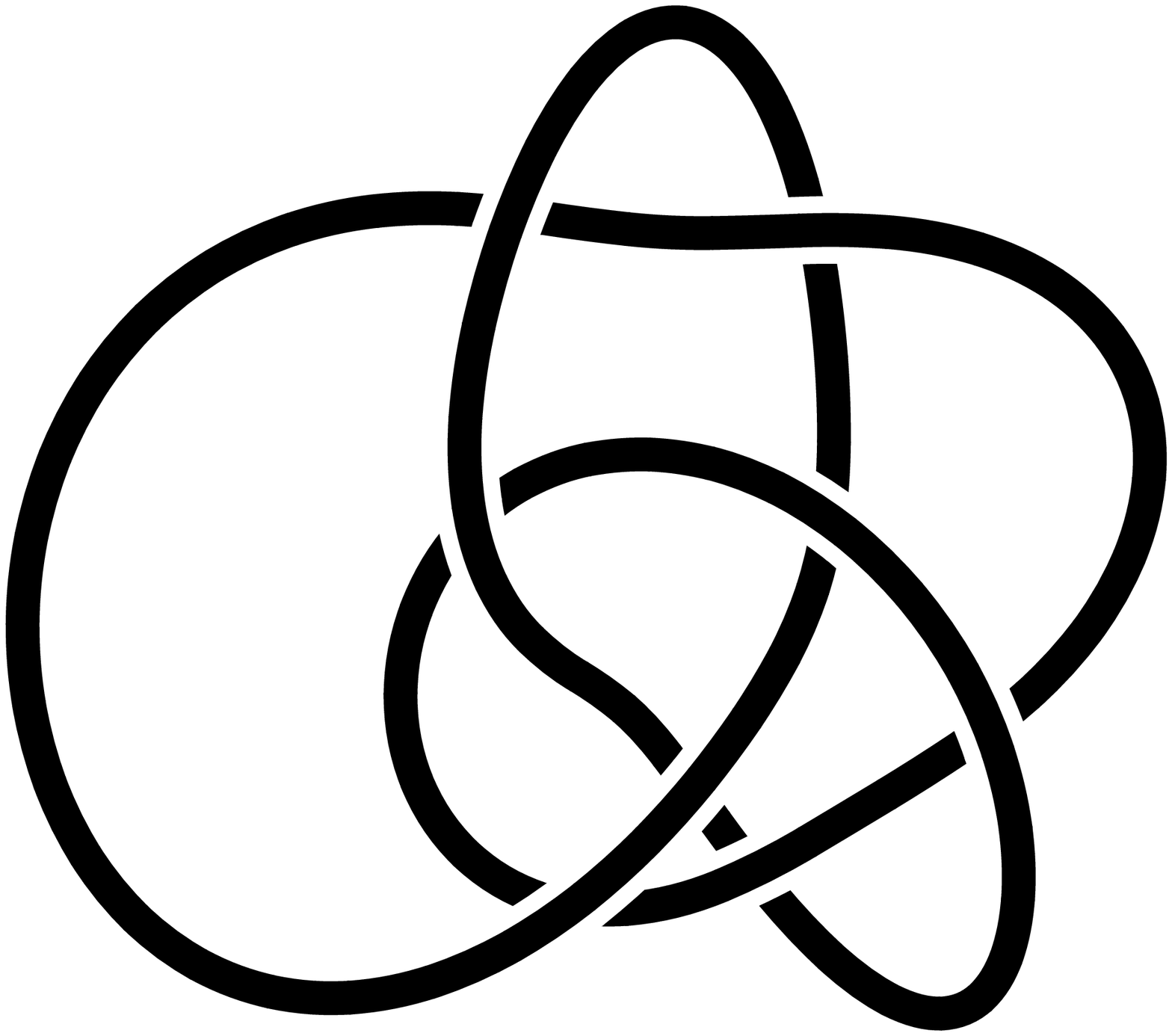}
$8_{19}$&
\includegraphics[keepaspectratio=1,height=2cm]{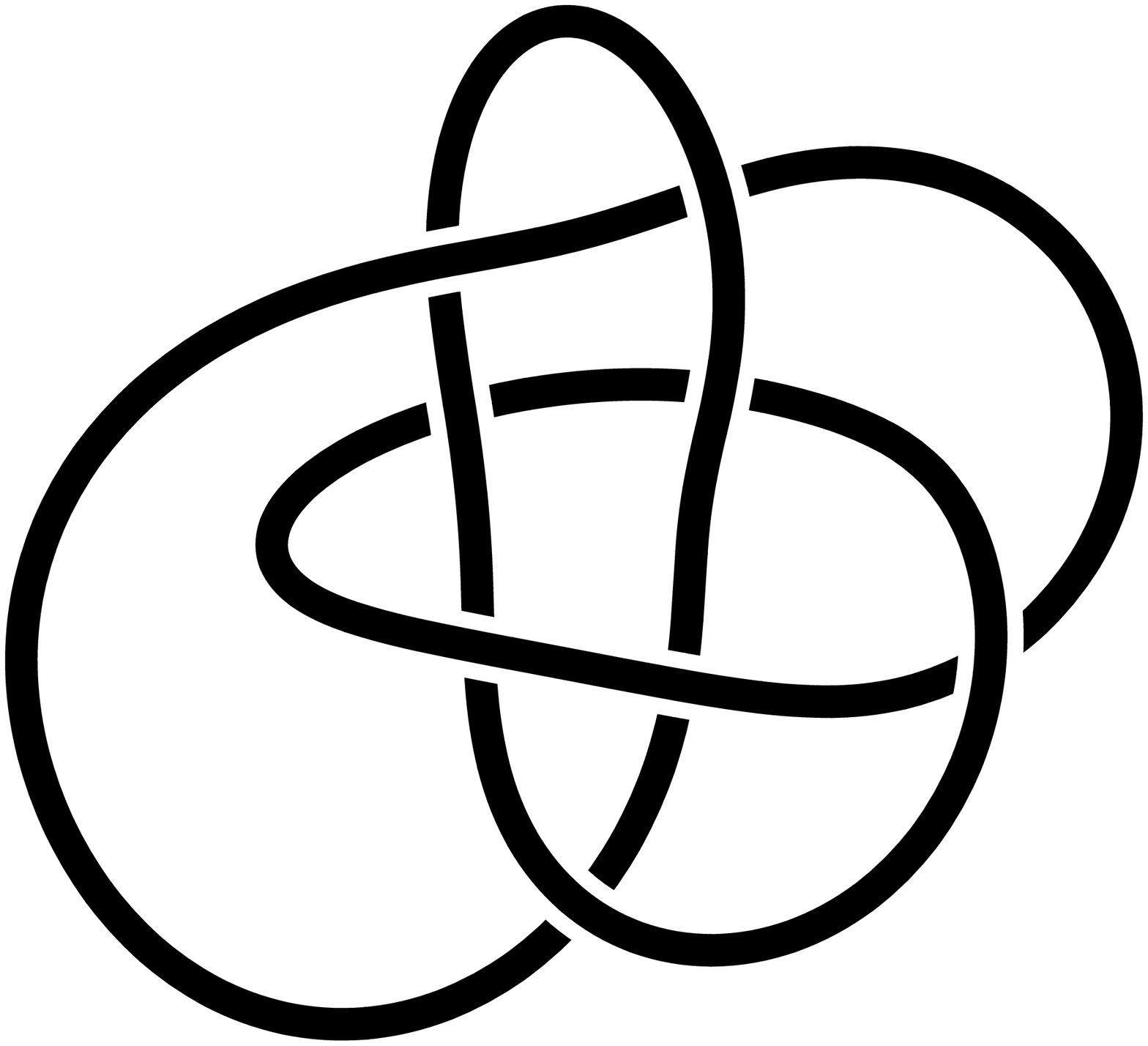}
$8_{20}$&
\includegraphics[keepaspectratio=1,height=2cm]{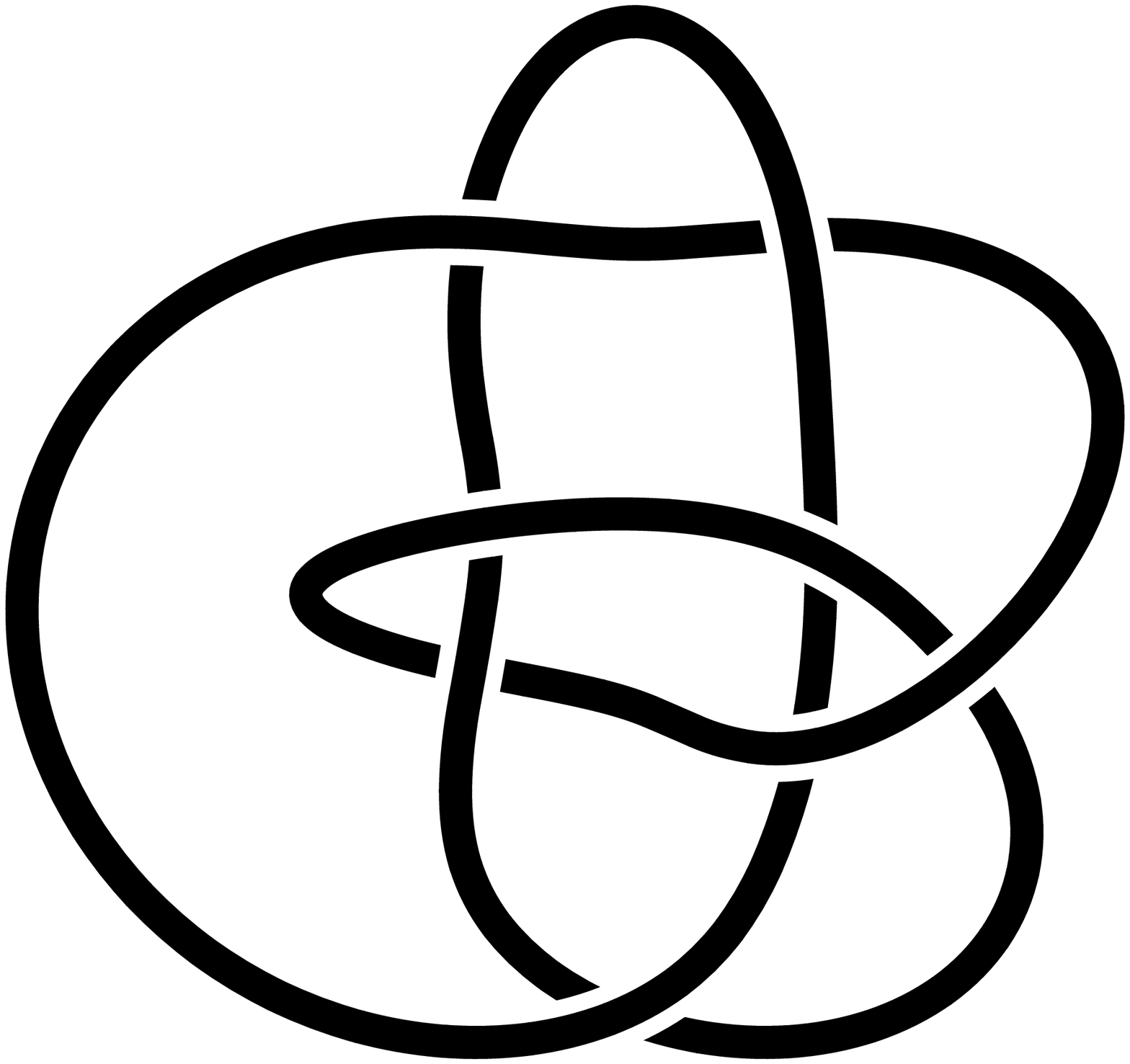}
$8_{21}$\\ \ &&&&\\
\includegraphics[keepaspectratio=1,height=2cm]{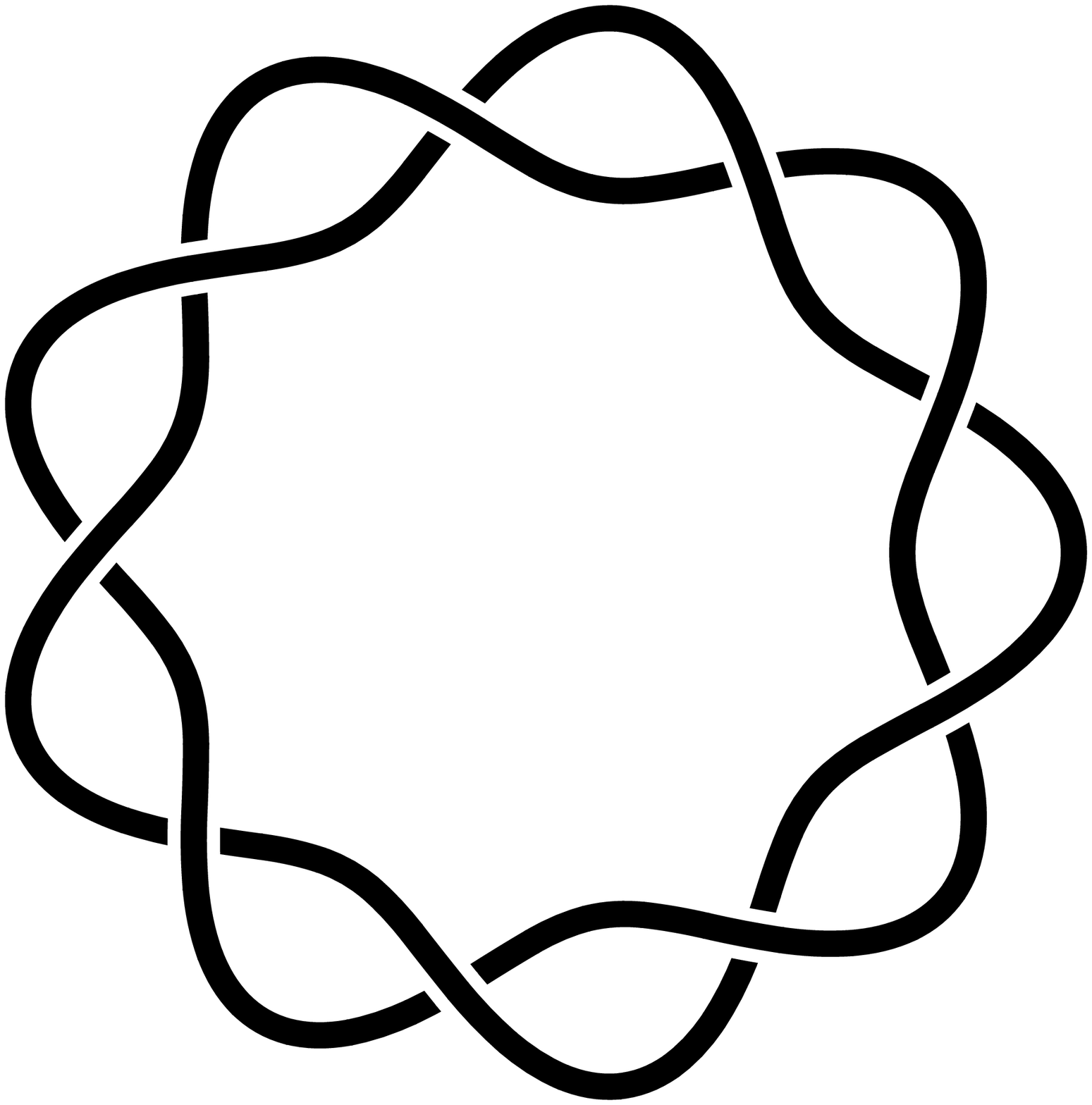}
$9_{1}$&
\includegraphics[keepaspectratio=1,height=2cm]{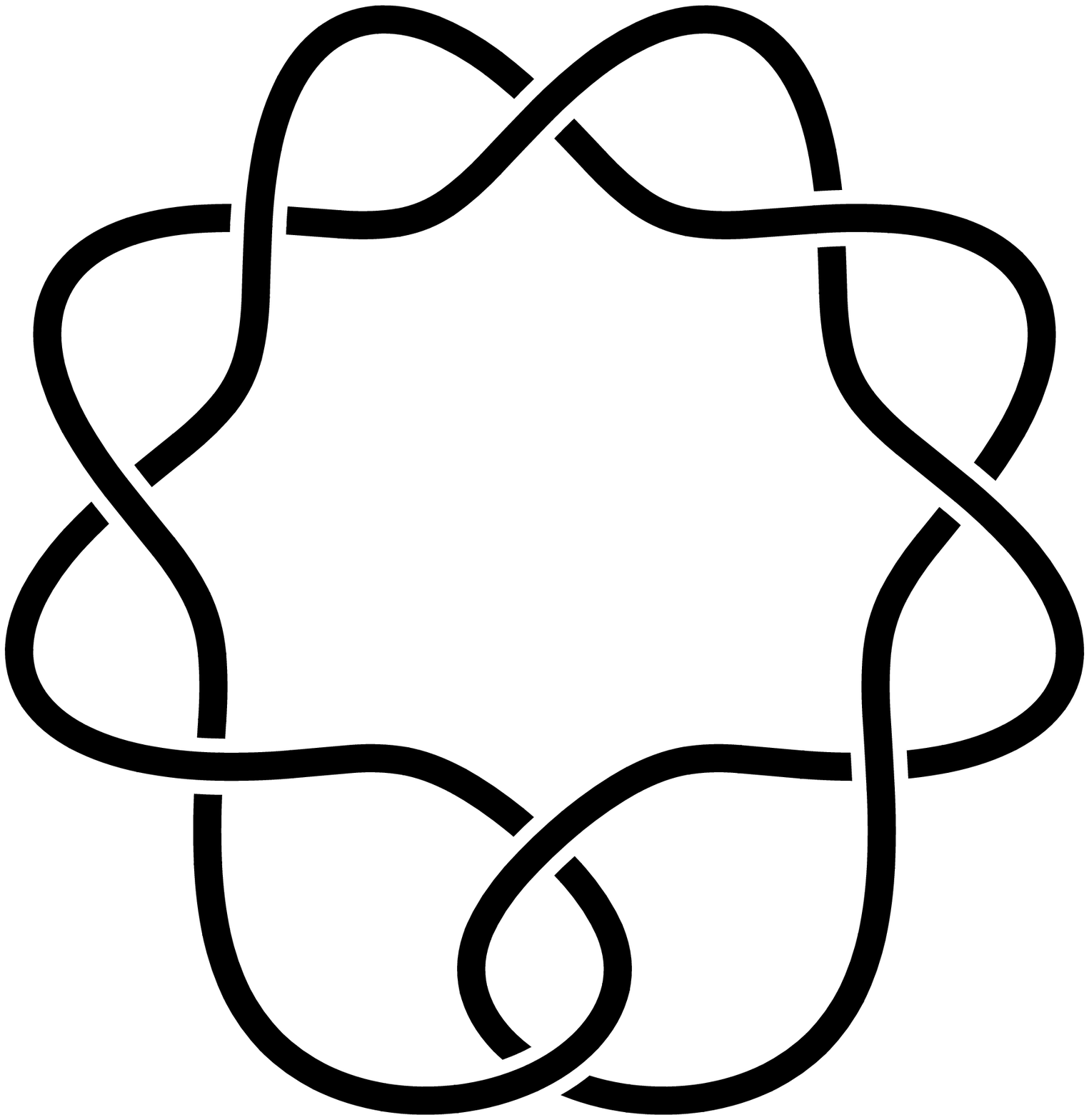}
$9_{2}$&
\includegraphics[keepaspectratio=1,height=2cm]{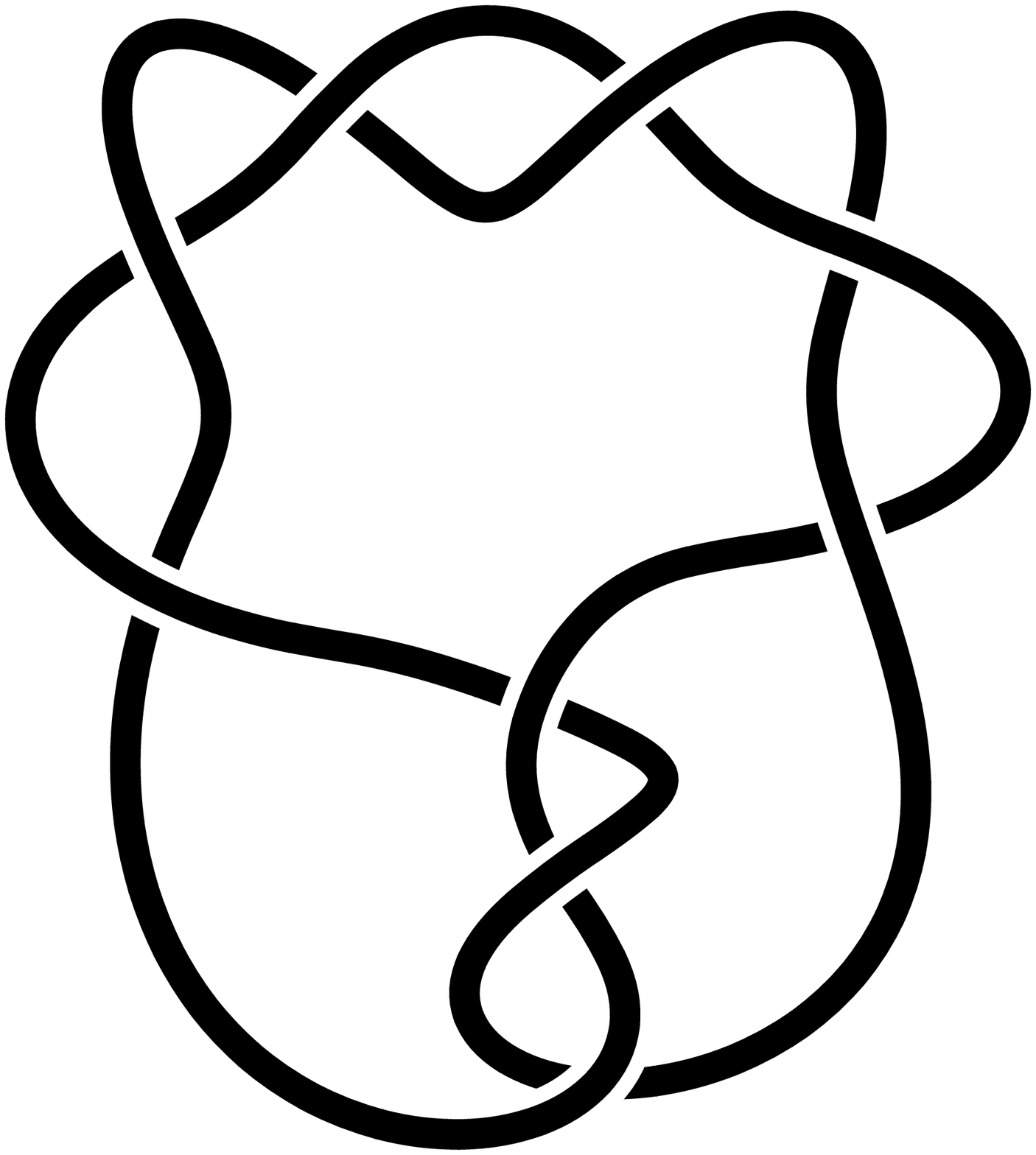}
$9_{3}$&
\includegraphics[keepaspectratio=1,height=2cm]{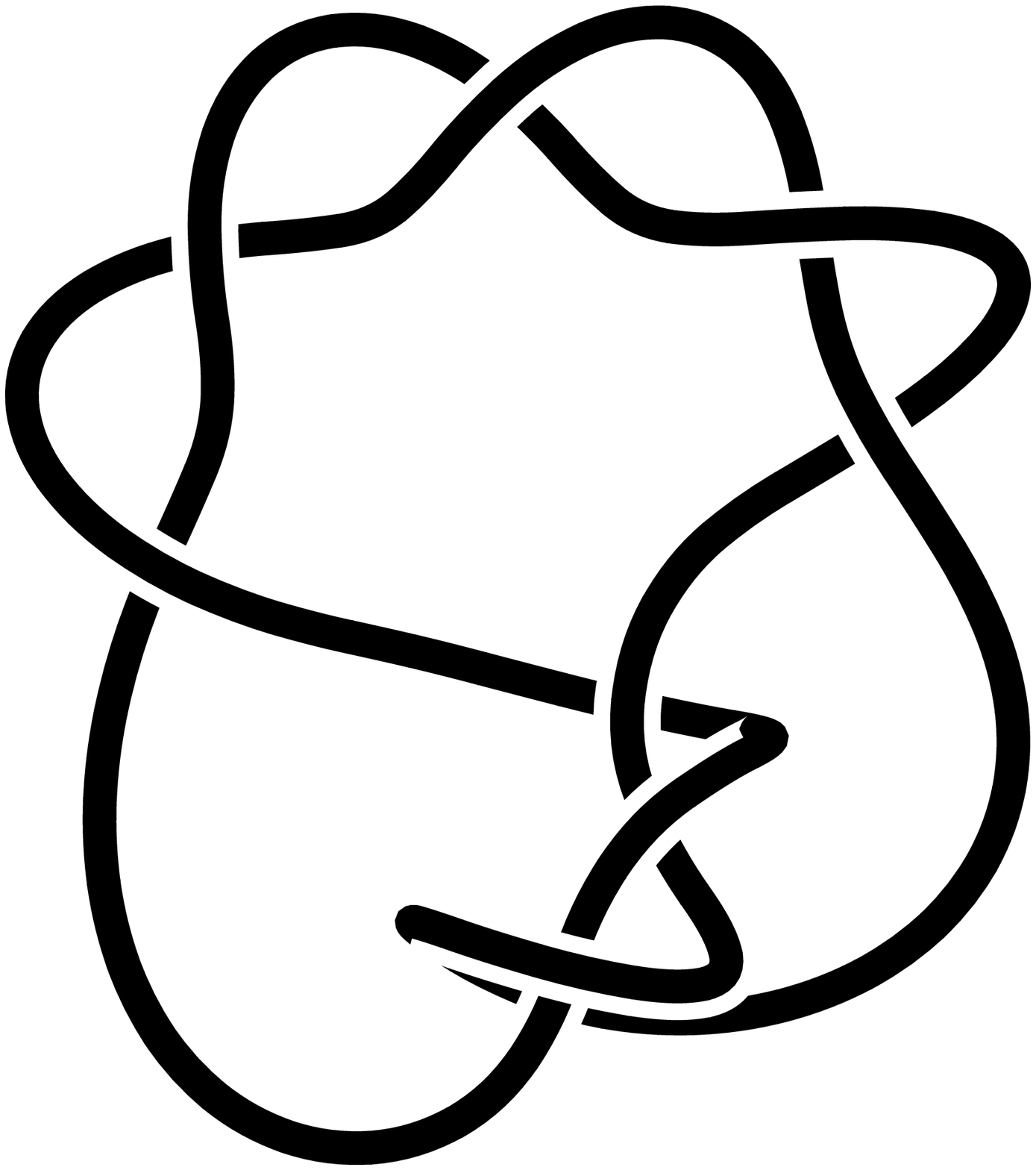}
$9_{4}$&
\includegraphics[keepaspectratio=1,height=2cm]{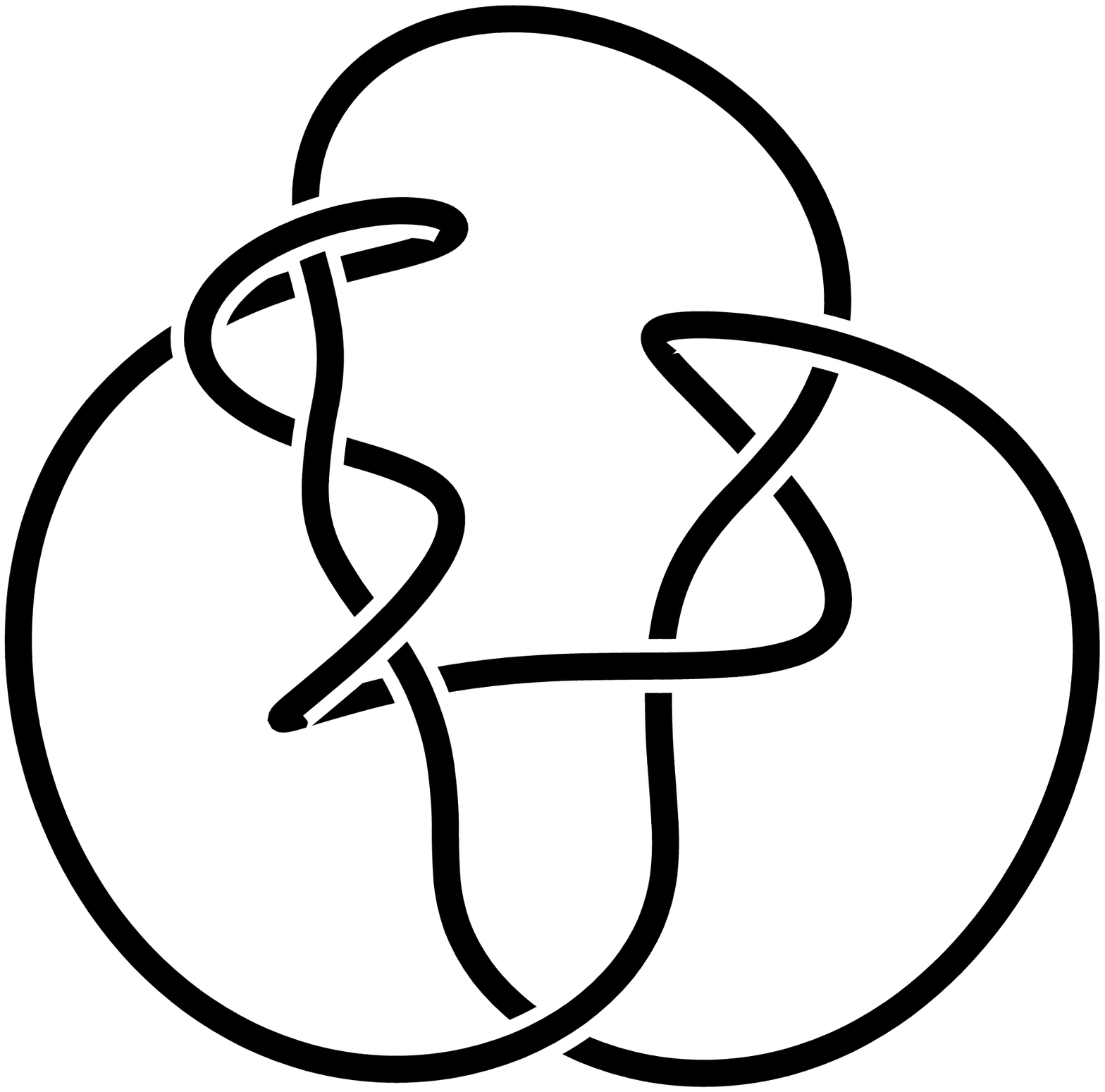}
$9_{5}$
\end{tabular}

\begin{tabular}{ccccc}
\includegraphics[keepaspectratio=1,height=2cm]{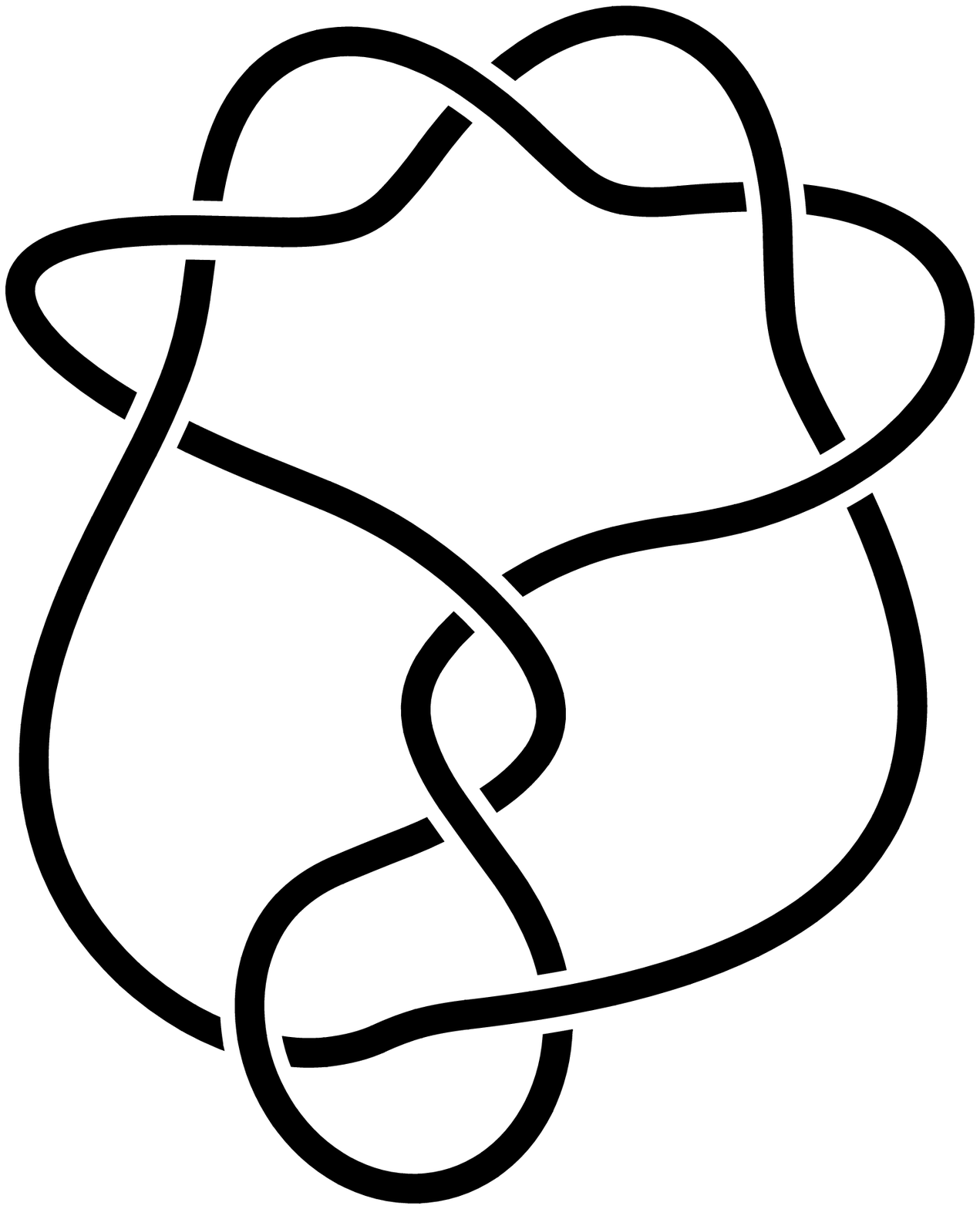}
$9_{6}$&
\includegraphics[keepaspectratio=1,height=2cm]{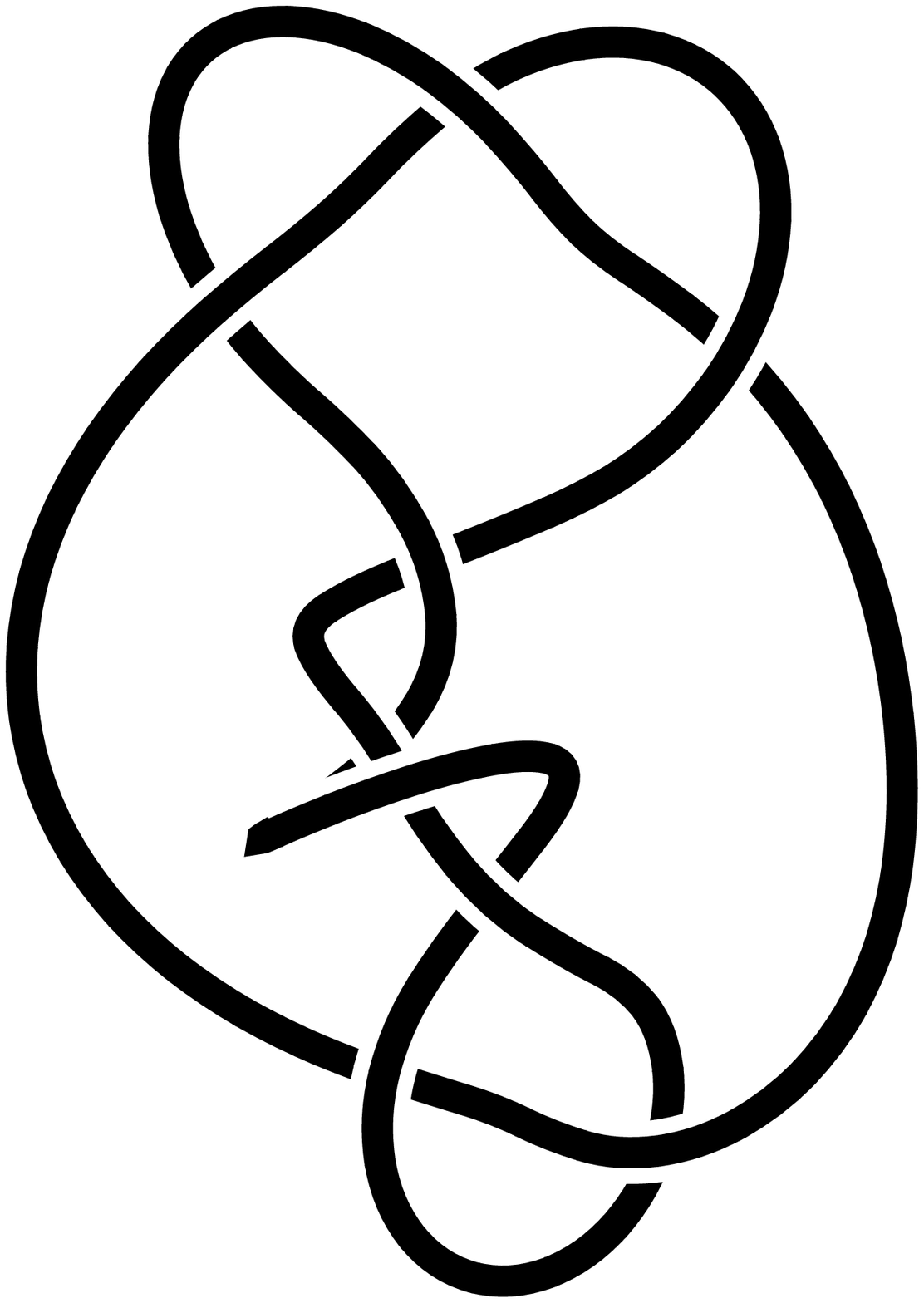}
$9_{7}$&
\includegraphics[keepaspectratio=1,height=2cm]{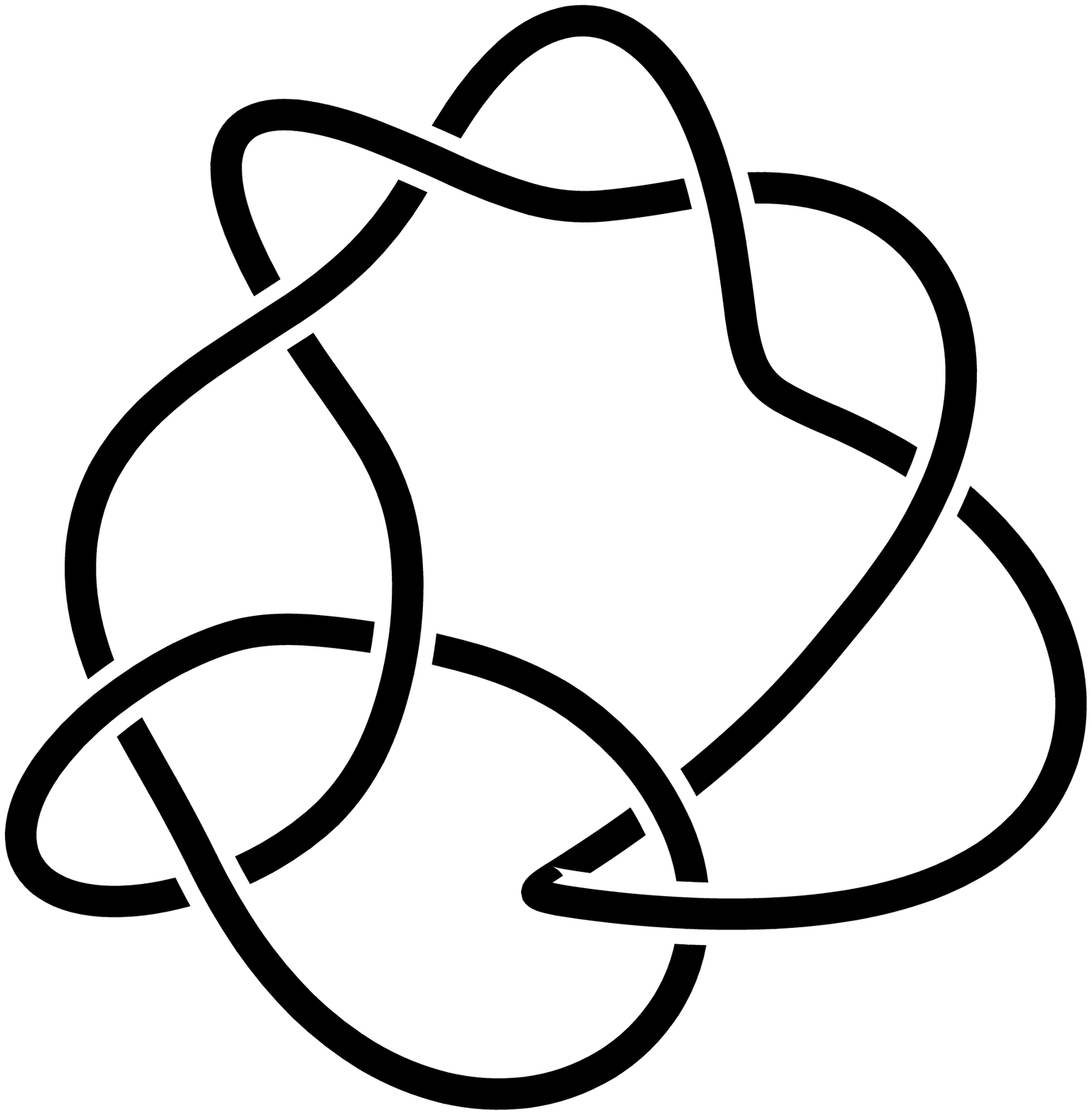}
$9_{8}$&
\includegraphics[keepaspectratio=1,height=2cm]{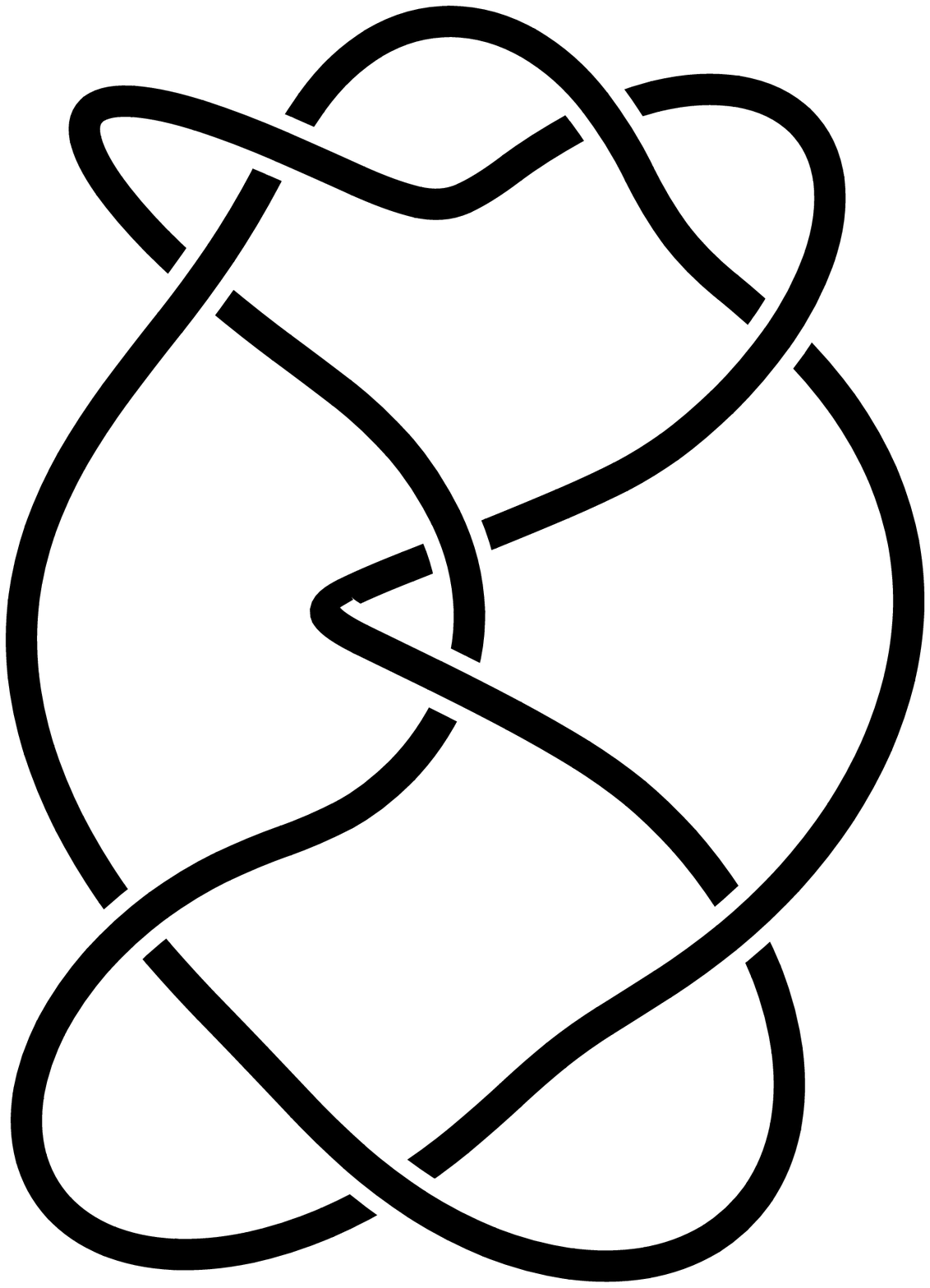}
$9_{9}$&
\includegraphics[keepaspectratio=1,height=2cm]{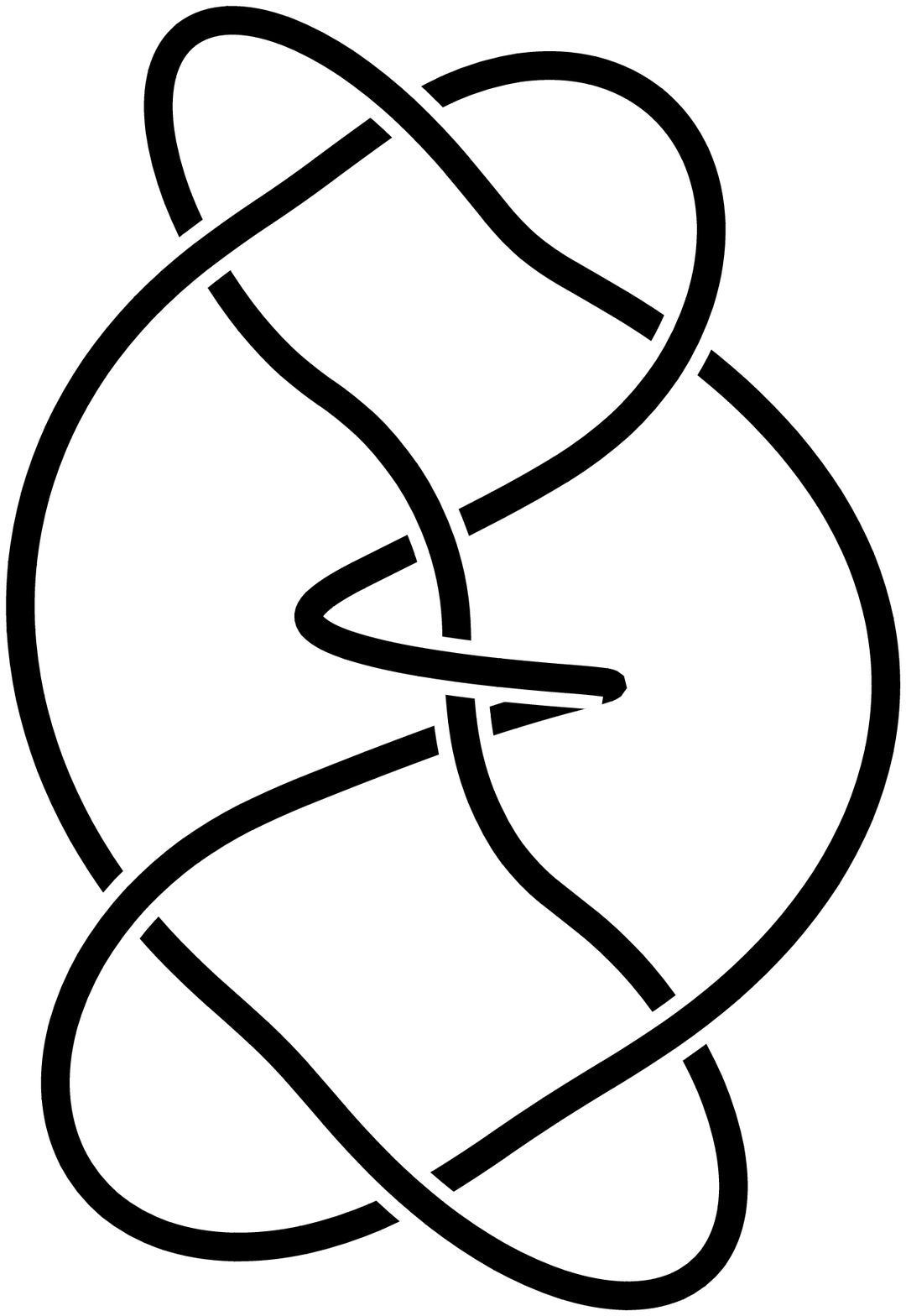}
$9_{10}$\\ \ &&&&\\
\includegraphics[keepaspectratio=1,height=2cm]{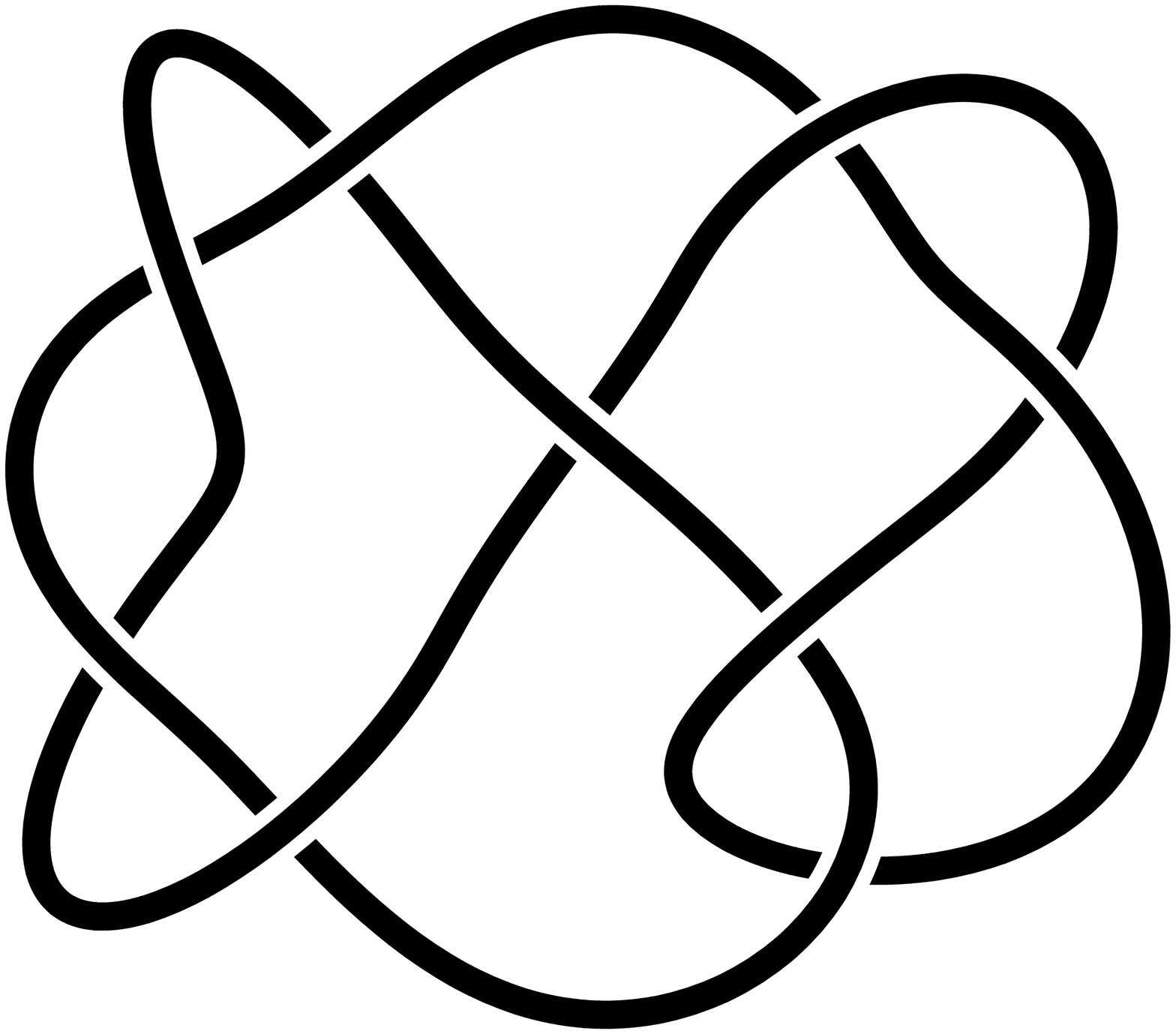}
$9_{11}$&
\includegraphics[keepaspectratio=1,height=2cm]{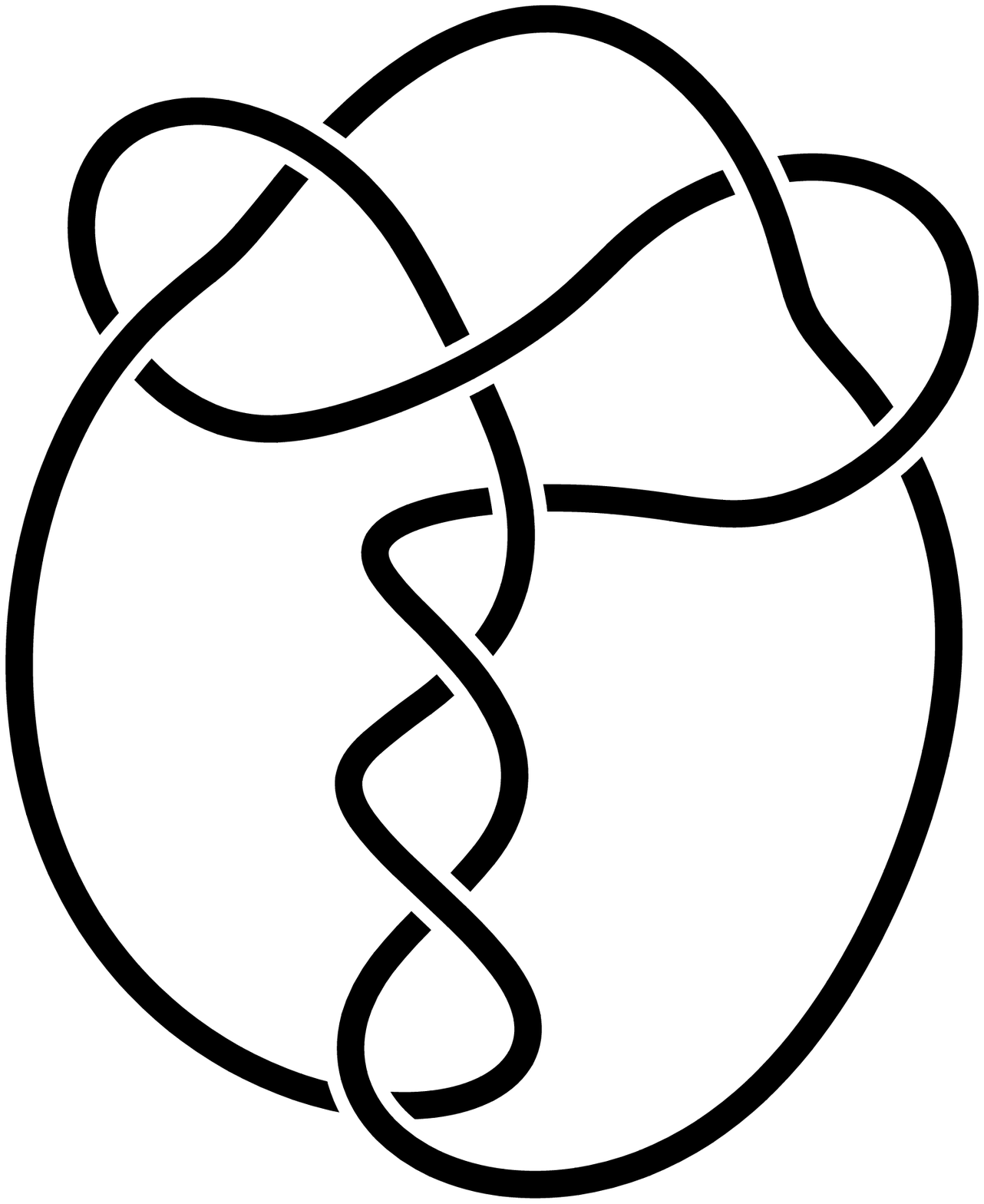}
$9_{12}$&
\includegraphics[keepaspectratio=1,height=2cm]{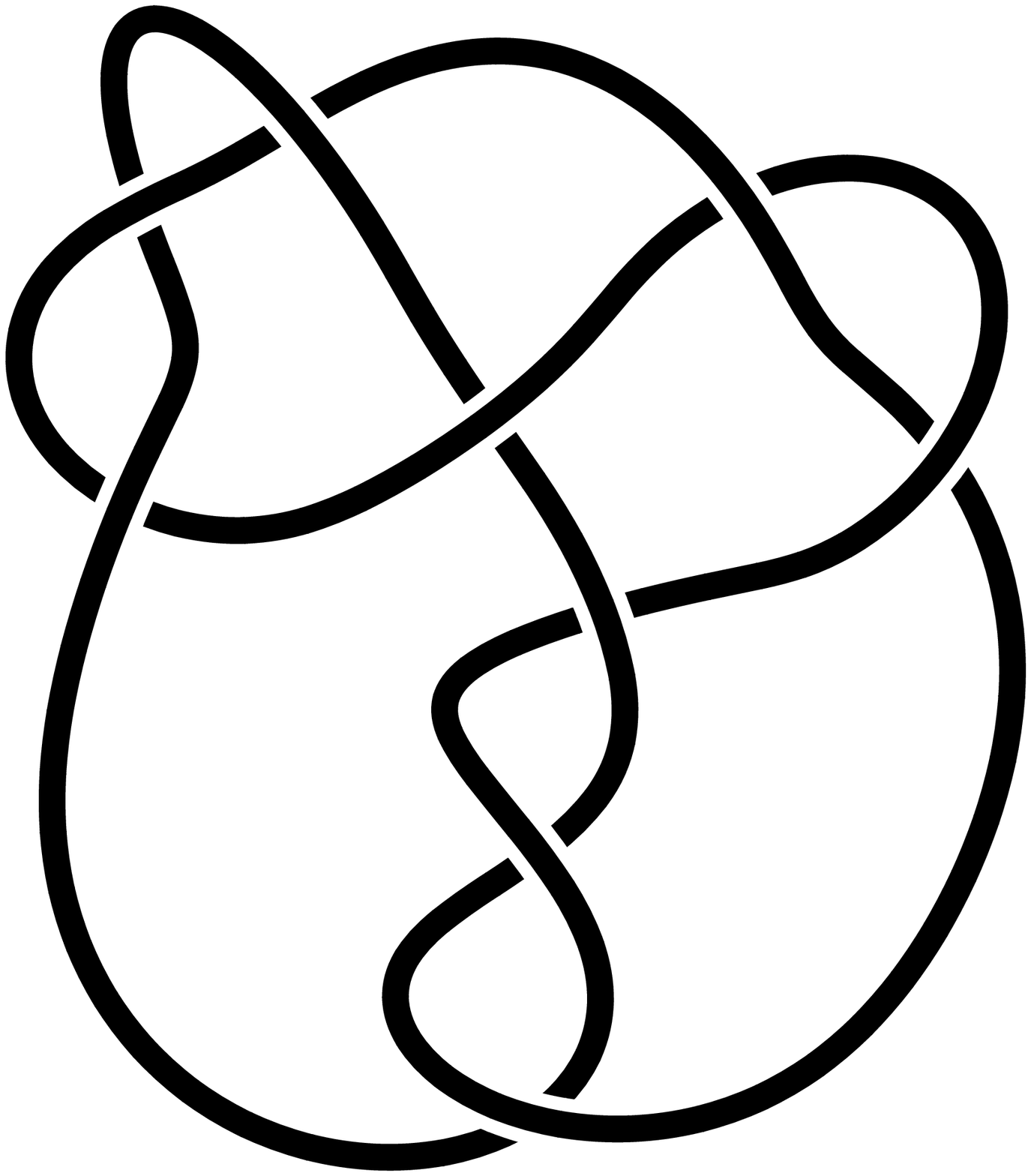}
$9_{13}$&
\includegraphics[keepaspectratio=1,height=2cm]{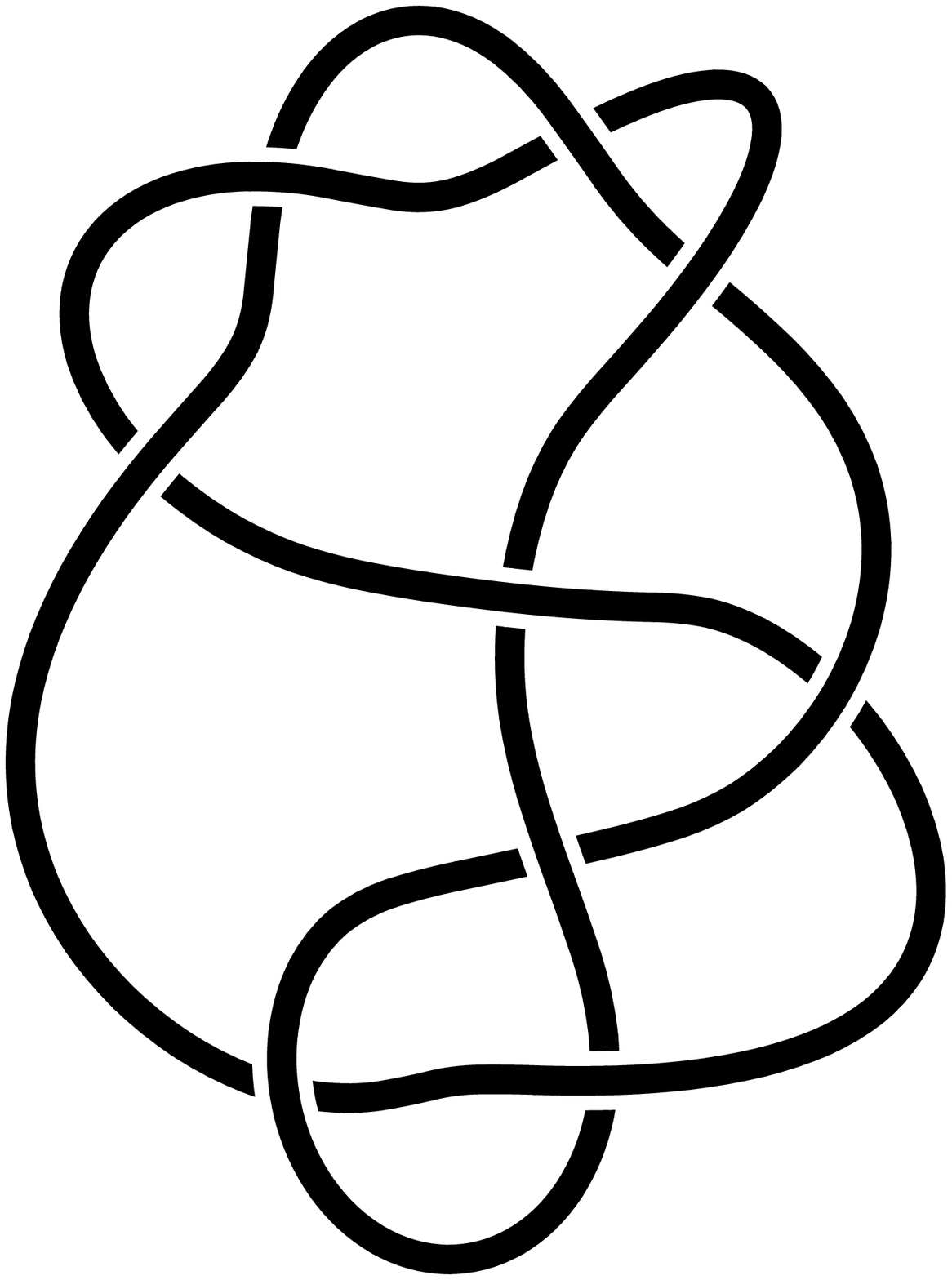}
$9_{14}$&
\includegraphics[keepaspectratio=1,height=2cm]{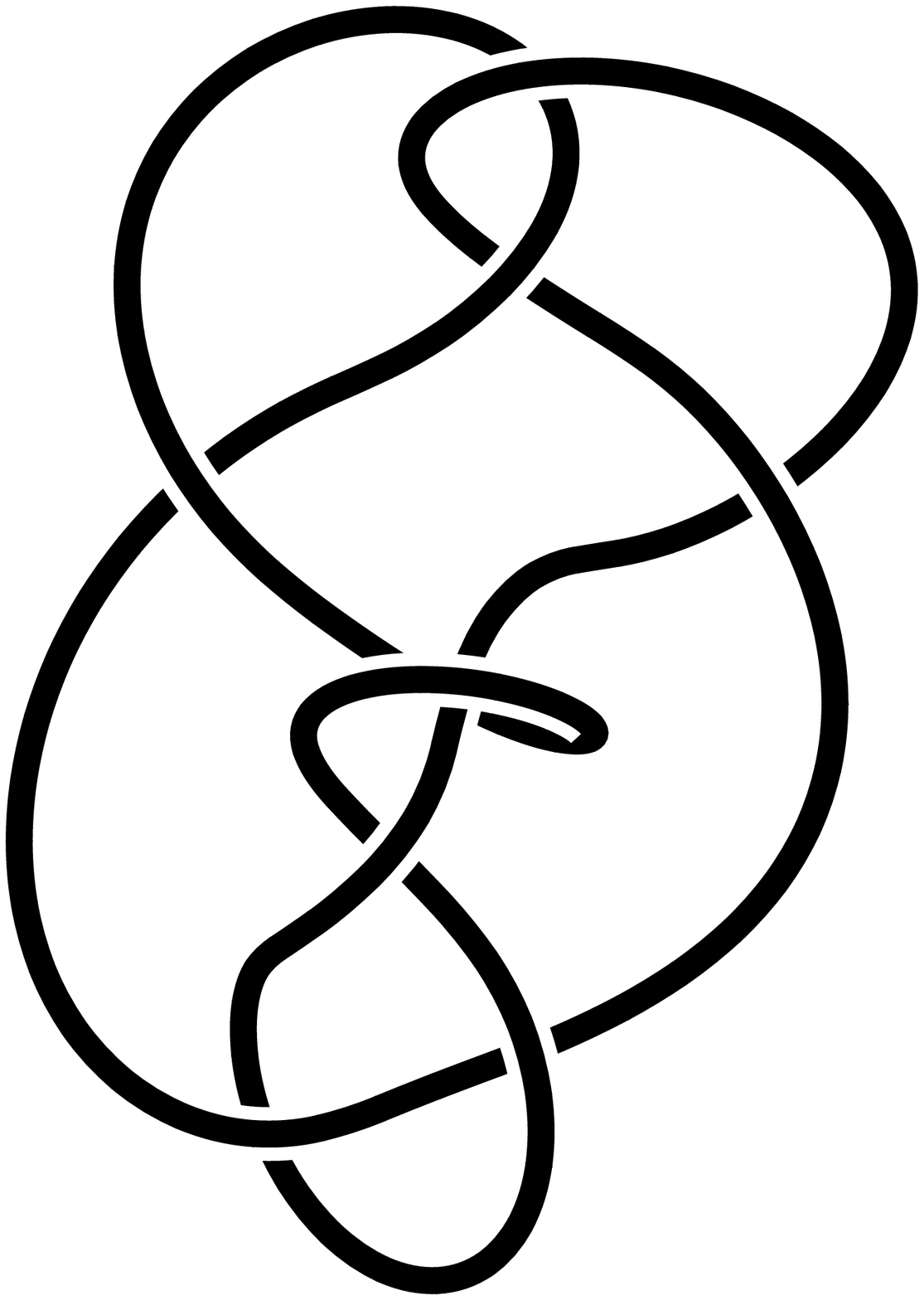}
$9_{15}$\\ \ &&&&\\
\includegraphics[keepaspectratio=1,height=2cm]{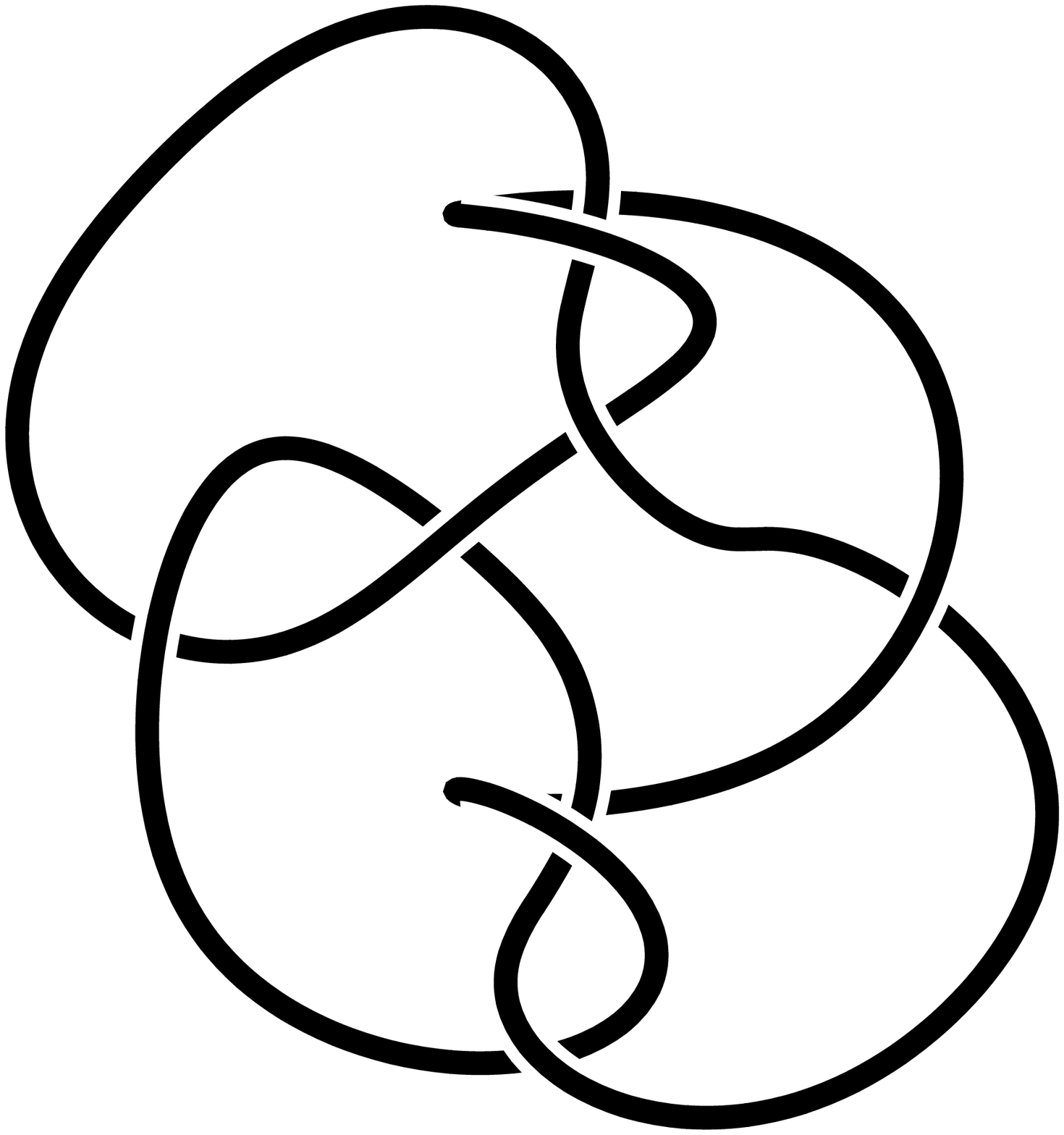}
$9_{16}$&
\includegraphics[keepaspectratio=1,height=2cm]{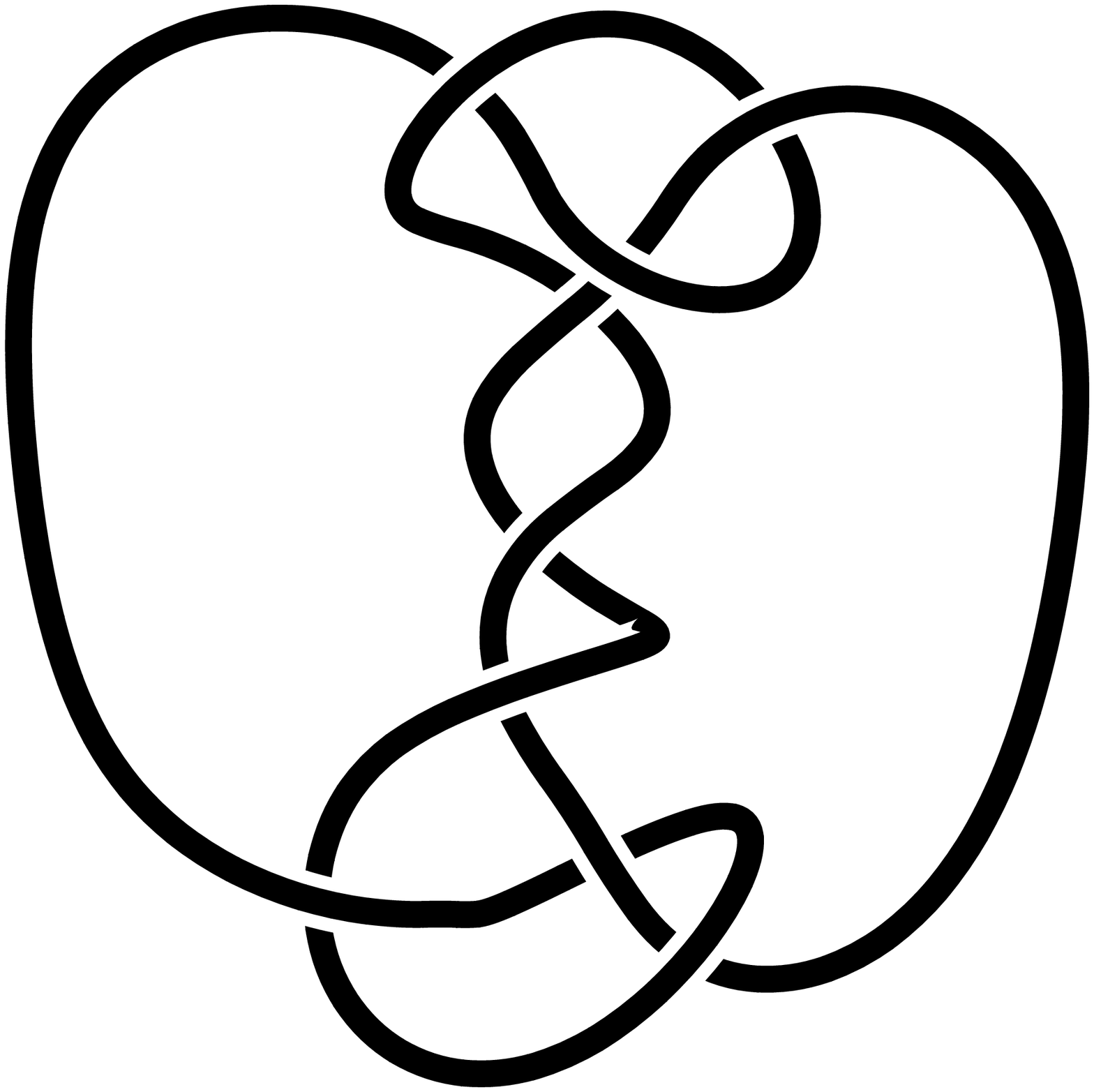}
$9_{17}$&
\includegraphics[keepaspectratio=1,height=2cm]{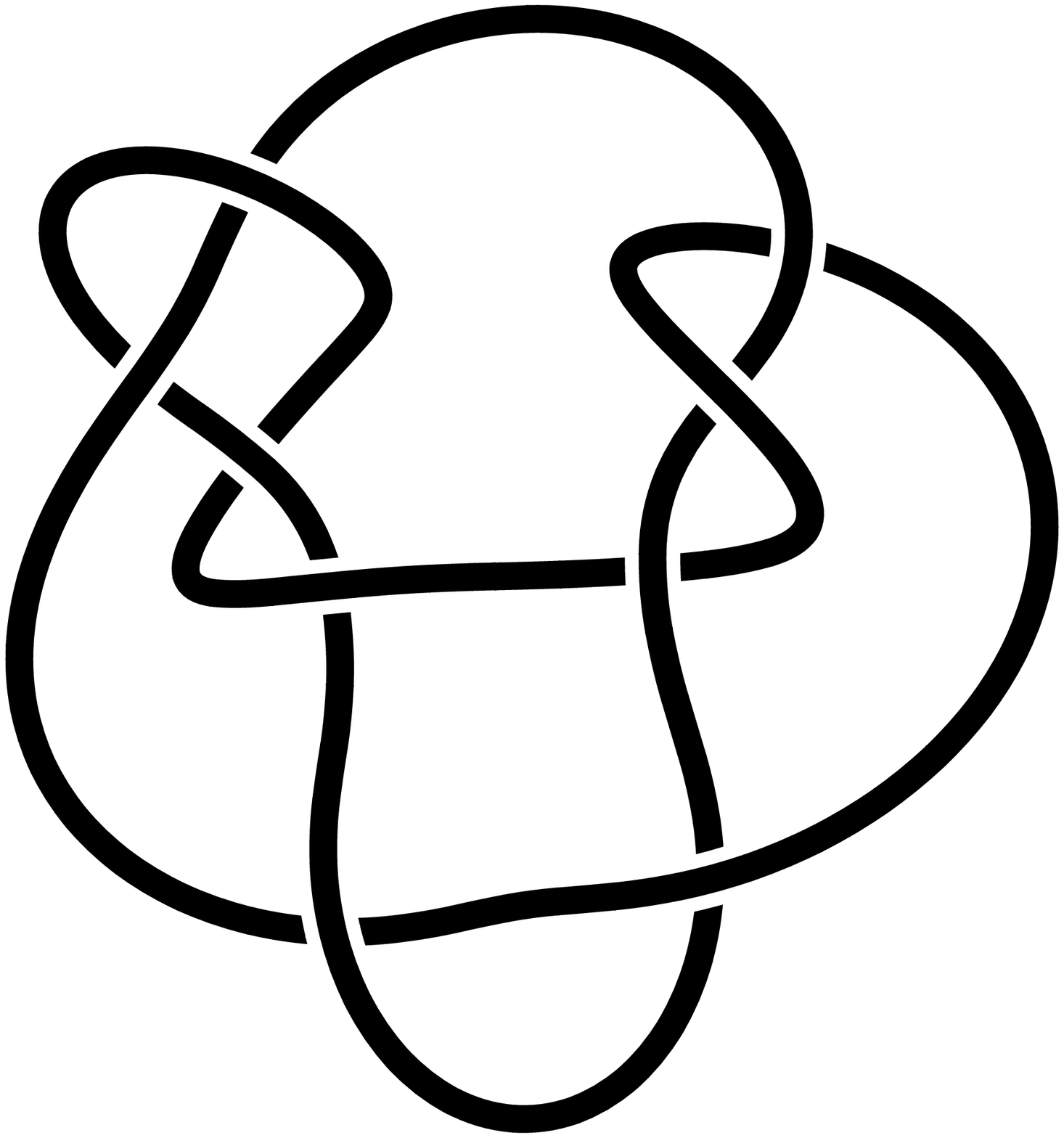}
$9_{18}$&
\includegraphics[keepaspectratio=1,height=2cm]{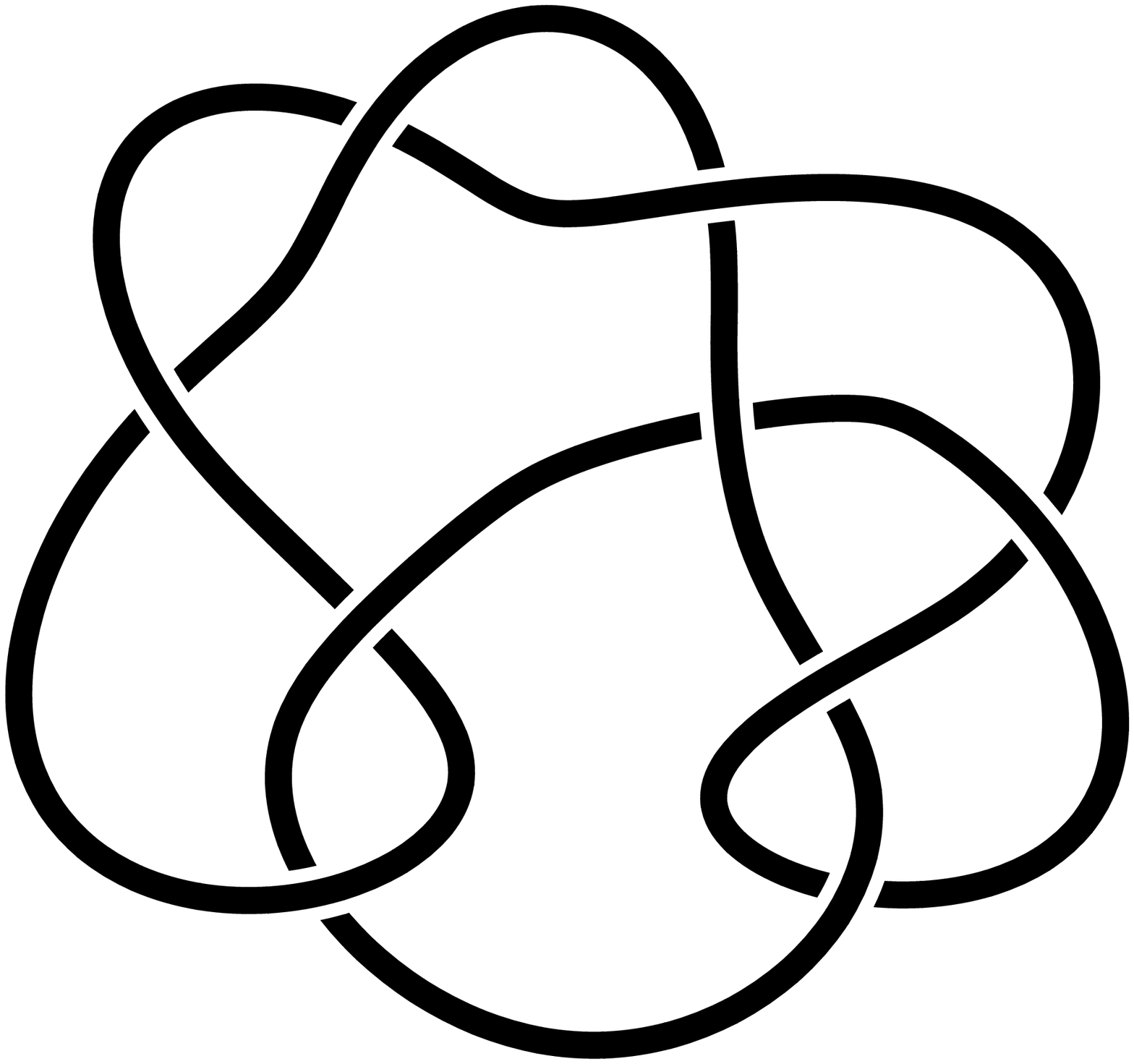}
$9_{19}$&
\includegraphics[keepaspectratio=1,height=2cm]{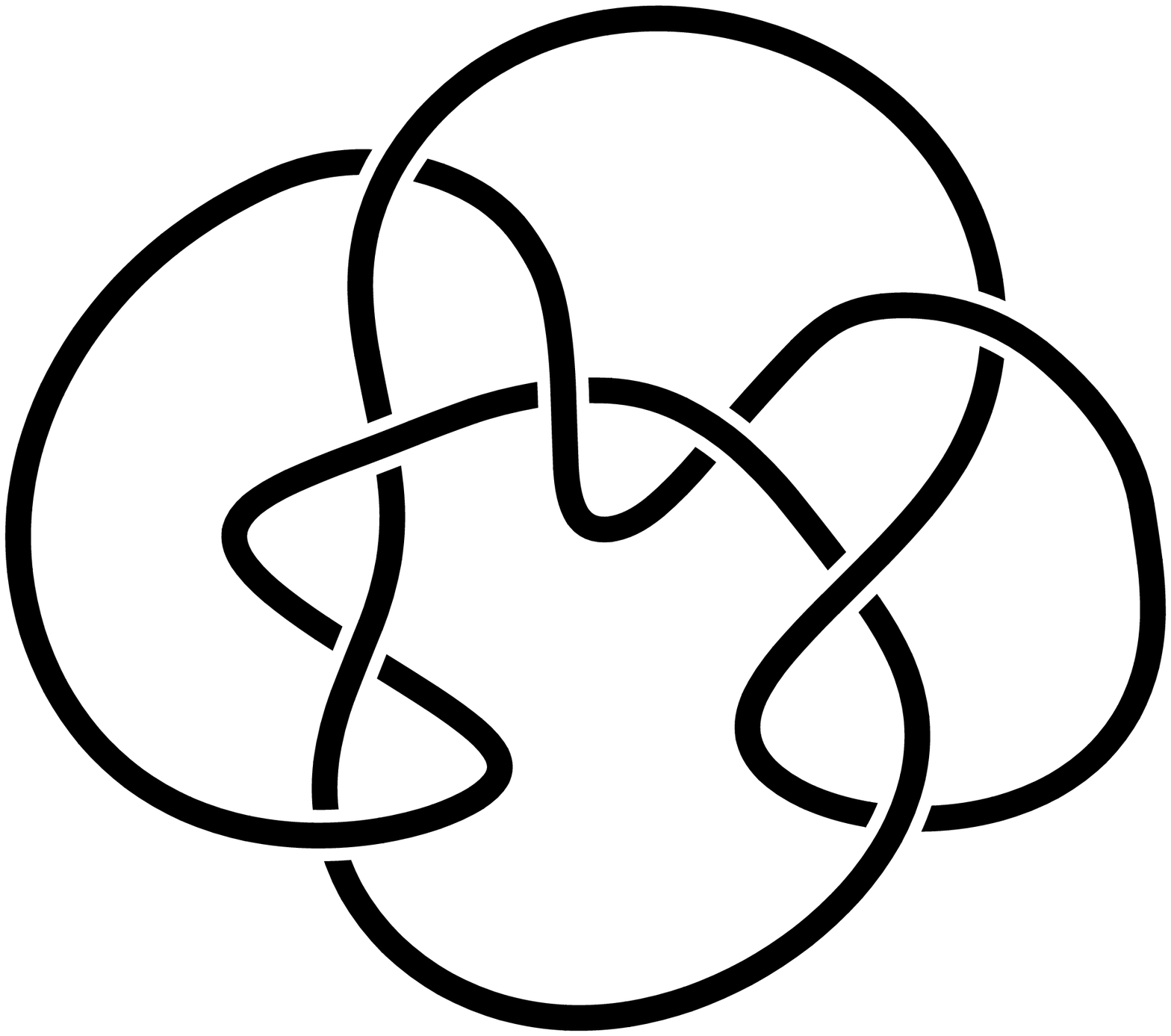}
$9_{20}$\\ \ &&&&\\
\includegraphics[keepaspectratio=1,height=2cm]{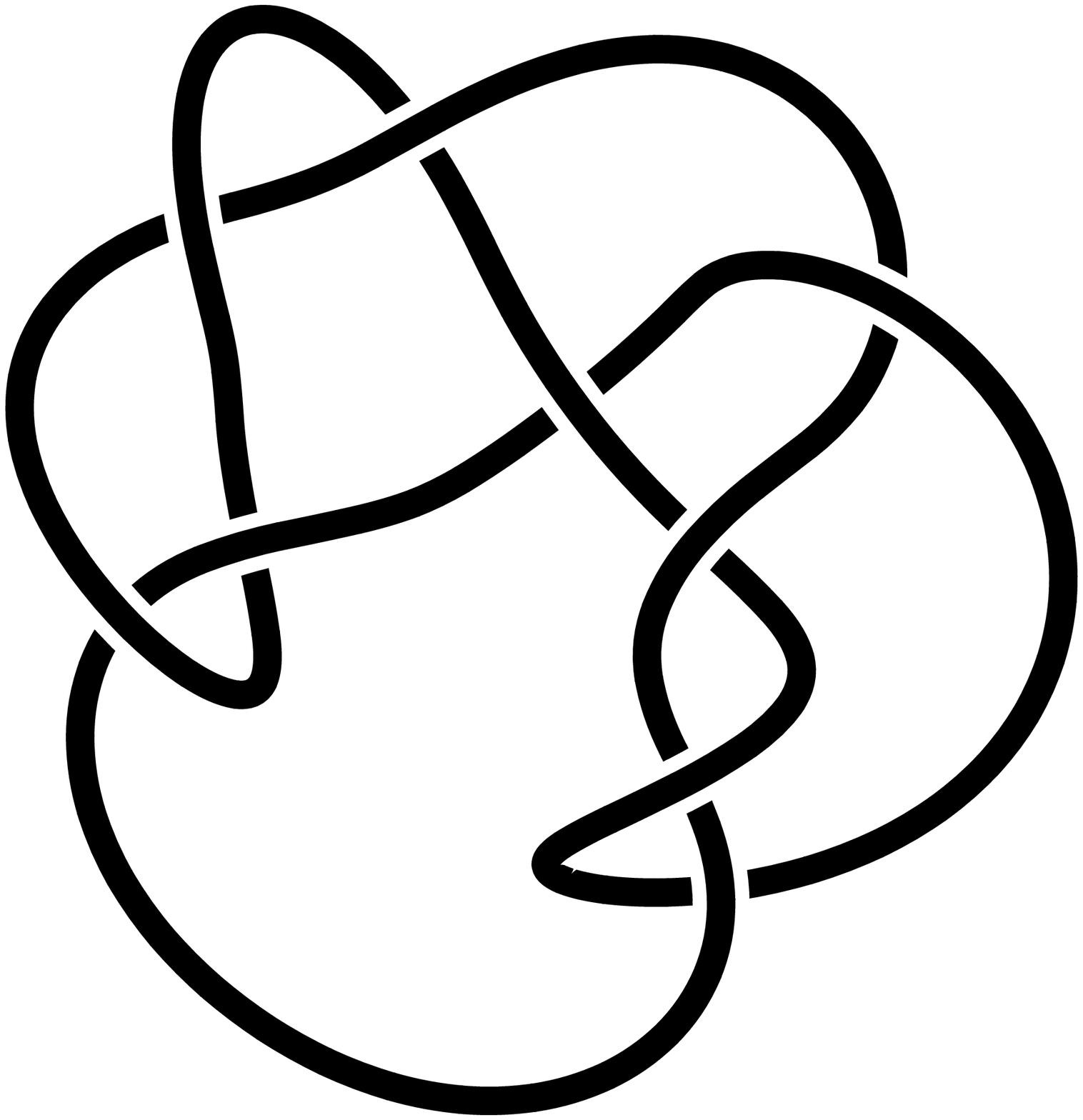}
$9_{21}$&
\includegraphics[keepaspectratio=1,height=2cm]{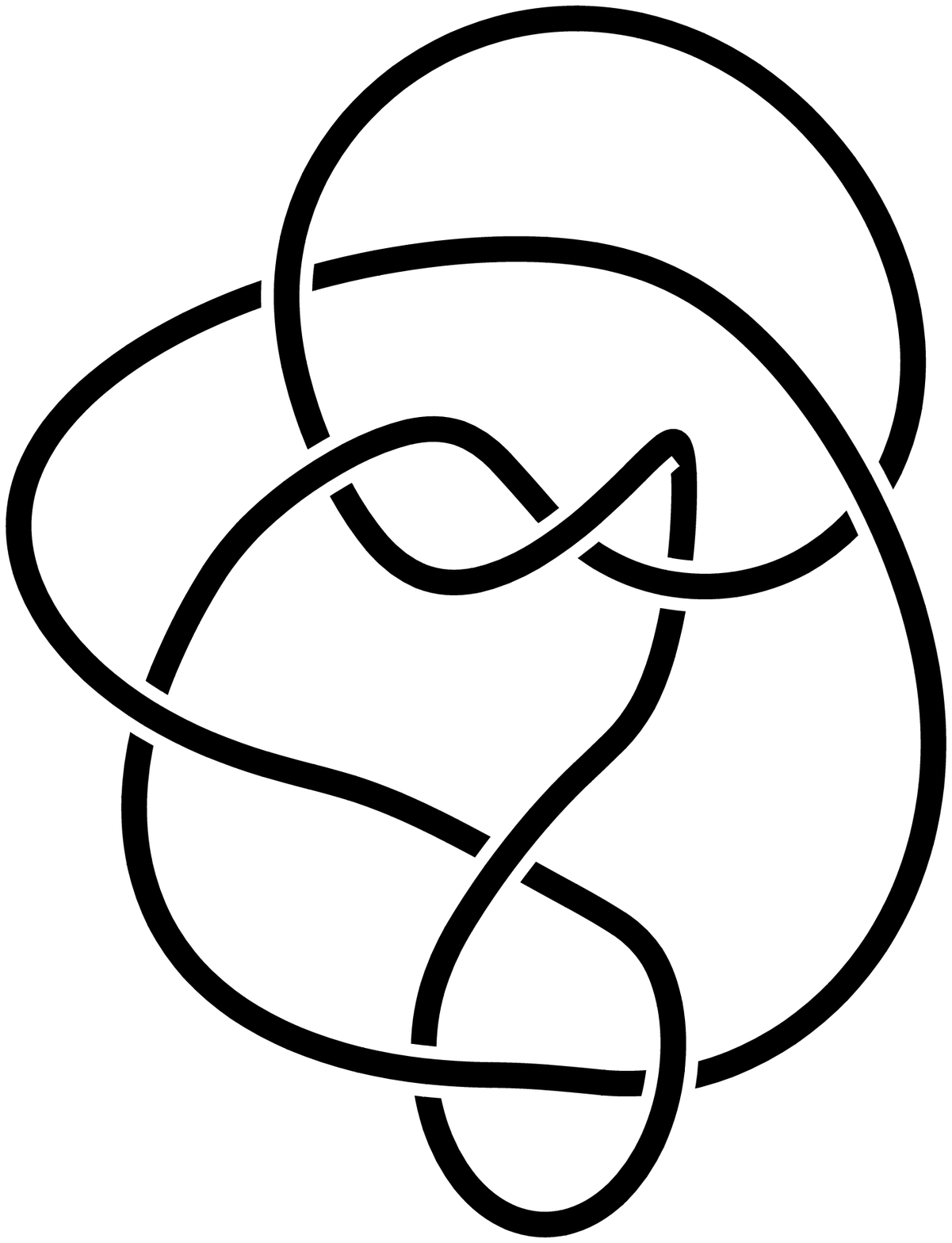}
$9_{22}$&
\includegraphics[keepaspectratio=1,height=2cm]{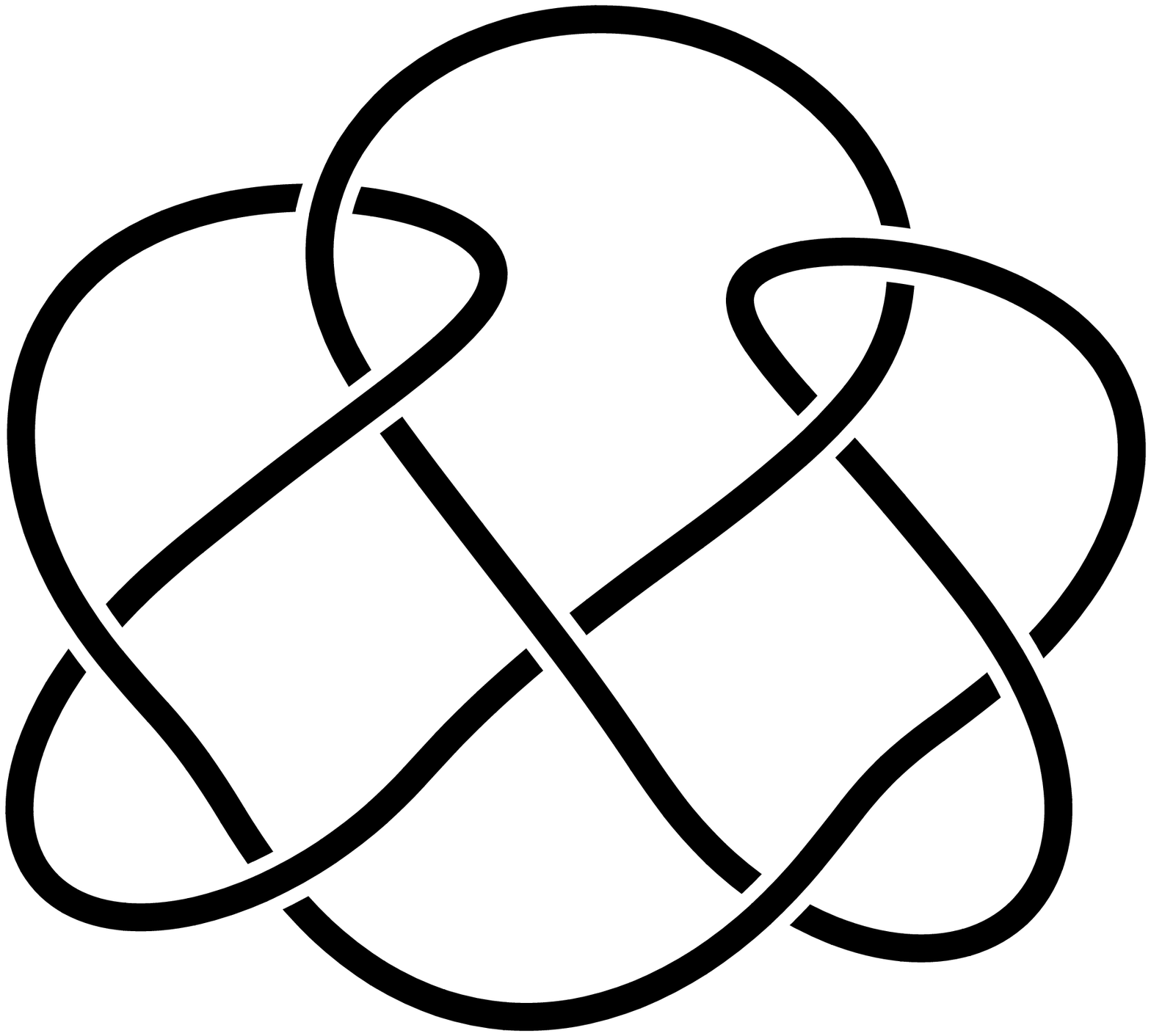}
$9_{23}$&
\includegraphics[keepaspectratio=1,height=2cm]{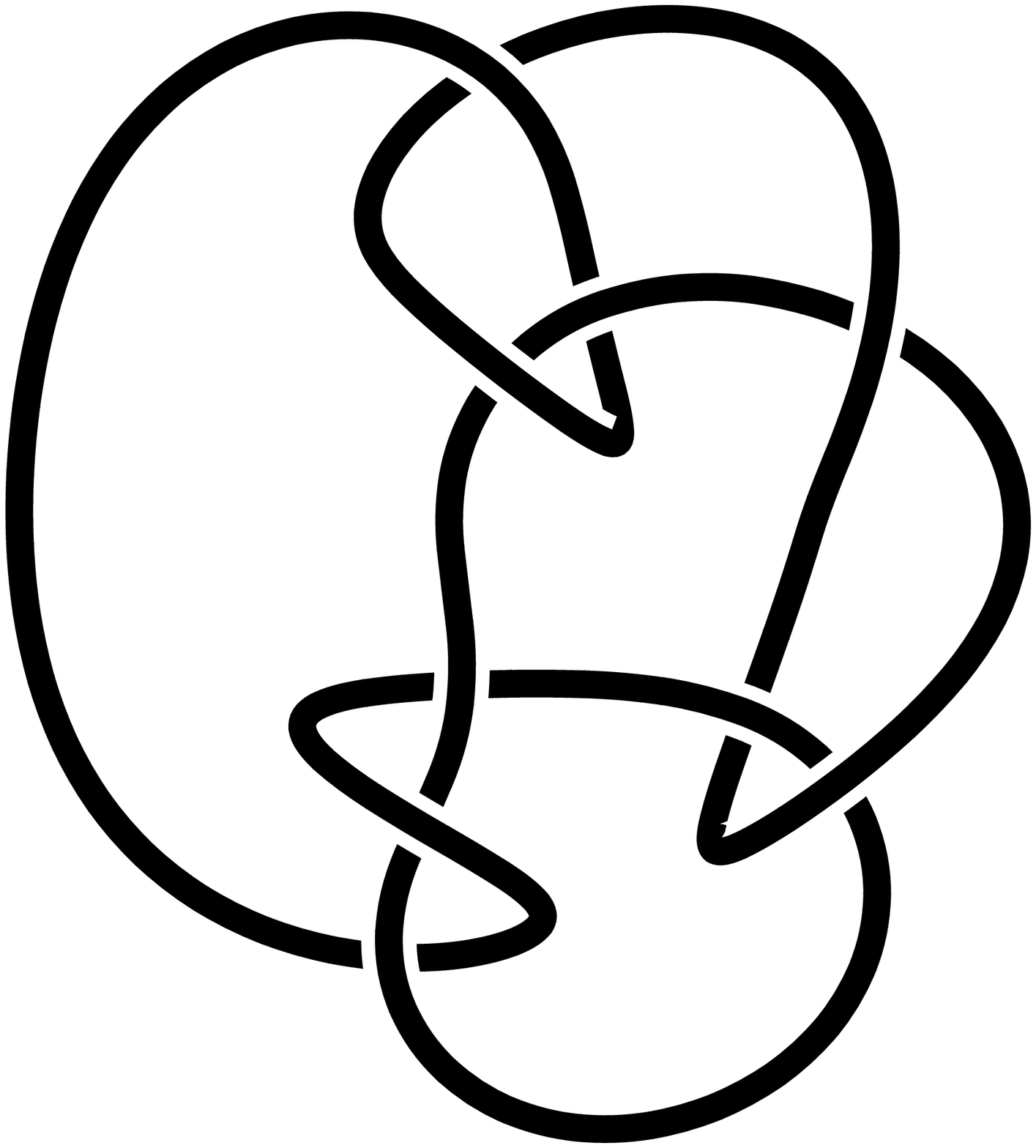}
$9_{24}$&
\includegraphics[keepaspectratio=1,height=2cm]{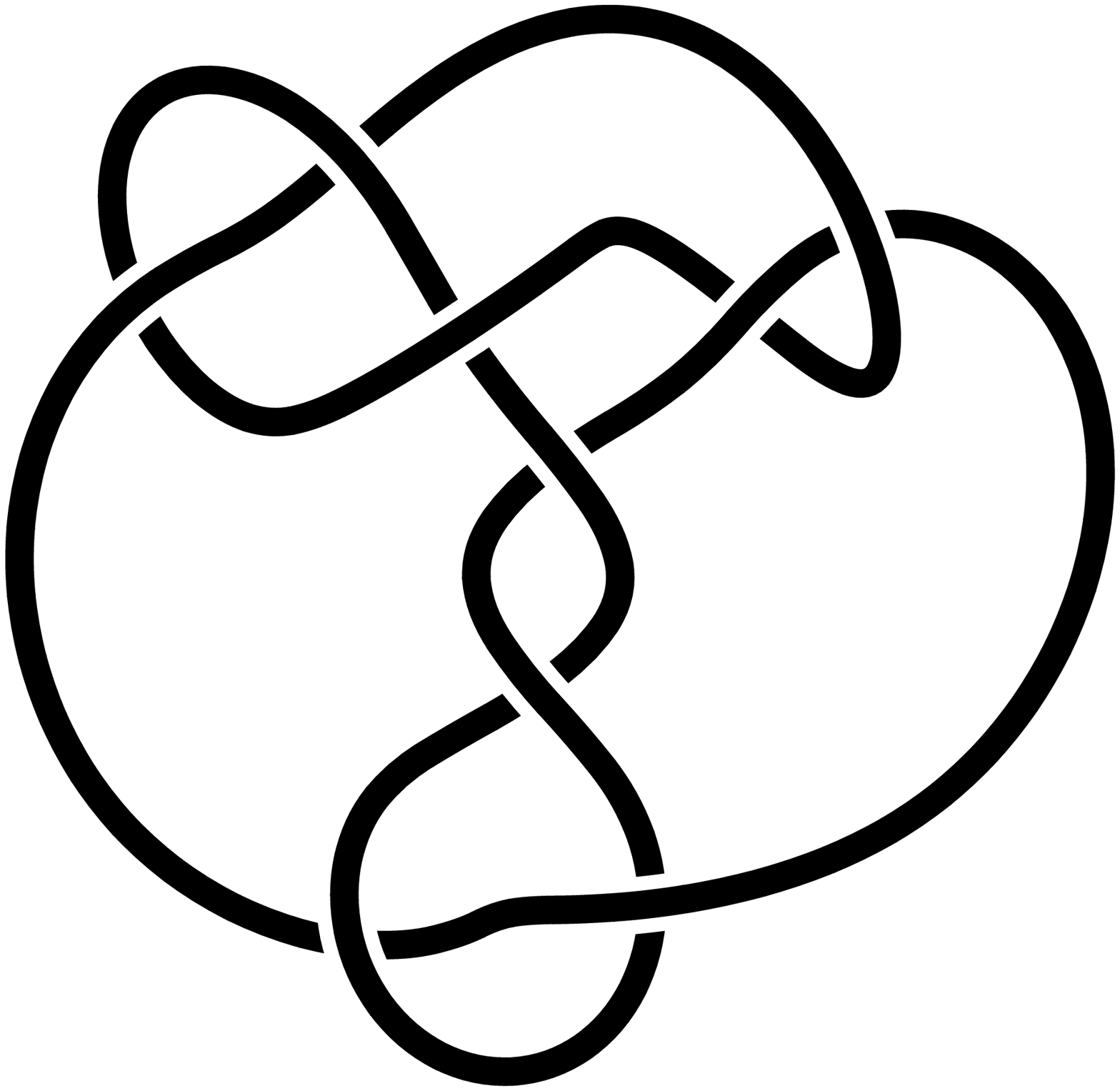}
$9_{25}$\\ \ &&&&\\
\includegraphics[keepaspectratio=1,height=2cm]{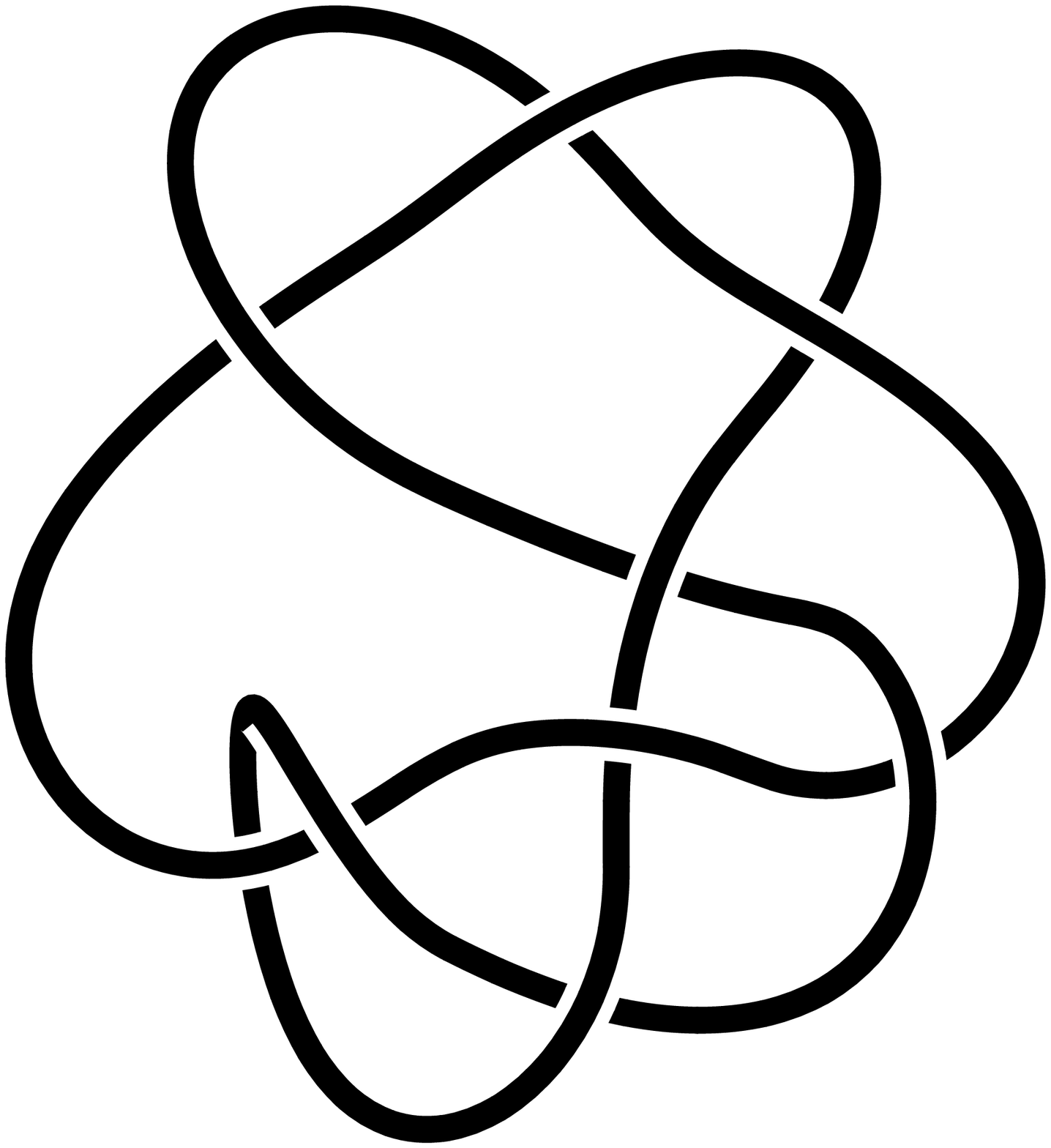}
$9_{26}$&
\includegraphics[keepaspectratio=1,height=2cm]{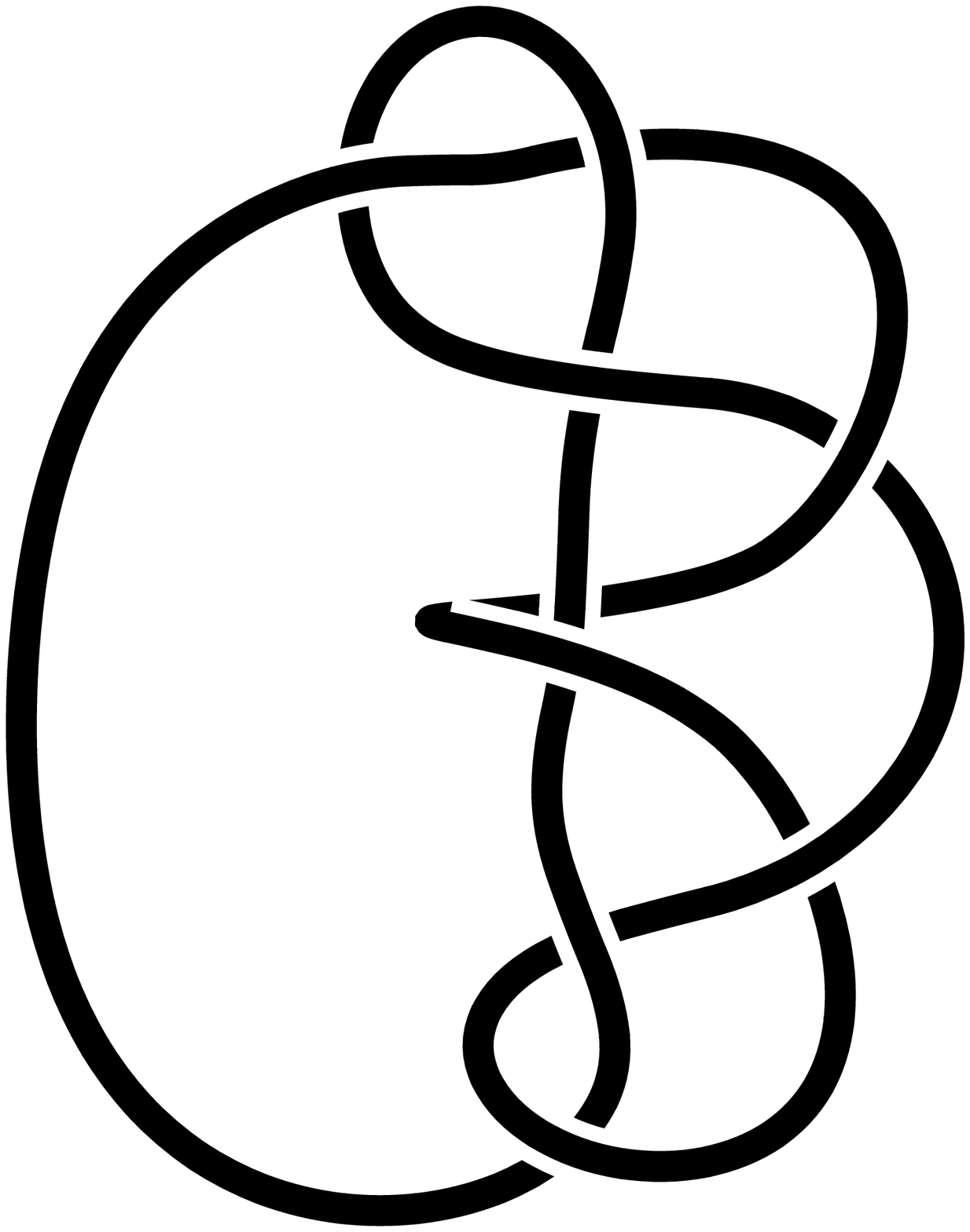}
$9_{27}$&
\includegraphics[keepaspectratio=1,height=2cm]{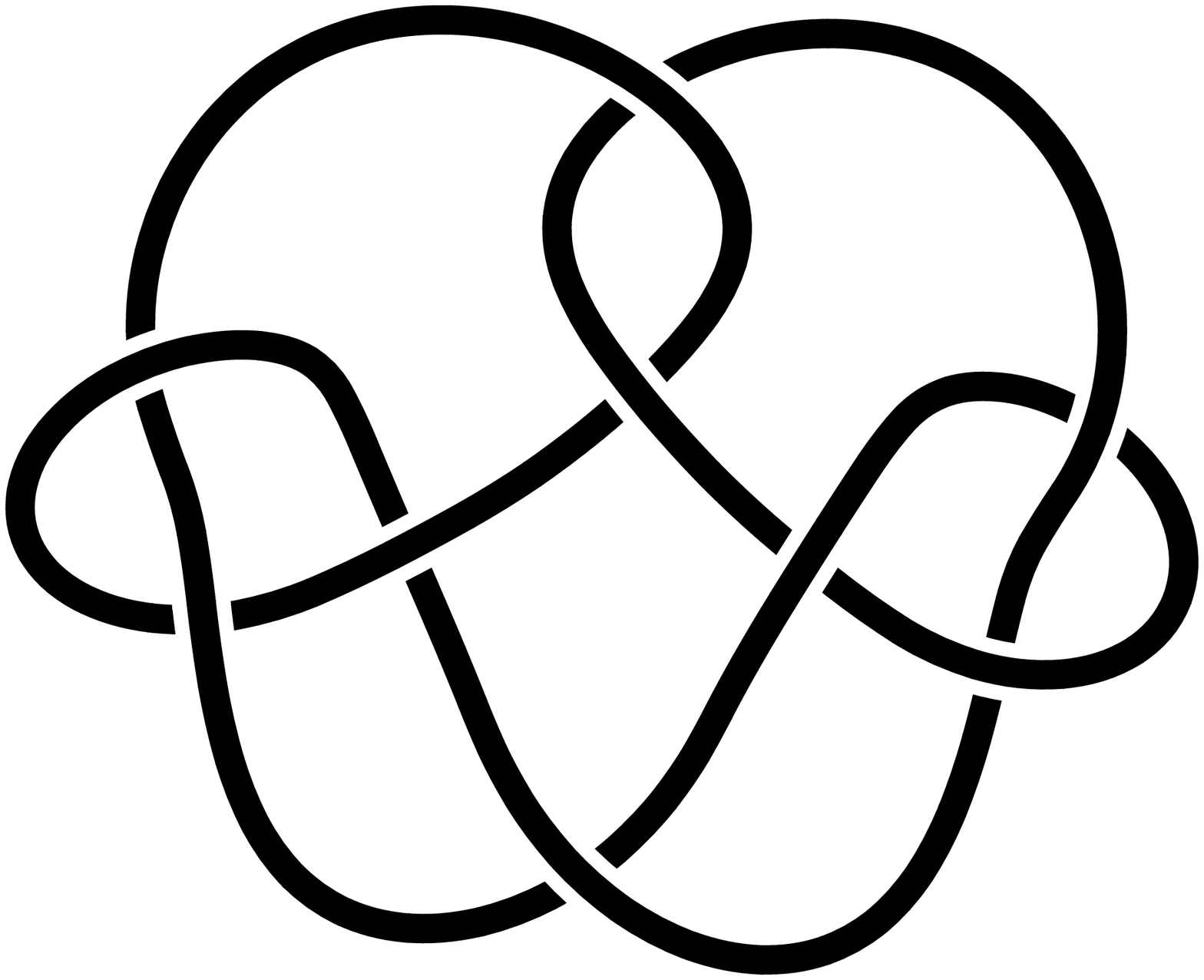}
$9_{28}$&
\includegraphics[keepaspectratio=1,height=2cm]{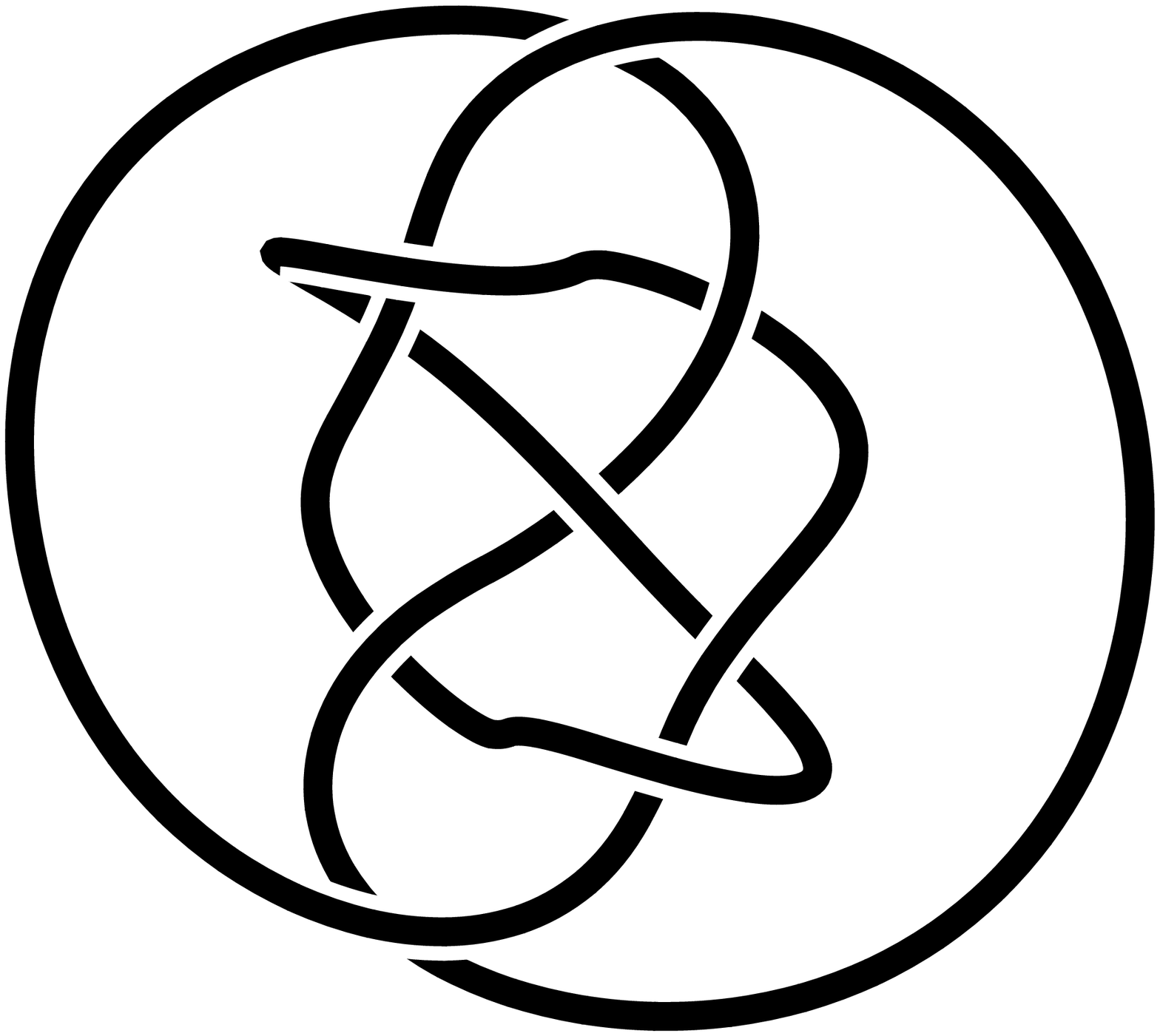}
$9_{29}$&
\includegraphics[keepaspectratio=1,height=2cm]{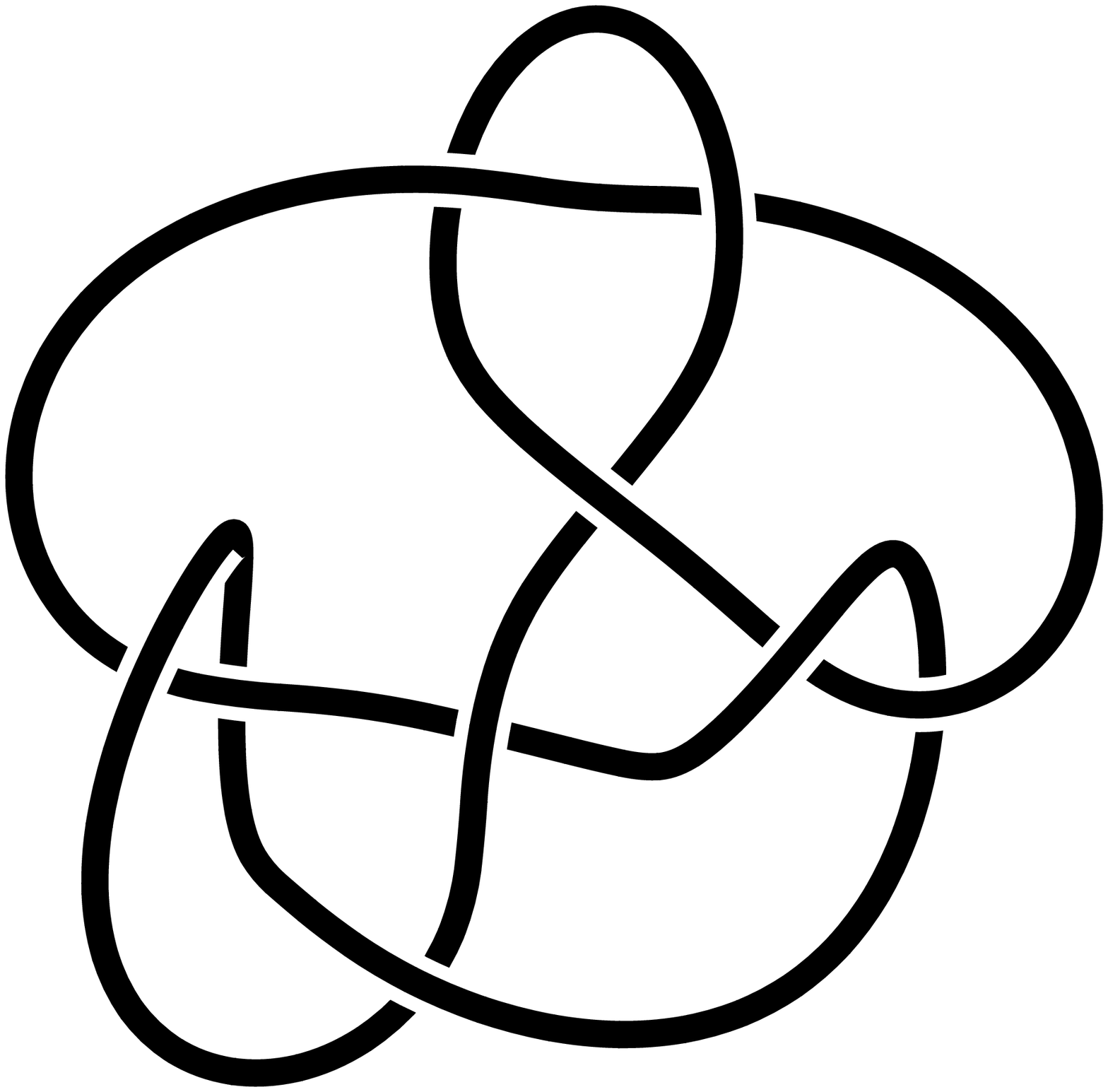}
$9_{30}$\\ \ &&&&\\
\includegraphics[keepaspectratio=1,height=2cm]{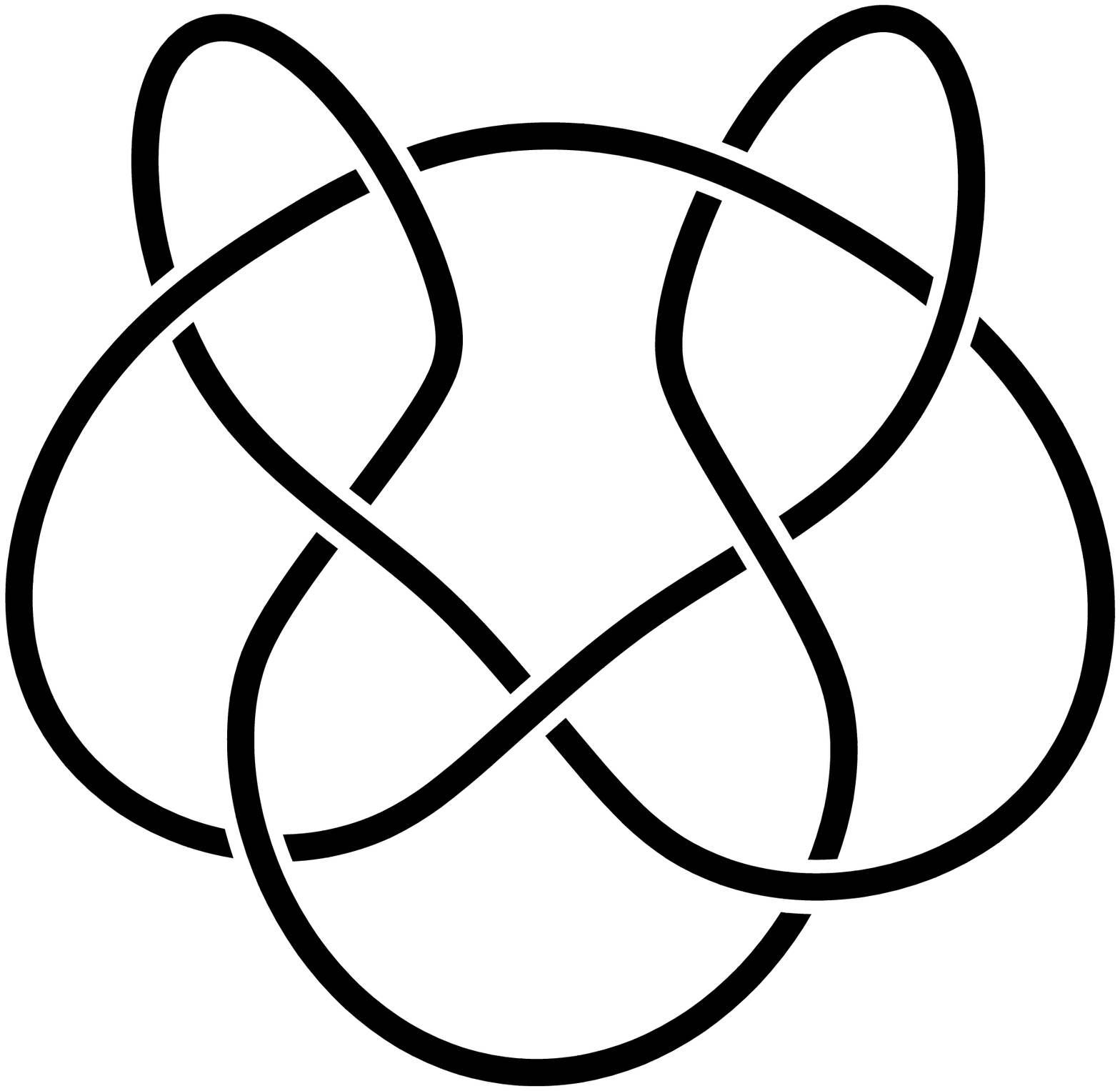}
$9_{31}$&
\includegraphics[keepaspectratio=1,height=2cm]{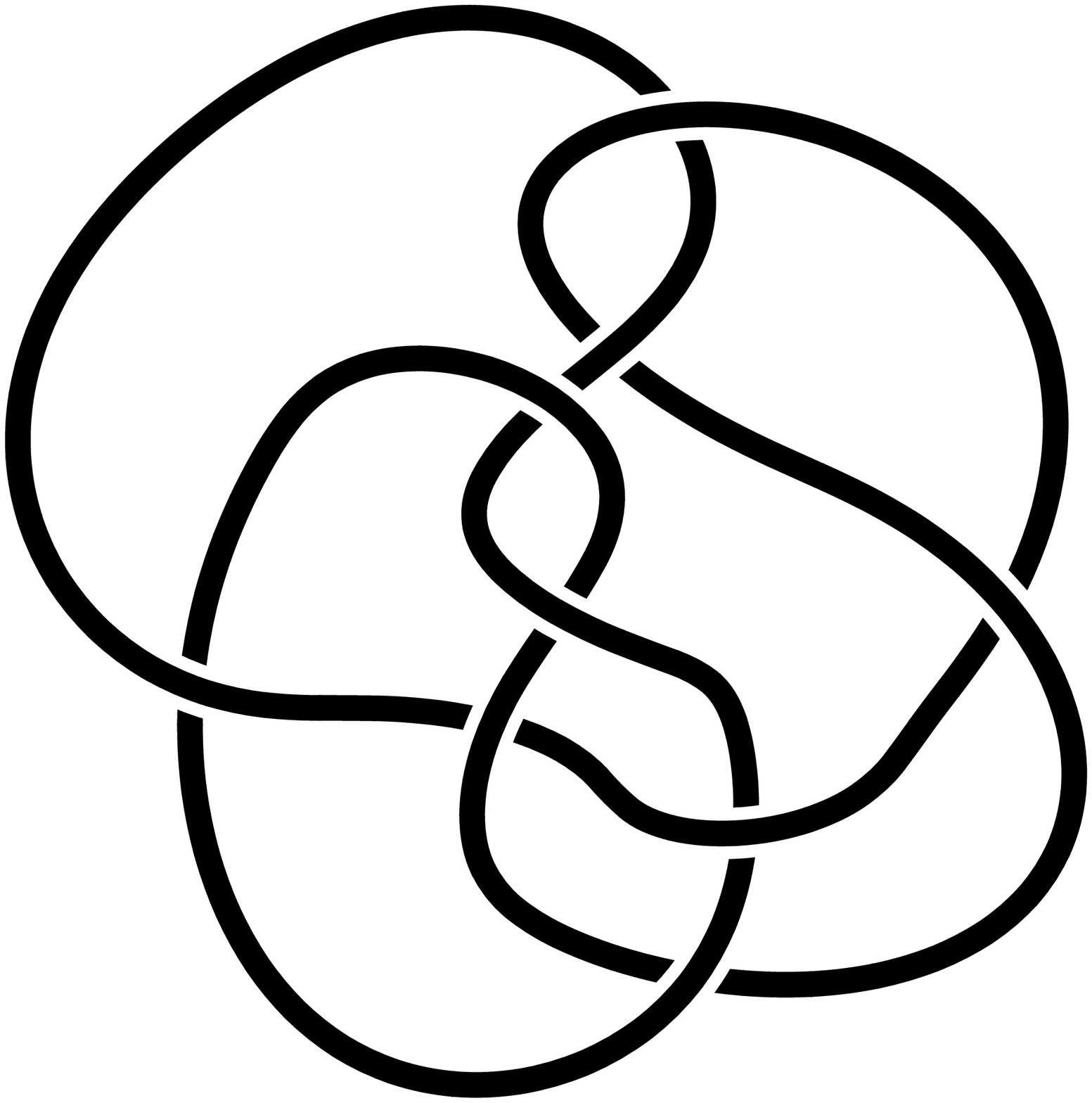}
$9_{32}$&
\includegraphics[keepaspectratio=1,height=2cm]{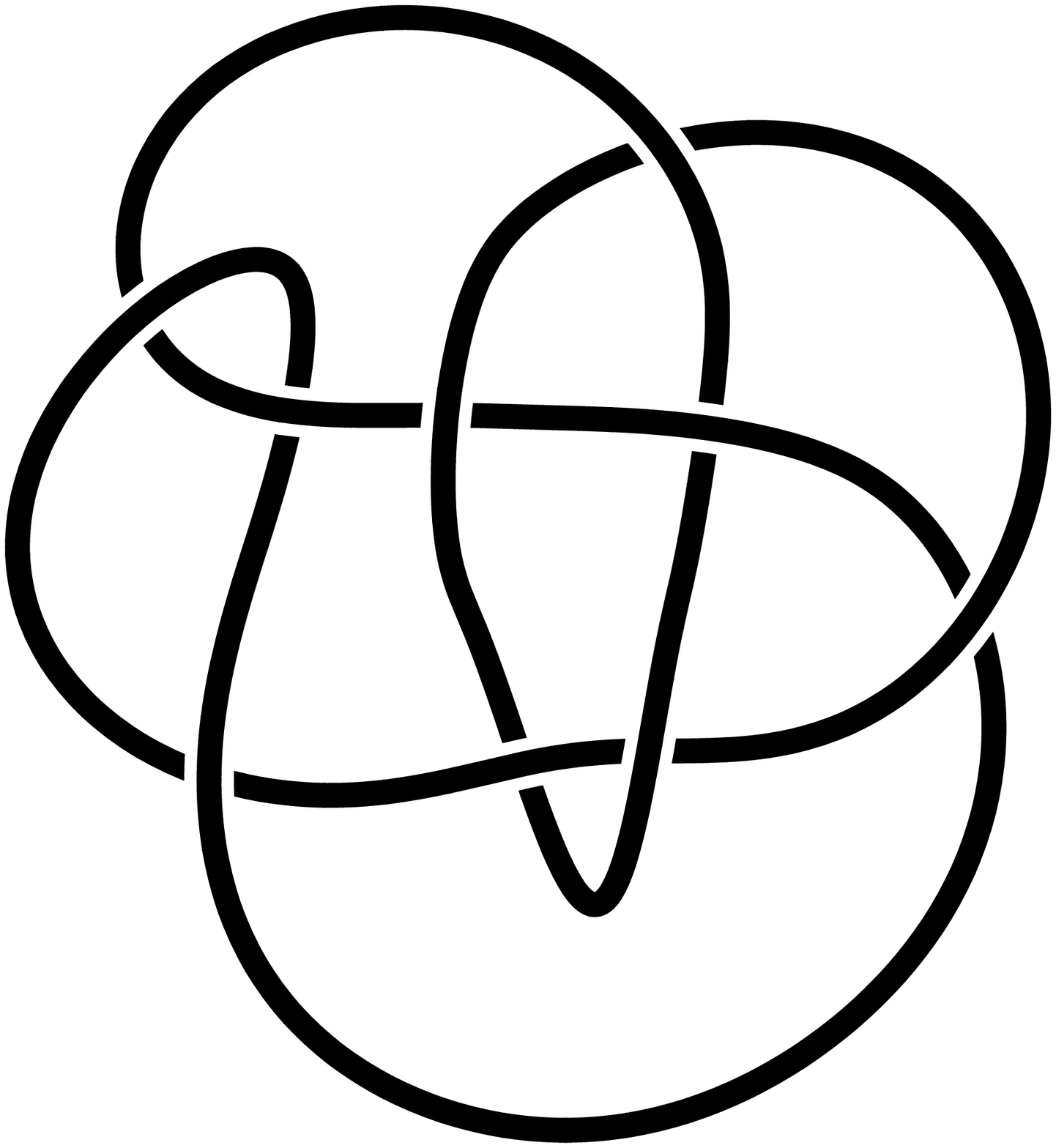}
$9_{33}$&
\includegraphics[keepaspectratio=1,height=2cm]{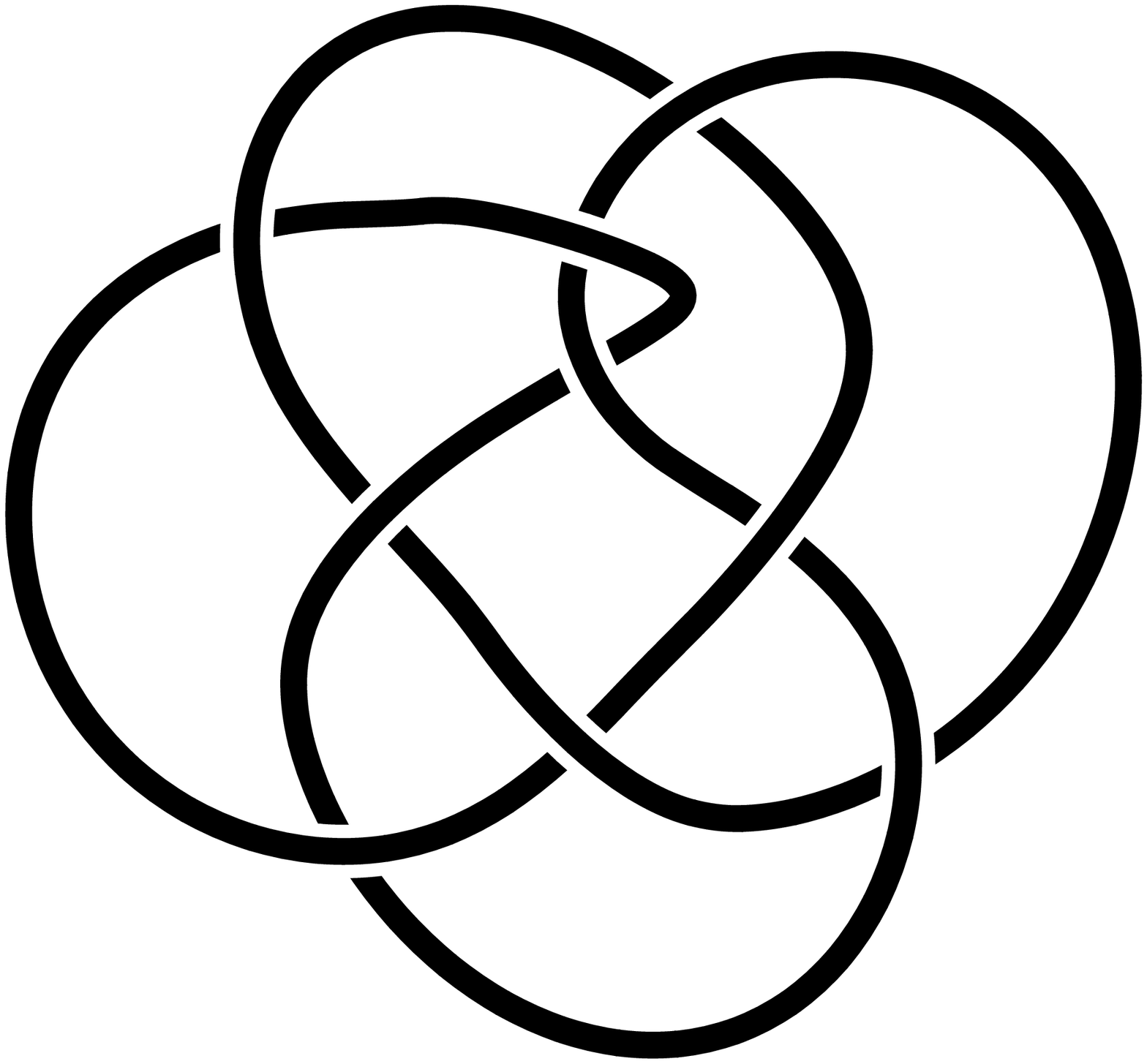}
$9_{34}$&
\includegraphics[keepaspectratio=1,height=2cm]{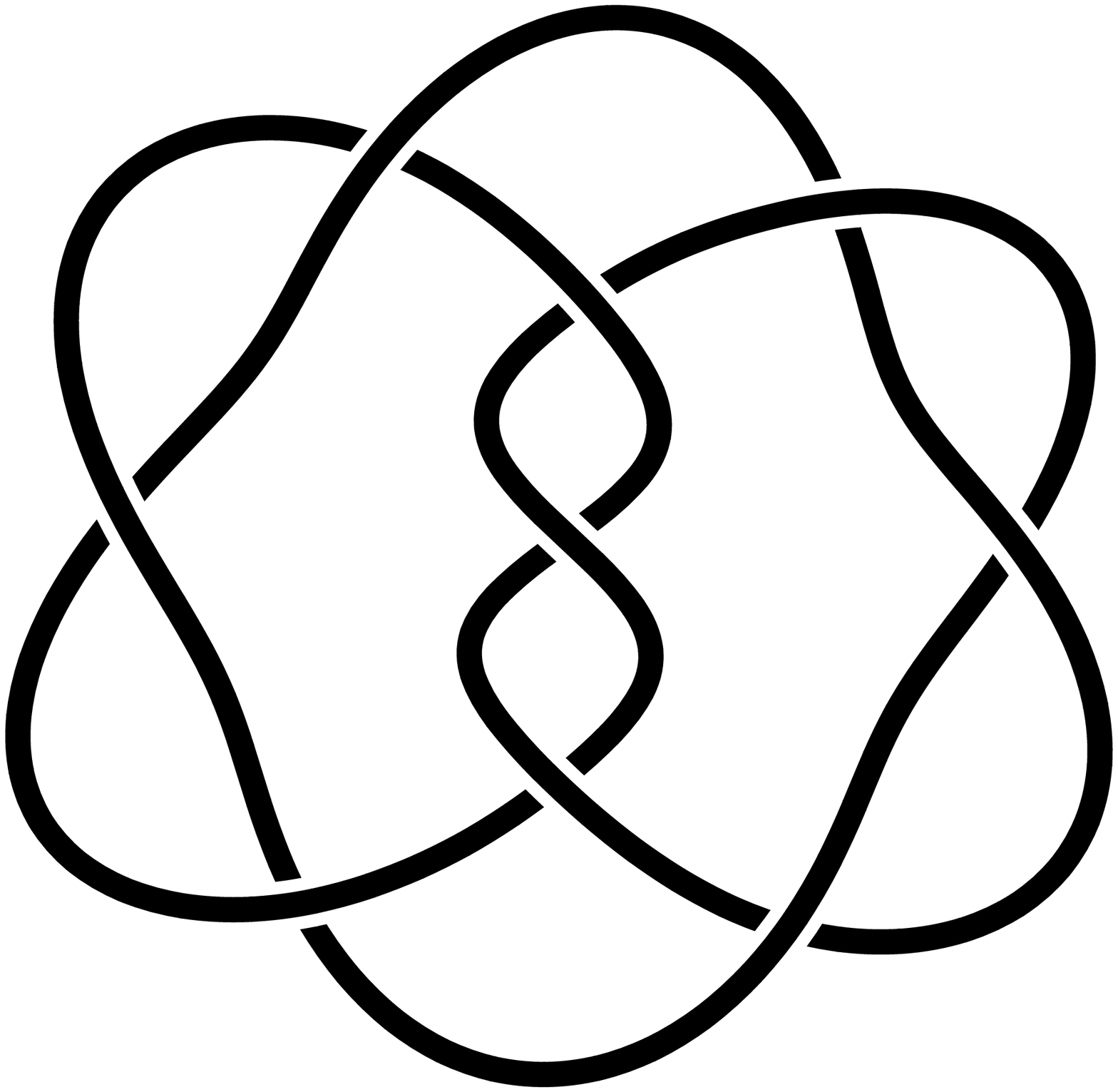}
$9_{35}$\\ \ &&&&\\
\includegraphics[keepaspectratio=1,height=2cm]{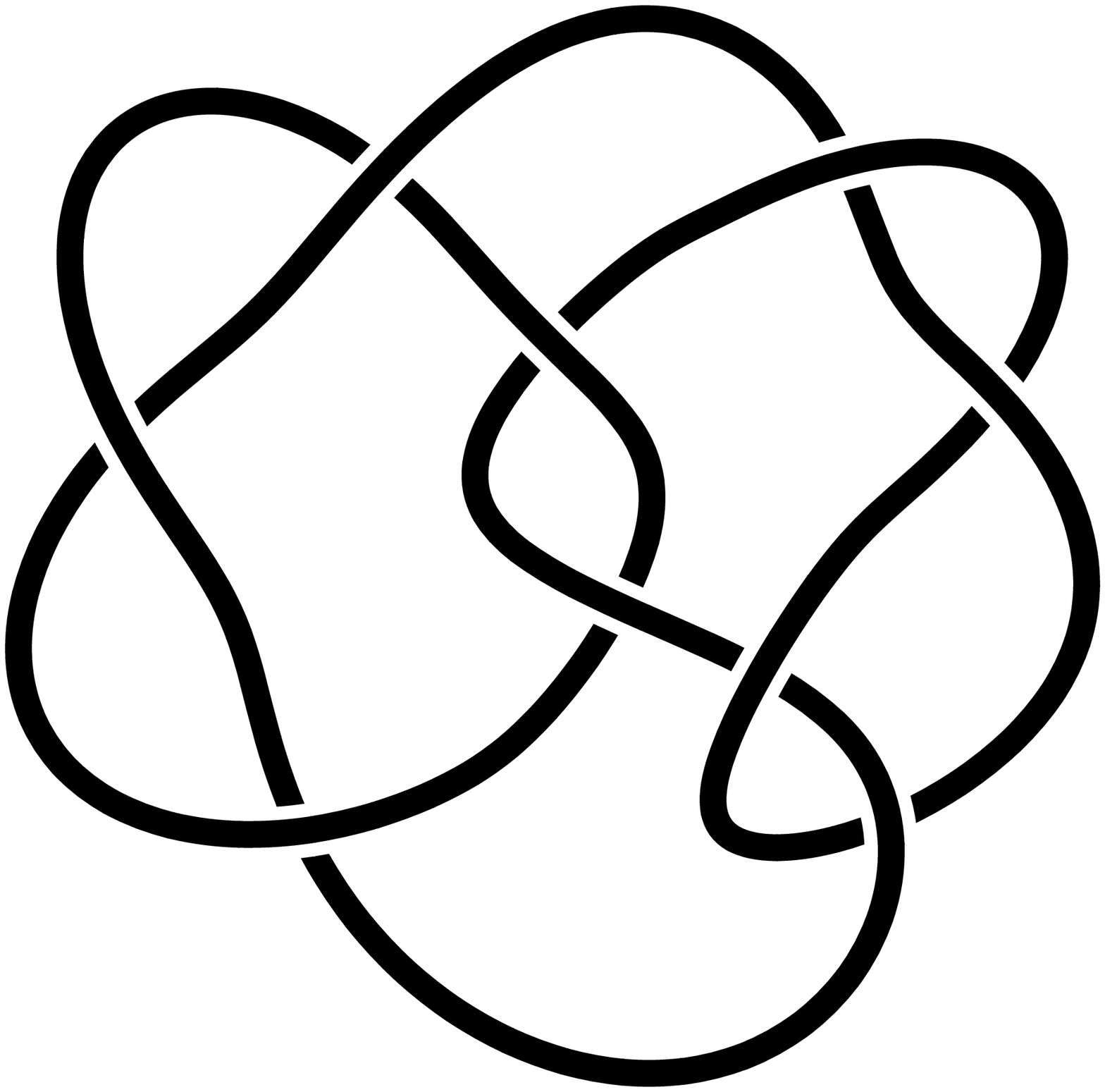}
$9_{36}$&
\includegraphics[keepaspectratio=1,height=2cm]{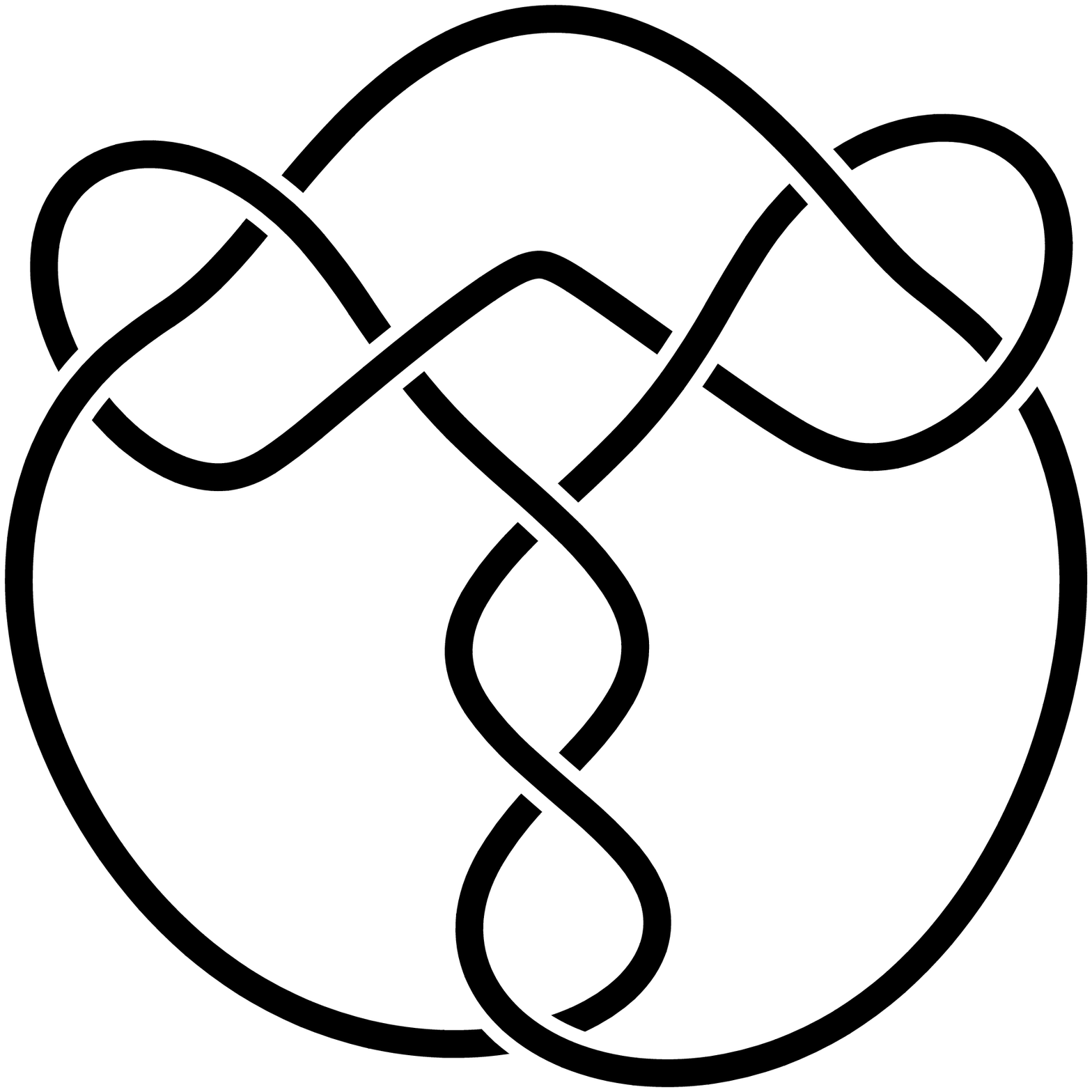}
$9_{37}$&
\includegraphics[keepaspectratio=1,height=2cm]{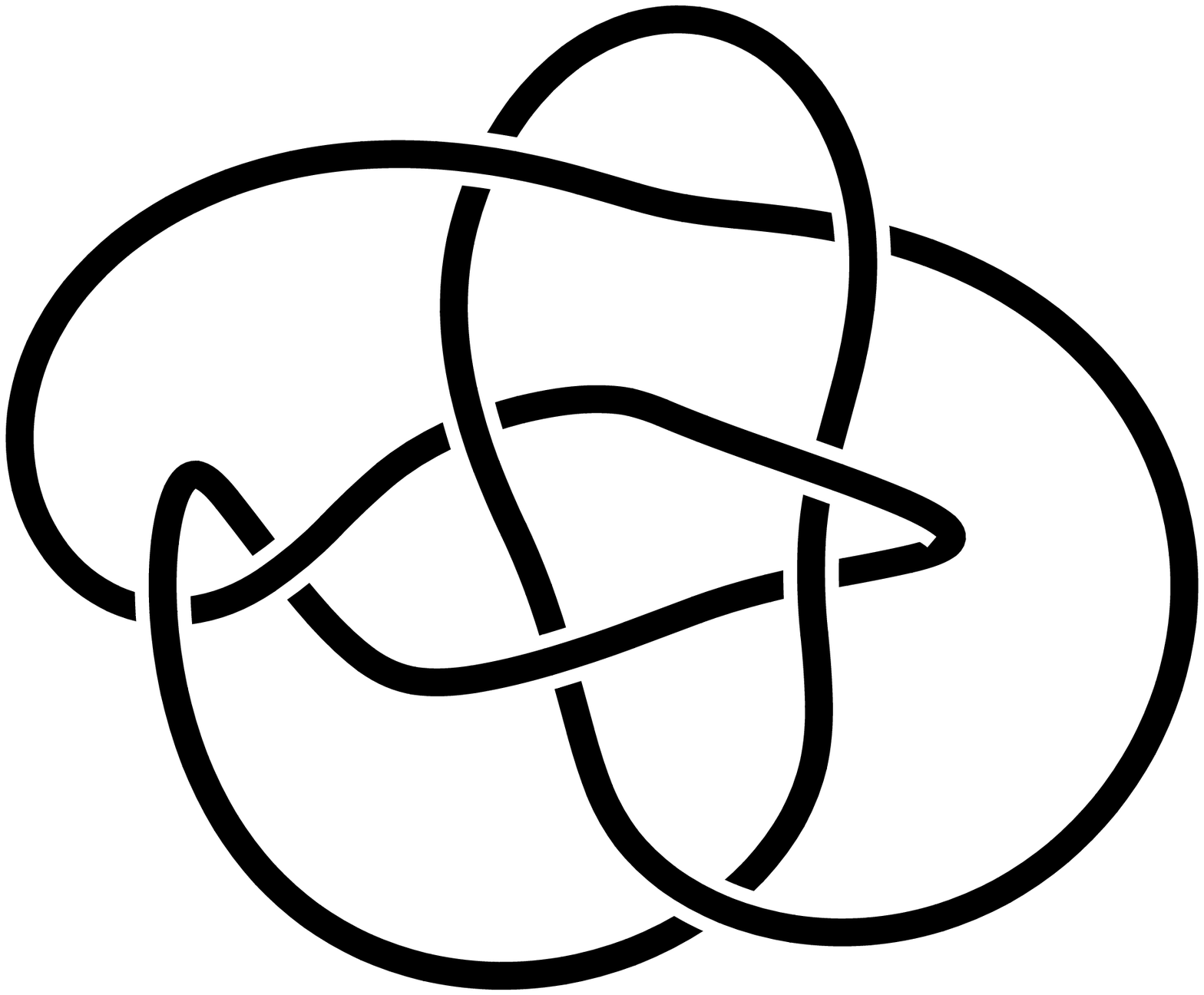}
$9_{38}$&
\includegraphics[keepaspectratio=1,height=2cm]{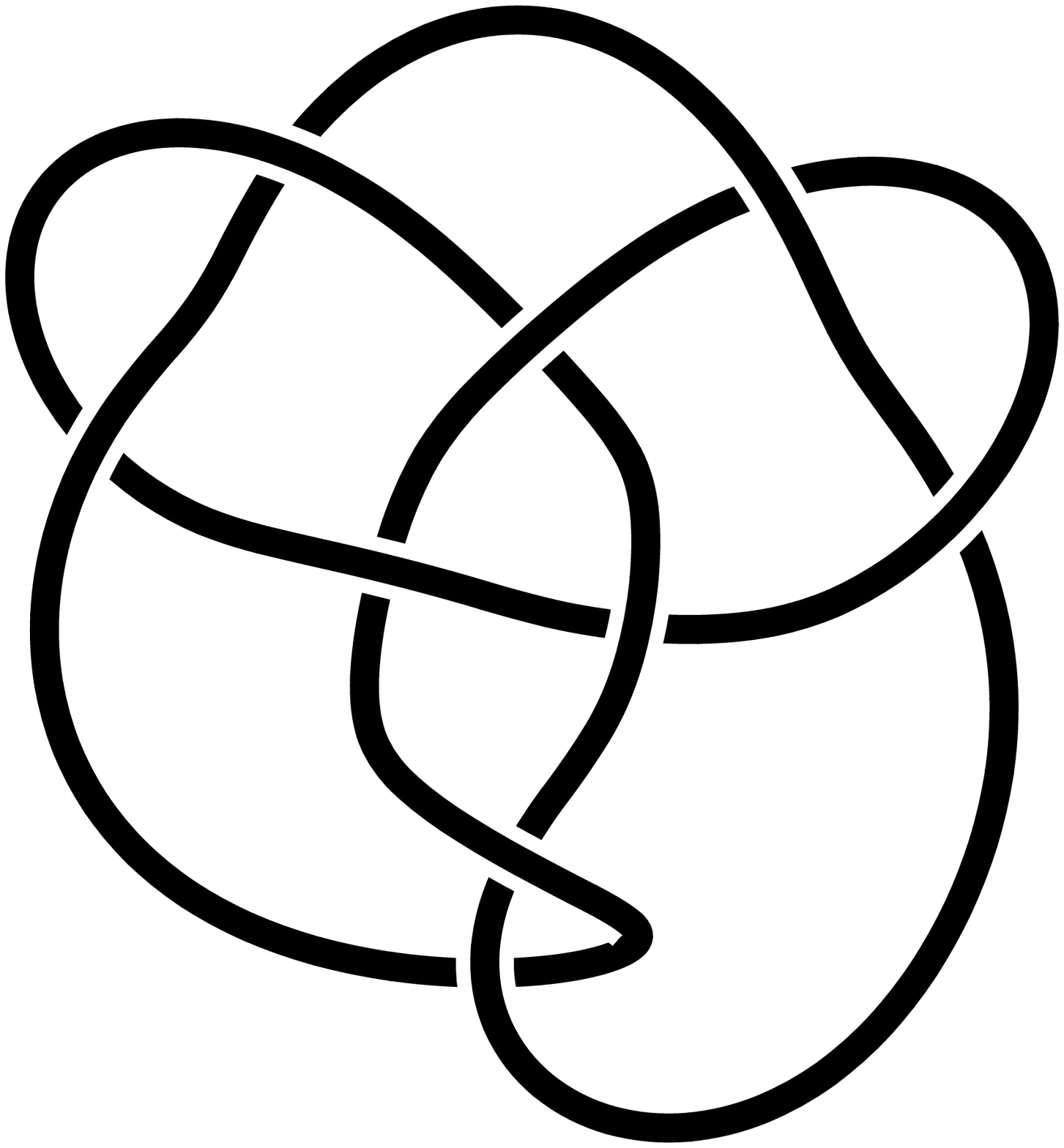}
$9_{39}$&
\includegraphics[keepaspectratio=1,height=2cm]{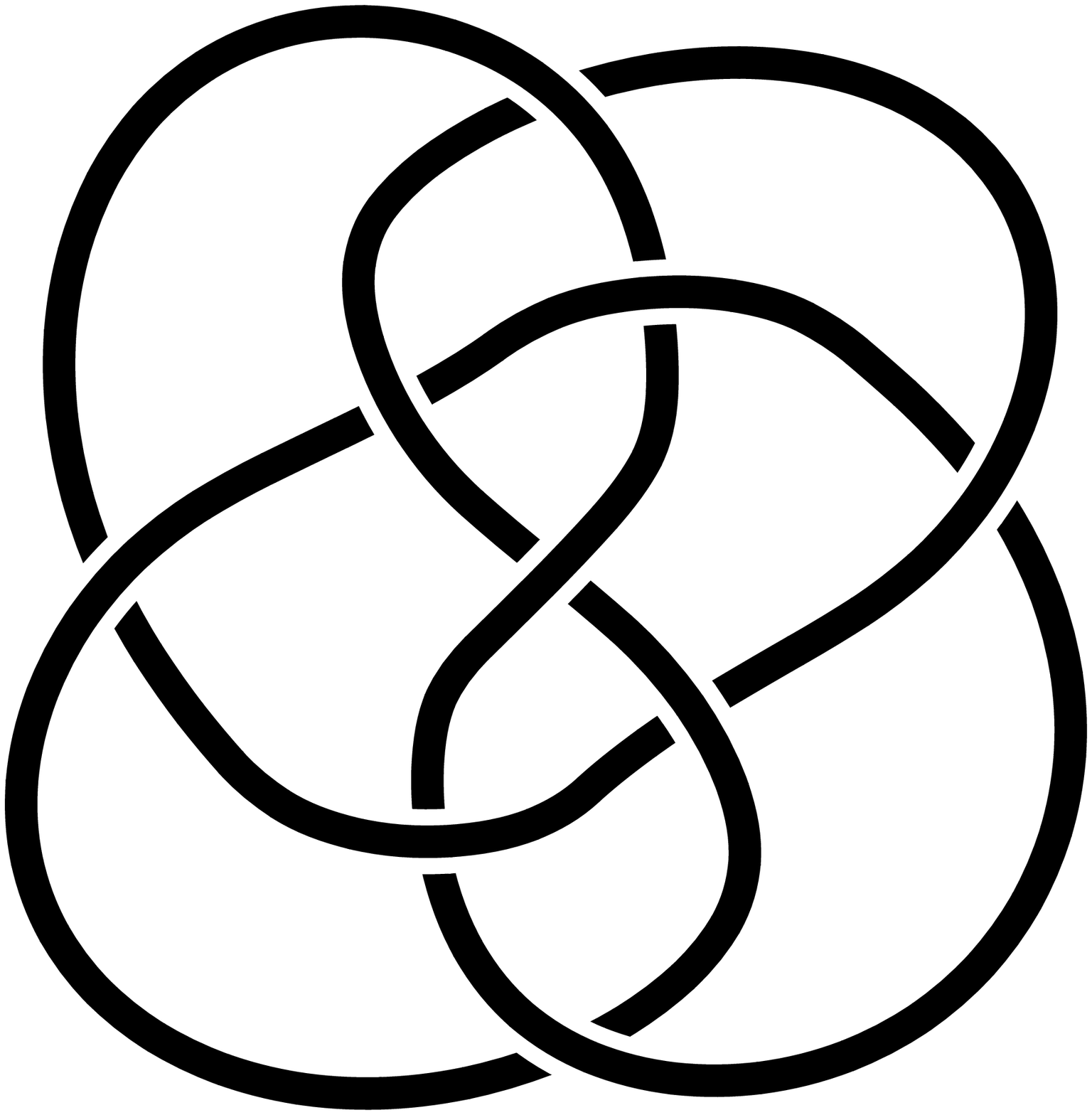}
$9_{40}$\\ \ &&&&\\
\includegraphics[keepaspectratio=1,height=2cm]{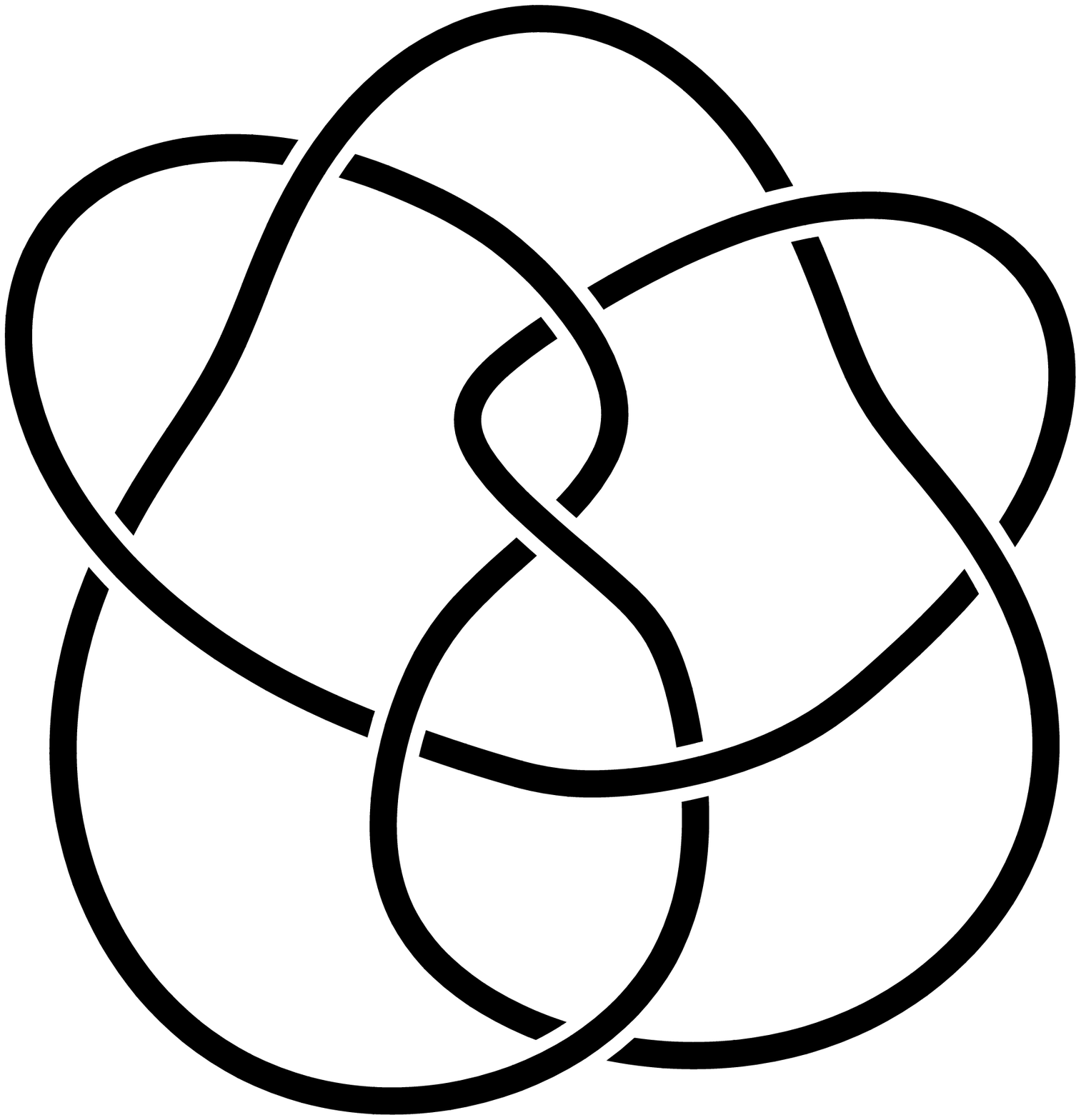}
$9_{41}$&
\includegraphics[keepaspectratio=1,height=2cm]{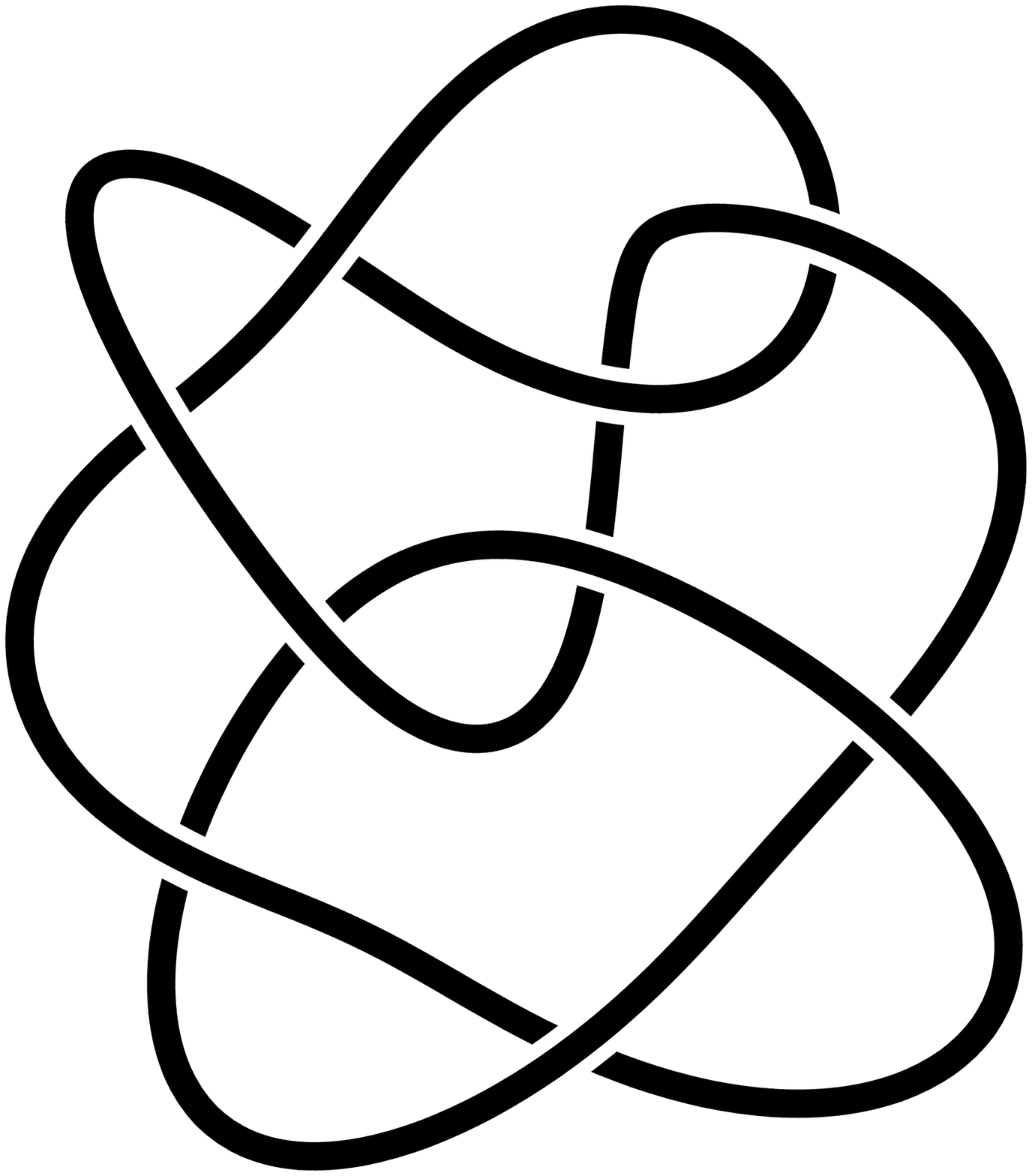}
$9_{42}$&
\includegraphics[keepaspectratio=1,height=2cm]{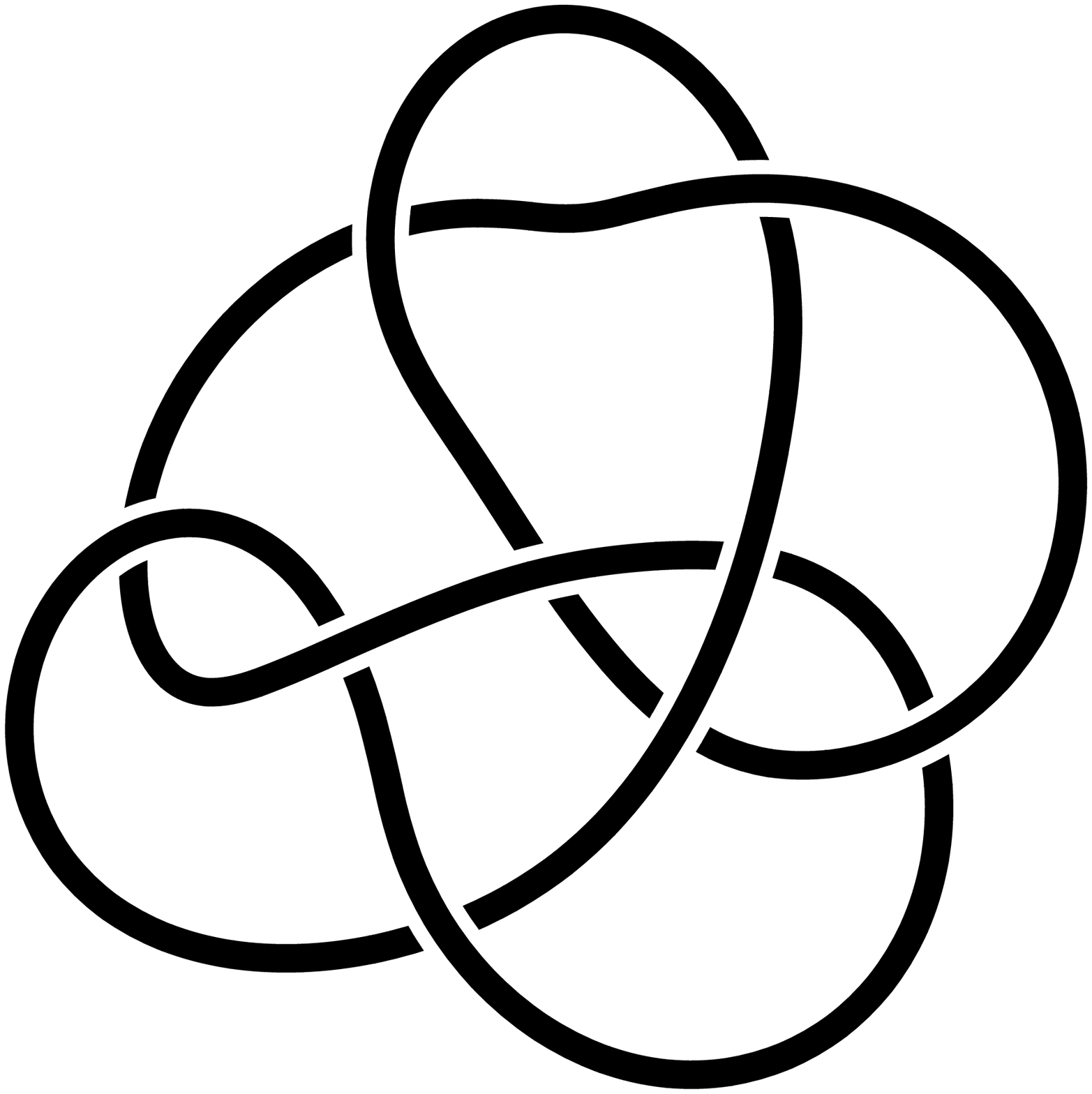}
$9_{43}$&
\includegraphics[keepaspectratio=1,height=2cm]{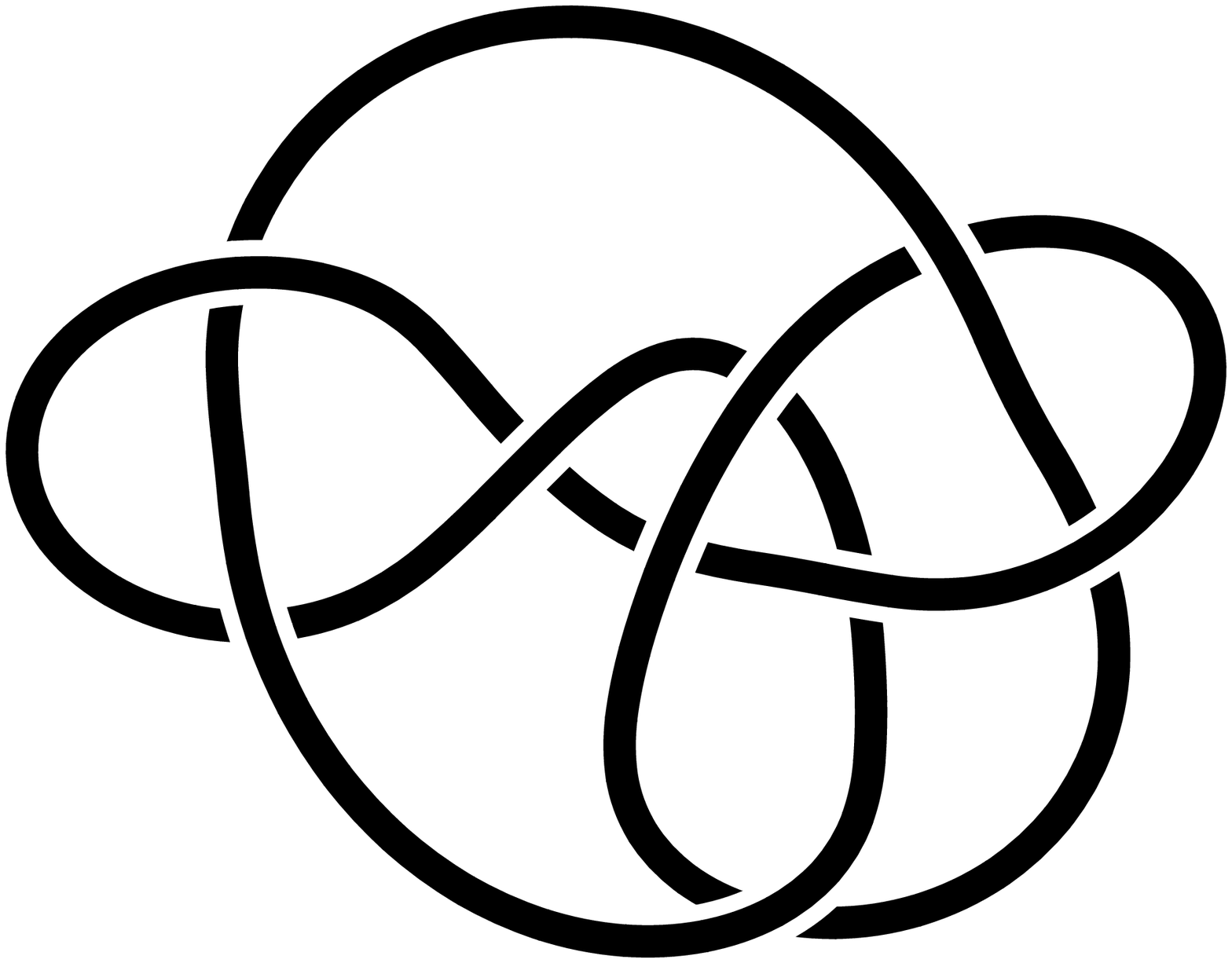}
$9_{44}$&
\includegraphics[keepaspectratio=1,height=2cm]{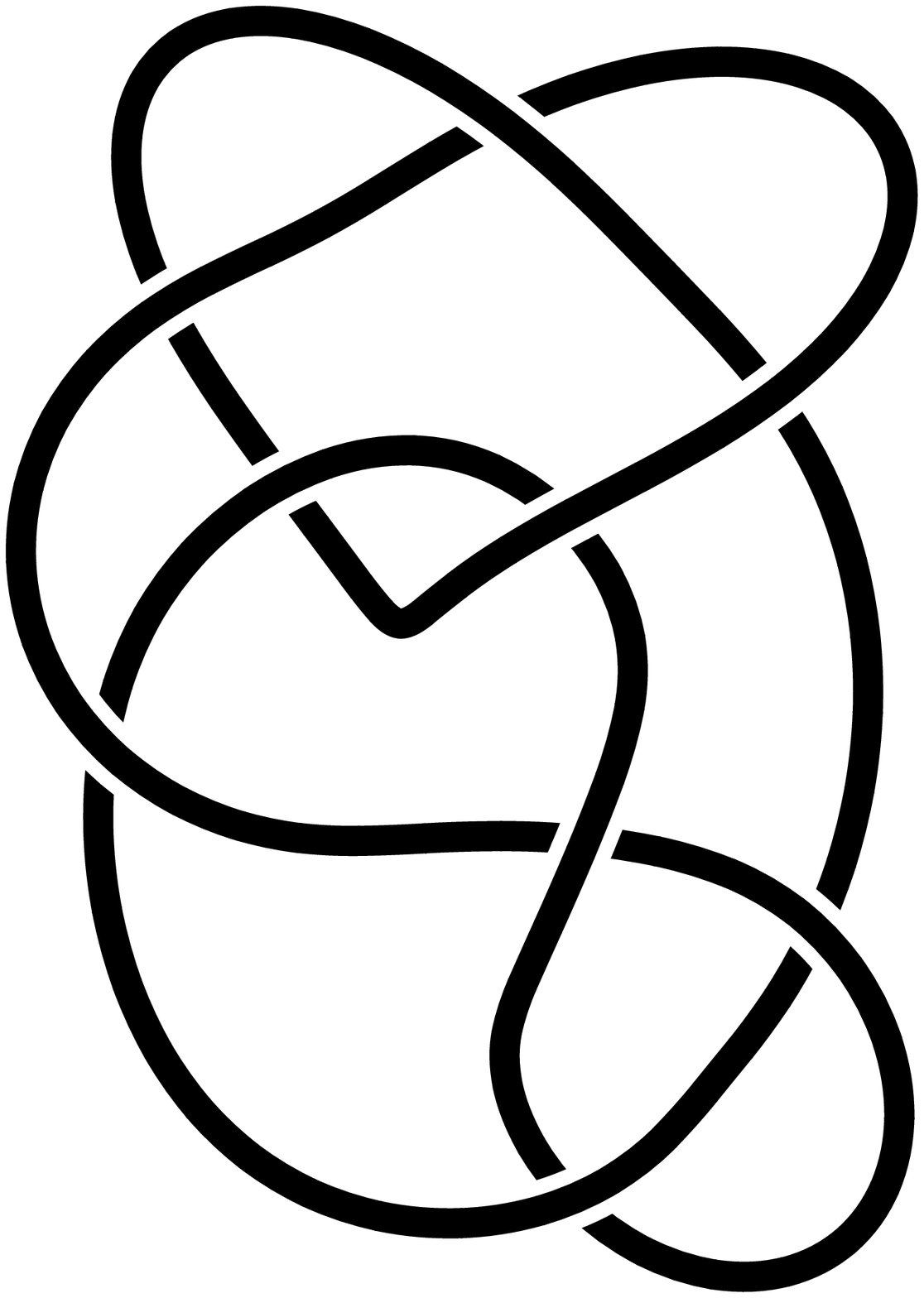}
$9_{45}$\\ \ &&&&\\
\includegraphics[keepaspectratio=1,height=2cm]{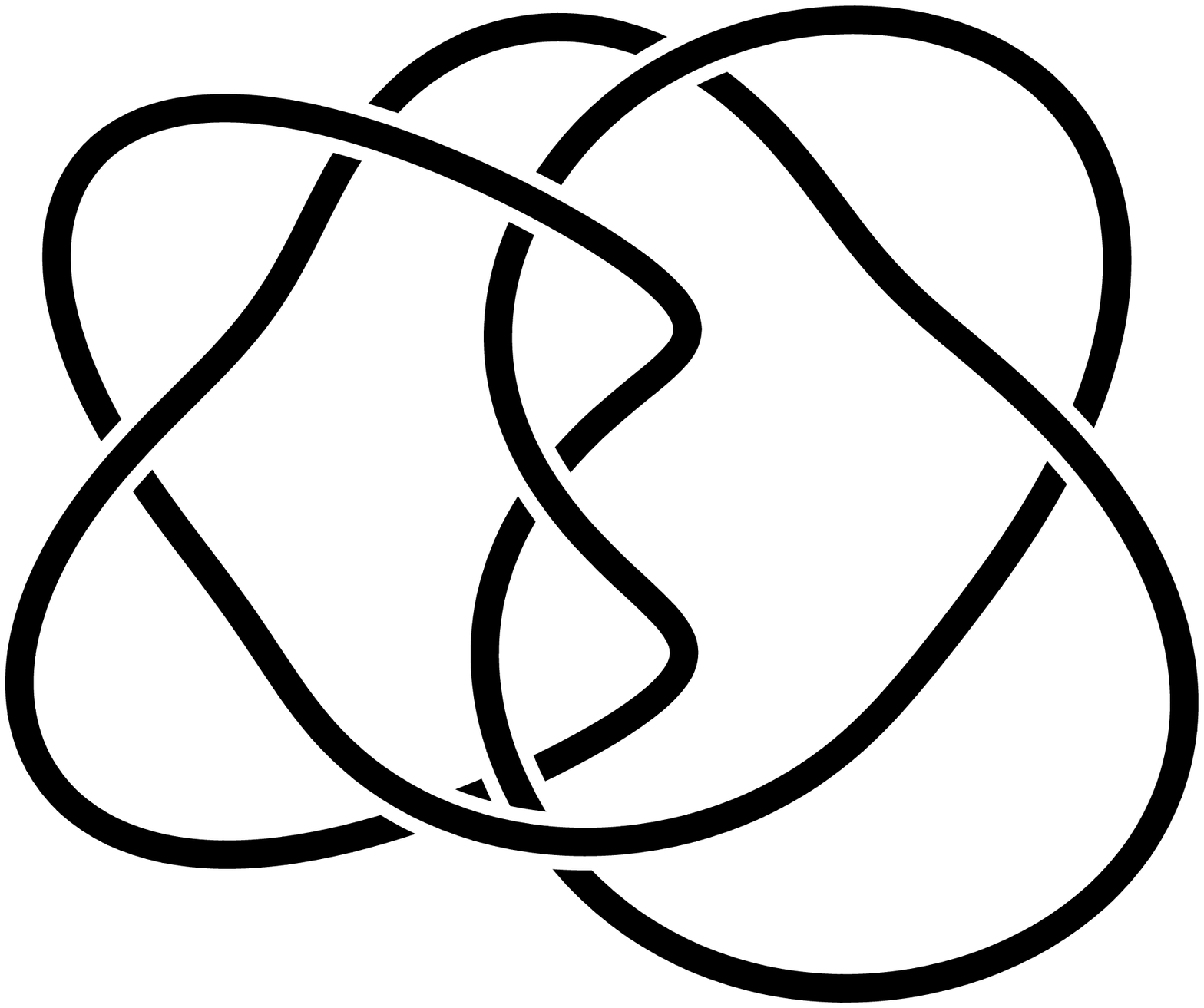}
$9_{46}$&
\includegraphics[keepaspectratio=1,height=2cm]{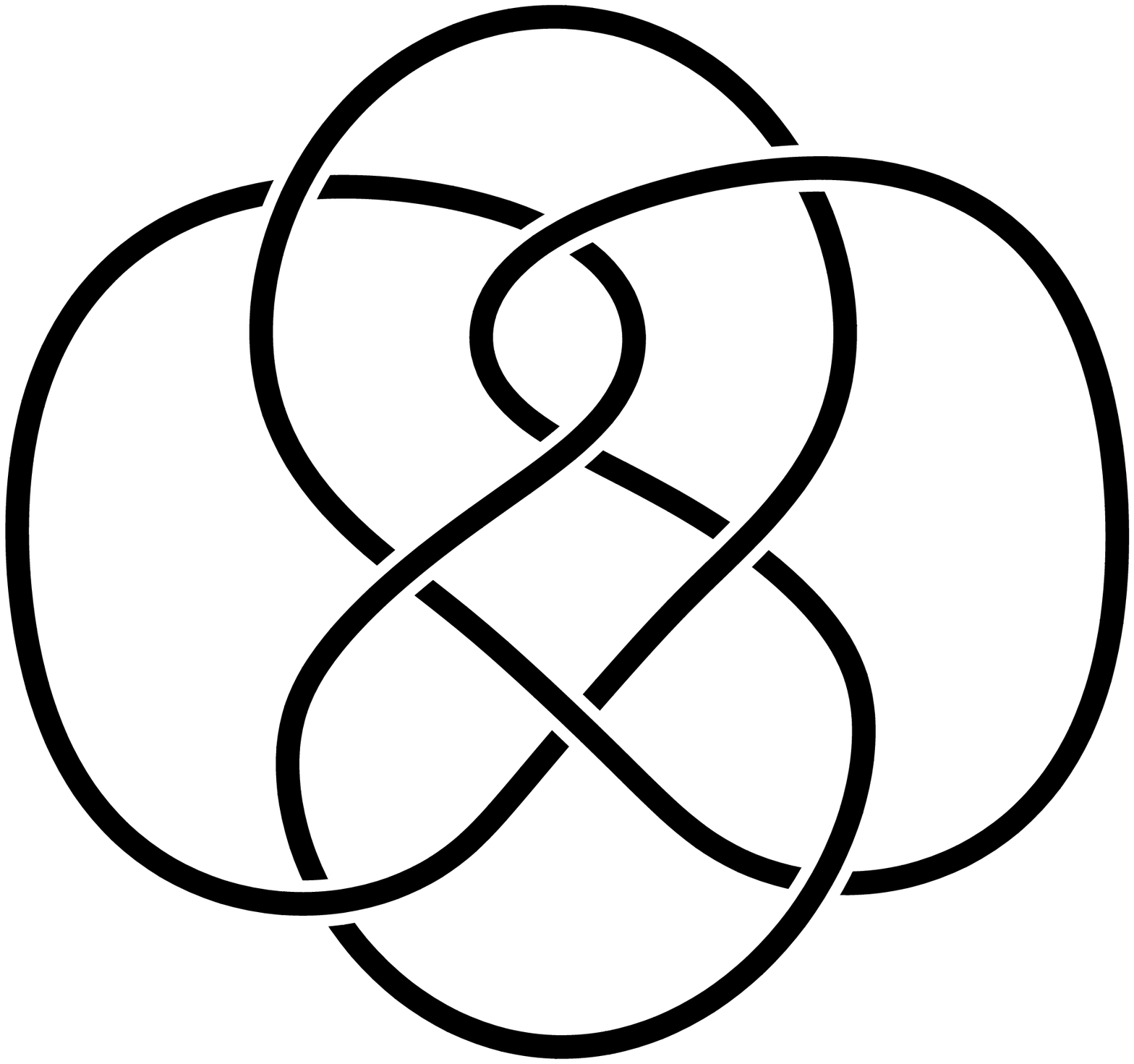}
$9_{47}$&
\includegraphics[keepaspectratio=1,height=2cm]{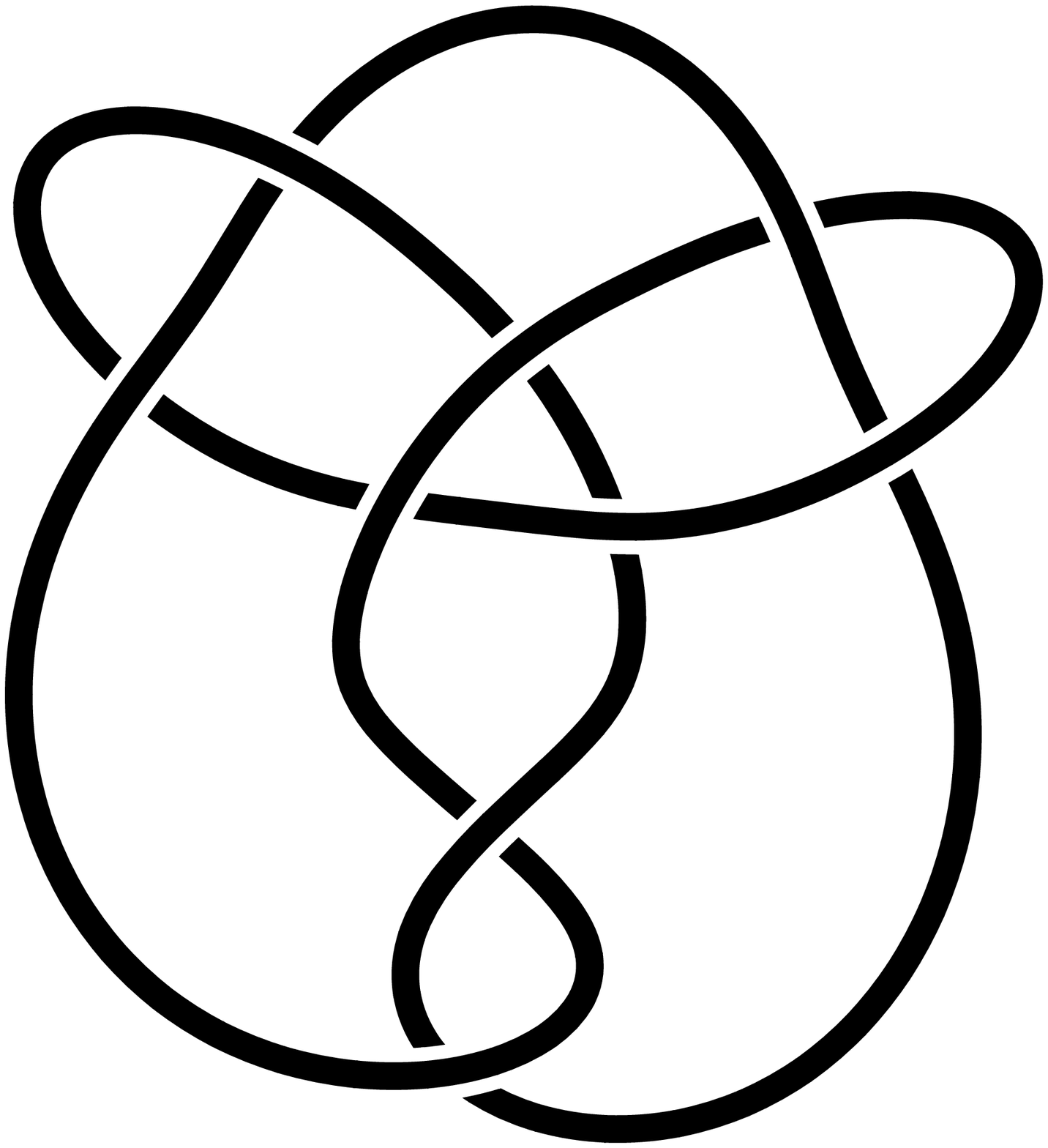}
$9_{48}$&
\includegraphics[keepaspectratio=1,height=2cm]{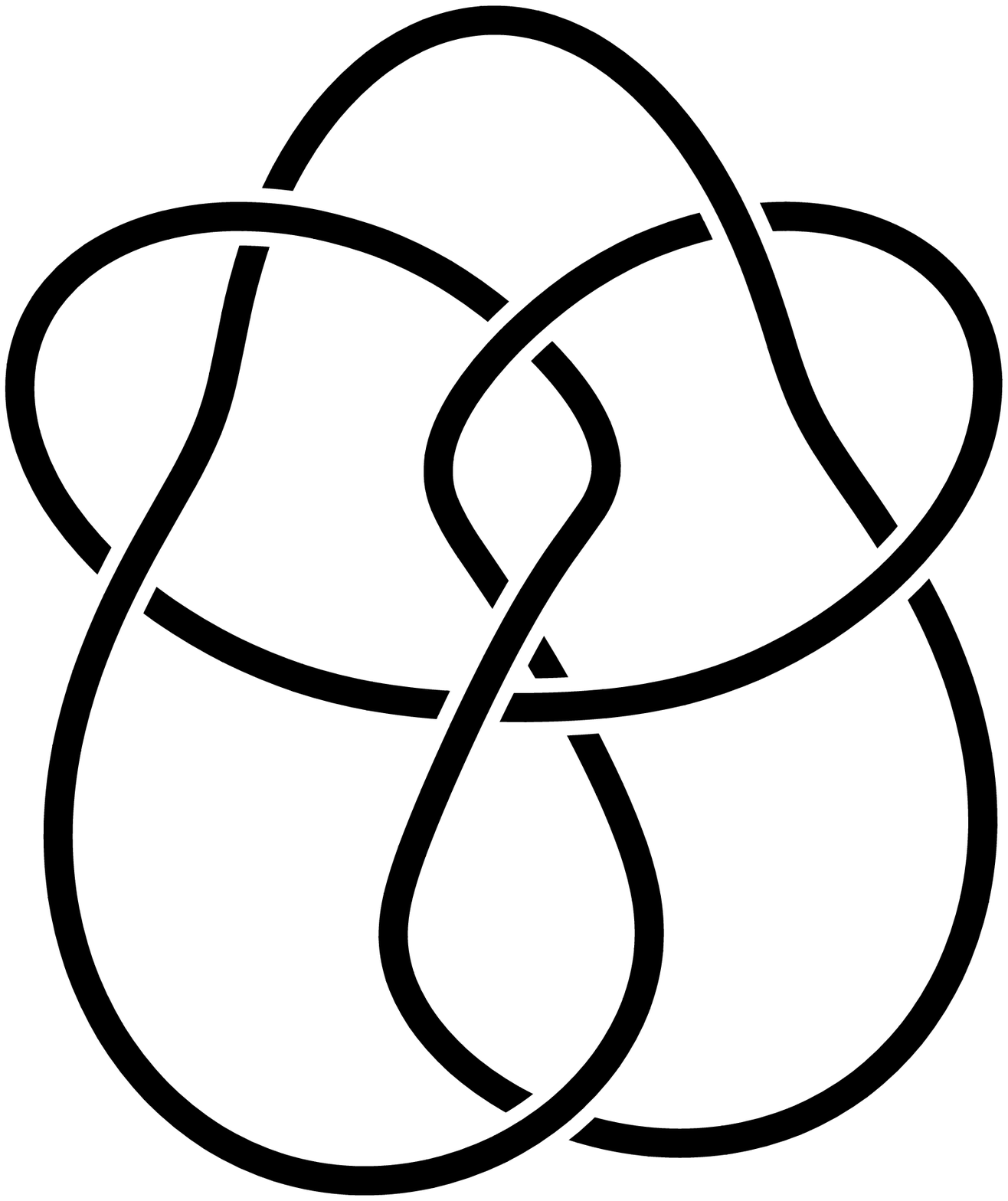}
$9_{49}$&
\end{tabular}

\cleardoublepage
\phantomsection

\end{document}